\documentclass[a4paper,
]{amsart}
\usepackage{amsmath,amsthm,amssymb,mathtools,mathrsfs} 
\usepackage{bm}
\usepackage{tikz-cd}
\usetikzlibrary{positioning}
\usepackage[all]{xy}
\usepackage{soul}
\usepackage{hyperref}
 \definecolor{urlcolor}{rgb}{0,.145,.698}
 \definecolor{linkcolor}{rgb}{.7,0.10,0.2}
 \definecolor{citecolor}{rgb}{.12,.54,.11}
\hypersetup{
 breaklinks=true, 
 colorlinks=true,
 urlcolor=urlcolor,
 linkcolor=linkcolor,
 citecolor=citecolor,
}
\usepackage{comment}
\newtheorem{innercustomthm}{Assumption}
\newenvironment{assumption}[1]
 {\renewcommand\theinnercustomthm{#1}\innercustomthm}
 {\endinnercustomthm}
 
\usepackage{setspace}
\usepackage{xcolor}
\usepackage[shortlabels]{enumitem}
\usepackage[utf8]{inputenc}
\numberwithin{equation}{section}
\usepackage[top=2.5cm,
bottom=2.5cm,
left=2.5cm,
right=2.5cm]{geometry}

\usepackage[numbers]{natbib}
\newtheorem{theorem}{Theorem}[section]
\newtheorem{corollary}[theorem]{Corollary}
\newtheorem{proposition}[theorem]{Proposition}
\newtheorem{proposition-definition}[theorem]{Proposition-Definition}
\newtheorem{lemma}[theorem]{Lemma}

\newtheorem{conj}[theorem]{Conjecture}

\theoremstyle{definition} 
\newtheorem{definition}[theorem]{Definition}

\newtheorem{theorem-definition}[theorem]{Theorem-Definition}

\theoremstyle{remark} 
\newtheorem{remark}[theorem]{Remark}
\newtheorem{example}[theorem]{Example}

\newcommand{\longto}{\longrightarrow} 
\newcommand{\longfrom}{\longleftarrow}
\newcommand{\isoto}{\stackrel{\sim}{\to}}
\newcommand{\longisoto}{\stackrel{\sim}{\longrightarrow}}
\newcommand{\into}{\hookrightarrow}

\renewcommand{\implies}{\Rightarrow}

\let\ul\underline

\newcommand{\abs}[1]{\lvert #1 \rvert}
\renewcommand{\bar}{\overline}
\renewcommand{\emptyset}{\varnothing}
\renewcommand{\geq}{\geqslant}
\renewcommand{\leq}{\leqslant}
\renewcommand{\subset}{\subseteq}
\renewcommand{\setminus}{\smallsetminus}
\renewcommand{\tilde}{\widetilde}
\renewcommand{\vec}{\mathbf}

\DeclareMathOperator{\lmod}{-mod}
\DeclareMathOperator{\real}{re}
\DeclareMathOperator{\iso}{iso}
\DeclareMathOperator{\rat}{\mathbf{rat}}
\DeclareMathOperator{\hyp}{hyp}
\DeclareMathOperator{\imag}{im}
\DeclareMathOperator{\ungr}{ug}

\DeclareMathOperator{\UEA}{\mathbf{U}}

\newcommand{\BC}{\mathbb{C}}
\newcommand{\BD}{\mathbb{D}}

\newcommand{\BG}{\mathbb{G}}

\newcommand{\BL}{\mathbb{L}}

\newcommand{\BN}{\mathbb{N}}

\newcommand{\BZ}{\mathbb{Z}}

\newcommand{\BoA}{\mathbf{A}}
\newcommand{\BoC}{\mathbf{C}}
\newcommand{\BoF}{\mathbf{F}}

\newcommand{\BoN}{\mathbf{N}}
\newcommand{\BoQ}{\mathbf{Q}}
\newcommand{\ulBoQ}{\underline{\mathbf{Q}}}
\newcommand{\BoZ}{\mathbf{Z}}
\newcommand{\BoL}{\mathbf{L}}

\newcommand{\bsX}{\bs{\mathfrak{X}}}

\newcommand{\bsM}{\bs{\mathfrak{M}}}

\newcommand{\bsE}{\bs{\mathfrak{E}}}

\newcommand{\bsF}{\bs{\mathfrak{F}ilt}}

\newcommand{\CA}{\mathcal{A}}
\newcommand{\CB}{\mathcal{B}}
\newcommand{\CC}{\mathcal{C}}
\newcommand{\CD}{\mathcal{D}}
\newcommand{\CE}{\mathcal{E}}
\newcommand{\CF}{\mathcal{F}}
\newcommand{\CG}{\mathcal{G}}
\newcommand{\CH}{\mathcal{H}}
\newcommand{\CI}{\mathcal{I}}

\newcommand{\CK}{\mathcal{K}}
\newcommand{\CL}{\mathcal{L}}
\newcommand{\CM}{\mathcal{M}}
\newcommand{\CN}{\mathcal{N}}
\newcommand{\CO}{\mathcal{O}}
\newcommand{\CP}{\mathcal{P}}
\newcommand{\CQ}{\mathcal{Q}}
\newcommand{\CR}{\mathcal{R}}

\newcommand{\CT}{\mathcal{T}}
\newcommand{\CU}{\mathcal{U}}
\newcommand{\CV}{\mathcal{V}}

\newcommand{\CX}{\mathcal{X}}

\newcommand{\FE}{\mathfrak{E}}

\newcommand{\FL}{\mathfrak{L}}
\newcommand{\FM}{\mathfrak{M}}
\newcommand{\FN}{\mathfrak{N}}

\newcommand{\FP}{\mathfrak{P}}

\newcommand{\FU}{\mathfrak{U}}
\newcommand{\FV}{\mathfrak{V}}

\newcommand{\FX}{\mathfrak{X}}
\newcommand{\FY}{\mathfrak{Y}}

\newcommand{\Filt}{\mathfrak{Filt}}

\newcommand{\SA}{\mathscr{A}}
\newcommand{\SB}{\mathscr{B}}
\newcommand{\SC}{\mathscr{C}}

\newcommand{\SG}{\mathscr{G}}

\newcommand{\Higgs}{\mathbf{Higgs}}
\DeclareMathOperator{\Dol}{Dol}
\DeclareMathOperator{\Tr}{Tr}

\DeclareMathOperator{\BoMo}{BM}
\DeclareMathOperator{\ICA}{IH^*\!}
\DeclareMathOperator{\sta}{Stck}
\DeclareMathOperator{\B}{B\!}
\DeclareMathOperator{\K}{K_0}
\DeclareMathOperator{\Tan}{T}
\DeclareMathOperator{\abso}{abs}
\DeclareMathOperator{\prim}{prim}

\newcommand{\ad}{\operatorname{ad}}
\newcommand{\Aut}{\mathrm{Aut}}
\newcommand{\Alg}{\mathrm{Alg}}
\newcommand{\D}{\mathcal{D}}

\DeclareMathOperator{\dg}{\mathrm{dg}}
\DeclareMathOperator{\bdd}{\mathrm{b}}

\newcommand{\Cohst}{\mathfrak{Coh}}
\newcommand{\Cohsp}{\mathcal{C}\!oh}
\newcommand{\cc}{\mathrm{c}}
\newcommand{\ch}{\mathrm{ch}}
\newcommand{\cl}{\mathrm{cl}}
\newcommand{\coker}{\operatorname{coker}}

\newcommand{\cone}{\mathrm{cone}}
\newcommand{\crit}{\mathrm{crit}}
\newcommand{\cusp}{\mathrm{cusp}}
\newcommand{\dd}{\mathbf{d}}
\newcommand{\ee}{\mathbf{e}}
\newcommand{\mm}{\mathbf{m}}
\newcommand{\dimvec}{\operatorname{\mathbf{dim}}}

\newcommand{\Ext}{\operatorname{Ext}}

\newcommand{\ext}{\operatorname{ext}}
\newcommand{\Free}{\operatorname{Free}}
\newcommand{\GL}{\mathrm{GL}}

\newcommand{\Hom}{\operatorname{Hom}}
\newcommand{\RHom}{\operatorname{RHom}}
\newcommand{\IC}{\mathcal{IC}}
\DeclareMathOperator{\IP}{IP}

\newcommand{\id}{\operatorname{id}}

\newcommand{\JH}{\mathtt{JH}}
\newcommand{\Hit}{\mathtt{h}}
\newcommand{\KK}{\operatorname{K}}
\newcommand{\Lie}{\mathrm{Lie}}

\newcommand{\MHM}{\mathrm{MHM}}

\newcommand{\Mod}{\mathrm{Mod}}
\newcommand{\Msp}{\CM} % Moduli Space
\newcommand{\Mst}{\FM} % Moduli Stack
\newcommand{\Nil}{\mathrm{Nil}}
\newcommand{\nil}{\mathrm{nil}}
\newcommand{\Ob}{\operatorname{Ob}}
\newcommand{\Perf}{\operatorname{Perf}}
\newcommand{\Perv}{\operatorname{Perv}}

\newcommand{\pH}[1]{{^\mathfrak{p}\!\CH^{#1}}} %takes an argument for cohomological degree, if no argument is desired then input "\," as the argument
\newcommand{\pr}{\mathrm{pr}}
\newcommand{\ptau}[1]{{^\mathfrak{p}\!\tau^{#1}}} % ditto
\newcommand{\QuivDim}{\mathrm{QuivDim}}
\newcommand{\Quot}{\operatorname{Quot}}
\newcommand{\rank}{\operatorname{rank}}
\newcommand{\relCoHA}{\mathscr{A}}
\newcommand{\ulrelCoHA}{\underline{\mathscr{A}}}
\newcommand{\relEnv}[1]{\UEA(#1)}

\newcommand{\ulrelBPSLie}[1]{\underline{\BPS}_{#1,\Lie}}
\newcommand{\ulrelBPSalg}[1]{\underline{\BPS}_{#1,\Alg}}
\newcommand{\Rep}{\operatorname{Rep}}
\newcommand{\R}{\mathrm{R}} % right derived

\newcommand{\Spec}{\operatorname{Spec}}
\newcommand{\SSN}{\mathcal{SSN}}

\newcommand{\supp}{\operatorname{supp}}
\newcommand{\todd}{\operatorname{td}}
\newcommand{\Tot}{\operatorname{Tot}}

\newcommand{\vir}{\mathrm{vir}}

\DeclareMathOperator{\Betti}{B}
\DeclareMathOperator{\sst}{sst}
\DeclareMathOperator{\Exp}{Exp}
\DeclareMathOperator{\Id}{Id}
\DeclareMathOperator{\Coh}{Coh}
\newcommand{\Jac}{\operatorname{Jac}}

\newcommand{\BPS}{\mathcal{BPS}} %BPS sheaf
\newcommand{\ulBPS}{\underline{\BPS}}
\DeclareMathOperator{\aBPS}{BPS} %notation for ``absolute'' BPS things (i.e derived global sections of of BPS sheaf things)
\DeclareMathOperator{\Sym}{Sym}
\DeclareMathOperator{\vrank}{vrank}
\newcommand{\pt}{\mathrm{pt}}
\DeclareMathOperator{\HO}{H}
\newcommand{\cms}{/\!\!/}
\newcommand{\bs}[1]{\boldsymbol{#1}}

\title[BPS Lie algebras for totally negative $2$-Calabi--Yau categories]{BPS Lie algebras for totally negative $\mathbf{2}$-Calabi--Yau categories and nonabelian Hodge theory for stacks}
\date{\today}

\author{Ben Davison}
\address{School of Mathematics, University of Edinburgh, Edinburgh, UK}
\email{Ben.Davison@ed.ac.uk}

\author{Lucien Hennecart}
\address{Laboratoire Ami\'enois de Math\'ematique Fondamentale et Appliqu\'ee, CNRS UMR 7352, Universit\'e de Picardie Jules Verne, 33 rue Saint Leu, 80000 Amiens, France}
\email{Lucien.Hennecart@u-picardie.fr}

\author{Sebastian Schlegel Mejia}
\address{Ecole Polytechnique Fédérale de Lausanne (EPFL), CH-1015 Lausanne, Switzerland}
\email{sebastian.schlegelmejia@epfl.ch}
\begin{document}

\maketitle

\begin{abstract}
We define and study a sheaf-theoretic cohomological Hall algebra for suitably geometric Abelian categories $\mathcal{A}$ of homological dimension at most two, and a sheaf-theoretic BPS algebra under the conditions that $\mathcal{A}$ is 2-Calabi--Yau and has a good moduli space.  We define and study the BPS Lie algebra of arbitrary 2-Calabi--Yau categories $\mathcal{A}$ for which the Euler form is negative on all pairs of non-zero objects, which recovers the BPS algebra as its universal enveloping algebra for such ``totally negative'' 2CY categories. We show that for totally negative 2CY categories the BPS algebra is freely generated by intersection complexes of certain coarse moduli spaces as above, and the Borel--Moore homology of the stack of objects in such $\mathcal{A}$ satisfies a Yangian-type PBW theorem for the BPS Lie algebra. In this way we prove the cohomological integrality theorem for these categories.
 
 We use our results to prove that for $C$ a smooth projective curve, and for $r$ and $d$ not necessarily coprime, there is a nonabelian Hodge isomorphism between the Borel--Moore homologies of the stack of rank $r$ and degree $d$ Higgs bundles, and the appropriate stack of twisted representations of the fundamental group of $C$. In addition we prove the Bozec--Schiffmann positivity conjecture for totally negative quivers; we prove that their polynomials counting cuspidal functions in the constructible Hall algebra for $Q$ have positive coefficients, strengthening the positivity theorem for the Kac polynomials of such quivers.
\end{abstract}

\setcounter{tocdepth}{1}

\tableofcontents

\section{Introduction}
\subsection{Background}
2-Calabi--Yau (2CY) categories are central objects of study in geometric representation theory, nonabelian Hodge theory, supersymmetric gauge theory and algebraic geometry. They arise as the categories of representations of preprojective algebras of quivers, Higgs bundles and local systems on smooth projective complex curves, and coherent sheaves on K3 and Abelian surfaces. For $\mathcal{A}$ a 2CY Abelian category belonging to any of these subjects, a central object of study is the moduli stack of objects $\mathfrak{M}_{\mathcal{A}}$. In this paper we study $\HO^{\BoMo}_*(\mathfrak{M}_{\mathcal{A}},\BoQ)$, the Borel--Moore homology of this stack. It encodes various enumerative invariants of the category $\CA$ as well as topological information regarding the stack $\FM_{\CA}$ and its moduli space $\CM_{\CA}$.

In geometric representation theory, this Borel--Moore homology arises as the Hall algebra of raising operators for Nakajima quiver varieties \cite{nakajima1994instantons,nakajima1998quiver,grojnowski1996instantons}, and as such is the target of numerous morphisms from (half) quantum groups \cite{schiffmann2020cohomological,schiffmann2017cohomological}. In nonabelian Hodge theory \cite{hitchin1987self,donaldson1987twisted,corlette1988flat,simpson1992higgs}, for \textit{smooth} moduli spaces, this Borel--Moore homology can be identified with the cohomology of the various moduli spaces appearing on different sides (Betti, de Rham and Dolbeault) of the nonabelian Hodge correspondence. Therefore, the isomorphism between the Borel--Moore homologies of the Betti and Dolbeault stacks follows from the homeomorphism between the moduli spaces. In this paper, we are more interested in the singular and stacky case, in which Borel--Moore homology and cohomology diverge and the connection between the topology of the stacks and that of the moduli spaces is much more involved. The question of the comparison of the \emph{moduli stacks} rather than the moduli spaces was first raised by Simpson \cite[Question p.12]{simpson1996hodge}.

For a $2$CY category $\CA$, we prove general results linking $\HO^{\BoMo}_*(\mathfrak{M}_{\mathcal{A}},\BoQ)$ with the intersection cohomology $\ICA(\CM_{\mathcal{A}},\BoQ)$ of coarse moduli spaces of objects in $\mathcal{A}$, under the assumption that $\mathcal{A}$ is \textit{totally negative}, i.e. under the assumption that the Euler form for $\mathcal{A}$ is strictly negative for pairs of nonzero objects. The cohomological integrality theorem (Corollary~\ref{Y_corollary}) reconstructs the former in terms of the latter and involves the \emph{BPS Lie algebra} of the category, which we define. This generalizes to any $2$CY category the BPS Lie algebra of preprojective algebras of quivers \cite{davison2020bps}.

Totally negative 2CY categories are important in geometry and representation theory. In geometry, these are the categories arising as semistable coherent sheaves on a quasiprojective symplectic surface (K3, Abelian, cotangent bundle of a smooth projective curve of genus $g\geq 2$) for some fixed slope and given polarization \cite{huybrechts2010geometry}. Their moduli spaces of objects have attracted considerable attention. In particular, categories of semistable Higgs bundles of fixed slopes on curves of genus $\geq 2$ are totally negative. This allows us to study the nontrivial cases of the nonabelian Hodge isomorphisms for stacks in \S\ref{NAHT_sec}. We refer to \cite{davison2023nonabelian} for the more elementary cases of genus $0$ and $1$ curves. In representation theory, the relevance of totally negative quivers was highlighted by Lusztig in his study of canonical bases of (generalised) Kac--Moody algebras \cite{lusztig1991quivers,kashiwara1993global}. He noticed in \cite{lusztig1993tight} that these are the quivers for which the canonical basis is solely composed of \emph{monomials}, which is a peculiar and particularly nice behaviour in the theory of canonical bases. Totally negative quivers were also considered in Bozec--Schiffmann's study of the cuspidal functions in the constructible Hall algebras of quivers \cite{bozec2019counting}.

We study the Borel--Moore homology $\HO^{\mathrm{BM}}_{*}(\FM_{\CA},\BoQ)$ by first developing the cohomological Hall algebra (CoHA) structure on this vector space for \textit{general}, sufficiently geometric, categories of homological dimension at most 2, generalising \cite{schiffmann2013cherednik,yang2018cohomological,sala2020cohomological,davison2020bps,kapranov2019cohomological}. We then show that in the totally negative 2CY case, $\HO^{\BoMo}_*(\mathfrak{M}_{\mathcal{A}},\BoQ)$ is a kind of (half) Yangian associated to the free Lie algebra generated by $\ICA(\CM_{\mathcal{A}})$. The precise statement is given by Corollary \ref{Y_corollary} below. A key role is played by our refinement of the CoHA structure on $\HO^{\BoMo}_*(\mathfrak{M}_{\mathcal{A}},\BoQ)$, for $\mathcal{A}$ any suitably geometric Abelian category of projective homological dimension at most 2. This refinement is an algebra object in the category of constructible complexes of sheaves (or even mixed Hodge modules) on the good moduli space of objects in $\mathcal{A}$, which recovers the CoHA structure on $\HO^{\mathrm{BM}}_{*}(\FM_{\CA},\BoQ)$ after taking derived global sections (see \S \ref{subsubsection:introrelativeCoHA}). Working over the good moduli space allows us to use the powerful techniques of mixed Hodge modules, and establish our results via local-global principles relying on the local study of the BPS Lie algebra over the good moduli space.

These results have strong consequences in each of the subjects mentioned above -- representation theory, nonabelian Hodge theory and algebraic geometry. For instance, noting that the various categories associated to a smooth complex projective curve of genus $g\geq 2$ within nonabelian Hodge theory are totally negative, we are able to extend the classical nonabelian Hodge isomorphisms for the cohomology of smooth moduli spaces to an isomorphism of Borel--Moore homology for singular moduli stacks (see Theorem \ref{NAHT_main_thm}). In addition, we prove the Bozec--Schiffmann positivity conjecture \cite{bozec2019counting} for totally negative quivers, stating that the polynomials counting cuspidal functions in the constructible Hall algebra of $Q$ have positive coefficients --- a strengthening of the positivity theorem regarding the Kac polynomials of such quivers, \cite{hausel2013positivity} (see \S \ref{cuspidals_sec} for an exact statement, and the link with Kac polynomials). This follows from applying our main theorem to the totally negative 2CY category $\Rep(\Pi_Q)$ of representations of the preprojective algebra of $Q$.

The categories $\Rep(\Pi_Q)$ for totally negative quivers $Q$ serve as local prototypes for \textit{all} totally negative 2CY categories possessing good moduli spaces \cite{davison2021purity}. Examples include deformed preprojective algebras \cite{crawley1998noncommutative,crawley2022deformed}, multiplicative preprojective algebras \cite{crawley2006multiplicative}, local systems on Riemann surfaces, semistable Higgs bundles of fixed slope, and semistable coherent sheaves of fixed slope on K3 and Abelian surfaces. In this paper we deal with all of these examples, by reducing to the local case of preprojective algebras of quivers, and using our sheaf-theoretic refinement of the CoHA structure on $\HO^{\BoMo}_*(\mathfrak{M}_{\mathcal{A}},\BoQ)$.

\subsubsection{From Bogomol'nyi--Prasad--Sommerfield algebras to Borcherds--Kac--Moody algebras} In 1996 Harvey and Moore \cite{harvey1996algebras} proposed a construction of an algebra of Bogomol'nyi--Prasad--Sommerfield (BPS) states associated to certain $N=2$ $4$-dimensional supersymmetric gauge theories. This proposal in part motivated later work of Kontsevich and Soibelman \cite{kontsevich2010cohomological}, in which they rigorously defined a cohomological Hall algebra associated to any smooth algebra with potential, or noncommutative Landau--Ginzburg model in physics language. Smooth algebras with potential define $3$-Calabi--Yau categories \cite{ginzburg2006calabi}.

As part of the proposal in \cite{harvey1996algebras} a strong connection between algebras of BPS states and Borcherds--Kac--Moody (BKM) \cite{kac1990infinite,borcherds1988generalized} algebras was posited. This connection is perhaps most easy to state (in the Kontsevich--Soibelman programme) in the case of a Dynkin ADE quiver $Q$: in this case one may consider a certain tripled quiver $\tilde{Q}$ with cubic potential $\tilde{W}$, such that the Kontsevich--Soibelman CoHA is the universal enveloping algebra of one half of the current algebra for the associated simple Lie algebra. In general the picture is more complicated and developed in \cite{davison2023bps}, but at least one can say that for a \textit{general} quiver $Q$, the CoHA $\HO^*\!\mathscr{A}_{\tilde{Q},\tilde{W}}$ built via the Kontsevich--Soibelman construction contains the universal enveloping algebra of the Kac--Moody Lie algebra associated to $Q$ \cite{davison2020bps}. We note that via dimensional reduction \cite{davison2017critical,kontsevich2010cohomological}, the Kontsevich--Soibelman CoHA for the pair $(\tilde{Q},\tilde{W})$ is isomorphic to the CoHA of $\Pi_Q$-modules considered in this paper.

In a \textit{generalised} Kac--Moody algebra (a.k.a. BKM algebra) one has imaginary roots, as well as the usual real simple roots of a Kac--Moody Lie algebra. The imaginary roots, in turn, come in two flavours: isotropic and hyperbolic. The salient feature of these roots, for the purposes of this introduction, are that the Lie sub-algebra generated by the positive hyperbolic roots is \textit{free}. In \cite{harvey1998algebras}, a proposal is given for the construction of generalised Kac--Moody Lie algebras in more general geometrically interesting cases than those captured by studying $\Pi_Q$-modules, namely in the study of coherent sheaves on K3 surfaces $S$. Note that by Kinjo's generalisation \cite{kinjo2022dimensional} of the dimensional reduction isomorphism, the Borel--Moore homology of $\mathfrak{Coh}_{p(t)}^{\sst}(S)$ (where $p(t)$ specifies the normalised Hilbert polynomial of semistable sheaves under consideration) should itself be isomorphic to a Kontsevich--Soibelman-style critical cohomological Hall algebra of BPS states for the threefold $S\times\BoA^1$, which has recently been defined in \cite{kinjo2024cohomological}. One conclusion of this paper is that we can indeed rigorously construct generalised Kac--Moody Lie algebras out of CoHAs for very general $2$-Calabi--Yau categories, under a ``total negativity'' assumption that guarantees that all of the simple roots are hyperbolic, so that the expected BKM algebra is free.

\subsubsection{Mixed Hodge structures and refined invariants} The cohomological Hall algebra introduced in \cite{kontsevich2010cohomological}, as well as providing a mathematically rigorous approach to defining the BPS algebra, is a partial categorification of the BPS/DT invariants introduced and studied in \cite{joyce2011theory,kontsevich2008stability}. Donaldson--Thomas (DT) invariants were introduced in \cite{thomas2000holomorphic} via virtual fundamental classes, in order to provide enumerative invariants of moduli spaces of coherent sheaves on Calabi--Yau threefolds. It was subsequently realised \cite{behrend2009donaldson,joyce2011theory,szendroi2008thomas,joyce2007holomorphic,kontsevich2008stability,behrend2013motivic} that DT invariants could be assigned to much more general moduli stacks of objects in 3-Calabi--Yau categories, and that $q$-refinements of these invariants could be defined via the study of mixed Hodge structures on vanishing cycle cohomology, or the Hodge realisation of the motivic vanishing cycle, in the framework of \cite{kontsevich2008stability}. In particular, in order to \textit{define} refined BPS invariants, or Hall algebras partially categorifying them, it is important to work in the category of mixed Hodge structures (or mixed Hodge modules for the sheaf-theoretic CoHA), as opposed to underlying vector spaces or constructible complexes. In particular, the $q$-refinement is defined by replacing Euler characteristics throughout the theory with virtual Poincar\'e polynomials, recording the weights of mixed Hodge structures. For that reason we provide a lift of the CoHA theory to these richer categories. 

An added incentive for working at this level of refinement is that it means that we are able to use our results to study the weight filtration on the Borel--Moore homology of the Betti moduli stack appearing in nonabelian Hodge theory (see \S \ref{NAHT_intro} and \S \ref{NAHT_sec}). In this topic, the weight filtration on the Betti side is in general expected to be the counterpart of the perverse filtration on the Dolbeault side, as shown by the famous $\mathrm{P}=\mathrm{W}$ conjectures for moduli spaces \cite{de2012topology}, proven for smooth moduli spaces in \cite{maulik2022p,hausel2022p}, but still open for singular moduli spaces and stacks.

\subsection{The main results}
We next state our main results in a little more detail.
\subsubsection{Relative construction of the CoHA for 2-dimensional categories}

\label{subsubsection:introrelativeCoHA}

Let $\mathcal{A}$ be a $\BoC$-linear Abelian category of finite length satisfying the assumptions of \S\ref{subsubsection:assumptionCoHAproduct}. In particular, its stack of objects $\mathfrak{M}_{\mathcal{A}}$ is an Artin stack and $\mathcal{C}^{\bullet}$, the RHom complex shifted by one on the product $\mathfrak{M}_{\mathcal{A}}\times\mathfrak{M}_{\mathcal{A}}$, is perfect with tor-amplitude $[-1,1]$. We let $\chi_{\mathcal{A}}\colon \mathfrak{M}_{\mathcal{A}}\times\mathfrak{M}_{\mathcal{A}}\rightarrow \BoZ$ be the locally constant function given by the Euler form: its value on a pair $x,y$ of $\BoC$-points of $\mathfrak{M}_{\mathcal{A}}$ represented by objects $M$ and $N$ of $\mathcal{A}$ is $\chi_{\mathcal{A}}(x,y)=(M,N)_{\mathcal{A}}=-\vrank(\mathcal{C}^{\bullet}_{(x,y)})$, where $(-,-)_{\mathcal{A}}$ denotes the Euler form of the category $\mathcal{A}$ (see \S\ref{subsection:grading}).
It need not be symmetric (for example it will not be if $\mathcal{A}=\mathrm{Coh}(S)$ for $S$ a non-symplectic surface). We let $\pi_0(\mathfrak{M}_{\mathcal{A}})$ be the set of connected components of $\mathfrak{M}_{\mathcal{A}}$. It has a monoid structure induced by taking the direct sum of objects in $\mathcal{A}$, and we may view $\chi_{\mathcal{A}}$ as a bilinear form $\chi_{\mathcal{A}}\colon\pi_0(\mathfrak{M}_{\mathcal{A}})\times\pi_0(\mathfrak{M}_{\mathcal{A}})\rightarrow\BoZ$. 

Let $(\mathcal{M},\oplus)$ be a monoid object in the category of schemes and let $\varpi\colon\mathfrak{M}_{\mathcal{A}}\rightarrow\mathcal{M}$ be a morphism of monoids splitting short exact sequences (see \S\ref{subsubsection:assumptionCoHAproduct} for definitions). In particular, $\varpi(x)=\varpi(y)\oplus \varpi(z)$ if $x$ is a closed $\BoC$-point of $\mathfrak{M}_{\mathcal{A}}$ represented by an object $R$ of $\mathcal{A}$ which is an extension of objects $M,N$ of $\mathcal{A}$ represented by the $\BoC$-points $y,z$ of $\mathfrak{M}_{\mathcal{A}}$ respectively.

We define
\[
 \mathscr{A}_{\varpi}\coloneqq\varpi_*\BD(\BoQ_{\mathfrak{M}_{\mathcal{A}}}[-\chi_{\mathcal{A}}]),
\]
where $\BoQ_{\Mst_{\CA}}$ is the constant sheaf and $\BD$ is the Verdier duality functor and, for $a\in\pi_0(\FM_{\CA})$, $\BoQ_{\FM_{\CA}}[-\chi_{\CA}]_{|\FM_{\CA,a}}=\BoQ_{\FM_{\CA,a}}[-\chi_{\CA}(a,a)]$.  This is a constructible complex on $\mathcal{M}$, locally bounded below, i.e. $\mathscr{A}_{\varpi}\in \Ob(\CD_{\mathrm{c}}^+(\mathcal{M}))$.  We denote $\BoQ_{\mathfrak{M}_{\mathcal{A}}}^{\vir}=\BoQ_{\mathfrak{M}_{\mathcal{A}}}[-\chi_{\mathcal{A}}]$, i.e. the superscript ``vir'' encodes the shift of the constant sheaf by the virtual dimension of the stack. We let $\boxdot$ denote the symmetric monoidal structure on $\CD_{\mathrm{c}}^+(\mathcal{M})$ induced by the monoid structure $\oplus\colon\mathcal{M}\times\mathcal{M}\rightarrow \mathcal{M}$ (see \S\ref{subsection:monoidalstructure}).

We make some assumptions on the stacks that appear in this paper, namely we assume that each connected component of $\FM_{\CA}$ can be written as a global quotient stack. This is a rather mild assumption satisfied by all stacks we consider. Note however that it has been shown recently how one can work with mixed Hodge modules over more general stacks in \cite{tubach2024mixed}. Under these assumptions, we recall in \S \ref{section:perversesheaves} how to upgrade $\mathscr{A}_{\varpi}$ to a mixed Hodge module (MHM) complex
\[
\ulrelCoHA_{\varpi}\coloneqq \varpi_*\BD(\ulBoQ^{\vir}_{\mathfrak{M}_{\mathcal{A}}})
\]
i.e. a complex of mixed Hodge modules that recovers $\mathscr{A}_{\varpi}$ as its underlying constructible complex $\rat(\ulrelCoHA_{\varpi})$.

\begin{theorem}[\S\ref{subsubsection:ThecohaproductKV}]
\label{intro_cohaprod}
 For all $a,b\in\pi_0(\mathfrak{M}_{\mathcal{A}})$, there exist morphisms
 \[
 m_{\varpi,a,b}\colon \mathscr{A}_{\varpi,a}\boxdot\mathscr{A}_{\varpi,b}\rightarrow\mathscr{A}_{\varpi,a+b}[\chi_{\mathcal{A}}(b,a)-\chi_{\mathcal{A}}(a,b)]
 \]
in $\D_{\cc}^{+}(\CM)$ such that, combined together, these morphisms give a shifted multiplication on $\mathscr{A}_{\varpi}$, that is, the structure of a shifted associative algebra object in $\CD_{\mathrm{c}}^+(\mathcal{M})$. If $\chi_{\mathcal{A}}$ is even\footnote{If $\chi_{\CA}$ is not even, a square-root of the Tate twist is required, and so one needs to work in the larger category of \emph{monodromic mixed Hodge modules} as in \cite{davison2020cohomological} for example.}, then the above morphisms can be upgraded to morphisms
\[
 m_{\varpi,a,b}\colon\ulrelCoHA_{\varpi,a}\boxdot\ulrelCoHA_{\varpi,b}\rightarrow\ulrelCoHA_{\varpi,a+b}\otimes\BoL^{\chi_{\mathcal{A}}(a,b)/2-\chi_{\mathcal{A}}(b,a)/2}
\]
where $\BoL\coloneqq \HO_{\cc}^*(\mathbb{A}^1,\BoQ)$, endowing $\ulrelCoHA_{\varpi}$ with a twisted multiplication in $\D^+(\MHM(\CM))$.
\end{theorem}

In the theorem, ``twisted'' refers to the Tate twist appearing on the right-hand-side of the multiplication map, which depends on $a,b$ and vanishes when $\chi_{\CA}$ is symmetric.

By taking derived global sections $\HO^*(\Msp,\mathscr{A}_{\varpi})$, when $\mathcal{A}$ is the category of coherent sheaves on a surface, we recover the cohomological Hall algebra constructed by Kapranov and Vasserot in \cite{kapranov2019cohomological}. By considering $\HO^*(\CM,\ulrelCoHA_{\varpi})$ we lift their algebra to an algebra object in the category of mixed Hodge structures.

When $\mathfrak{M}_{\mathcal{A}}$ has a good moduli space $\JH\colon\mathfrak{M}_{\mathcal{A}}\rightarrow\mathcal{M}_{\mathcal{A}}$, the algebra with shifted multiplication in the derived category of constructible complexes on $\mathcal{M}_{\mathcal{A}}$ thus obtained, $\ulrelCoHA_{\JH}$, is called \emph{the relative cohomological Hall algebra of $\mathcal{A}$}.
When $\CM=\pt$, we recover the \emph{absolute cohomological Hall algebra} of $\mathcal{A}$. When $\mathcal{M}=\pi_0(\mathfrak{M}_{\mathcal{A}})$, and $\varpi$ is the class map $\cl\colon\mathfrak{M}_{\CA}\rightarrow \pi_0(\mathfrak{M}_{\mathcal{A}})$ sending points in $\mathfrak{M}_{\mathcal{A}}$ to their corresponding connected component, we recover the $\pi_0(\mathfrak{M}_{\mathcal{A}})$-grading on the absolute Hall algebra, via the identification between  complexes of constructible sheaves of rational spaces on the discrete space $\pi_0(\mathfrak{M}_{\mathcal{A}})$ and $\pi_0(\mathfrak{M}_{\mathcal{A}})$-graded complexes of $\BoQ$-vector spaces.

When $\mathcal{A}$ is a $2$-Calabi--Yau Abelian category, the Euler form is even and symmetric. Examples include the categories of representations of the derived preprojective algebra of a quiver, of local systems over a Riemann surface of genus $\geq 1$, or of semistable Higgs bundles on a smooth projective curve of genus $\geq 1$. In the 2CY case the multiplication is a morphism of complexes of mixed Hodge modules
\[
 m_{\varpi}\colon\ulrelCoHA_{\varpi}\boxdot\ulrelCoHA_{\varpi}\rightarrow \ulrelCoHA_{\varpi},
\]
and the multiplication respects the cohomological grading (the shift appearing in Theorem~\ref{intro_cohaprod} disappears).

When the category $\mathcal{A}$ considered is that of representations of the preprojective algebra of a quiver, and $\varpi=\JH$ is the affinization morphism, $\ulrelCoHA_{\varpi}$ coincides with the relative CoHA considered in \cite{davison2020bps}. By taking derived global sections, we get back the cohomological Hall algebra of the preprojective algebra constructed in \cite{schiffmann2020cohomological} (see \S\ref{subsection:comparison_preproj}). For semistable Higgs bundles of fixed slope, we recover the semistable cohomological Hall algebra of a smooth projective curve constructed in \cite{minets2020cohomological,sala2020cohomological}. It is explained in detail in Appendix~\ref{section:relativeCoHA} how to compare the multiplications on $\SA_{\varpi}$ defined via the $3$-term RHom complex and in the classical way \cite{davison2020bps}. In each of these cases, we obtain a lift of the cohomological Hall algebra to an algebra object in the category of complexes of mixed Hodge modules on $\CM_{\mathcal{A}}$.

\subsubsection{Relative BPS algebra}
\label{subsubsection:relativeBPSalgebra}
Let $\mathcal{A}$ be a $2$-Calabi--Yau Abelian category and let $\varpi=\JH\colon\mathfrak{M}_{\mathcal{A}}\rightarrow \mathcal{M}_{\mathcal{A}}$ be a good moduli space for the stack $\mathfrak{M}_{\mathcal{A}}$.
The constructible complex $\mathscr{A}_{\JH}$ is concentrated in nonnegative perverse degrees (Lemma~\ref{lemma:nonnegativeperverse}). Therefore, we obtain a multiplication morphism $m_{\JH}$ on the perverse sheaf $\pH{0}(\mathscr{A}_{\JH}) =\ptau{\leq 0}\mathscr{A}_{\JH}\in\Perv(\CM_{\CA})$. We define $\mathcal{BPS}_{\Alg}:=\pH{0}(\mathscr{A}_{\JH})$. The multiplication
\[
 \pH{0}(m_{\JH})={\ptau{\leq 0}m_{\JH}} \colon\mathcal{BPS}_{\Alg}\boxdot\mathcal{BPS}_{\Alg}\rightarrow\mathcal{BPS}_{\Alg},
\]
fits into a commutative square
\begin{equation}
\label{equation:commutation-mult}
 \begin{tikzcd}
	\mathcal{BPS}_{\Alg}\boxdot\mathcal{BPS}_{\Alg} & \mathcal{BPS}_{\Alg} \\
	{\mathscr{A}_{\JH}\boxdot\mathscr{A}_{\JH}} & {\mathscr{A}_{\JH}}
	\arrow["\pH{0}(m_{\JH})", from=1-1, to=1-2]
	\arrow["{m_{\JH}}", from=2-1, to=2-2]
	\arrow[from=1-1, to=2-1]
 \arrow[from=1-2, to=2-2]
\end{tikzcd}
\end{equation}
where the vertical arrows are obtained from the adjunction morphism $\ptau{\leq 0}\rightarrow \id$.

The algebra object $\mathcal{BPS}_{\Alg}$ in the category of perverse sheaves on $\mathcal{M}_{\mathcal{A}}$ is called \emph{the BPS algebra sheaf}. We define the \textit{BPS algebra MHM} $\tau^{\leq 0}\ulrelCoHA_{\JH}$ similarly. It is an algebra object in the category of mixed Hodge modules on $\CM_{\mathcal{A}}$ via the same construction. Applying the derived global sections functor we obtain the \emph{BPS algebra} $\aBPS_{\Alg}\coloneqq\HO^*(\CM_{\mathcal{A}},\BPS_{\Alg})$. By the decomposition theorem for 2CY categories \cite{davison2021purity}, the vertical maps of \eqref{equation:commutation-mult} are split morphisms of constructible sheaves, and so the natural morphism of algebras to the absolute CoHA $\aBPS_{\Alg}\rightarrow \HO^{\BoMo}_*(\mathfrak{M}_{\mathcal{A}},\BoQ^{\vir})$ is injective. In particular, the BPS algebra is a subalgebra of the \emph{absolute} cohomological Hall algebra recalled in \S \ref{subsection:relativecohaproductKV}.

\subsubsection{Freeness of the BPS-algebra for totally negative quivers}
\label{introduction:freenesstotallynegativequiver}
Let $Q=(Q_0,Q_1)$ be a quiver with set of vertices $Q_0$ and set of arrows $Q_1$. We let $\Pi_Q$ be the preprojective algebra of $Q$ (defined in \eqref{preproj_def}). We denote by $\langle-,-\rangle_{Q}$ the Euler form of $Q$ and by $(-,-)_{\mathscr{G}_2(Q)}$ the symmetrization of $\langle-,-\rangle_{Q}$ (see \S \ref{subsection:notationsquivreps} -- as indicated by the notation, it is also the Euler form of the \emph{derived preprojective algebra} $\SG_2(Q)$).
The Euler form and its symmetrization factor through the morphism $\K(\Rep(\Pi_Q))\rightarrow \BoZ^{Q_0}$ taking a module to its dimension vector.
We say that $Q$ is totally negative if for any pair of nonzero $\dd,\ee\in\BoN^{Q_0}$, $(\dd,\ee)_{\mathscr{G}_2(Q)}<0$. This is equivalent to the property that $Q$ has at least two loops at each vertex and at least one arrow between any two distinct vertices. This class of quivers appeared for example in \cite{bozec2019counting}, where the authors exploited the fact that this assumption implies the freeness of a different algebra, the \emph{constructible Hall algebra} of $Q$.

We let $\JH_{\Pi_Q}\colon\mathfrak{M}_{\Pi_Q}\rightarrow\mathcal{M}_{\Pi_Q}$ be the semisimplification map from the stack of representations of $\Pi_Q$ to the coarse moduli space. As in \S \ref{subsubsection:introrelativeCoHA}, we set
\[
 \mathscr{A}_{\JH}\coloneqq \JH_*\BD(\BoQ_{\mathfrak{M}_{\Pi_Q}}[-\chi_{\SG_2(Q)}]),
\]
where we denote by $\chi_{\mathscr{G}_2(Q)}$ the locally constant function taking a point representing a $\dd$-dimensional $\Pi_Q$-module $M$ to $\chi_{\SG_2(Q)}(\dd,\dd)=2\langle
\dd,\dd\rangle_Q$. The Abelian category $\mathcal{A}=\Rep(\Pi_Q)$ and its
stack of objects satisfy the Assumptions
\ref{p_assumption}-\ref{BPS_cat_assumption} in
\S\ref{section:modulistackobjects2CY}, so that we may define the relative
cohomological Hall algebra and the BPS algebra for this category, where the
ambient dg-category is the dg-category of perfect dg-modules over the derived
preprojective algebra $\mathscr{G}_2(Q)$ (see
\S\ref{section:thepreprojective}).

The constructible complex $\mathscr{A}_{\JH}$ carries the multiplication $m\colon\mathscr{A}_{\JH}\boxdot\mathscr{A}_{\JH}\rightarrow\mathscr{A}_{\JH}$ as in \S\ref{subsubsection:introrelativeCoHA}. This multiplication respects the perverse filtration induced by $\JH$. As in \S\ref{subsubsection:introrelativeCoHA} we consider $\pH{0}(\mathscr{A}_{\JH})$, the zeroth perverse cohomology of $\mathscr{A}_{\JH}$. As above (\S\ref{subsubsection:relativeBPSalgebra}), this is an algebra object in the (non-derived) category of perverse sheaves on $\mathcal{M}_{\Pi_Q}$, admitting a morphism to $\mathscr{A}_{\JH}$, in the category of algebra objects of $\CD_{\mathrm{c}}^+(\mathcal{M}_{\mathcal{A}})$.

We let $\Sigma_{\Pi_Q}\subset\BoN^{Q_0}$ be the set of dimension vectors $\dd$ such that $\Pi_Q$ admits a simple representation of dimension vector $\dd$. This set admits a combinatorial description given in \cite{crawley2001geometry}, recalled in \S \ref{CrBo_geom_sec}.  By \cite[Theorem 6.6]{davison2021purity}, for any $\dd\in\Sigma_{\Pi_Q}$, we have a natural monomorphism from the intersection complex $\mathcal{IC}(\mathcal{M}_{\Pi_Q,\dd})\rightarrow \pH{0}(\mathscr{A}_{\JH})$. We let
\[
\Free_{\boxdot-\Alg}\left(\bigoplus_{\dd\in\Sigma_{\Pi_Q}}\mathcal{IC}(\mathcal{M}_{\Pi_Q,\dd})\right)\in \Perv(\mathcal{M}_{\Pi_Q})
\]
be the free algebra object generated by the indicated direct sum, for the monoidal product $\boxdot$ (\S\ref{subsection:monoidalstructure}). By the universal property of free algebras, there is a unique morphism extending the above monomorphisms
\begin{equation}
\label{equation:extensionFreepreproj}
 \Phi_{\Pi_Q}\colon \Free_{\boxdot-\Alg}\left(\bigoplus_{\dd\in\Sigma_{\Pi_Q}}\mathcal{IC}(\mathcal{M}_{\Pi_Q,\dd})\right)\rightarrow \pH{0}(\mathscr{A}_{\JH}).
\end{equation}
The following is the special case $\mathcal{A}=\Rep(\Pi_Q)$ of our first main structural theorem on the relative CoHA (Theorem \ref{theorem:freenesstotneg2CY}).
\begin{theorem}[=Theorem~\ref{theorem:FreeAlgPreProj}]
\label{theorem:freenesspreprojective}
 Let $Q$ be a totally negative quiver. Then, the morphism $\Phi_{\Pi_Q}$ is an isomorphism of algebras in the tensor category of perverse sheaves on $\mathcal{M}_{\Pi_Q}$, with tensor structure defined by $\boxdot$. The lift of $ \Phi_{\Pi_Q}$ to an algebra morphism in the category of $\boxdot$-algebras in MHMs on $\CM_{\Pi_Q}$ is also an isomorphism.
\end{theorem}

We let $\Free_{\Alg}\left(\bigoplus_{\dd\in\Sigma_{\Pi_Q}}\ICA(\mathcal{M}_{\Pi_Q,\dd})\right)$ be the free algebra generated by the intersection cohomology of the moduli space of semisimple representations of $\Pi_Q$.
We define $\mathfrak{P}_{\JH}^0\!\HO^*\!\!\mathscr{A}\coloneqq\HO^*(\mathcal{M}_{\Pi_Q},\ptau{\leq 0}\mathscr{A}_{\JH})$: the zeroth piece of the perverse filtration on $\HO^{\BoMo}_*(\FM_{\Pi_Q},\BoQ^{\vir})$ induced by the morphism $\JH$. By taking the cohomology of the morphism $\Phi_{\Pi_Q}$, we obtain the following corollary.
\begin{corollary}
\label{corollary:BPSalgebraFree}
 Let $Q$ be a totally negative quiver. The morphism 
 \[
 \HO^*(\Phi_{\Pi_Q})\colon \Free_{\Alg}\left(\bigoplus_{\dd\in\Sigma_{\Pi_Q}}\ICA(\mathcal{M}_{\Pi_Q,\dd})\right)\rightarrow \mathfrak{P}_{\JH}^0\!\HO^*\!\!\mathscr{A}
 \]
 is an isomorphism of algebras. The lift of this morphism to a morphism of algebra objects in the category of mixed Hodge structures is also an isomorphism.
\end{corollary}

\subsubsection{Freeness of the BPS-algebra for totally negative 2CY categories}
\label{subsubsection:freenesstotallynegative}
Theorem \ref{theorem:freenesspreprojective} can be generalised to more general $2$-Calabi--Yau categories satisfying some assumptions ensuring good geometric behaviour. We let $\JH\colon\mathfrak{M}_{\mathcal{A}}\rightarrow\mathcal{M}_{\CA}$ be a good moduli space of objects in a $2$-Calabi--Yau Abelian category $\mathcal{A}$. We denote by $\chi_{\CA}$ the Euler form of $\CA$. We refer to \S \ref{section:modulistackobjects2CY} for the precise hypotheses needed.
We let $\mathscr{A}_{\JH}=\JH_*\BD(\BoQ_{\mathfrak{M}_{\CA}}[-\chi_{\CA}])$. This object is a complex of constructible sheaves on $\mathcal{M}_{\CA}$. The object $\mathscr{A}_{\JH}$ has a degree zero multiplication map since for such categories, the Euler form is symmetric (\S \ref{subsubsection:introrelativeCoHA}). It is graded by $\pi_0(\mathfrak{M}_{\mathcal{A}})$, the monoid of connected components of $\mathfrak{M}_{\mathcal{A}}$ (or $\mathcal{M}_{\CA}$). For $a\in\pi_0(\mathfrak{M}_{\mathcal{A}})$, $\JH_a\colon\mathfrak{M}_{\CA,a}\rightarrow\mathcal{M}_{\CA,a}$ denotes the restriction of $\JH$ to the $a$th connected component.

The constructible complex $\mathscr{A}_{\JH}=\JH_*\BD(\BoQ_{\mathfrak{M}_{\CA}}[-\chi])$ is an algebra in the monoidal category of constructible complexes on $\mathcal{M}_\CA$ and $\pH{0}(\mathscr{A}_{\JH})$ is an algebra for the multiplication $\pH{0}(m)$; see \S \ref{subsubsection:ThecohaproductKV} for details of the construction. We let $\Sigma_{\mathcal{A}}\subset \pi_0(\mathfrak{M}_{\mathcal{A}})$ be the set of elements $a$ of $\pi_0(\Mst_\CA)$ for which $\JH_a$ is a $\mathbf{G}_m$-gerbe over a nonempty subset of $\mathcal{M}_{\CA,a}$. A class $a\in\pi_0(\mathfrak{M}_{\mathcal{A}})$ is in $\Sigma_{\mathcal{A}}$ if and only if $\mathcal{A}$ has a simple object of class $a$. By \cite{davison2021purity}, for $a\in\Sigma_{\mathcal{A}}$, we have a canonical monomorphism of perverse sheaves
\[
 \mathcal{IC}(\mathcal{M}_{\CA,a})\hookrightarrow \pH{0}(\mathscr{A}_{\JH})= \BPS_{\Alg}.
\]
This monomorphism together with the multiplication $\pH{0}(m)$ induces a morphism of algebra objects in the tensor category of perverse sheaves on $\mathcal{M}_{\CA}$:
\begin{equation}
 \Phi_{\mathcal{A}}\colon \Free_{\boxdot-\Alg}\left(\bigoplus_{a\in\Sigma_{\mathcal{A}}}\mathcal{IC}(\mathcal{M}_{\CA,a})\right)\rightarrow \BPS_{\Alg},
\end{equation}
by the universal property of free algebras.

The following theorem is the generalization of Theorem~\ref{intro_cohaprod} to any $2$CY category $\CA$ satisfying the assumptions of \S\ref{section:modulistackobjects2CY} and \S\ref{subsubsection:ThecohaproductKV}.

\begin{theorem}[=Theorem~\ref{theorem:FreeAlg2CY}]
\label{theorem:freenesstotneg2CY}
 If $\mathcal{A}$ is a totally negative 2CY category, then the morphism $\Phi_{\mathcal{A}}$ is an isomorphism of algebras (in the tensor category of perverse sheaves on $\mathcal{M}_{\CA}$). The natural upgrade of this statement to the category $\MHM(\CM_{\mathcal{A}})$ of mixed Hodge modules also holds.
\end{theorem}
Taking derived global sections, the analogue of Corollary \ref{corollary:BPSalgebraFree} also holds.

\subsubsection{PBW isomorphism for the CoHA of a totally negative 2CY category}
\label{subsubsection:cohatotnegative}
Let $\mathcal{A}$ be a totally negative $2$-Calabi--Yau Abelian category admitting a stack of objects together with a good moduli space $\JH\colon\mathfrak{M}_{\mathcal{A}}\rightarrow \mathcal{M}_{\CA}$ satisfying the Assumptions \ref{p_assumption}-\ref{assumption:associativity}.
It will prove useful to twist the multiplication on $\ulrelCoHA_{\varpi}$ by a sign, in the 2CY case. Let $\KK(\Mst_{\CA})$ be the groupification of the monoid $\pi_0(\Mst_{\CA})$. Firstly, by the 2CY property and our assumptions, the Euler form on the category $\mathcal{A}$ descends to a symmetric bilinear form $\chi$ on $\KK(\Mst_{\CA})$, for which $\chi(a,a)$ is even for all $a$. As such we can find a bilinear form $\psi$ on $\KK(\Mst_{\CA})$ such that $\chi$ is the symmetrisation of $\psi$, modulo two, i.e. 
\begin{align}
\label{first_psi}
\chi(a,b)=\psi(a,b)+\psi(b,a)&\quad\quad\pmod{2}.
\end{align}
The existence of $\psi$ is straightforward as it suffices to work with the $\BoF_2$-vector space $\KK(\FM_{\CA})\otimes_{\BoZ}\BoF_2$. Then for $a,b\in \pi_0(\FM_{\CA})$ we define $m^{\psi}_{\JH,a,b}=(-1)^{\psi(a,b)}m_{\JH,a,b}$. We denote by $\ulrelCoHA_{\JH}^{\psi}$ the mixed Hodge module complex $\JH_*\BD\ulBoQ^{\vir}_{\FM_{\CA}}$ with the associative product twisted by these signs.

We define
\begin{align*}
\ulBPS_{\Lie}\coloneqq &\Free_{\boxdot-\Lie}\left(\bigoplus_{a\in\Sigma_{\mathcal{A}}}\ul{\mathcal{IC}}(\mathcal{M}_{\CA,a})\right)\in\MHM(\CM_{\CA}),&
\aBPS_{\Lie} \coloneqq&\Free_{\Lie}\left(\bigoplus_{a\in\Sigma_{\mathcal{A}}}\ICA(\mathcal{M}_{\CA,a})\right).
\end{align*}
We define $\ICA(\mathcal{M}_{\CA,a})=\HO(\mathcal{M}_{\CA,a},\ul{\mathcal{IC}}(\mathcal{M}_{\CA,a}))$, the (shifted) intersection cohomology of the possibly singular variety $\mathcal{M}_{\CA,a}$. Using the algebra structure on $\ulrelCoHA^{\psi}_{\JH}$ (which induces a Lie algebra structure by taking the commutator Lie bracket), the morphism $\ul{\mathcal{IC}}(\mathcal{M}_{\CA,a})\rightarrow \ulrelCoHA^{\psi}_{\JH}$ induces a morphism $\ulBPS_{\Lie}\rightarrow \ulrelCoHA^{\psi}_{\JH}$. Since the relative cohomological Hall algebra possesses a $\HO^*(\B\BoC^*,\BoQ)$-action coming from a choice of positive determinant line bundle on $\FM_{\CA}$ (Assumption~\ref{det_bun_assumption}) we obtain a morphism
\[
\ulBPS_{\Lie}\otimes \HO^*(\B\BoC^*,\BoQ)\rightarrow \ulrelCoHA^{\psi}_{\JH}.
\]
The algebra structure on $\ulrelCoHA^{\psi}_{\JH}$ provides us now with a canonical morphism of complexes of mixed Hodge modules on $\mathcal{M}_{\CA}$:
\[
 \tilde{\Phi}^{\psi}\colon \Sym_{\boxdot}\left(\ulBPS_{\Lie}\otimes \HO^*(\B\BoC^*,\BoQ)\right)\rightarrow \ulrelCoHA^{\psi}_{\JH}.
\]

\begin{theorem}[see~\S\ref{subsection:PBW-2CY}]
\label{theorem:pbwtotnegative2CY}
 The morphism $\tilde{\Phi}^{\psi}$ is an isomorphism of mixed Hodge module complexes on $\mathcal{M}$ (but not of algebra objects in general).
\end{theorem}
Taking derived global sections, we deduce the following corollary:
\begin{corollary}
\label{Y_corollary}
The PBW morphism
\[
\Sym\left(\aBPS_{\Lie}\otimes \HO^*\!(\B\BoC^*,\BoQ)\right)\rightarrow \HO^*(\CM_{\CA},\mathscr{A}^{\psi}_{\JH})=\HO^{\BoMo}_*(\mathfrak{M}_{\mathcal{A}},\BoQ^{\vir})\eqqcolon\HO^*\!\!\mathscr{A}^{\psi}
\]
provided by the absolute CoHA multiplication on $\HO^*\!\!\mathscr{A}^{\psi}$ is an isomorphism of $\pi_0(\FM_{\CA})$-graded, cohomologically graded vector spaces. The same statement holds at the level of cohomologically graded mixed Hodge structures.
\end{corollary}
At the level of underlying mixed Hodge structures, there is no difference between $\HO^*\!\!\mathscr{A}$ and $\HO^*\!\!\mathscr{A}^{\psi}$ (the difference only regards the multiplications). In particular, we deduce that the \textit{cohomological integrality conjecture} (which merely states that there is \textit{some} isomorphism of cohomologically graded mixed Hodge structures (not of algebra objects) $\HO^*\!\!\mathscr{A}\cong \Sym(V\otimes\HO(\B\BoC^*,\BoQ))$ with $V$ a $\pi_0(\FM_{\mathcal{A}})$-graded mixed Hodge structure, with finite-dimensional graded pieces) is true for $\HO^*\!\!\mathscr{A}$.

\subsubsection{Nonabelian Hodge theory for stacks}
\label{NAHT_intro}
One of our main applications of the study of cohomological Hall algebras of 2CY categories is an extension of nonabelian Hodge isomorphisms for curves to the Borel--Moore homology of singular stacks, which we now explain. Let $(r,d)\in\BoZ^2$ with $r>0$. We do not assume that $r$ and $d$ are coprime, and we may even allow $d=0$. Let $C$ be a smooth projective curve of genus $g$. Let $\mathfrak{M}_{r,d}^{\Dol}(C)$ be the Dolbeault moduli stack, that is the stack of rank $r$ and degree $d$ semistable Higgs bundles on $C$. We let $\mathcal{M}_{r,d}^{\Dol}(C)$ be the Dolbeault coarse moduli space. We have a morphism $p_{\Dol}\colon \mathfrak{M}_{r,d}^{\Dol}(C)\rightarrow\mathcal{M}_{r,d}^{\Dol}(C)$ taking a semistable Higgs bundle to the associated polystable Higgs bundle.

On the Betti side, let $p\in C$ be a fixed point. Let $\zeta_r$ be a fixed primitive $r$-th root of unity. Consider $\mathfrak{M}_{g,r,d}^{\Betti}$ the moduli stack of local systems on $C\setminus \{p\}$ whose monodromy around $p$ is given by multiplication by $\zeta_r^d$ (the Betti moduli stack). Let $\mathcal{M}_{g,r,d}^{\Betti}$ be the Betti coarse moduli space. We have the semisimplification map $p_{\Betti}\colon\mathfrak{M}_{g,r,d}^{\Betti}\rightarrow\mathcal{M}_{g,r,d}^{\Betti}$.

Classical nonabelian Hodge theory gives a homeomorphism (actually a real-analytic isomorphism)
\[
\Psi_{r,d}\colon \mathcal{M}_{r,d}^{\Dol}(C)\rightarrow\mathcal{M}_{g,r,d}^{\Betti}.
\]
\begin{theorem}[=Theorem~\ref{theorem:NAHT-body}]
\label{NAHT_main_thm}
We have a natural isomorphism
\[
 (\Psi_{r,d})_*(p_{\Dol})_*\BD\BoQ_{\mathfrak{M}_{r,d}^{\Dol}(C)}^{\vir}\cong (p_{\Betti})_*\BD\BoQ_{\mathfrak{M}_{g,r,d}^{\Betti}}^{\vir}
\]
in the constructible derived category $\mathcal{D}^+_{\mathrm{c}}(\mathcal{M}_{r,d}^{\Betti})$, and therefore, taking derived global sections and shifting, a natural isomorphism
\[
 \HO^{\BoMo}_*(\mathfrak{M}_{r,d}^{\Dol}(C),\BoQ)\cong \HO^{\BoMo}_*(\mathfrak{M}_{g,r,d}^{\Betti},\BoQ).
\]
\end{theorem}
For $g\leq 1$ this result is contained in \cite{davison2023nonabelian}, while for higher genera it is a consequence of our results on totally negative 2CY categories, along with the classical nonabelian Hodge homeomorphism. Note that, unlike the theorems in previous sections, there is no hope of upgrading this statement to an isomorphism of mixed Hodge module complexes, since $\Psi_{r,d}$ is (famously) not a morphism of complex algebraic varieties. In particular, the natural mixed Hodge structure on the intersection cohomology of the domain of $\Psi_{r,d}$ is pure, while the mixed Hodge structure on the intersection cohomology of the target is not. Furthermore, the mixed Hodge structure on $\HO^{\BoMo}_\ast (\Mst_{r,d}^{\Dol}(C),\BoQ)$ is pure by \cite[Theorem 7.22]{davison2021purity}, while the mixed Hodge structure on $\HO^{\BoMo}_{\ast} (\Mst_{g,r,d}^{\Betti},\BoQ)$ is not.

In the case in which $r$ and $d$ are coprime Theorem \ref{NAHT_main_thm} is a consequence of the existence of the homeomorphism $\Psi_{r,d}$, and the fact that the respective moduli stacks are $\BoC^*$-gerbes over their respective moduli schemes, and we have isomorphisms
\[
\HO^{\BoMo}_*(\mathfrak{M}_{r,d}^{\bullet},\BoQ)\cong \HO^{\BoMo}_*(\CM_{r,d}^{\bullet},\BoQ)\otimes\HO^*(\B \BoC^*,\BoQ)
\]
for $\bullet\in \{\Betti,\Dol\}$, abbreviating $\FM_{r,d}^{\Dol}=\FM_{r,d}^{\Dol}(C)$ and $\FM_{r,d}^{\Betti}=\FM_{g,r,d}^{\Betti}$. The non-coprime case is totally different: although classical nonabelian Hodge theory still gives a homeomorphism $\Psi_{r,d}$, the respective morphisms from the stacks to the moduli schemes are much more complicated. At the level of the stacks themselves, it seems likely that there is \textit{no} homeomorphism simply to apply the Borel--Moore homology functor to, see \cite[page 38]{simpson1994moduli}. Theorem~\ref{NAHT_main_thm} raises the question of whether the algebra structures on both sides coincide. This has been positively answered by the second author in \cite{hennecart2023nonabelian}, where the de Rham moduli space makes the bridge between the Betti and Dolbeault CoHAs via the Hodge--Deligne moduli space and the Riemann--Hilbert correspondence respectively.

\subsection{Acknowledgements}
All three authors were supported by the European Research Council starting grant “Categorified Donaldson--Thomas theory” No. 759967. BD was in addition supported by a Royal Society University Research Fellowship. SSM was in addition supported by the Max Planck Institute for Mathematics and the Swiss National Science Foundation [no. 218340]. We would like to thank Camilla Felisetti, Michael Groechenig, Shivang Jindal, Tasuki Kinjo, Davesh Maulik, Andrei Okounkov, Olivier Schiffmann and Yan Soibelman for useful conversations.

\subsection{Conventions and notations}
\begin{itemize}
\item
We write $\HO_{\BoC^{\ast}} \coloneqq \HO^*_{\BoC^{\ast}}(\pt,\BoQ) = \HO^*\!(\B\BoC^{\ast},\BoQ)$ for the $\BoC^*$-equivariant cohomology of the point. It is isomorphic to $\BoQ[x]$, a polynomial ring in one variable of degree $\deg(x)=2$.
\item
If an Abelian or triangulated $k$-linear category $\mathcal{A}$ is fixed, for objects $M,N$ of $\mathcal{A}$ we denote by $(M,N)_{\mathcal{A}}\coloneqq\sum_{i\in\BoZ}(-1)^i \dim_k(\Ext^i(M,N))$ the Euler form.
\item
We define $\BoN=\BoZ_{\geq 0}$.
\item
If $\CF$ is a complex of sheaves or mixed Hodge modules on a variety or stack $X$, we will often abbreviate the derived global sections functor, writing $\HO^*\!\CF\coloneqq \HO^*(X,\CF)$. Similarly, we write $\HO^i\!\CF\coloneqq \HO^i(X,\CF)$.
\item
If $A=kQ/\langle R\rangle$ is the free path algebra of a quiver $Q$ modulo relations $R$ contained in the two-sided ideal generated by the arrows, and $i$ is a vertex of $Q$, we denote by $S_i$ the $A$-module with dimension vector $1_i$, for which all of the arrows act via multiplication by zero. If $\dd=(d_i)_{i\in Q_0}, \in \BoN^{Q_0}$ we define $S_{\dd}\coloneqq \bigoplus_{i\in Q_0} S_i^{\oplus d_i}$.
\item If $\spadesuit$ is an object that depends on a category $\CA$, then we may sometimes make explicit its dependence by a subscript $\spadesuit_{\CA}$. When $\CA = \Rep(\Pi_Q)$ is the category of representations of the preprojective algebra, then we write $\spadesuit_{\Pi_Q}$ instead of $\spadesuit_{\Rep(\Pi_Q)}$.
We write $\relCoHA_{\CA} = \relCoHA_{\JH} = \relCoHA_{\JH_{\CA}}$, where $\JH_{\CA}\colon\FM_{\CA}\rightarrow\CM_{\CA}$ denotes the good moduli space morphism of the stack of objects in $\CA$.
\item All functors are derived.
\item We write dga for differential graded algebra, and cdga for commutative differential graded algebra.
\item All stacks will be classical stacks, unless explicitly stated otherwise. Where a stack appears in bold script (e.g. $\bsX$) it is a possibly derived stack, where it appears in normal script (e.g. $\FX=t_0(\bsX)$) it is a classical stack. 
\item Quotient stacks are denoted without brackets. GIT quotients are denoted with $\cms$.
\end{itemize}

\section{The preprojective algebra of a quiver}
\label{section:thepreprojective}
The preprojective algebra of a quiver, introduced by Gel'fand and Ponomarev in \cite{gel1979model}, and its category of representations play a central role in this paper. In this section we recall constructions associated to this category which are used throughout the paper.

\subsection{The preprojective algebra}
\label{subsection:preprojectivealgebra}

Let $Q=(Q_0,Q_1,s,t)$ be a quiver, i.e. a set $Q_0$ of vertices, and a set $Q_1$ of arrows, along with a pair of morphisms $s,t\colon Q_1\rightarrow Q_0$ taking arrows to their sources and targets respectively. Let $\overline{Q}=(Q_0,\overline{Q}_1,\overline{s},\overline{t})$ be the double of $Q$. It has the same set of vertices as $Q$ and $\overline{Q}_1=Q_1\sqcup Q_1^*$, where $Q_1^*$ is such that $Q^*=(Q_0,Q_1^*)$ is the opposite quiver of $Q$. If $\alpha\in Q_1$, the corresponding opposite arrow is denoted $\alpha^*\in Q_1^*$. For a field $k$ we denote by $k Q$ the path algebra of $Q$. A basis of $kQ$ is given by the paths in $Q$ (including paths of length zero at each $i\in Q_0$, denoted $e_i$), with multiplication given by concatenation of paths. We let $\Pi_Q$ be the \emph{preprojective algebra} of $Q$:
\begin{equation}
\label{preproj_def}
 \Pi_Q\coloneqq\BoC\overline{Q}/\langle\rho\rangle,
\end{equation}
where $\rho=\sum_{\alpha\in Q_1}[\alpha,\alpha^*]$ is the \emph{preprojective relation}.

For a dimension vector $\dd\in\BoN^{Q_0}$, we let $X_{Q,\dd}\coloneqq\bigoplus_{i\xrightarrow{\alpha}j\in Q_1}\Hom(\BoC^{d_i},\BoC^{d_j})$ be the \emph{representation space} of $\BoC Q$. It is acted on by the product of linear groups $\GL_{\dd}\coloneqq \prod_{i\in Q_0}\GL_{d_i}$ by conjugation at the vertices. The action of $\GL_{\dd}$ on $\Tan^*\!X_{Q,\dd}\cong X_{\overline{Q},\dd}\coloneqq \bigoplus_{i\xrightarrow{\alpha}j\in \overline{Q}_1}\Hom(\BoC^{d_i},\BoC^{d_j})$ is Hamiltonian and the quadratic moment map is

\begin{equation*}
\mu_{\vec{d}} \colon \Tan^*\!X_{\dd} \longrightarrow \mathfrak{gl}_{\dd}, \quad\quad
(x,x^*) \longmapsto \sum_{\alpha\in Q_1}[x_{\alpha},x_{\alpha^*}].
\end{equation*}
It is precisely given by evaluation of the preprojective relation $\rho$ on representations. The stack of $\dd$-dimensional representations of $\Pi_Q$ is $\mathfrak{M}_{\Pi_Q,\dd}\simeq \mu_{\dd}^{-1}(0)/\GL_{\dd}$. We let $\mathcal{M}_{\Pi_Q,\dd}\coloneqq \mu_{\dd}^{-1}(0)\cms\GL_{\dd}=\Spec(\BoC[\mu^{-1}(0)]^{\GL_{\dd}})$ be the affinization. Its closed points parametrise semisimple $\dd$-dimensional representations of $\Pi_Q$. We let $\JH_{\Pi_Q,\dd}\colon\mathfrak{M}_{\Pi_Q,\dd}\rightarrow\mathcal{M}_{\Pi_Q,\dd}$ be the semisimplification/affinization map. It is a good moduli space for the stack $\mathfrak{M}_{\Pi_Q,\dd}$ in the sense of \cite{alper2013good}. Let
\[
\JH_{\Pi_Q}=\bigsqcup_{\dd\in\BoN^{Q_0}}\JH_{\Pi_Q,\dd}\colon\mathfrak{M}_{\Pi_Q}=\bigsqcup_{\dd\in\BoN^{Q_0}}\mathfrak{M}_{\Pi_Q,\dd}\rightarrow\mathcal{M}_{\Pi_Q}=\bigsqcup_{\dd\in\BoN^{Q_0}}\mathcal{M}_{\Pi_Q,\dd}
\]
be the semisimplification map from the stack of all finite dimensional representations of $\Pi_Q$ to its coarse moduli space.

\subsection{The derived preprojective algebra}

The derived preprojective algebra is a differential graded algebra that is a dg-version of the classical preprojective algebra described in \S \ref{subsection:preprojectivealgebra}, see \cite{keller2008calabi}. Let $Q=(Q_0,Q_1)$ be a quiver, $\overline{Q}$ its double (as in \S \ref{subsection:preprojectivealgebra}) and $\tilde{Q}=(Q_0,\tilde{Q}_1)$ the tripled quiver. It is obtained from $\overline{Q}$ by adding a loop $u_i$ at each vertex $i\in Q_0$. We define a nonpositive grading on $\BoC\tilde{Q}$ by assigning the degree $0$ to arrows of $\overline{Q}$ and degree $-1$ to the additional loops $u_i$ at the vertices. We set
\[
 d(u_i)=e_i\left(\sum_{\alpha\in Q_1}[\alpha,\alpha^*]\right)e_i=e_i\rho e_i.
\]
Via the Leibniz rule, $d$ admits a unique extension to a differential on $\BoC\tilde{Q}$. The \emph{derived preprojective algebra} is the differential graded algebra $\mathscr{G}_2(Q)\coloneqq(\BoC\tilde{Q},d)$.  It is by definition concentrated in nonpositive degrees and the isomorphism $\HO^0(\mathscr{G}_2(Q))\cong \Pi_Q$ induces a morphism
\begin{equation}
\label{equation:derivedpptopp}
 \mathscr{G}_2(Q)\rightarrow \Pi_Q
\end{equation}
of differential graded algebras, where $\Pi_Q$ is given the zero differential.  The derived preprojective algebra is a $2$-Calabi--Yau dg-algebra: it can be realised as the $2$-Calabi--Yau completion of the path algebra $\BoC Q$ of $Q$, as in \cite{keller2011deformed}.

For $A$ a differential graded algebra, we denote by $\Perf_{\dg}(A)$ the dg-category of perfect left $A$-modules, i.e. the compact objects in the dg-category $\D_{\dg}(\Mod^{A})$. It follows from the 2-Calabi--Yau property of $\mathscr{G}_2(Q)$ that $\Perf_{\dg}(\SG_2(Q))$ carries a \emph{left 2CY} structure in the sense of \cite{brav2019relative}, to which we refer the reader for more details.

\begin{remark}
 When $Q$ is non-Dynkin, the map \eqref{equation:derivedpptopp} is a quasi-isomorphism and there is no need to appeal to the derived preprojective algebra. The ambient left 2CY dg-category appearing in \S\ref{subsubsection:categorical} can be taken to be $\Perf_{\dg}(\Pi_Q)$. When $Q$ is Dynkin (for example, $Q^{(0)}\coloneqq (Q_0=\{*\},Q_1=\emptyset)$), this fails: for example, for $Q^{(0)}$ the preprojective algebra is isomorphic to $\BoC$ and is of homological dimension $0$.
\end{remark}

\subsection{The direct sum map}
\begin{proposition}
\label{proposition:directsumproperpreprojective}
The direct sum map $\oplus\colon\mathcal{M}_{\Pi_Q}\times\mathcal{M}_{\Pi_Q}\rightarrow\mathcal{M}_{\Pi_Q}$ is finite.
\end{proposition}
\begin{proof}
We have a closed immersion $\mathcal{M}_{\Pi_Q}\rightarrow \mathcal{M}_{\overline{Q}}$ to the moduli space of representations of the double quiver $X_{\overline{Q}}=\bigcup_{\dd\in\BoN^{Q_0}}X_{\overline{Q},\dd}\cms\GL_{\dd}$ and a Cartesian square
\[
 \begin{tikzcd}
	{\mathcal{M}_{\Pi_Q}\times\mathcal{M}_{\Pi_Q}} & {\mathcal{M}_{\Pi_Q}} \\
	{\mathcal{M}_{\overline{Q}}\times\mathcal{M}_{\overline{Q}}} & {\mathcal{M}_{\overline{Q}}}
	\arrow["\oplus", from=1-1, to=1-2]
	\arrow["\oplus", from=2-1, to=2-2]
	\arrow[from=1-2, to=2-2]
	\arrow[from=1-1, to=2-1]
	\arrow["\lrcorner"{anchor=center, pos=0.125}, draw=none, from=1-1, to=2-2]
\end{tikzcd}
\]
where the direct sum map $\oplus\colon\mathcal{M}_{\overline{Q}}\times\mathcal{M}_{\overline{Q}}\rightarrow\mathcal{M}_{\overline{Q}}$ is the direct sum morphism for semisimple $\BoC\overline{Q}$-modules, which is finite by \cite[Lemma 2.1]{meinhardt2019donaldson}. The finiteness of $ \oplus\colon\mathcal{M}_{\Pi_Q}\times\mathcal{M}_{\Pi_Q}\rightarrow\mathcal{M}_{\Pi_Q}$ follows by base change.
\end{proof}

\subsection{The RHom-complex}

\subsubsection{Projective resolution of the preprojective algebra}
\label{subsubsection:projectiveresolutionPP}
We recall the classical resolution of $\Pi_Q$ as a $\Pi_Q$-bimodule.

We define
\begin{align*}
 P_0=\bigoplus_{i\in Q_0}\Pi_Qe_i\otimes_{\BoC}e_i\Pi_Q, & &P_1=\bigoplus_{i\xrightarrow{\alpha}j\in\overline{Q}_1}\Pi_Qe_j\otimes_{\BoC}e_i\Pi_Q.
\end{align*}
By \cite[Proposition 2.4]{crawley2022deformed}, we have an exact sequence of $\Pi_Q$-bimodules
\begin{equation}
\label{equation:projresolutionPiQ}
 P_0\xrightarrow{f}P_1\xrightarrow{g} P_0\xrightarrow{h} \Pi_Q\rightarrow 0
 \end{equation}
 where $f,g$ and $h$ are as in \cite{crawley2022deformed}. In particular, $h$ is the multiplication. This sequence is not exact on the left in general but it is when $Q$ is a non-Dynkin quiver \cite[Theorem 2.7]{crawley2022deformed}. If $M$ is a finite-dimensional representation of $\Pi_Q$, it induces an exact sequence of projective $\Pi_Q$-modules by applying $\otimes_{\Pi_Q}M$ to \eqref{equation:projresolutionPiQ} on the right:
 \begin{equation}
 \label{equation:projresM}
 P_0\otimes_{\Pi_Q}M\xrightarrow{f} P_1\otimes_{\Pi_Q}M\xrightarrow{g} P_0\otimes_{\Pi_Q}M\rightarrow M\rightarrow 0,
\end{equation}
where by abuse of notation  the maps are still denoted $f,g,h$. If $N$ is a finite-dimensional representation of $\Pi_Q$, by applying the functor $\Hom(-,N)$ to \eqref{equation:projresM}, we get a short exact sequence
\begin{equation}
\label{equation:sesHomcomplexe}
\Hom_{\Pi_Q}(P_0\otimes_{\Pi_Q}M,N)\xrightarrow{g} \Hom_{\Pi_Q}(P_1\otimes_{\Pi_Q}M,N)\xrightarrow{f} \Hom_{\Pi_Q}(P_0\otimes_{\Pi_Q}M,N)
\end{equation}
We note as in the proof of \cite[Proposition 2.6]{crawley2022deformed} that $\Hom_{\Pi_Q}(P_0\otimes_{\Pi_Q}M,N)\cong\bigoplus_{i\in Q_0}\Hom_{\BoC}(e_iM,e_iN)$ and $\Hom_{\Pi_Q}(P_1\otimes_{\Pi_Q}M,N)\cong\bigoplus_{i\xrightarrow{\alpha}j\in\overline{Q}_1}\Hom_{\BoC}(e_iM,e_jN)$. 
In particular, these $\Pi_Q$-bimodules only depend on the full subquiver of $Q$ supporting the dimension vectors of the representations $M$ and $N$.

\subsubsection{The RHom complex for preprojective algebras}
\label{subsubsection:rhompreprojectivealgebra}
By performing the construction of \S \ref{subsubsection:projectiveresolutionPP} over $\mathfrak{M}_{\Pi_Q}\times\mathfrak{M}_{\Pi_Q}$, we get a presentation of the RHom complex as a strictly $[-1,1]$-perfect complex. That is, we obtain an explicit presentation of the RHom complex on $\mathfrak{M}_{\Pi_Q,\dd^{(1)}}\times\mathfrak{M}_{\Pi_Q,\dd^{(2)}}$ for $\dd^{(1)},\dd^{(2)}\in\BoN^{Q_0}$ by a 3-term complex of vector bundles, where RHom is calculated in the category of finite-dimensional $\mathscr{G}_2(Q)$-modules.

The construction goes as follows. Let $V_{\dd^{(j)}}=\mu_{\dd^{(j)}}^{-1}(0)\times \BoC^{\dd^{(j)}}$ ($j=1,2$) be the $\GL_{\dd^{(j)}}$-equivariant $Q_0$-graded vector bundles with fibre $\BoC^{\dd^{(j)}}$ on $\mu_{\dd^{(j)}}^{-1}(0)$. We let $\Hom(V_{\dd^{(2)}},V_{\dd^{(1)}})=\bigoplus_{i\in Q_0}\Hom((V_{\dd^{(2)}})_i,(V_{\dd^{(1)}})_i)$ be the $\GL_{\dd^{(1)}}\times\GL_{\dd^{(2)}}$-equivariant vector bundle on $\mu^{-1}_{\dd^{(1)}}(0)\times\mu^{-1}_{\dd^{(2)}}(0)$ of $Q_0$-graded morphisms from $V_{\dd^{(2)}}$ to $V_{\dd^{(1)}}$. We define a complex of vector bundles on $\mu^{-1}_{\dd^{(1)}}(0)\times \mu^{-1}_{\dd^{(2)}}(0)$:
\begin{equation}
\label{equation:RHompreprojective}
 \Hom(V_{\dd^{(2)}},V_{\dd^{(1)}})\xrightarrow{d}\bigoplus_{i\xrightarrow{\alpha}j\in \overline{Q}_1}\Hom((V_{\dd^{(2)}})_i,(V_{\dd^{(1)}})_j)\xrightarrow{\mu} \Hom(V_{\dd^{(2)}},V_{\dd^{(1)}}).
\end{equation}
Let $x=(x_{\alpha})_{\alpha\in \overline{Q}_1}\in\mu_{\dd^{(1)}}^{-1}(0)$ and $y=(y_{\alpha})_{\alpha\in \overline{Q}_1}\in\mu_{\dd^{(2)}}^{-1}(0)$. For $z\in\Hom(V_{\dd^{(2)}},V_{\dd^{(1)}})_{(x,y)}=\Hom(\BoC^{\dd^{(2)}},\BoC^{\dd^{(1)}})$, we define
\[
 d_{(x,y)}z=(z_jy_{\alpha}-x_{\alpha}z_i)_{i\xrightarrow{\alpha} j\in\overline{Q}_1}
\]
and
for $t=(t_{\alpha},t_{\alpha^*})_{\alpha\in Q_1}\in \left(\bigoplus_{i\xrightarrow{\alpha}j\in \overline{Q}_1}\Hom((V_{\dd^{(2)}})_i,(V_{\dd^{(1)}})_j)\right)_{(x,y)}$, we define
\[
 \mu_{(x,y)}(t)=\sum_{\alpha\in Q_1}[x_{\alpha}+y_{\alpha}+t_{\alpha},x_{\alpha^*}+y_{\alpha^*}+t_{\alpha^*}]
\]
where for $\alpha\in \overline{Q}_1$, $x_{\alpha}, y_{\alpha}, t_{\alpha}$ are seen as endomorphisms of $\BoC^{\dd^{(1)}+\dd^{(2)}}$ via the decomposition $\BoC^{\dd^{(1)}+\dd^{(2)}}\cong \BoC^{\dd^{(1)}}\oplus\BoC^{\dd^{(2)}}$. As these maps are $\GL_{\dd^{(1)}}\times\GL_{\dd^{(2)}}$-equivariant, the complex \eqref{equation:RHompreprojective} induces a $3$-term complex $\mathcal{H}_{\dd^{(1)},\dd^{(2)}}$ of vector bundles on $\mathfrak{M}_{\Pi_Q,\dd^{(1)}}\times\mathfrak{M}_{\Pi_Q,\dd^{(2)}}$.

The following is a very special case of Corollary \ref{der_exact_cor}.
\begin{proposition}
 The $3$-term complex $\mathcal{H}_{\dd^{(1)},\dd^{(2)}}$ obtained from \eqref{equation:RHompreprojective} is a complex of vector bundles quasi-isomorphic to the RHom complex on $\mathfrak{M}_{\Pi_Q,\dd^{(1)}}\times\mathfrak{M}_{\Pi_Q,\dd^{(2)}}$ introduced in general in \S \ref{subsubsection:categorical}, where RHom is calculated in the category of finite-dimensional $\mathscr{G}_2(Q)$-modules.
\end{proposition}

This explicit presentation of the RHom complex for preprojective algebras of quivers is crucial to compare the CoHA product on $\JH_*\BD\BoQ_{\FM_{\Pi_Q}^{\vir}}$ obtained later via $3$-term complexes (\S\ref{subsubsection:ThecohaproductKV}) with the classical construction \cite{yang2018cohomological,schiffmann2020cohomological}, see Appendix~\ref{section:relativeCoHA}.

\section{Perverse sheaves and mixed Hodge modules}
\label{section:perversesheaves}
We assume that the reader is comfortable with the formalism of constructible derived categories and perverse sheaves and has a working familiarity with mixed Hodge modules, for instance by having read the introductory \cite{saito1989introduction}. We nonetheless give some reminders of basic constructions in the theory here.
\subsection{Tensor structure on the constructible derived category of a monoid object of schemes}
\label{subsection:monoidalstructure}
Recall that a monoid $(\CM,\oplus,\eta)$ in a monoidal category $\mathscr{C}$ is the data of an object $\CM$ of $\mathscr{C}$, a morphism $\oplus\colon\CM\times\CM\rightarrow \CM$ and a morphism $\eta\colon\mathbf{1}_{\mathscr{C}}\rightarrow \CM$ from the monoidal unit satisfying the compatibility conditions of an associative algebra. We sometimes omit $\eta$ from the notation. Let $(\mathcal{M},\oplus,\eta)$ be a monoid object in the category of complex schemes (with monoidal structure given by Cartesian product) such that $\pi_0(\oplus)\colon \pi_0(\CM)^{\times 2}\rightarrow \pi_0(\CM)$ has finite fibres. We allow $\mathcal{M}$ to have infinitely many connected components. We assume that each connected component of $\CM$ is a finite type separated complex scheme.

We define $\CD_{\mathrm{c}}^+(\mathcal{M})$ to be the full subcategory of the derived category of constructible complexes on $\CM$ that have bounded below cohomological amplitude on each connected component of $\CM$. The category $\CD_{\mathrm{c}}^+(\mathcal{M})$ has a monoidal structure $\boxdot$ defined by
\[
 \mathscr{F}\boxdot\mathscr{G}\coloneqq\oplus_*(\mathscr{F}\boxtimes\mathscr{G}).
\]
If $\oplus$ is a finite morphism of schemes (and so $\oplus_*$ is perverse exact), the same formula gives a monoidal structure on $\Perv(\mathcal{M})$, the category of perverse sheaves on $\mathcal{M}$. The monoidal unit is given by $\eta_*\BoQ_{\pt}$, where $\eta\colon\pt\rightarrow \CM$ is the unit morphism for $\CM$.
\smallbreak
Let us moreover assume that $\CM$ is a commutative monoid, i.e. if $s\colon \Msp\times \Msp\rightarrow \Msp\times\Msp$ is the morphism swapping the factors, then $\oplus\circ s=\oplus$. Then the monoidal categories defined above are naturally upgraded to symmetric monoidal categories.

\begin{definition}
\label{ma_alg_def}
A \emph{$\boxdot$-algebra} in $\Perv(\CM)$ (or more generally in $\CD_{\mathrm{c}}^{+}(\CM)$) is a monoid in the respective tensor (monoidal) category, with monoidal structure given by $\boxdot$. It is a triple $(\CF,m\colon\CF\boxdot\CF\rightarrow \CF,\iota\colon\eta_*\BoQ_{\pt}\rightarrow \CF)$, with $\CF\in \Perv(\CM)$ or $\CF\in\CD^+_{\mathrm{c}}(\CM)$, such that $m$ and $\iota$ satisfy the usual commutativity and unitality properties demanded of an associative algebra.
If $(\CM,\oplus)$ is a commutative monoid, then a \emph{$\boxdot$-Lie algebra} in $\Perv(\CM)$ (or $\CD^+_{\mathrm{c}}(\CM)$) is a pair $(\CF,c\colon \CF\boxdot \CF\rightarrow \CF)$ with $\CF\in \Perv(\CM)$ or $\CF\in\CD^+_{\mathrm{c}}(\CM)$, such that the Lie bracket $c$ is antisymmetric and satisfies the Jacobi identity.
\end{definition}

\subsection{Mixed Hodge modules}
\label{MHM_sec}

If $\CM$ is a separated and reduced complex scheme, we define as in \cite{saito1990mixed} the bounded derived category $\CD^{\mathrm{b}}(\MHM(\CM))$ of algebraic mixed Hodge modules on $\CM$. There is a faithful functor $\rat\colon\MHM(\CM)\rightarrow \Perv(\CM)$ taking a mixed Hodge module to its underlying perverse sheaf. We denote also by $\rat$ the functor obtained between the derived categories.

A mixed Hodge module $\CF\in\MHM(\CM)$ carries an ascending weight filtration $W_{\bullet}\CF$, and we say that $\CF$ is pure of weight $n$ if $W_{n-1}\CF=0$ and $W_{n}\CF=\CF$. An object $\CF\in\CD^{\mathrm{b}}(\MHM(\CM))$ is called \emph{pure} if the mixed Hodge module $\mathcal{H}^i(\CF)$ is pure of weight $i$ for every $i$. By Saito's theory, pure mixed Hodge modules (of fixed weight) form a semisimple category, and pure objects of $\CD^{\mathrm{b}}(\MHM(\CM))$ are preserved by taking direct images along projective morphisms.

We define the Tate MHM $\BoL\coloneqq \HO^{*}_{\cc}(\BoA^1_{\BoC},\BoQ)$, with its canonical pure weight zero mixed Hodge structure. Precisely, $\BoL$ is a cohomologically graded mixed Hodge structure, concentrated in cohomological degree $2$, and the second cohomology mixed Hodge structure of $\BoL$ is pure, of weight 2. We denote by $\BoL^n$ the $n$th tensor power of $\BoL$ in the category of cohomologically graded mixed Hodge structures, and we allow $n\leq 0$, since $\BoL$ is invertible in this tensor category. If $X$ is a variety, there is a natural upgrade of $\BoQ_X$ to a mixed Hodge module complex on $X$, namely $(X\rightarrow \pt)^*\BoL^0$, which we denote $\ulBoQ_X$.
\begin{lemma}
\label{Homs_Lem}
Let $X$ be a separated, reduced complex scheme, then the natural morphisms
\[
\Hom_{\CD^{\mathrm{b}}(\MHM(X))}(\ulBoQ_X,\ulBoQ_X)\xrightarrow{\alpha=\rat}\Hom_{\CD^{\mathrm{b}}(\Perv(X))}(\BoQ_X,\BoQ_X)\xrightarrow{\beta}\BoQ^{\pi_0(X)}
\]
are isomorphisms.
\end{lemma}
\begin{proof}
It is clear that $\beta$ is an isomorphism: if $X$ is connected then $\Hom_{\CD^{\mathrm{b}}(\Perv(X))}(\BoQ_X,\BoQ_X)\cong\HO^0(X,\BoQ)\cong\BoQ$ (and we use this isomorphism to define $\beta$). Also, $\beta\alpha$ has a right-inverse $\gamma$, since for any object $\CF$ in the tensor category $\CD^{\mathrm{b}}(\MHM(X))$ there is an embedding $\BoQ\hookrightarrow \Hom(\CF,\CF)$ of scalar endomorphisms, and we may write $\ulBoQ_X\cong \bigoplus_{X'\in\pi_0(X)}\ulBoQ_{X'}$.  So all that remains is to show that $\beta\alpha$ is injective, and this reduces to the case in which $X$ is connected, which we now assume. By adjunction we have
\[
\Hom_{\CD^{\mathrm{b}}(\MHM(X))}(\ulBoQ_X,\ulBoQ_X)\cong \Hom_{\CD^{\mathrm{b}}(\MHM(\pt))}\left(\BoL^0,a_*\ulBoQ_X\right)\cong \Hom_{\MHM(\pt)}(\BoL^0,\HO^0(X,\ulBoQ))\cong\underline{\BoQ}^{\pi_0(X)}
\]
where $a\colon X\rightarrow \pt$ is the structure morphism, and we are done.
\end{proof}

If $X$ is an even-dimensional separated, finite type integral scheme, and $X^{\mathrm{sm}}$ is its smooth locus, there is a unique extension of the pure weight zero mixed Hodge module $\ulBoQ_{X^{\mathrm{sm}}}\otimes\BoL^{-\dim(X)/2}$ on $X^{\mathrm{sm}}$ to a simple pure weight zero mixed Hodge module on $X$, denoted $\underline{\IC}(X)$. Then $\rat(\underline{\IC}(X))=\IC(X)$, the intersection complex on $X$. If $X$ is odd-dimensional (a case not needed in the present paper), the extension of the theory to monodromic mixed Hodge modules is required to define the square-root $\BoL^{1/2}$ of the the Tate twist $\BoL$.

\subsubsection{Unbounded complexes}
\label{unbounded_cplx_sec}
As in \cite{davison2020bps} we work in the bigger category $\CD^+(\MHM(\CM))$ of locally bounded below complexes of mixed Hodge modules on $\CM$. If $\CM$ is connected, this is defined to be the limit of the diagram of categories $\CD_n$ for $n\in \BoZ$, with $\CD_n=\CD^{\mathrm{b}}(\MHM(\CM))$ for all $n$, and arrows from $\CD_n$ to $\CD_{n'}$ provided by the truncation functors $\tau^{\leq n'}$. For general $\CM$ we define $\CD^+(\MHM(\CM))\coloneqq \prod_{\CN\in\pi_0(\CM)}\CD^+(\MHM(\CN))$. By Saito's theory, pure objects of $\CD^+(\MHM(\CM))$ are preserved by taking direct image along projective morphisms. See \cite[\S 2.1.4]{davison2020bps} for the detailed construction, and proofs.

\subsubsection{Mixed Hodge modules on stacks}
\label{MHMs_on_stacks}
At the time this paper was written, the full six-functor formalism for derived categories of mixed Hodge modules on stacks had not been developed. While this paper was under revision, such a formalism has been developped in \cite{tubach2024mixed}. Here we remind the reader of the standard workaround which allows us to work with mixed Hodge modules over stacks independently of this formalism.
\begin{definition}
\label{acyclic_cover_def}
Given a finite type Artin stack $\FX$, we say that $\FX$ has an \textit{acyclic cover} if for all $N$ there is a morphism $f_N\colon X_N\rightarrow \FX$ such that 
\begin{enumerate}
\item
each $f_N$ is smooth
\item
As $N\to \infty$, the minimum $i$ such that $\pH{i}(\cone(\BD\BoQ_{\FX}\rightarrow f_{N,*}\BD\BoQ_{X_N}[-2\dim f_N]))\neq 0$ tends to infinity.
\end{enumerate}
\end{definition}
It is easy to check that if $\FX$ has an acyclic cover, and $\FY$ is finite type and representable over $\FX$, then $\FY$ has an acyclic cover (provided by varieties $\FY\times_{\FX} X_N$). If $\FX$ has an acyclic cover, and $a\colon \FX\rightarrow \CX$ is a morphism to a variety, we define $a_*\BD\ulBoQ_{\FX}$ by setting 
\begin{align*}
\tau^{\leq i}a_*\BD\ulBoQ_{\FX}=\tau^{\leq i}(af_N)_*(\BD\ulBoQ_{X_N}\otimes \BoL^{\dim(f_N)}) && \text{for $N \gg 0$}.
\end{align*}
Then there is a natural isomorphism $\rat(\tau^{\leq i}a_*\BD\ulBoQ_{\FX})\cong \ptau{\leq i}a_*\BD\BoQ_{\FX}$. Likewise, if $\bsX=t_0(\FX)$ is the underlying classical stack of a derived stack with the cotangent complex $\BL_{\bsX}$ a perfect complex, with even virtual rank, and $\FX$ has an acyclic cover, we define $a_*\BD\ulBoQ^{\vir}_{\FX}$ by setting
\begin{align*}
\tau^{\leq i}a_*\BD\ulBoQ^{\vir}_{\FX}=\tau^{\leq i}(af_N)_*(\BD\ulBoQ_{X_N}\otimes \BoL^{\dim(f_N)+\vrank(\BL_{\bsX})/2}) &&\text{for $N \gg 0$},
\end{align*}
that is we just twist eveything by $\BoL^{\vrank(\BL_{\bsX})}$. The main point of the construction is that we do not consider directly or even define the dualizing complex of mixed Hodge modules $\BD\underline{\BoQ}_{\FX}$, but that we define and work direcly with the pushforward to a scheme $a_*\BD\underline{\BoQ}_{\FX}$.

It is well known that if $\FX=X/ G$ is a global quotient of a variety by a linear algebraic group, then $\FX$ has an acyclic cover provided by approximations to the Borel construction. See \cite{davison2020bps} for more details. Therefore, these constructions apply to quotient stacks.

\subsubsection{Action of unipotent group scheme}
\label{subsubsection:action-unipotent-group-scheme}
Let $\pi\colon \FY \rightarrow \FX$ be a morphism of stacks and $\FV$ a vector bundle on $\FX$. Then $\FV$ is a smooth group scheme over $\FX$. We assume that it acts on $\FY$ via $\FV\times_{\FX} \FY \rightarrow \FY$ and we let $\pi\colon \FY \rightarrow \FY/\FV$ be the natural map to the quotient stack. Then, the adjunction morphism $\id\rightarrow\pi_*\pi^*$ is an isomorphism, and so the pullback $\pi^* \colon \CD_{\mathrm{c}}(\FY/\FV) \rightarrow \CD_{\mathrm{c}}(\FY)$ between the constructible derived categories is fully faithful. The corresponding functor between derived categories of monodromic mixed Hodge modules is also fully faithful. These follow from the fact that $\pi$ is smooth and the fibers of $\pi$ are contractible, and so have trivial cohomology. In particular, if $\pi\colon \FY\rightarrow\FX$ is a torsor for a unipotent group $U$, then the pullback $\pi^*$ is fully faithful. This follows from the previous statement where $\FV=\FX\times U$ is the trivial bundle with fiber $U$ over $\FX$.

\subsubsection{Monoids}
\label{subsubsection:monoids}
If $(\mathcal{M},\oplus,\eta)$ is a monoid object in the category of complex schemes such that $\pi_0(\oplus)$ has finite fibres, the formulas of \S \ref{subsection:monoidalstructure} define monoidal structures on $\D^+(\MHM(\mathcal{M}))$, and also on $\MHM(\CM)$ if $\oplus$ is moreover finite.  The forgetful functor $\rat\colon\D^+(\MHM(\mathcal{M}))\rightarrow \D^{+}_{\cc}(\mathcal{M})$ is monoidal, as is $\rat\colon \MHM(\mathcal{M})\rightarrow\Perv(\mathcal{M})$ if $\oplus$ is finite. They are moreover symmetric monoidal in case $\CM$ is a symmetric monoid; see \cite{maxim2011symmetric} and \cite[Sec.3.2]{davison2020cohomological} for details. We define $\boxdot$-algebras and $\boxdot$-Lie algebras in $\CD^+(\MHM(\CM))$ and $\MHM(\CM)$ as in Definition~\ref{ma_alg_def}.

\subsection{Restriction to submonoids}
\label{strict_mf_sec}
Let $\imath\colon(\mathcal{N},\oplus)\rightarrow (\mathcal{M},\oplus)$ be a morphism of monoids in the category of schemes. We assume $\CN$ is a saturated submonoid, in the sense that the following diagram is Cartesian:
\begin{equation}
 \label{equation:diagrampullbackmonoidal}
 \begin{tikzcd}
	{\mathcal{N}\times\mathcal{N}} & {\mathcal{N}} \\
	{\mathcal{M}\times \mathcal{M}} & {\mathcal{M}.}
	\arrow["\imath\times\imath"', from=1-1, to=2-1]
	\arrow["\oplus", from=1-1, to=1-2]
	\arrow["\oplus", from=2-1, to=2-2]
	\arrow["\imath", from=1-2, to=2-2]
\end{tikzcd}
\end{equation}

\begin{lemma}
\label{lemma:strictmonoidalfunctor}
The exceptional pullback functors
\begin{align*}
 \imath^!\colon& \mathcal{D}_{\cc}^{+}(\mathcal{M})\rightarrow \mathcal{D}_{\cc}^{+}(\mathcal{N}),&
 &\imath^!\colon\CD^+(\MHM(\CM))\rightarrow \CD^+(\MHM(\CN))
\end{align*}
are strict monoidal functors.
\end{lemma}
\begin{proof}
 This is the base change isomorphism $\oplus_*(\imath\times\imath)^!\cong \imath^!\oplus_*$ for the Cartesian diagram \eqref{equation:diagrampullbackmonoidal}. See \cite[(4.4.3)]{saito1990mixed} for the lift to mixed Hodge modules.
\end{proof}

\subsection{Free algebras}
Let $(\CM,\oplus,\eta)$ be a monoid object in the category of complex schemes. Assume furthermore that the unit morphism $\eta\colon \pt\rightarrow \CM$ is an isomorphism onto a connected component $\CM_0$, with complement $\CM_{>}$, and the morphism $\coprod_{n\geq 1} \CM_{>}^{\times n}\rightarrow \CM$ induced by the monoidal structure on $\CM$ is finite. The free $\boxdot$-algebra $\Free_{\boxdot-\Alg}(\CF)$ generated by an object $\CF \in \D^+_{\cc}(\CM_{>})$ is the initial $\boxdot$-algebra amongst morphisms from $\CF$ to underlying complexes of $\boxdot$-algebras.
As an object in $\D_{\cc}^{+}(\CM)$ we have
\begin{equation}
\label{equation:freealgdirectsum}
 \Free_{\boxdot-\Alg}(\CF) = \bigoplus_{n \geq 0} \CF^{\boxdot n}\,.
\end{equation}
The multiplication is induced by the identifications $\CF^{\boxdot m}\boxdot\CF^{\boxdot n}\cong \CF^{\boxdot m+n}$. If $\CF\in\CD^+(\MHM(\CM_{>}))$ is a complex of mixed Hodge modules, we define the free algebra it generates analogously.
\begin{proposition}
\label{proposition:FreeSemiSimp}
Assume that $\oplus$ is finite. If $\CF \in \Perv(\CM_{>})$ is a semisimple perverse sheaf, then so is $\Free_{\boxdot-\Alg}(\CF)$. If $\CF\in\MHM(\CM)$ is pure of weight zero, then so is $\Free_{\boxdot-\Alg}(\CF)$.
\end{proposition}
\begin{proof}
Since $\oplus$ is finite, $\oplus_\ast$ sends (semisimple) perverse sheaves to (semisimple) perverse sheaves. The categories of semisimple perverse sheaves are closed under direct sum and external products. The proposition now follows from \eqref{equation:freealgdirectsum}. The proof for mixed Hodge modules is identical, noting that $\oplus$ preserves purity since it is finite.
\end{proof}

For later use, we specialise Proposition~\ref{proposition:FreeSemiSimp} to the context in which we will use it.  First recall the definition of $\Sigma_{\CA}$ from \S \ref{subsubsection:freenesstotallynegative}.
\begin{corollary}
Let $\CA$ be an Abelian category satisfying Assumptions~\ref{gms_assumption} and \ref{ds_fin} from \S\ref{section:modulistackobjects2CY}. In particular, we assume that $\oplus$ is finite (Assumption~\ref{ds_fin}). The perverse sheaf $\Free_{\boxdot-\mathrm{Alg}}\left(\bigoplus_{a \in \Sigma_{\CA}} \IC(\Msp_{\CA,a})\right)$ on $\Msp_{\CA}$ is semisimple. The mixed Hodge module $\Free_{\boxdot-\Alg}\left(\bigoplus_{a \in \Sigma_{\CA}} \underline{\IC}(\Msp_{\CA,a})\right)$ is pure.
\end{corollary}
\begin{proof}
This follows immediately from Proposition \ref{proposition:FreeSemiSimp} by semisimplicity and purity of IC-complexes.
\end{proof}

\section{Virtual pullbacks and quasi-smooth morphisms}
\label{vpb_section}

This section develops the core technical material required for the construction of the product of the cohomological Hall algebras considered in this paper. We introduce a class of morphisms along which we can define virtual pullback morphisms of mixed Hodge modules that satisfy desirable properties.

Similar constructions appear for example in \cite{olsson2015borel}, for $\ell$-adic cohomology. Since we need refined Gysin morphisms and the corresponding functorialities at the level of complexes of mixed Hodge modules, we recall the constructions in this section.

\subsection{Virtual pullback for strongly l.c.i. morphisms}
\label{lci_pb_sec}
Let $f\colon X\rightarrow Y$ be a local complete intersection, i.e. we assume that we can write $f$ as a composition $ba$ where $b\colon U\rightarrow Y$ is smooth of dimension $d$, and $a\colon X\rightarrow U$ is a regular embedding of codimension $c$.  For simplicity, and because this is the only type of situation that we will have to handle in this paper, we assume furthermore that $a$ is a section of a smooth morphism $p$. We will call morphisms that are compositions of smooth maps and sections of smooth maps \emph{strongly l.c.i. morphisms}. Note that the inclusion $\{xy=0\}\subset \BoA^2$ is a (local) complete intersection that cannot be written as such a composition \cite{Hennecart2023LCIMorphisms}.

The condition on $b$ results in a canonical isomorphism $b^!\ulBoQ_Y\cong \ulBoQ_U\otimes\BoL^{-d}$, while the condition on $a$ yields a canonical isomorphism $a^!\ulBoQ_U\cong a^!p^!\ulBoQ_X\otimes\BoL^c\cong \ulBoQ_X\otimes \BoL^{c}$. Composing, we obtain the isomorphism
\[
\alpha\colon f^!\ulBoQ_Y=(ba)^!\ulBoQ_Y\cong \ulBoQ_X\otimes\BoL^{c-d}.
\]
Applying adjunction and Verdier duality, we thus define the \textit{pullback morphism}
\[
\BD\ulBoQ_Y\xrightarrow{}f_*\BD\ulBoQ_X\otimes\BoL^{d-c}.
\]
Let $b^{\circ}\colon U^{\circ}\rightarrow Y$, $a^{\circ}\colon X\rightarrow U^{\circ}$ be a different choice of decomposition exhibiting $f$ as a complete intersection, under our standing assumption that $a^{\circ}$ is a section of a smooth morphism. We construct in the same way an isomorphism
\[
\beta\colon f^!\ulBoQ_Y=(b^{\circ}a^{\circ})^!\ulBoQ_Y\cong \ulBoQ_X\otimes\BoL^{c-d}.
\]
We claim that $\zeta\coloneqq\beta\alpha^{-1}=\id$. By Lemma \ref{Homs_Lem} this can be checked in the category of constructible sheaves. Moreover, by the same lemma, it is sufficient to check that $\rat(\zeta)\colon\BoQ_X\rightarrow \BoQ_X$ is the identity at a single point of $X$. We can therefore reduce to the following situation: $X=\BoA^{m}\times\BoA^{m'}$, $U^{(\circ)}=\BoA^m\times\BoA^{m'}\times\BoA^{n}\times\BoA^{n^{(\circ)}}$ and $Y=\BoA^m\times\BoA^n$ with $a^{(\circ)}$ and $b^{(\circ)}$ the natural inclusions and projections, and then the claim is easy to verify. We will make occasional use of the following lemma, comparing pullbacks along strongly l.c.i. morphisms with trivial excess intersection bundle.
\begin{lemma}
\label{zei_lemma}
Let
\[
\begin{tikzcd}
X&Y\\
Z&W
\arrow["f",from=1-1,to=1-2]
\arrow["a",from=1-1,to=2-1]
\arrow["b",from=1-2,to=2-2]
\arrow["g",from=2-1,to=2-2]
\arrow["\lrcorner"{anchor=center, pos=0.125}, draw=none, from=1-1, to=2-2]
\end{tikzcd}
\]
be a Cartesian diagram of morphisms of schemes, and assume that $f$ and $g$ are strongly l.c.i. morphisms of the same relative dimension $r$. Then the isomorphisms
\begin{align*}
\BD\ulBoQ_Y\rightarrow f_*\BD\ulBoQ_X\otimes\BoL^r\\
b^!(\BD\ulBoQ_W\rightarrow g_*\BD\ulBoQ_Z\otimes\BoL^r)
\end{align*}
are equal, after making the natural identifications $b^!\BD\ulBoQ_W=\BD\ulBoQ_Y$ and $b^!g_*\BD\ulBoQ_Z=f_*a^!\BD\ulBoQ_Z=f_*\BD\ulBoQ_X$.
\end{lemma}
\begin{proof}
Again, we use Lemma \ref{Homs_Lem} to reduce to the same statement, but for constructible sheaves, and then reduce to checking at a single point of $Y$. We may therefore assume that the morphisms are all composed out of inclusions and projections of affine spaces of appropriate dimensions, at which point the claim is easy to check.
\end{proof}

\subsection{Refined pullback for vector bundles with section}
We explain the contruction of the virtual pullback from \S\ref{lci_pb_sec} in the particular case of a vector bundle with a section, for which we derive some functorial properties.

\subsubsection{Definition of the refined pullback}
Let $\mathfrak{X}$ be an Artin stack and $\mathcal{E}$ be a vector bundle of rank $r$ on $\mathfrak{X}$. Let $\mathfrak{E}=\Tot_{\FX}(\mathcal{E})$ be the total space of $\mathcal{E}$ and $\pi\colon\mathfrak{E}\rightarrow\mathfrak{X}$ be the projection. Let $s\colon\mathfrak{X}\rightarrow \mathfrak{E}$ be a section of $\mathcal{E}$. We let $0_{\mathfrak{E}}$ be the zero section of $\mathfrak{E}$. Consider the pullback diagram

\begin{equation}
 \label{equation:refinedpullback}
 \begin{tikzcd}
	{\mathfrak{X}_s} & {\mathfrak{X}} \\
	{\mathfrak{X}} & {\mathfrak{E}}
	\arrow["s", from=1-2, to=2-2]
	\arrow["{0_{\mathfrak{E}}}"', from=2-1, to=2-2]
	\arrow["{i_s}", from=1-1, to=1-2]
	\arrow["{i_s}"', from=1-1, to=2-1]
	\arrow["\lrcorner"{anchor=center, pos=0.125}, draw=none, from=1-1, to=2-2]
\end{tikzcd}
\end{equation}
where $\mathfrak{X}_s$ is the zero locus of $s$. We will define the refined Gysin morphism as a morphism of complexes of sheaves\footnote{Here and throughout the rest of the paper, where we write the potentially ambiguous $\BD \BoQ_{X}[d]$, we mean $(\BD \BoQ_{X})[d]$, and not $\BD(\BoQ_{X}[d])$.}
\begin{equation}
\label{equation:refinedGysinmorphism}
v_s\colon \BD \BoQ_{\mathfrak{X}}\rightarrow(i_s)_*\BD \BoQ_{\mathfrak{X}_s}[2r].
\end{equation}
We have the natural isomorphism
\[
s^*\BD \BoQ_{\mathfrak{E}}\cong \BD \BoQ_{\mathfrak{X}}[2r].
\]
Indeed, since $\pi$ is smooth of relative complex dimension $r$, we have $\pi^!=\pi^*[2r]$ and $\pi\circ s=\id_{\mathfrak{X}}$, so $s^*\BD \BoQ_{\mathfrak{E}}=\BD (s^!\BoQ_{\mathfrak{E}})\cong\BD (s^!\pi^!\BoQ_{\mathfrak{X}}[-2r])=\BD \BoQ_{\mathfrak{X}}[2r]$. By the adjunction $(s^*,s_*)$, we obtain a morphism $\BD \BoQ_{\mathfrak{E}}\rightarrow s_*\BD \BoQ_{\mathfrak{X}}[2r]$. By applying $0_{\mathfrak{E}}^!$ and using the base change isomorphism $0_{\mathfrak{E}}^!s_*\cong (i_s)_*i_s^!$, we obtain the morphism
\[
\BD \BoQ_{\mathfrak{X}}\xrightarrow{v_s} (i_s)_* \BD \BoQ_{\mathfrak{X}_s}[2r].
\]
This is the refined Gysin morphism \eqref{equation:refinedGysinmorphism}.
\smallbreak
Let $\FX$ be a scheme.  At the level of functors between mixed Hodge module complexes, we have the canonical isomorphism $\pi^!\cong \pi^{\ast}\otimes \BoL^{-r}$, from which we define the morphism
\[
\BD\ulBoQ_{\FX}\rightarrow (i_s)_*\BD\ulBoQ_{\FX_s}\otimes\BoL^{-r}
\]
in similar fashion. If $\FX$ is a stack, then the construction in the formalism \S\ref{MHMs_on_stacks} gives, for a morphism $a\colon\FX\rightarrow\CX$ to a scheme, the virtual pullback $a_*\BD\underline{\BoQ}_{\FX}\xrightarrow{a_*v_s} a_*(i_s)_*\BD\underline{\FX}_s\otimes\BoL^{-r}$.

\subsubsection{Base change for refined pullbacks}
\label{subsubsection:functorialityGysin}

Let $f\colon \mathfrak{Y}\rightarrow\mathfrak{X}$ be a representable morphism of stacks, $\mathcal{E}$ a rank $r$ vector bundle on $\mathfrak{X}$ and $s$ a section of $\mathcal{E}$. We obtain the vector bundle $f^*\mathcal{E}$ on $\mathfrak{Y}$ together with the section $f^*s$.

We have a pullback square
\begin{equation}
\label{Gysin_pb_diagram}
 \begin{tikzcd}
	{\mathfrak{Y}_{f^*s}} & {\mathfrak{X}_s} \\
	{\mathfrak{Y}} & {\mathfrak{X}}
	\arrow["{f_s}", from=1-1, to=1-2]
	\arrow["{i_{f^*s}}"', from=1-1, to=2-1]
	\arrow["{i_s}", from=1-2, to=2-2]
	\arrow["f"', from=2-1, to=2-2]
	\arrow["\lrcorner"{anchor=center, pos=0.125}, draw=none, from=1-1, to=2-2]
\end{tikzcd}
\end{equation}
and functoriality of the refined Gysin morphism:

\begin{proposition}
\label{proposition:functorialityvirtualpullback}
 Let $v_s\colon \BD \BoQ_{\mathfrak{X}}\rightarrow (i_s)_*\BD \BoQ_{\mathfrak{X}_s}[2r]$ be the refined Gysin morphism for $(\mathcal{E},s)$ and let $v_{f^*s}\colon \BD \BoQ_{\mathfrak{Y}}\rightarrow (i_{f^*s})_*\BD \BoQ_{\mathfrak{Y}_{f^*s}}[2r]$ be the refined Gysin morphism for $(f^*\mathcal{E},f^*s)$. Then, $v_{f^*s}=f^!v_s$ using the canonical identification given by the base change isomorphism $f^!(i_s)_*\cong (i_{f^*s})_*f_s^!$.

 If $f\colon \FY\rightarrow\FX$ is a morphism of schemes, the same result holds at the level of morphisms of complexes of mixed Hodge modules with the shift $[2r]$ replaced by the Tate twist $\BoL^{-r}$.
\end{proposition}
\begin{proof}
 This follows by base change along $\overline{f}$ in the following diagram:
 \[
 \begin{tikzcd}
	{\mathfrak{Y}_{f^*s}} && {\mathfrak{Y}} \\
	& {\mathfrak{Y}} & {} & {f^*\mathfrak{E}} \\
	{\mathfrak{X}_s} & {} & {\mathfrak{X}} \\
	& {\mathfrak{X}} && {\mathfrak{E}}
	\arrow["{i_{f^*s}}", from=1-1, to=1-3]
	\arrow["{f_s}"', from=1-1, to=3-1]
	\arrow[""{name=0, anchor=center, inner sep=0}, "f"'{pos=0.7}, from=2-2, to=4-2]
	\arrow["{0_{\mathfrak{E}}}", from=4-2, to=4-4]
	\arrow["{0_{f^*\mathfrak{E}}}"{pos=0.7}, from=2-2, to=2-4]
	\arrow["{\overline{f}}", from=2-4, to=4-4]
	\arrow["{i_s}"{pos=0.3}, from=3-2, to=3-3]
	\arrow[no head, from=1-3, to=2-3]
	\arrow["f", from=2-3, to=3-3]
	\arrow["{f^*s}", from=1-3, to=2-4]
	\arrow["{i_{f^*s}}", from=1-1, to=2-2]
	\arrow["{i_s}"', from=3-1, to=4-2]
	\arrow["s", from=3-3, to=4-4]
	\arrow[bend left=10,"{\pi_{\mathfrak{E}}}", from=4-4, to=4-2]
	\arrow[bend left=10, "{\pi_{f^*\mathfrak{E}}}"{pos=0.7}, from=2-4, to=2-2]
	\arrow[shorten >=7pt, no head, from=3-1, to=0]
\end{tikzcd}
 \]
The bottom face of the cube is \eqref{equation:refinedpullback} while the top face is the version of \eqref{equation:refinedpullback} for $\FY$, the vector bundle $f^*\mathfrak{E}$ and $f^*s$.
\end{proof}
The next corollary follows similarly by base change:
\begin{corollary}
\label{bc_Gysin}
In the diagram \eqref{Gysin_pb_diagram}, define $g=i_s\circ f_s=f\circ i_{f^*s}$.
\begin{enumerate}
\item
Assume that $f$ is proper. Then the diagram (in which the horizontal maps are obtained by adjunction)
\[
\begin{tikzcd}
g_\ast\BD \BoQ_{\FY_{f^*s}}[2r] & (i_s)_{\ast}\BD \BoQ_{\FX_s}[2r]
\\
f_{\ast}\BD \BoQ_{\FY} & \BD \BoQ_{\FX}
\arrow[from=2-1,to=1-1,"f_{\ast}v_{f^*s}"]
\arrow[from=2-2,to=1-2,"v_{s}"]
\arrow[from=1-1,to=1-2]
\arrow[from=2-1,to=2-2]
\end{tikzcd}
\]
commutes.
\item
Assume that $f$ is smooth. Then the diagram (in which the horizontal maps are obtained by adjunction)
\[
\begin{tikzcd}
g_*\BD \BoQ_{\FY_{f^*s}}[2r] & (i_s)_{\ast}\BD \BoQ_{\FX_s}[2r+2\dim(f)]
\\
f_*\BD \BoQ_{\FY} & \BD \BoQ_{\FX}[2\dim(f)]
\arrow[from=2-1,to=1-1,"f_*v_{f^*s}"]
\arrow[from=2-2,to=1-2,"v_{s}[2\dim(f)\rbrack"]
\arrow[from=1-2,to=1-1]
\arrow[from=2-2,to=2-1]
\end{tikzcd}
\]
commutes.
\end{enumerate}
If $f\colon\FY\rightarrow\FX$ is a morphism of schemes, the same results hold at the level of mixed Hodge modules with the shift $[2\star]$ replaced by the Tate twist $\otimes\BoL^{-\star}$.
\end{corollary}

\subsubsection{Composition and direct sum of vector bundles}
Let $\FX$ be an Artin stack, let $\CE'$ and $\CE''$ be vector bundles on $\FX$,
with sections $s'$ and $s''$ of $\CE'$ and $\CE''$ respectively. We define
$\CE=\CE'\oplus \CE''$. It has the section $s'\boxplus s''$ induced by
$s'$ and $s''$. We define $\FE,\FE',\FE''$ to be the
total spaces of $\CE,\CE',\CE''$ respectively. There is a Cartesian diagram of
projections
\[
\begin{tikzcd}
\FE&\FE'\\
\FE''&\FX.
\arrow[from=1-1,to=1-2,"p'"]
\arrow[from=1-1,to=2-1,"p''"']
\arrow[from=1-2,to=2-2,"\pi'"]
\arrow[from=2-1,to=2-2,"\pi''"]
\arrow["\lrcorner"{anchor=center, pos=0.125}, draw=none, from=1-1, to=2-2]
\end{tikzcd}
\]
\begin{proposition}
\label{sum_comp}
With notation as above, we denote by $t$ the restriction of $s''$ to $\FX_{s'}$. Then the following diagram of refined Gysin morphisms commutes
\[
\xymatrix{
\BD \BoQ_{\FX}\ar[r]^-{v_{s'}}\ar[dr]_{v_{s'\boxplus s''}}& (i_{s'})_{*}\BD \BoQ_{\FX_{s'}}\ar[d]^{(i_{s'})_*v_{t}[2\rank(\CE')]}[2\rank(\CE')]\\
&(i_{s'\boxplus s''})_*\BD \BoQ_{\FX_{s'\boxplus s''}}[2\rank(\CE)].
}
\]
If $\FX$ is a scheme, the same statement is true at the level of mixed Hodge module complexes with $[2\rank(-)]$ replaced by $\otimes\BoL^{-\rank(-)}$.
\end{proposition}
\begin{proof}
Set $r=\pi''^*s'$, a section of $\pi''^*\CE'$. Then we consider the following commutative diagram in which all squares are Cartesian:
\[
\begin{tikzcd}
{\FX_{s'\boxplus s''}}\arrow[r]\arrow[d,"i_t"]\arrow[dd, bend right=35, "i_{s'\boxplus s''}" ']
&{\FX_{s''}}\arrow{r}\arrow[d, "i_{s''}"]
&{\FX}\arrow{d}[swap]{s''}\arrow[dd, bend left=35, "s'\boxplus s''"]\\
{\FX_{s'}}\arrow{r}{f=i_{s'}}\arrow{d}{i_{s'}}
&{\FX}\arrow{r}\arrow{d}{s'}\arrow[r,"0_{\CE''}"]
&{\FE''}\arrow{d}[swap]{r}\\
{\FX}\arrow{r}{0_{\CE'}}\arrow[rr, bend right=35, "0_{\CE}"]
&{\FE'}\arrow{r}{0_{\pi'^*\CE''}}&{\FE.}
\arrow["\lrcorner"{anchor=center, pos=0.125}, draw=none, from=1-1, to=2-2]
\arrow["\lrcorner"{anchor=center, pos=0.125}, draw=none, from=1-2, to=2-3]
\arrow["\lrcorner"{anchor=center, pos=0.125}, draw=none, from=2-1, to=3-2]
\arrow["\lrcorner"{anchor=center, pos=0.125}, draw=none, from=2-2, to=3-3]
\end{tikzcd}
\]
By Proposition \ref{proposition:functorialityvirtualpullback} there are
equalities of morphisms $v_{s'}=(0_{\CE''})^!v_r$ and $(i_{s'})_*v_{t}=(i_{s'})_*f^!v_{s''}$ and the proposition follows.
\end{proof}

\subsection{Total space of $3$-term complexes}

Let $\FM$ be an Artin stack and
\[
\CC^{-1}\xrightarrow{d^{-1}}\CC^0\xrightarrow{d^0}\CC^1
\]
be a $3$-term complex of vector bundles over $\FM$ concentrated in cohomological degrees $-1,0,1$.

We have the
following vector bundle with section over the stack $\FM$
\begin{equation*}
	\begin{tikzcd}
		\Tot_{\Tot_{\FM}(\CC^0)}(\pi^* \CC^1) \ar[r] &
		\Tot_{\FM}(\CC^0) \ar[l, bend right,"s_{d^{0}}"']
	\end{tikzcd}
\end{equation*}
Here $\pi\colon \Tot_{\FM}(\CC^0) \to \FM$ is the projection.
The vector bundle $\Tot_{\FM}(\CC^{-1})$ over $\FM$ is a smooth unipotent group scheme acting via $d^{-1}$ on $\Tot_{\FM}(\CC^0)$.
Since $d^0\circ d^{-1} = 0$, $\Tot_{\FM}(\CC^{-1})$ acts on $s_{d^0}^{-1}(0)$.
The (classical truncation) of the \emph{total space} of $\CC^{\bullet}$ is defined to be
the quotient (over $\FM$) of $s^{-1}_{d^{0}}(0)$ by $\Tot_{\FM}(\CC^{-1})$ and denoted by $\Tot_{\FM}(\CC^{\bullet})$.

\subsection{Equivariant virtual pullback}
\label{subsection:virtual-pb-equivariant}

In this section, we explain how the virtual pullback for three term complexes of vector bundles $\CC^\bullet = (\CC^{-1} \xrightarrow{d^{-1}} \CC^0 \xrightarrow{d^0}\CC^1)$ on a classical Artin stack $\FM$ can be understood equivariantly with respect to the $\Tot_{\FM}(\CC^{-1})$-action on $s_{d^0}^{-1}(0)$. Since $\CC^{-1}$ is a vector bundle over $\FM$, it defines a smooth unipotent group scheme. Moreover, the naively truncated complex $\CC^{\geq 0}=(\CC^0\rightarrow\CC^1)$ is such that $\Tot_{\FM}(\CC^{\geq0})\rightarrow \FM$ is a representable morphism. It admits an action of $\Tot_{\FM}(\CC^{-1})$ such that there is a canonical identification $\Tot_{\FM}(\CC^{\bullet})\simeq \Tot_{\FM}(\CC^{\geq0})/\Tot_{\FM}(\CC^{-1})$.

\begin{proposition}
\label{proposition:virtualpullback-unipotentpart}
We assume that $\FM$ is a scheme, or that it is a stack and we work with constructible complexes. We let $\FY\coloneqq\Tot_{\FM}(\CC^{\bullet})$ and $\FX\coloneqq\Tot_{\FM}(\CC^{\geq 0})$. We denote by $\pi\colon \FX\rightarrow\FY$ the projection. Then, the pullback morphism
\[
 \BD\underline{\BoQ}_{\FY}\xrightarrow{v}\pi_*\BD\underline{\BoQ}_{\FX}\otimes\BoL^{-\dim \pi}
\]
is an isomorphism (where $\dim\pi=\dim\CC^{-1}$). In particular, if $\pi'\colon\Tot_{\FM}(\CC^{\bullet})\rightarrow\FM$ is the projection, two morphisms $\BD\underline{\BoQ}_{\FM}\rightarrow\pi'_*\BD\underline{\BoQ}_{\FY}\otimes\BoL^{N}$ (for a fixed $N\in\BoZ$) coincide if and only if their postcompositions with $\pi'_*v\otimes\BoL^N$ coincide.
\end{proposition}
\begin{proof}
Since $\pi$ is the quotient by a smooth unipotent group scheme, the pullback $\pi^*$ is fully faithful (see \S\ref{subsubsection:action-unipotent-group-scheme}). Therefore, the adjunction morphism $\id\rightarrow\pi_*\pi^*$ is an isomorphism. Moreover, since $\pi$ is smooth of relative dimension $\dim\pi$, there is a natural isomorphism $\pi^*\BD\underline{\BoQ}_{\FY}\cong \BD\underline{\BoQ}_{\FX}\otimes\BoL^{-\dim\pi}$. This gives the isomorphism $v$ of the proposition. The second part of the proposition is straightforward since $\pi'_*v\otimes \BoL^N$ is also an isomorphism.
\end{proof}

\begin{proposition}
\label{proposition:naivetruncation-associativitytool}
 Let $\CC^{\bullet}$ (resp. $\CD^{\bullet}$) be perfect complexes of vector bundles over a classical Artin stack $\FM$ (resp. $\FN$) which are $3$-term complexes of vector bundles concentrated in degrees $[-1,1]$. We assume that there is a morphism of stacks $\gamma\colon \Tot_{\FM}(\CC^{\bullet})\rightarrow\FN$. We consider the diagram
% https://q.uiver.app/#q=WzAsNyxbMSwyLCJcXEZNIl0sWzEsMSwiXFxUb3Rfe1xcRk19KFxcQ0Nee1xcYnVsbGV0fSkiXSxbMiwxLCJcXFRvdF97XFxGTX0oXFxDQ157XFxnZXEgMH0pIl0sWzEsMCwiXFxUb3Rfe1xcVG90X3tcXEZNfShcXENDXntcXGJ1bGxldH0pfShcXGdhbW1hXipcXENEXntcXGJ1bGxldH0pIl0sWzIsMCwiXFxUb3Rfe1xcVG90X3tcXEZNfShcXENDXntcXGdlcSAwfSl9KChcXGdhbW1hXFxjaXJjXFxwaSleKlxcQ0Ree1xcYnVsbGV0fSkiXSxbMywwLCJcXFRvdF97XFxUb3Rfe1xcRk19KFxcQ0Nee1xcZ2VxIDB9KX0oKFxcZ2FtbWFcXGNpcmNcXHBpKV4qXFxDRF57XFxnZXEgMH0pIl0sWzAsMSwiXFxGTiJdLFszLDEsIlxcYmV0YSIsMl0sWzEsMCwiXFxhbHBoYSIsMl0sWzIsMCwiXFxhbHBoYSciXSxbMiwxLCJcXHBpIiwyXSxbNCwzLCJcXHBpJyIsMl0sWzQsMiwiXFxiZXRhJyJdLFs1LDQsIlxccGknJyIsMl0sWzUsMiwiXFxiZXRhJyciXSxbNCwxLCIiLDEseyJzdHlsZSI6eyJuYW1lIjoiY29ybmVyIn19XSxbMSw2LCJcXGdhbW1hIiwyXV0=
\[\begin{tikzcd}
	& {\Tot_{\Tot_{\FM}(\CC^{\bullet})}(\gamma^*\CD^{\bullet})} & {\Tot_{\Tot_{\FM}(\CC^{\geq 0})}((\gamma\circ\pi)^*\CD^{\bullet})} & {\Tot_{\Tot_{\FM}(\CC^{\geq 0})}((\gamma\circ\pi)^*\CD^{\geq 0})} \\
	\FN & {\Tot_{\FM}(\CC^{\bullet})} & {\Tot_{\FM}(\CC^{\geq 0})} \\
	& \FM
	\arrow["\beta"', from=1-2, to=2-2]
	\arrow["{\pi'}"', from=1-3, to=1-2]
	\arrow["\lrcorner"{anchor=center, pos=0.125, rotate=-90}, draw=none, from=1-3, to=2-2]
	\arrow["{\beta'}", from=1-3, to=2-3]
	\arrow["{\pi''}"', from=1-4, to=1-3]
	\arrow["{\beta''}", from=1-4, to=2-3]
	\arrow["\gamma"', from=2-2, to=2-1]
	\arrow["\alpha"', from=2-2, to=3-2]
	\arrow["\pi"', from=2-3, to=2-2]
	\arrow["{\alpha'}", from=2-3, to=3-2]
\end{tikzcd}\]
Then, two morphisms $\BD\underline{\BoQ}_{\FM}\rightarrow(\alpha\circ\beta)_*\BD\BoQ_{\Tot_{\Tot_{\FM}(\CC^{\bullet})}(\gamma^*\CD^{\bullet})}\otimes\BoL^{N}$ (for some fixed $N\in\BoZ$) coincide if and only if their postcompositions with $(\alpha\circ\beta)_*v$ coincide, where $v$ is the smooth pullback morphism $\BD\underline{\BoQ}_{\Tot_{\Tot_{\FM}(\CC^{\bullet})}(\gamma^*\CD^{\bullet})}\xrightarrow{v} \pi'_*\pi''_*\BD\underline{\BoQ}_{\Tot_{\Tot_{\FM}(\CC^{\geq 0})}((\gamma\circ\pi)^*\CD^{\geq 0})}\otimes\BoL^{-\dim \pi'\circ\pi''}$.

Moreover, the postcomposition with $(\alpha\circ\beta)_*v$ of the morphism
\[
 \BD\underline{\BoQ}_{\FM}\rightarrow(\alpha\circ\beta)_*\BD\BoQ_{\Tot_{\Tot_{\FM}(\CC^{\bullet})}(\gamma^*\CD^{\bullet})}\otimes\BoL^{N}
\]
($N=\vrank(\CC^{\bullet})+\vrank(\CD^{\bullet})$) given by the composition of the virtual pullbacks for $\alpha$ and $\beta$, is the composition of the virtual pullbacks for $\alpha'$ and $\beta''$.
\end{proposition}
\begin{proof}
The first part of the proposition follows from the same argument as Proposition~\ref{proposition:virtualpullback-unipotentpart} since the adjunction morphism $\id\rightarrow(\pi'\circ\pi'')_*(\pi'\circ\pi'')^*$ is an isomorphism.

The last part of the proposition comes from the fact that in the diagram of the proposition, all horizontal maps are smooth, vertical and diagonal arrows are total spaces of $3$ or $2$-term complexes and the square is Cartesian with trivial excess intersection bundle, using Lemma~\ref{zei_lemma}.
\end{proof}

\subsection{Virtual pullbacks under quasi-isomorphisms}

\subsubsection{Functoriality under global equivalences}

\begin{proposition}
\label{gqe_prop}
Let $\CC^{\bullet}$ and $\CC'^{\bullet}$ be two $3$-term complexes over an Artin stack $\FM$. Assume we are given the following commutative diagram
 \[
 \begin{tikzcd}
	{\FN = \Tot_{\FM}(\CC^{\bullet})}& \FM\\
	{\FN'=\Tot_{\FM}(\CC'^{\bullet})}&
	\arrow["\pi", from=1-1, to=1-2]
	\arrow["f"', from=1-1, to=2-1]
	\arrow["{\pi'}"', from=2-1, to=1-2]
\end{tikzcd}
 \]
in which $\pi,\pi'$ are the natural projections and $f$ is a global equivalence between them (i.e. an equivalence induced by a morphism of complexes $\CC^{\bullet}\rightarrow\CC'^{\bullet}$).  Then the following diagram commutes
\[
\begin{tikzcd}
\BoQ_{\FM}&\pi_* \BoQ_{\FN}[-2\vrank(\CC^{\bullet})]\\
&\pi'_{\ast}\BoQ_{\FN'}[-2\vrank(\CC'^{\bullet})]
\arrow["v_{\CC}",from=1-1,to=1-2]
\arrow["v_{\CC'}"',from=1-1,to=2-2]
\arrow["\cong",from=1-2,to=2-2]
\end{tikzcd}
\]
where the vertical isomorphism is induced by the equivalence $f$. Assume, in addition, that $\FM$ has an acyclic cover, and we are given a morphism $a\colon\FM\rightarrow \CM$ to a scheme.  Then the following diagram of complexes of mixed Hodge modules over $\CM$ commutes
 \[
\begin{tikzcd}
{a_*\ulBoQ_{\FM}}\arrow{dr}[swap]{a_*v_{\CC'}}&{a_*\pi_*\ulBoQ_{\FN}}\otimes\BoL^{\vrank(\CC^{\bullet})}\\
&{a_*\pi'_*\ulBoQ_{\FN'}}\otimes\BoL^{\vrank(\CC'^{\bullet})}.
\arrow["{a_*v_{\CC}}", from=1-1, to=1-2]
\arrow["{\cong}", from=1-2, to=2-2]
\end{tikzcd}
\]

\end{proposition}
\begin{proof}
We prove the constructible complex version of the theorem, the mixed Hodge module version is proved in the same way by considering acyclic covers as in \S\ref{MHMs_on_stacks}. We have a commutative diagram
 \[
 \begin{tikzcd}
	{\FM} & {\Tot(\mathcal{C}'^{\leq 0})} & {\Tot(\mathcal{C}^{\leq 0})} & Z \\
	& {\Tot(\pi'^{\leq 0*}\mathcal{C}^1)} & {\Tot((\pi^{\leq 0})^*\mathcal{C}^1)} & {\Tot(\mathcal{C}^{\leq 0})} \\
	& {\Tot((\pi'^{\leq 0})^*\mathcal{C}'^1)} & {} & {\Tot(\mathcal{C}'^{\leq 0})}
	\arrow["{{(\pi')}^{\leq 0}}", from=1-2, to=1-1]
	\arrow["{f^{\leq 0}}", from=1-3, to=1-2]
	\arrow["q",from=1-4, to=1-3]
	\arrow["u", from=1-2, to=2-2]
	\arrow["{(\pi'^{\leq 0})^*f^1}", from=2-2, to=3-2]
	\arrow["{s_{d'^1}}"', from=3-4, to=3-2]
	\arrow["0"', from=1-3, to=2-3]
	\arrow["q",from=1-4, to=2-4]
	\arrow["{s_{d^1}}"', from=2-4, to=2-3]
	\arrow["v"', from=2-3, to=2-2]
	\arrow["{f^{\leq 0}}", from=2-4, to=3-4]
	\arrow["\pi^{\leq 0}"', bend right=30, from=1-3, to=1-1]
	\arrow["0"', bend right=70, from=1-2, to=3-2]
	\arrow["\lrcorner"{anchor=center, pos=0.125, rotate=-90}, draw=none, from=1-3, to=2-2]
	\arrow["\lrcorner"{anchor=center, pos=0.125, rotate=-90}, draw=none, from=1-4, to=2-3]
	\arrow["\lrcorner"{anchor=center, pos=0.025, rotate=-90}, draw=none, from=2-4, to=3-2]
\end{tikzcd}
 \]
where all the squares are Cartesian (since $f$ is a quasi-isomorphism), and we have left out some subscripts in order to reduce clutter. In particular, $Z=\Tot_{\FM}(\CC^{\bullet})\cong \Tot_{\FM}(\CC'^{\bullet})$. The maps out of $Z$ are the natural inclusions, the map $v$ is the natural morphism induced by $f^0$ coming from the identification $\Tot((\pi^{\leq 0})^*\mathcal{C}^1)\cong \Tot((f^{\leq 0})^*(\pi'^{\leq 0})^*\mathcal{C}^1)$. The map $u$ is the zero section of the vector bundle $\Tot((\pi'^{\leq 0})^*\mathcal{C}^1)\rightarrow \Tot(\mathcal{C}'^{\leq 0})$. The maps in this diagram satisfy the following additional properties:
\begin{enumerate}
 \item 
 \label{lci_p1}The map $f^{\leq 0}$ is strongly l.c.i. of codimension $\vrank(\mathcal{C}'^{\leq 0})-\vrank(\mathcal{C}^{\leq 0})$.  We provide an explicit factorisation as in \S \ref{lci_pb_sec}.  We consider the two term complex $\mathcal{E}^{\bullet}=\mathcal{C}^{-1}\xrightarrow{g} \mathcal{C}^0\oplus \mathcal{C}'^0$ with $g=d^{-1}\oplus (f^0\circ d^{-1})$.  Then the graph embedding $\Tot(\mathcal{C}^{\leq 0})\rightarrow \Tot(\mathcal{E})$ is a closed embedding, and a section of the smooth projection $\Tot(\mathcal{E})\rightarrow \Tot(\mathcal{C}^{\leq 0})$.  There is a morphism of complexes $l\colon \mathcal{E}^{\bullet}\rightarrow \mathcal{C}'^{\leq 0}$ with $l^{-1}=f^{-1}$ and $l^0=\pi_{\mathcal{C}'^0}$.  We claim that $l$ is representable and smooth.  For this, it is sufficient to check the statement for the morphism $\overline{l}$ obtained after pulling back along $\Tot(\mathcal{C}'^0)\rightarrow \Tot(\mathcal{C}'^{\leq 0})$.  Since $f$ is a quasi-isomorphism, the morphism $r\colon \mathcal{C}^{-1}\rightarrow \mathcal{C}'^{-1}\oplus\mathcal{C}^0\oplus\mathcal{C}'^0$ given by $f^{-1}\oplus d^{-1}\oplus (f^0\circ d^{-1})$ is injective, and we have a surjective morphism of vector bundles $q\colon \mathcal{C}'^{-1}\oplus\mathcal{C}^0\oplus\mathcal{C}'^0/\mathcal{C}^{-1}\rightarrow \mathcal{C}'^0$.  The morphism $\overline{l}$ is obtained by applying $\Tot_{\FM}(-)$ to $q$, and is thus smooth.
 \item The map $v$ is strongly l.c.i. of codimension $\vrank(\mathcal{C}'^{\leq 0})-\vrank(\mathcal{C}^{\leq 0})$: the argument for this is exactly as in (\ref{lci_p1}).
 \item The map $s_{d^1}$ is strongly l.c.i. of codimension $\rank(\mathcal{C}^1)$ as the section of a vector bundle,
 \item The map $s_{d'^1}$ is strongly l.c.i. of codimension $\rank(\mathcal{C}'^1)$ as the section of a vector bundle.
\end{enumerate}

Since $\mathcal{C}^{\bullet}$ and $\mathcal{C}'^{\bullet}$ are quasi-isomorphic, we have in particular $\vrank(\CC^{\bullet})=\vrank(\CC'^{\bullet})$, which means that $v\circ s_{d^1}$ and $s_{d'^1}$ are both of the same codimension: $\rank(\mathcal{C}'^{1})$. The maps $f^{\leq 0}$ and $v$ are also of the same codimension and hence these Cartesian squares have no excess intersection bundle. Via Lemma \ref{zei_lemma}, we obtain that the virtual pullbacks by $\pi^{\leq 0}\circ q$ and $(\pi'^{\leq 0})\circ (f^{\leq 0}\circ q)$ coincide.
\end{proof}

\begin{lemma}
\label{lemma:cover_stack_rp}
 Let $\FX$ be a finite type Artin stack having the resolution property. Then, there exists a smooth morphism $f\colon\FY\rightarrow\FX$ which is a torsor for a unipotent algebraic group, where $\FY$ is a quotient stack of an affine scheme by a reductive algebraic group.
\end{lemma}
\begin{proof}
Since $\FX$ has the resolution property, it may be presented as a quotient stack $X/H$ where $H$ is an affine algebraic group and $X$ an affine scheme \cite{gross2017tensor}. Write $H\cong G\ltimes U$ where $G$ is a reductive group and $U$ is the unipotent radical of $H$. Then, we set $\FY\coloneqq X/G$. The morphism $f\colon X/G\rightarrow X/H$ satisfies the conditions of the lemma.
\end{proof}

\begin{proposition}
\label{glob_res_prop}
Let $\FM$ be an Artin stack having the resolution property. For $\CC^{\bullet}$ and $\CC'^{\bullet}$ two $3$-term complexes of vector bundles over $\FM$, an isomorphism $\CC^{\bullet}\cong \CC'^{\bullet}$ in $\CD(\Coh(\FM))$ induces an equivalence of stacks $f\colon t_0(\Tot_{\FM}(\CC^{\bullet}))\simeq t_0(\Tot_{\FM}(\CC'^{\bullet}))$
and the diagram of complexes of constructible sheaves
\[
\begin{tikzcd}
\BoQ_{\FM}&p_* \BoQ_{t_0(\Tot_{\FM}(\CC'^{\bullet}))}[2\vrank(\CC^{\bullet})]\\
&p_*f_*\BoQ_{t_0(\Tot_{\FM}(\CC^{\bullet}))}[2\vrank(\CC^{\bullet})]
\arrow["v_{\CC'}",from=1-1,to=1-2]
\arrow["v_{\CC}"',from=1-1,to=2-2]
\arrow["\cong",from=1-2,to=2-2]
\end{tikzcd}
\]
commutes, with $p\colon t_0(\Tot_{\FM}(\CC'^{\bullet}))\rightarrow\FM$ the projection. If $a\colon\FM\rightarrow \CM$ is a morphism to a scheme, the same statement remains true at the level of direct image mixed Hodge modules on $\CM$.
\end{proposition}
\begin{proof}
We let $f\colon\FY\rightarrow\FX$ be a torsor for a unipotent algebraic group such that $\FY$ is the quotient of an affine scheme by a reductive algebraic group (Lemma~\ref{lemma:cover_stack_rp}). Then, it suffices to prove the statement of the proposition after pulling back to $\FY$ thanks to \S\ref{subsubsection:action-unipotent-group-scheme}, i.e. for $f^*\CC^{\bullet}$ and $f^*\CC'^{\bullet}$. Since $f^*\CC^{\bullet}$ and $f^*\CC'^{\bullet}$ are quasi-isomorphic complexes of projective objects, there is a morphism of complexes $f^*\CC^{\bullet}\rightarrow f^*\CC'^{\bullet}$ that is a quasi-isomorphism. We are then reduced to Proposition~\ref{gqe_prop}.
\end{proof}

\section{Moduli stacks of objects in $2$-dimensional categories}
\label{section:modulistackobjects2CY}

In this section we introduce the setting and notation for categories and their moduli spaces. We explain in detail the assumptions we make on our categories and their moduli theory, in order to define cohomological Hall algebras, and then in the 2CY case, define BPS algebras.

\subsection{Geometric setup}
\label{subsubsection:gsetup}

\subsubsection{Categorical setup}
\label{subsubsection:categorical}
We let $\mathscr{C}$ be a $\BoC$-linear dg-category, $\bm{\mathfrak{M}}\subset \bm{\mathfrak{M}}_{\mathscr{C}}$ a $1$-Artin open substack of the stack of objects of $\mathscr{C}$ (in the sense of \cite{toen2007moduli}) and let $\mathfrak{M}= t_0(\bm{\mathfrak{M}})$ be the classical truncation. We assume that $\mathfrak{M}$ parametrises objects of an admissible (in the sense of \cite[\S 1.2]{beilinson2018faisceaux}) finite length Abelian subcategory $\mathcal{A}$ of the homotopy category $\HO^0(\mathscr{C})$. For $X=\Spec(A)$, $X$-points of $\bm{\mathfrak{M}}_{\mathscr{C}}$ correspond to pseudo-perfect $\mathscr{C}\otimes_{\BoC} A$-modules $N$, i.e. bimodules $N$ that are perfect as $A$-modules. Given a pair of such modules we obtain the dg $A$-module $\RHom_{\mathscr{C}\otimes_{\BoC} A}(N,N')$, in this way defining the RHom complex on $\bm{\mathfrak{M}}_{\mathscr{C}}^{\times 2}$ and thus, by restriction, on $\mathfrak{M}^{\times 2}$.

\subsubsection{Base monoid}
\label{base_monoid_sec}
We let $\mathscr{C}$, $\mathfrak{M}$ and $\mathcal{A}$ be as in \S\ref{subsubsection:categorical}. We assume that we have a monoid object $(\mathcal{M},\oplus)$ in the category of Artin stacks and a morphism $\varpi\colon\mathfrak{M}\rightarrow \mathcal{M}$ such that the following diagram commutes:
\[
 \begin{tikzcd}
	{\mathfrak{M}\times\mathfrak{M}} & {\mathfrak{Exact}} & {\mathfrak{M}} \\
	{\mathcal{M}\times\mathcal{M}} && {\mathcal{M}}
	\arrow["\varpi\times\varpi"', from=1-1, to=2-1]
	\arrow["q"', from=1-2, to=1-1]
	\arrow["p", from=1-2, to=1-3]
	\arrow["\varpi", from=1-3, to=2-3]
	\arrow["\oplus", from=2-1, to=2-3]
\end{tikzcd}
\]
where $\mathfrak{Exact}$ denotes the stack of short exact sequences in $\mathcal{A}$, $q$ is the projection onto the first and last terms of a short exact sequence and $p$ is the projection onto the middle term. In all the examples that we consider, $\CM$ will actually be a scheme. We sometimes say that $\mathcal{M}$ satisfying the above condition \emph{splits short exact sequences}. If $\CM_a$ is a connected component of $\CM$, we define $\FM_a\coloneqq \varpi^{-1}(\CM_a)$.

\subsection{Assumptions for the construction of the relative CoHA product}
\label{subsubsection:assumptionCoHAproduct}
For the construction of the relative Hall algebra product, we need two assumptions: one for the construction of the pushforward along $p$, one for the construction of the virtual pullback by $q$.

%The following assumption is required in the definition of pushforward by $p$.
\begin{assumption}{1}
\label{p_assumption}
(Assumption for the construction of the pushforward by $p$). Let $\mathfrak{Exact}$ be the stack of short exact sequences in $\mathcal{A}$. We let $p\colon\mathfrak{Exact}\rightarrow \mathfrak{M}$ be the morphism forgetting the extreme terms of the short exact sequence. We assume that $p$ is a proper and representable morphism of stacks.
\end{assumption}

Next we turn to the virtual pullback along $q$. 
Let $\CH^{\bullet}$ be the RHom complex on $\FM\times\FM$ (\S\ref{subsubsection:categorical}), and set $\CC^{\bullet}=\CH^{\bullet}[1]$. The stack $\bsE=\Tot_{\FM\times\FM}(\CC^{\bullet})$ carries a universal bundle of (shifted) homomorphisms of objects in $\mathscr{C}$, so that we have a natural identification $t_0(\bsE)=\mathfrak{Exact}$ (see \cite[Proposition 2.3.4]{kapranov2019cohomological} for the stack of coherent sheaves on surfaces, or \S \ref{new_exts_sec}), and a morphism $\bs{q}\colon \bsE\rightarrow \FM\times\FM$ given by sending the morphism $f\colon \rho''\rightarrow \rho'$  to $(\rho',\rho'')$. Then we have $q=t_0(\bs{q})$.

\begin{assumption}{2}
\label{q_assumption1}
We assume that $\FM$ has the resolution property, i.e. every coherent sheaf on $\FM$ is a quotient of a vector bundle.
For all $\CM_a,\CM_b\in\pi_0(\CM)$ we assume that the complex $\CC^{\bullet}$ is quasi-isomorphic to a $3$-term complex of vector bundles over $\FM_a\times\FM_b$.
\end{assumption}
Assumption \ref{q_assumption1} is all that is required to define the Hall product in \S\ref{subsection:relativecohaproductKV}.

A convenient consequence of this assumption is that by Totaro's criterion \cite{totaro2004resolution} (and its generalization to non-normal stacks \cite{gross2017tensor}) it guarantees that all stacks appearing have an acyclic cover, since they are global quotient stacks. We are then free to define direct images of mixed Hodge module complexes from these stacks following \S \ref{unbounded_cplx_sec}.

\subsection{Assumptions for construction of the BPS algebra}
\label{BPS_assumptions_sec}
For the construction of the BPS sheaf/MHM and the BPS algebra, we impose stronger restrictions on the stack $\FM$ and the categories $\mathscr{C}$ and $\mathcal{A}$. The first two assumptions are geometric:
\begin{assumption}{3}
\label{gms_assumption}
The stack $\FM$ has a good moduli space $\FM\xrightarrow{\JH}\CM$ in the sense of \cite{alper2013good}, and $\CM$ is a scheme.
\end{assumption}
The scheme $\mathcal{M}$ is endowed with a monoid structure given by the direct sum:
\begin{align*}
 \oplus\colon\mathcal{M}\times\mathcal{M}&\longrightarrow\mathcal{M}\\
 (x,y)&\longmapsto x\oplus y
\end{align*}
obtained from the direct sum $\oplus\colon\mathfrak{M}\times\mathfrak{M}\rightarrow\mathfrak{M}$ and universality of the good moduli space $\JH\colon\mathfrak{M}\rightarrow\mathcal{M}$. 
\begin{assumption}{4}
\label{ds_fin}
We assume that the morphism $\oplus$ is finite.
\end{assumption}
Given that the morphism $\oplus$ is easily seen to be quasi-finite, this amounts to requesting that $\oplus$ is a proper morphism.

The final assumption is categorical, and is the vital ingredient in \textit{formality} and the local neighbourhood theorem (Theorem \ref{theorem:neighbourhood}): see \cite{davison2021purity} for details.
\begin{assumption}{5}
\label{BPS_cat_assumption}
Given a collection of points $x_1,\ldots,x_r\in\CM$ parametrising simple objects $\underline{\CF}=\{\CF_1,\ldots,\CF_r\}$ of $\CA$, the full dg subcategory of $\mathscr{C}$ containing $\underline{\CF}$ carries a right 2-Calabi--Yau (2CY) structure in the sense of \cite{brav2019relative}.
\end{assumption}
In most, but not all, examples, the right 2CY structure can be uniformly derived (using results of \cite{brav2019relative}) from the presence of a \textit{left} 2CY structure on the ambiant category $\mathscr{C}$. For $\mathscr{C}$ the dg-category of coherent sheaves on a smooth symplectic surface, a left 2CY structure is induced by a trivialisation of the canonical bundle $K_S$. If $\mathscr{C}$ is the dg-category of perfect modules over a dga $A$, a left 2CY structure on $\mathscr{C}$ is equivalent to a 2CY structure on $A$ in the usual sense. We refer the reader to \cite{brav2019relative} for details and proofs.

In the rest of this section we recall some constructions and facts regarding the moduli spaces that appear under the above assumptions.

\subsection{Geometry of moduli spaces in the 2CY case}
\label{2CY_background}

\subsubsection{Serre subcategories}
\label{subsubsection:serre}
Let $\FM\xrightarrow{\JH}\CM$ be as in Assumption \ref{gms_assumption} a good moduli space for the stack of objects $\FM$ as in \S \ref{subsubsection:categorical}. Fix a locally closed saturated (\S\ref{strict_mf_sec}) submonoid $\mathcal{M}'\subset \mathcal{M}$. We define $\mathcal{B}$ to be the full Abelian subcategory of $\mathcal{A}$ generated by the objects represented by the closed points of $\mathcal{M}'$ under taking extensions. Typically, $\mathcal{B}$ will be defined to be the full subcategory of objects of $\mathcal{A}$ satisfying some support or nilpotency condition. See \S \ref{subsection:hierarchy} for examples crucial for this paper. Since $\CM$ parametrises semisimple objects in $\CA$, it follows that $\CB$ is a Serre subcategory (an Abelian subcategory stable under extensions, subobjects and quotients).

Let $\mathcal{B}$ be a Serre subcategory of $\mathcal{A}$ corresponding to a locally closed submonoid $\mathcal{M}_{\mathcal{B}}\subset \mathcal{M}$. We let $\mathfrak{M}_{\mathcal{B}}\subset \mathfrak{M}$ be the substack parametrising objects of $\mathcal{B}$, which we define via the following Cartesian square
\[
\begin{tikzcd}
	{\mathfrak{M}_{\mathcal{B}}} & {\mathfrak{M}} \\
	{\mathcal{M}_{\mathcal{B}}} & {\mathcal{M}}.
	\arrow[from=1-1, to=1-2]
	\arrow["\JH", from=1-2, to=2-2]
	\arrow[from=2-1, to=2-2]
	\arrow[from=1-1, to=2-1]
	\arrow["\lrcorner"{anchor=center, pos=0.125}, draw=none, from=1-1, to=2-2]
\end{tikzcd}
\]

\subsubsection{Grading}
\label{subsection:grading}
Let $\mathcal{B}$ be a fixed Serre subcategory of $\mathcal{A}$. Let $\pi_0(\mathcal{M}_{\mathcal{B}})$ be the monoid of connected components of $\mathcal{M}_{\mathcal{B}}$. The monoid structure $\oplus\colon \pi_0(\Msp_{\CB})\times \pi_0(\Msp_{\CB})\rightarrow \pi_0(\Msp_\CB)$ is induced by the direct sum: for any $a,b\in \pi_0(\mathcal{M}_{\mathcal{B}})$, there exists a unique $c\in\pi_0(\mathcal{M}_{\mathcal{B}})$ such that for any $(x,y)\in a\times b$, $x\oplus y\in c$. We set $a\oplus b=c$.

The monoid of connected components of the stack $\mathfrak{M}_{\mathcal{B}}$ coincides with $\pi_0(\Msp_{\CB})$ (since the fibers of $\JH$ are connected, by \cite[Theorem 4.16 vii)]{alper2013good}). For $a\in\pi_0(\Msp_\CB)$, we denote by $\Mst_{\CB,a}$ and $\Msp_{\CB,a}$ the corresponding connected components.

For any object $\mathcal{F}$ of $\mathcal{B}$ (resp. closed $\BoC$-point $x$ of $\mathfrak{M}_{\mathcal{B}}$ or $\mathcal{M}_{\mathcal{B}}$), we let $[\mathcal{F}]$ (resp. $[x]$) be the connected component of the point of $\FM_{\mathcal{B}}$ corresponding to $\mathcal{F}$ (resp. $x$). We will refer to $[\mathcal{F}]$ (resp. $[x]$) as \emph{the class} of $\mathcal{F}$ (resp. $x$).

By Assumption~\ref{q_assumption1}, the restriction $\mathcal{C}^{\bullet}_{a,b}$ of the shifted RHom complex to $\mathfrak{M}_a\times\mathfrak{M}_b$ can be represented by a bounded complex of vector bundles. Consequently, the Euler form $(\mathcal{F},\mathcal{G})_{\SC}=\sum_{i\in \BoZ}(-1)^i\ext^1(\mathcal{F},\mathcal{G})=-\vrank(\CC_{a,b}^{\bullet})$ is constant for objects $\mathcal{F},\mathcal{G}$ of $\mathcal{B}$, corresponding to points in fixed connected components $\Mst_{\CB,a},\Mst_{\CB,b}$ of $\Mst_{\CB}$. Therefore, the Euler form factors through $\pi_0(\Msp_{\CB})$:
\begin{equation*}
(-,-)_{\mathscr{C}}\colon \pi_0(\Msp_{\CB})\times\pi_0(\Msp_{\CB})\longrightarrow\BoZ.
\end{equation*}

\subsubsection{Geometric setup for the local neighbourhood theorem}
\label{subsubsection:setuplocal}
We give the hypotheses we need to formulate Theorem \ref{theorem:neighbourhood} (see \cite{davison2021purity}). Let $\mathscr{C}$, $\mathfrak{M}$ and $\CM$ be as above satisfying Assumptions \ref{gms_assumption} and \ref{BPS_cat_assumption}. Then the closed points of $\mathcal{M}$ are in bijection with semisimple objects of the category $\mathcal{A}$ and the map $\JH$ sends a $\BoC$-point of $\mathfrak{M}$ corresponding to an object $\mathcal{F}$ of $\mathcal{A}$ to the $\BoC$-point of $\mathcal{M}$ corresponding to the associated graded of $\mathcal{F}$ with respect to some Jordan--H\"older filtration. Since $\mathcal{A}$ is of finite length, each $\mathcal{F}_i$ is a simple object of the category $\mathcal{A}$.

\subsubsection{$\Sigma$-collections} \label{subsubsection:sigma}
Let $\mathcal{F}$ be an object in a $\BoC$-linear dg category $\SC$. We say that $\mathcal{F}$ is a \emph{$\Sigma$-object} if for some $g\geq 0$,
\[
 \dim \Ext^i_{\SC}(\mathcal{F},\mathcal{F})=
 \left\{
 \begin{aligned}
 &1 \quad&\text{ if $i=0,2$}\\
 &2g \quad&\text{ if $i=1$}\\
 &0 \quad&\text{ otherwise.}
 \end{aligned}
 \right.
\]
Let $\underline{\mathcal{F}}=\{\mathcal{F}_1,\hdots,\mathcal{F}_r\}$ be a collection of objects of $\SC$. We say that $\underline{\mathcal{F}}$ is a \emph{$\Sigma$-collection} if each $\mathcal{F}_t$ is a $\Sigma$-object and for any $m\neq n$, $\Hom(\mathcal{F}_m,\mathcal{F}_n)=0$.

\subsubsection{The Ext-quiver of a $\Sigma$-collection}
\label{subsubsection:extquiver}

Let $\underline{\mathcal{F}}=\{\mathcal{F}_1,\hdots,\mathcal{F}_r\}$ be a $\Sigma$-collection of objects of $\SC$. The Ext-quiver of $\underline{\mathcal{F}}$ is the quiver $\overline{Q}_{\underline{\mathcal{F}}}$ having as set of vertices $\underline{\mathcal{F}}$, and $\ext^1(\mathcal{F}_i,\mathcal{F}_j)$ arrows from $i$ to $j$.

If $\ext^1(\mathcal{F}_i,\mathcal{F}_j)=\ext^1(\mathcal{F}_j,\mathcal{F}_i)$ for any $1\leq i,j\leq r$ and this number is even if $i=j$, then $\overline{Q}_{\underline{\mathcal{F}}}$ is the double of some (non-unique) quiver $Q_{\underline{\mathcal{F}}}$. This condition is satisfied if $\underline{\CF}$ generates a subcategory of $\SC$ which admits a right $2$CY structure as in Assumption~\ref{BPS_cat_assumption}. We will refer to $Q_{\underline{\mathcal{F}}}$ as \emph{a half of} $\overline{Q}_{\underline{\mathcal{F}}}$, or as \emph{a half of the Ext-quiver of $\underline{\mathcal{F}}$}.

We assume that we have a good moduli space and make Assumptions \ref{gms_assumption} and \ref{BPS_cat_assumption}. We have a morphism of monoids
\begin{equation*}
\begin{split}
\psi_{\underline{\CF}}\colon \BoN^{\underline{\CF}} &\longrightarrow \pi_0(\CM)\\
\mm& \longmapsto \sum_{i=1}^rm_i[\CF_i]
\end{split}
\end{equation*}
The pullback by $\psi_{\underline{\mathcal{F}}}$ of the Euler form of $\SC$ on $\pi_0(\Msp)$ is the same as the pullback of the Euler form of $\SG_2(Q_{\underline{\CF}})$ on $\BoN^{\underline{\mathcal{F}}}$. We also have a morphism of monoids
\[
 \imath_{\underline{\mathcal{F}}}\colon \BoN^{\underline{\mathcal{F}}}\longrightarrow \mathcal{M}
\]
mapping $\mm\in\BoN^{\underline{\mathcal{F}}}$ to the point of $\mathcal{M}$ corresponding to the object $\bigoplus_{i=1}^r\mathcal{F}_i^{\oplus m_i}$.

\subsubsection{The local neighbourhood theorem}

We recall the local neighbourhood theorem for $2$-Calabi--Yau categories proved in \cite[Theorem 5.11]{davison2021purity}.

\begin{theorem}
\label{theorem:neighbourhood}
 Let $\JH\colon\mathfrak{M}\rightarrow\mathcal{M}$ be a good moduli space locally of finite type with reductive stabilizer groups such that $\mathfrak{M}=t_0(\bm{\mathfrak{M}})$ is an Artin stack, where $\bm{\mathfrak{M}}\subset\bm{\mathfrak{M}}_{\mathscr{C}}$ is an open substack of the stack of objects of a dg-category $\mathscr{C}$. Let $x$ be a closed $\BoC$-valued point of $\mathfrak{M}$ corresponding to an object $\mathcal{F}=\bigoplus_{1\leq i\leq r}\mathcal{F}_i^{\oplus m_i}$ and assume that the full dg-subcategory of $\mathscr{C}$ containing $\mathcal{F}_i$, $1\leq i\leq r$ carries a right 2CY structure and $\underline{\mathcal{F}}=\{\mathcal{F}_i\}_{1\leq i\leq r}$ is a $\Sigma$-collection. We let $Q$ be a half Ext-quiver of the collection $\underline{\mathcal{F}}=\{\mathcal{F}_i:1\leq i\leq r\}$. Then, there is a pointed affine $\GL_{\mm}$-variety $(H,y)$, and an analytic neighbourhood $U\subset H$ of $y$ such that if $\mathcal{U}_x$ (resp. $\mathfrak{U}_x$) denotes the image of $U$ by the canonical map $H\rightarrow H\cms\GL_{\mm}$ (resp. $H\rightarrow H/\GL_{\mm}$), there is a commutative diagram
 \[
\begin{tikzcd}
	{(\mathfrak{M}_{\mm}(\Pi_Q),0_{\mm})} && {(\mathfrak{U}_{x},y)} && {(\mathfrak{M},x)} \\
	\\
	{(\mathcal{M}_{\mm}(\Pi_Q),0_{\mm})} && {(\mathcal{U}_{x},p(y))} && {(\mathcal{M},\JH(x))}
	\arrow["{\JH_{\mm}}"', from=1-1, to=3-1]
	\arrow["p", from=1-3, to=3-3]
	\arrow["\JH", from=1-5, to=3-5]
	\arrow["{\tilde{\jmath}_x}", from=1-3, to=1-5]
	\arrow["{\jmath_x}", from=3-3, to=3-5]
	\arrow["{\jmath_{0_{\mm}}}"', from=3-3, to=3-1]
	\arrow["{\tilde{\jmath}_{0_{\mm}}}"', from=1-3, to=1-1]
	\arrow["\lrcorner"{anchor=center, pos=0.125, rotate=-90}, draw=none, from=1-3, to=3-1]
	\arrow["\lrcorner"{anchor=center, pos=0.125}, draw=none, from=1-3, to=3-5]
\end{tikzcd}
 \]
 in which the horizontal morphisms are analytic open immersions and the squares are Cartesian.

\end{theorem}
The original version of this theorem gives an \'etale local description of the map $(\mathfrak{M},x)\rightarrow (\mathcal{M},\JH(x))$ but for our purposes, it is most convenient to work with the analytic topology.
Under the assumptions given in \S \ref{subsubsection:setuplocal} (Assumptions \ref{gms_assumption} and \ref{BPS_cat_assumption}), the theorem applies for any $\BoC$-valued closed point of $\mathfrak{M}$.

\subsubsection{Geometry of the good moduli space}
\label{CrBo_geom_sec}
In this section, we deduce some geometric information regarding the good moduli space of the stack of objects of a $2$-Calabi--Yau category using the local neighbourhood theorem and \cite{crawley2001geometry}.

For $a\in\pi_0(\CM_{\CA})$, we let $p(a)=2-(a,a)_{\SC}$. We let $\Sigma_{\mathcal{A}}$ be the set of non-zero $a\in \pi_0(\CM_{\CA})$ such that for any decomposition $a=\sum_{j=1}^sa_j$ with $a_j\in \pi_0(\CM_{\CA})$ and $s\geq 2$ we have $p(a)>\sum_{j=1}^sp(a_j)$. Finally, we let $\mathcal{M}_{\mathcal{A},a}^s$ be the open locus of $\mathcal{M}_{\mathcal{A},a}$ over which $\JH_a$ is a $\BoC^*$-gerbe. It parametrises simple objects of $\mathcal{A}$ of class $a$.

\begin{proposition}
\label{proposition:geometryofgms}
Let $a\in\pi_0(\mathcal{M}_{\mathcal{A}})$. We have the following properties:
\begin{enumerate}
 \item The set $\Sigma_{\mathcal{A}}$ is the set of $a\in\pi_0(\mathcal{M}_{\mathcal{A}})$ for which $\JH_a$ is generically a $\BoC^*$-gerbe,
 \item If $a\in\Sigma_{\mathcal{A}}$, then $\mathcal{M}_{\mathcal{A},a}$ is irreducible of dimension $p(a)$ and its smooth locus is $\mathcal{M}_{\mathcal{A},a}^s$,
 \item If $\underline{\mathcal{F}}$ is a $\Sigma$-collection in $\mathcal{A}$ and $Q_{\underline{\mathcal{F}}}$ a half of the Ext-quiver of $\underline{\mathcal{F}}$, then $\psi_{\underline{\mathcal{F}}}^{-1}(\Sigma_{\mathcal{A}})= \Sigma_{\Pi_{Q_{\underline{\CF}}}}$.
\end{enumerate}
\end{proposition}
\begin{proof}
 The proof is straightforward, combining Theorem \ref{theorem:neighbourhood} and \cite[Theorem 1.2]{crawley2001geometry} together with the compatibility of Euler forms with $\imath_{\underline{\mathcal{F}}}$, defined in \S\ref{subsubsection:extquiver}.
\end{proof}

\subsection{Determinant line bundles}

\begin{assumption}{6}
	\label{det_bun_assumption}
	Suppose $\Mst$ satisfies Assumptions \ref{gms_assumption} and \ref{BPS_cat_assumption} so that Theorem~\ref{theorem:neighbourhood} applies to $\Mst$. 
	For every closed point $x \in \Mst$ denote by  $i_x \colon \B \BoC^{\times} \into \Mst$ the 
	inclusion of the scaling automorphisms $\BoC^{\times} \subset \Aut_{\Mst}(x)$ at $x$. 

	We assume there is a line bundle $\CL$ on $\Mst$ such that $i_x^* \CL$ has non-zero degree, i.e., if $l\colon \Mst \to \B \BoC^{\times}$ is the map given by $\CL$, then for all closed points $x \in \Mst$ we have $(i_x\circ l)^* u \neq 0 \in \HO^2(\B \BC^{\times})$ where $u \in \HO^2(\B \BoC^{\times})$ is a generator of the cohomology ring.

	We call such an $\CL$ a \emph{positive determinant line bundle}.
\end{assumption}

\section{Totally negative $2$-Calabi--Yau categories}
In this paper, we are interested in certain $2$-Calabi--Yau Abelian categories $\mathcal{A}$ that arise in the general context described in \S\ref{subsubsection:gsetup}-\S\ref{BPS_assumptions_sec}, which we call \emph{totally negative}. Namely, we say that a full Abelian subcategory $\mathcal{A}$ of the homotopy category of a dg category $\SC$ is \emph{totally negative} if for any pair of nonzero objects $M,N$ of $\mathcal{A}$, $( M,N)_{\mathscr{C}}<0$.

\subsection{The preprojective algebra of a totally negative quiver}

Let $Q=(Q_0,Q_1)$ be a quiver. Following Bozec--Schiffmann \cite[Section 1.2]{bozec2019counting}, we say that $Q$ is \emph{totally negative} if its symmetrised Euler form $(\dd,\ee)_{Q}=\langle\dd,\ee\rangle_Q+\langle\ee,\dd\rangle_Q$ is totally negative, that is $(\dd,\ee)_{Q}<0$ for any $\dd,\ee\in\BoN^{Q_0}\setminus\{0\}$. Unraveling the definition of the Euler form, this is equivalent to the requirement that any vertex of $Q$ carries at least two loops and any two vertices are connected by at least one arrow. The category of representations of $\Pi_Q$ is then a totally negative $2$-Calabi--Yau category, as it is a full Abelian subcategory of the homotopy category of the dg-category of representations of the dg-algebra $\SG_2(Q)$ whose Euler form is $(-,-)_{Q}$. These totally negative $2$-Calabi--Yau categories constitute the building blocks of more general totally negative $2$-Calabi--Yau categories, by the local neighbourhood theorem (Theorem \ref{theorem:neighbourhood}) and the following proposition. 

\begin{proposition}
\label{proposition:totnegextquiv}
The Ext-quiver of any $\Sigma$-collection (\S \ref{subsubsection:sigma}) in a totally negative $2$-Calabi--Yau Abelian category $\mathcal{A}$ is the double of a totally negative quiver.
\end{proposition}
\begin{proof}
 Let $\underline{\mathcal{F}}=\{\mathcal{F}_1,\hdots,\mathcal{F}_r\}$ be a $\Sigma$-collection in $\mathcal{A}$ and let $\overline{Q}$ be the Ext-quiver of $\underline{\mathcal{F}}$. Then, for any $i\neq j$, $\ext_{\SC}^1(\mathcal{F}_i,\mathcal{F}_j)=\hom_{\SC}(\mathcal{F}_i,\mathcal{F}_j)+\ext_{\SC}^2(\mathcal{F}_i,\mathcal{F}_j)-(\mathcal{F}_i,\mathcal{F}_j)_{\SC}=-(\mathcal{F}_i,\mathcal{F}_j)_{\SC}>0$ so $\overline{Q}$ has an arrow from $\mathcal{F}_i$ to $\mathcal{F}_j$. Furthermore, for any $1\leq i\leq r$, $\ext_{\SC}^1(\mathcal{F}_i,\mathcal{F}_i)=2-(\mathcal{F}_i,\mathcal{F}_i)_{\SC}>2$ and since $\mathcal{A}$ is a $2$-Calabi--Yau category, $\ext_{\SC}^1(\mathcal{F}_i,\mathcal{F}_i)$ is even. Therefore, $\overline{Q}$ has at least $4$ loops at each vertex.
\end{proof}

\subsection{Representations of the fundamental group of a surface}

\subsubsection{The stack of representations of the fundamental group of a surface}
Let $\Sigma_g$ be a closed orientable topological surface without boundary, of genus $g$. We fix a point $p\in \Sigma_g$. Let $r,d$ be integers with $r>0$. If $d\neq 0$, we write $r=\underline{r}\gcd(r,d)$ and $d=\underline{d}\gcd(r,d)$, fix throughout a primitive $\underline{r}$th root of unity $\zeta_{\ul{r}}$, and set $\lambda=\zeta_{\underline{r}}^{\underline{d}}$. If $d=0$ we set $\lambda=1$. The fundamental group of $\Sigma_g\setminus\{p\}$ is
\[
 \pi_1(\Sigma_g\setminus\{p\})=\left\langle l,x_i,y_i,1\leq i\leq g\mid l^{-1}\prod_{i=1}^g(x_iy_ix_i^{-1}y_i^{-1})=1\right\rangle.
\]
Its group algebra $\BoC[\pi_1(\Sigma_g\setminus\{p\})]$ is the algebra generated by $l,x_i,y_i$ for $1\leq i\leq g$, localised at $x_i,y_i,l$, with the relation $l^{-1}\prod_{i=1}^gx_iy_ix_i^{-1}y_i^{-1}=1$.

We are interested in local systems on $\Sigma_g\setminus\{p\}$ whose monodromy around $p$ is given by the multiplication by $\lambda$, and therefore in representations of the algebra $\BoC[\pi_1(\Sigma_g\setminus\{p\}),\lambda]=\BoC[\pi_1(\Sigma_g\setminus\{p\})]/\langle l-\lambda\rangle$. Note that the rank of any such local system is divisible by $\underline{r}$ if $d\neq 0$.  We denote by 
\[
Z^{\Betti}_{g,r,d}\subset \GL_{r}^{\times 2g}
\]
the variety cut out by the matrix valued equation $\prod_{i=1}^g\prod_{i=1}^gA_iB_iA_i^{-1}B_i^{-1}=\lambda\cdot \Id_{r\times r}$. This variety is acted on by $\GL_r$ via the simultaneous conjugation action. Then we set
\[
\CM^{\Betti}_{g,r,d}\coloneqq\Spec\left(\BoC[Z^{\Betti}_{g,r,d}]^{\GL_r}\right)
\]
where $\BoC[Z^{\Betti}_{g,r,d}]$ denotes the algebra of regular functions on $Z^{\Betti}_{g,r,d}$. We define $\FM^{\Betti}_{g,r,d}$ to be the stack of $r$-dimensional representations of $\BoC[\pi_1(\Sigma_g\setminus\{p\}),\lambda]$. There is an equivalence of stacks $\FM^{\Betti}_{g,r,d}\simeq Z^{\Betti}_{g,r,d}/\GL_r$. In particular, $\Msp^{\Betti}_{g,r,d}$ is the affinization of $\FM^{\Betti}_{g,r,d}$, and the affinization morphism $\JH\colon \FM^{\Betti}_{g,r,d}\rightarrow \CM^{\Betti}_{g,r,d}$ is a good moduli space.

\subsubsection{The dg-category of representations of the fundamental group of a surface}
\label{subsubsection:dgrepfundgroup}
Recall that a topological space $M$ is called \textit{acyclic} if its universal cover is contractible.
\begin{theorem}[{\cite[Corollary 6.2.4]{davison2012superpotential}}]
 Let $M$ be a compact orientable acyclic manifold of dimension $d$ without boundary. Then the fundamental group algebra $\BoC[\pi_1(M)]$ is a (left) $d$-Calabi--Yau algebra.
 \end{theorem}
 As an immediate corollary, we obtain the following.

 \begin{corollary}
 \label{RS_CY}
 For any $g\geq 1$, the group algebra $\BoC[\pi_1(\Sigma_g)]$ of the fundamental group of a Riemann surface of genus $g$,
\[
 \pi_1(\Sigma_g)=\left\langle x_i,y_i,1\leq i\leq g\mid \prod_{i=1}^gx_iy_ix_i^{-1}y_i^{-1}=1\right\rangle
\]
 is a (left) $2$-Calabi--Yau algebra.
 \end{corollary}
\begin{proof}
 Such a surface is compact, orientable, and acyclic.
\end{proof}

 The cohomological Hall algebra of representations of the fundamental group algebra of a Riemann surface was defined and studied in terms of brane tiling algebras in \cite{davison2016cohomological}, where it was shown that it satisfies the cohomological integrality theorem (regardless of genus). It is shown in \cite{mistry2022cohomological} that the CoHA defined this way agrees with the more direct construction analogous to the Schiffmann--Vasserot definition \cite[§4.3]{schiffmann2013cherednik} of the CoHA of compactly supported sheaves on $\BoA^2$, and thus (adapting the arguments of \S \ref{subsection:comparison_preproj}) to the CoHA defined in this paper in terms of the RHom complex.  A categorified version of this Hall algebra appears also in \cite{porta2022two}. 

The analogue of Corollary \ref{RS_CY} is expected to be true for $\BoC[\pi_1(\Sigma_g\setminus\{p\}),\lambda]$ with $g\geq 1$, although we have not been able to locate a precise reference. We use instead the fact that these algebras are localizations of multiplicative preprojective algebras, which are known to be $2$-Calabi--Yau algebras \cite{kaplan2023multiplicative}.

Let $Q=(Q_0,Q_1)=S_g$ be the quiver with one vertex and $g$ loops $\alpha_1,\hdots,\alpha_g$. The doubled quiver $\overline{Q}$ has arrows $\alpha_i,\alpha_i^*$, $1\leq i\leq g$. Let $A$ be the universal localisation of $\BoC \overline{Q}$ with respect to the elements $1+\alpha_i\alpha_i^*$ and $1+\alpha_i^*\alpha_i$ for $1\leq i\leq g$. We let
\[
 \Lambda^{\lambda}(Q)\coloneqq A\Big/\left\langle\lambda^{-1}\prod_{i=1}^g(1+\alpha_i\alpha_i^*)(1+\alpha_i^*\alpha_i)^{-1}-1\right\rangle
\]
for $\lambda\in\BoC^*$ be the corresponding multiplicative preprojective algebra. Since we have an explicit presentation of $\Lambda^{\lambda}(Q)$ as a quotient of a localised quiver algebra by a set of relations, we may define the derived multiplicative preprojective algebra $\tilde{\Lambda}^{\lambda}(Q)$ as in Appendix \ref{Dalg_sec}. This is the same as the derived multiplicative preprojective algebra defined in \cite{kaplan2023multiplicative}. The stack of finite-dimensional representations of $\tilde{\Lambda}^{\lambda}(Q)$ satisfies Assumptions \ref{p_assumption}-\ref{ds_fin}, \ref{det_bun_assumption}, \ref{assumption:associativity} by Appendix \ref{algebra_constr_sec}, and so all that remains is to consider the 2CY property (Assumption~\ref{BPS_cat_assumption}).

We let $\Lambda^{\lambda}(Q)'$ be the universal localisation of $\Lambda^{\lambda}(Q)$ at $\alpha_i$, $1\leq i\leq g$.

\begin{proposition}[{\cite[Proposition 2]{crawley2013monodromy}}]
\label{proposition:fundamentalgrouppreprojec}
 We have an isomorphism of algebras
 \[
 \BoC[\pi_1(\Sigma_g\setminus\{p\}),\lambda]\cong \Lambda^{\lambda}(Q)'.
\]
\end{proposition}
\begin{proof}
 Crawley-Boevey only states that these two algebras are Morita equivalent, i.e. that the category $\Rep(\BoC[\pi_1(\Sigma_g\setminus\{p\}),\lambda])$ is equivalent to the category $\Rep(\Lambda^{\lambda}(Q)')$. But the arguments of his proof give an isomorphism between the algebras. The isomorphism sends $x_i$ to $\alpha_i$ and $y_i$ to $\alpha_i^{-1}+\alpha_i^*$.
\end{proof}

Thanks to this proposition, we can realise the stack of finite-dimensional representations of $\BoC[\pi_1(\Sigma_g\setminus \{p\}),\lambda]$ as an open substack of the stack of finite-dimensional representations of $\Lambda^{\lambda}(Q)$. Letting $\mathscr{D}=\Perf_{\dg}(\Lambda^{\lambda}(Q)')$ be the dg-category of perfect $\Lambda^{\lambda}(Q)'$-modules, we have a fully faithful localisation functor
\[
 \mathscr{D}\rightarrow \mathscr{C},
\]
and $\mathscr{D}$ is quasi-equivalent to $\Perf_{\dg}(\BoC[\pi_1(\Sigma_g\setminus\{p\}),\lambda])$. We can therefore rely on the favourable properties of the multiplicative preprojective algebra $\Lambda^{\lambda}(Q)$:
\begin{proposition}[{\cite[Theorem 1.2+Proposition 4.4]{kaplan2023multiplicative}}]
$\Lambda^{\lambda}(Q)$ is a (left) $2$-Calabi--Yau algebra, and the natural morphism of dgas $\tilde{\Lambda}^{\lambda}(Q)\rightarrow \Lambda^{\lambda}(Q)$ is a quasi-isomorphism.
\end{proposition}
By \cite{brav2019relative} it follows that the full subcategory containing any collection of simple $\Lambda^{\lambda}(Q)$-modules carries a right 2CY structure, and so the category of $\Lambda^{\lambda}(Q)$-modules satisfies Assumption \ref{BPS_cat_assumption}. Restricting to collections of simple objects arising from $\BoC[\pi_1(\Sigma_g\setminus\{p\}),\lambda]$-modules via Proposition \ref{proposition:fundamentalgrouppreprojec} we deduce that the category of $\BoC[\pi_1(\Sigma_g\setminus\{p\}),\lambda]$-modules also satisfies Assumption \ref{BPS_cat_assumption}.

\subsubsection{The RHom-complex}
We have an explicit projective resolution of the multiplicative preprojective algebra $\Lambda^{\lambda}(Q)$ as a $\Lambda^{\lambda}(Q)$-bimodule and it gives an explicit presentation of the RHom-complex (\S\ref{subsubsection:categorical}) on the square $\mathfrak{M}_{\Lambda^{\lambda}(Q)}\times\mathfrak{M}_{\Lambda^{\lambda}(Q)}$ of the stack of finite-dimensional representations of $\Lambda^{\lambda}(Q)$. By considering the fully faithful embedding $\mathscr{D}\rightarrow\mathscr{C}$, the RHom complex on $\mathfrak{M}^{\Betti}_{g,r,d}\times\mathfrak{M}^{\Betti}_{g,r,d}$ is the restriction to $\mathfrak{M}^{\Betti}_{g,r,d}\times\mathfrak{M}^{\Betti}_{g,r,d}$ of the RHom-complex on $\mathfrak{M}_{\Lambda^{\lambda}(Q)}\times\mathfrak{M}_{\Lambda^{\lambda}(Q)}$. The latter can be described explicitly thanks to the $3$-term projective resolution of $\Lambda^{\lambda}(Q)$ as a $\Lambda^{\lambda}(Q)$-bimodule, $P_{\bullet}=P_0\xrightarrow{\alpha}P_1\xrightarrow{\beta} P_0$ given in \cite[Proposition 3.12]{kaplan2023multiplicative} (since $Q=S_g$ contains a cycle, see \cite{kaplan2023multiplicative}). 
It is the same as the 3-term complex from Corollary \ref{dalg_dhom} calculating $\Ext^i_{\tilde{\Lambda}^{\lambda}(Q)}(M,N)$.

\subsubsection{The Euler form}
By Proposition \ref{proposition:fundamentalgrouppreprojec} and the discussion following it (\S\ref{subsubsection:dgrepfundgroup}), we have a fully faithful functor from the derived category of representations of the deformed fundamental group algebra $\BoC[\pi_1(\Sigma_g\setminus\{p\}),\lambda]$ to the derived category of representations of the multiplicative preprojective algebra $\Lambda^{\lambda}(Q)$.

We determine the Euler form of the multiplicative preprojective algebra. 

\begin{lemma}
The Euler form of $\Rep(\Lambda^{\lambda}(Q))$ is given by
\begin{equation}
\label{equation:eulerformfundamentalgroup}
 \begin{split}
 \BoZ\times\BoZ&\longrightarrow\BoZ\\
 (d,e)&\longmapsto 2(1-g)de
 \end{split}
\end{equation}
i.e. for finite dimensional representations $M,N$ of $\Lambda^{\lambda}(Q)$, $(M,N)_{\Lambda^{\lambda}(Q)}=2(1-g)\dim(M)\dim(N)$.
\end{lemma}
\begin{proof}
This follows from Corollary \ref{dalg_dhom}.
\end{proof}

\begin{corollary}
 The Euler form of the category of finite-dimensional representations of the deformed fundamental group algebra $\BoC[\pi(\Sigma_g\setminus\{p\}),\lambda]$ is given by \eqref{equation:eulerformfundamentalgroup}, i.e. for finite dimensional representations $M,N$ of $\BoC[\pi(\Sigma_g\setminus\{p\}),\lambda]$, $(M,N)_{\BoC[\pi(\Sigma_g\setminus\{p\}),\lambda]}=2(1-g)\dim(M)\dim(N)$. In particular, if $g\geq 2$ this category is a totally negative 2CY category.
\end{corollary}
\begin{proof}
This is an immediate consequence of the fully faithful embedding $\mathscr{D}\rightarrow\mathscr{C}$ which induces a fully faithful embedding of the triangulated categories $\D^{\bdd}(\Rep(\BoC[\pi_1(\Sigma_g\setminus\{p\}),\lambda]))\rightarrow \D^{\bdd}(\Rep(\Lambda^{\lambda}(Q)))$.
\end{proof}

\subsection{Semistable Higgs bundles}
\label{Higgs_background_sec}

\subsubsection{The stack of Higgs sheaves}
Let $C$ be a complex smooth projective curve. A \emph{Higgs sheaf} on $C$ is a pair $(\mathcal{F},\theta)$ of a coherent sheaf $\mathcal{F}$ on $C$ together with an $\mathcal{O}_C$-linear map $\theta\colon\mathcal{F}\rightarrow \mathcal{F}\otimes K_C$ (the \emph{Higgs field}) where $K_C$ is the canonical bundle of $C$. The rank $r$ and degree $d$ of a Higgs sheaf are defined to be the rank and degree of the underlying coherent sheaf, while the slope $\mu(\mathcal{F},\theta)=\mu(\mathcal{F})$ is likewise defined to be $d/r$. We let $\BoZ^{2,+}=\{(r,d)\in \BoZ^2\mid r>0 \text{ or }(r=0 \text{ and }d\geq 0)\}$. We let $\Higgs(C)$ be the (Abelian) category of Higgs sheaves on $C$ and let $\mathfrak{Higgs}(C)=\bigsqcup_{(r,d)\in\BoZ^{2,+}}\mathfrak{Higgs}_{(r,d)}(C)$ be the stack of Higgs sheaves on $C$.

\subsubsection{Semistable Higgs sheaves}
Let $(\mathcal{F},\theta)$ be a Higgs sheaf. It is called semistable if for any subsheaf $\mathcal{G}\subset \mathcal{F}$ such that $\theta(\mathcal{G})\subset \mathcal{G}\otimes K_C$, we have the inequality of slopes
\[
 \mu(\mathcal{G})=\frac{\mathrm{deg}(\mathcal{G})}{\mathrm{rank}(\mathcal{G})}\leq \mu(\mathcal{F})=\frac{\mathrm{deg}(\mathcal{F})}{\mathrm{rank}(\mathcal{F})}.
\]
For $\theta\in \BoQ\cup\{\infty\}$, we denote by $\Higgs^{\sst}_{\theta}(C)$ the full subcategory of the category of Higgs sheaves, containing those Higgs sheaves that are semistable of slope $\theta$, along with the zero Higgs sheaf. It is an Abelian subcategory of $\Higgs(C)$, the category of all Higgs sheaves. The stack of semistable Higgs sheaves of rank $r$ and degree $d$, $\Mst^{\Dol}_{r,d}(C)$, is an open substack of $\mathfrak{Higgs}(C)$. It is of finite type and can be realised as a global quotient stack. We spell this out in the next section.

\subsubsection{The BNR correspondence}
\label{subsubsection:BNRcorrespondence}
We may consider the category of Higgs sheaves as a subcategory of the category of coherent sheaves on a quasiprojective surface via the BNR-correspondence \cite{beauville1989spectral,schaub1998courbes}. Let $S=\Tan^*\!C$ and $\overline{S}=\mathbf{P}_C(\mathcal{O}_C\oplus K_C)$. The projective surface $\overline{S}$ is a smooth compactification of $S$. We let $D_{\infty}$ be the complement of the open immersion $S\subset\overline{S}$. We let $\pi\colon \Tan^*\!C\rightarrow C$ and $\overline{\pi}\colon \overline{S}\rightarrow C$ be the natural projections.

There is an equivalence of categories between the category of coherent sheaves $\mathcal{F}$ on $S$ for which $\pi_*\mathcal{F}=\mathrm{R}^0\pi_*\mathcal{F}$ is coherent and the category of Higgs sheaves on $C$. Therefore, the category of Higgs sheaves on $C$ is equivalent to the category of coherent sheaves on $\overline{S}$ whose support does not intersect $D_{\infty}$. Furthermore, semistability of the corresponding Higgs bundle is equivalent to Gieseker-semistability for the polarization of $\overline{S}$ given by an arbitrary choice of very ample class. Via the BNR correspondence, we consider $\Mst^{\Dol}_{r,d}(C)$ as an open substack of the stack of semistable coherent sheaves on $\overline{S}$. Likewise, we consider the coarse moduli space $\CM^{\Dol}_{r,d}(C)$ as an open subscheme of the coarse moduli scheme of semistable sheaves on $\overline{S}$.

\subsubsection{A dg-enhancement of the category of Higgs sheaves}
We let $\CD^{\mathrm{b}}_{\dg}(\Coh(\overline{S}))$ be the dg-enhancement of the derived category of coherent sheaves on $\overline{S}$ (classically constructed as the dg-quotient of the pretriangulated dg-category of complexes in $\Coh(\overline{S})$ by the full pretriangulated dg-subcategory of acyclic complexes). We define $\CD_{\dg}^{\mathrm{b}}(\Higgs(C))$ as the full pretriangulated dg-subcategory of $\CD^{\mathrm{b}}_{\dg}(\Coh(S))$ of bounded complexes whose cohomology sheaves are coherent after applying $\pi_*$. 
\begin{proposition}
Fix a slope $\theta\in\BoQ$. The category $\Higgs^{\sst}_{\theta}(C)$, considered as a subcategory of the dg-category $\CD_{\dg}^{\mathrm{b}}(\Higgs(C))$, along with its stack of objects, satisfy Assumptions \ref{p_assumption}-\ref{det_bun_assumption}.
\end{proposition}
\begin{proof}
This follows, via the BNR correspondence, from the statement for semistable coherent sheaves on $S$, which is established in Appendix \ref{geometry_constr_sec} (using that $K_S\cong \CO_S$ for Assumption \ref{BPS_cat_assumption}).
\end{proof}

\subsubsection{The Euler form}
 
The Euler form of the category of Higgs sheaves on a smooth projective curve factors through the numerical Grothendieck group of the curve
\begin{equation*}
\begin{split}
\K(\Higgs(C))&\longrightarrow \BoZ^2\\
 \overline{\mathcal{F}}=[(\mathcal{F},\theta)]&\longmapsto(\rank(\mathcal{F}),\deg(\mathcal{F})).
\end{split}
\end{equation*}
It is given by
\begin{equation}
\label{equation:EulerformHiggs}
 (\overline{\mathcal{F}},\overline{\mathcal{G}}):=\sum_{i\in\BoZ}(-1)^i\ext^i(\overline{\mathcal{F}},\overline{\mathcal{G}})=2(1-g)\mathrm{rank}(\overline{\mathcal{F}})\mathrm{rank}(\overline{\mathcal{G}}).
\end{equation}

\begin{proposition}
\label{g_t_n}
 For any $g\geq 2$, the Abelian category $\Higgs^{\sst}_{\theta}(C)$ of semistable Higgs bundles of fixed finite slope $\theta$ on a smooth projective curve of genus $\geq 2$ is totally negative.
\end{proposition}
\begin{proof}
Immediate from \eqref{equation:EulerformHiggs}.
\end{proof}

\subsection{Semistable sheaves on symplectic surfaces}
\label{SSSS}
Let $S$ be a smooth projective symplectic surface, i.e. $S$ is a K3 or Abelian surface. Let $\CO_S(1)$ be an ample line bundle on $S$. 
We explain the usual setup for moduli spaces on such a surface.
We endow $\HO^*(S,\BoZ)$ with the quadratic form 
\begin{equation*}
(v_0,v_1,v_2)(w_0,w_1,w_2) = v_1w_1 - v_0w_2 - v_2w_0\text{, for $v_i,w_i \in \HO^{2i}(S,\BoZ)$.}
\end{equation*}
The Mukai vector of a coherent sheaf $\CF$ is defined to be 
\begin{equation*}
v(\CF) = (\rank(\CF),\cc_1(\CF),\ch_2(\CF)+\rank(\CF)) = \ch(\CF) \cdot \sqrt{\todd(S)} \in \HO^*(S,\BoZ).
\end{equation*}
For any two coherent sheaves $\CE,\CF \in \Coh(S)$ the Euler form is given by the Riemann--Roch formula
\begin{equation}
\label{eq:EFsurface}
\chi(\CE,\CF) = - v(\CE)v(\CF) = -\cc_1(\CE)\cc_1(\CF) + 2\rank(\CE)\rank(\CF) + \rank(\CE)\ch_2(\CF) + \ch_2(\CE)\rank(\CF).
\end{equation}

Thus the Mukai vector $v=v(\CF)$ of a coherent sheaf $\CF$ determines its Hilbert polynomial which we denote by $P_{v}(t)$. As in §\ref{subsection:propernessp} we define Gieseker semistable sheaves on $S$ with respect to $\CO_S(1)$. Let $\Cohst^{\sst}_{v}(S) \subset \Cohst^{\sst}_{P_v(t)}(S)$ be the open and closed substack of Gieseker semistable coherent sheaves on $S$ with Mukai vector $v$.

Let $v_0 \in \HO^*(S,\BoZ)$ be a primitive Mukai vector and suppose that $\CO_S(1)$ is $v_0$-generic (i.e. such that all Gieseker semistable sheaves with Mukai vector $v_0$ are automatically stable). Let $\Coh^{\sst}(S,v_0) \subset \Coh(S)$ be the full Abelian subcategory of pure dimension 1 Gieseker semistable sheaves with Mukai vector a multiple of $v_0$. 
We call $\Coh^{\sst}(S,v_0)$ the category of \emph{one dimensional semistable sheaves of slope $v_0$}.
Since we demand sheaves in $\Coh^{\sst}(S,v_0)$ to be pure of dimension one, the category $\Coh^{\sst}(S,v_0)$ can be nonzero only if $v_0= (0,\beta,\chi)$ for $0 \neq \beta \in \HO^2(S)$.
\begin{proposition}
Let $v_0 \in \HO_{\leq 2}(S)$ (where we have identified $\HO_i(S)$ with $\HO^{4-i}(S)$ via Poincar\'e duality) be a Mukai vector for which $\CO_S(1)$ is $v_0$-generic. Suppose $v_0^2 > 0$, then the category of one dimensional semistable sheaves of slope $v_0$ is totally negative.
\end{proposition}
An example of such a Mukai vector $v_0$ is constructed as follows. Suppose the general smooth curve in the linear system $\abs{\CO_{S}(1)}$ has genus $g\geq 2$. Then the Mukai vector $v_0=(0,[C],\chi)$ for a smooth curve $C \in \abs{\CO_S(1)}$ and $\chi \in \BoZ$ satisfies $v_0^2 = 2g-2 > 0.$

\section{Cohomological Hall algebras of 2-dimensional categories}
\label{subsection:relativecohaproductKV} 

\subsection{The CoHA product}
\label{subsubsection:ThecohaproductKV}
In this section, we define the CoHA product, generalising the one given in \cite{kapranov2019cohomological}. We let $\mathscr{C}$, $\mathcal{A}$, $\varpi\colon \mathfrak{M}\rightarrow \mathcal{M}$ be as in \S \ref{subsubsection:categorical}. Then we define as in the introduction
\begin{align*}
\mathscr{A}_{\varpi}&\coloneqq\bigoplus_{\CM_a\in\pi_0(\CM)}\varpi_*\BD\BoQ_{\mathfrak{M}_{\mathcal{A},a}}[(a,a)_{\mathscr{C}}]\in\CD^+_{\mathrm{c}}(\CM)
\\
\ulrelCoHA_{\varpi}&\coloneqq \bigoplus_{\CM_{a} \in \pi_{0}(\CM)} \varpi_{\ast} \BD \ul{\BoQ}_{\Mst_{\CA,a}} \otimes \BoL^{-(a,a)_{\SC}/2}\in\CD^+(\MHM(\CM))
.
\end{align*}
We assume that the assumptions stated in \S \ref{subsubsection:assumptionCoHAproduct}, which we briefly recall, are satisfied. Firstly, we require that the forgetful morphism $p\colon \mathfrak{Exact}\rightarrow \FM$ from the stack of short exact sequences of objects parametrised by points of $\FM$, to $\FM$, which remembers only the middle term in the sequence, is representable and proper (Assumption \ref{p_assumption}).
 Secondly, we assume that the forgetful morphism $q\colon \mathfrak{Exact} \rightarrow \FM \times \FM$ is the classical truncation of the total space of a $3$-term complex (Assumption \ref{q_assumption1}).

Pick $a,b\in\pi_0(\CM)$ and let $a+b\in \pi_0(\CM)$ denote the connected component containing $\pi_0(\oplus)(a,b)$. By Assumption \ref{q_assumption1} we can present the morphism $q\colon \mathfrak{Exact}_{a,b}\rightarrow \FM_a\times \FM_b$ as the classical truncation of the total space of a $3$-term complex
\[
\Tot_{\FM_a\times\FM_b}(\CC_{a,b}^{\bullet})\rightarrow \FM_a\times \FM_b.
\]
For any $a,b\in\pi_0(\FM_{\CA})$, we fix a $3$-term complex $\CC_{a,b}^{\bullet}$ giving a presentation of $q$. We explain in Corollary~\ref{der_exact_cor} a canonical choice when $\CA$ is the category of representation of an algebra and we explain a set of choices when $\CA$ is the category of coherent sheaves on a surface. The main point is that these choices give an unambiguous definition of the virtual pullback by the morphism $q$.

Recall that $\mathcal{C}^{\bullet}_{a,b}$ is quasi-isomorphic to the RHom complex shifted by one on $\mathfrak{M}_a\times\mathfrak{M}_b$ for $a,b\in\pi_0(\mathfrak{M}_{\mathcal{A}})$. As such we have the equality
\[
(a,b)_{\mathscr{C}}=-\vrank(\CC^{\bullet}_{a,b}).
\]
We recall that as part of our starting data (\S\ref{base_monoid_sec}) we have the commutative diagram
\begin{equation}
 \begin{tikzcd}
	{\mathfrak{M}_a\times\mathfrak{M}_b} & {t_0(\Tot_{\FM\times\FM}(\mathcal{C}^{\bullet}_{a,b}))\simeq \mathfrak{Exact}_{a,b}} & {\mathfrak{M}_{a+b}} \\
	{\mathcal{M}_a\times\mathcal{M}_b} && {\mathcal{M}_{a+b}}.
	\arrow["{\varpi_a\times\varpi_b}"', from=1-1, to=2-1]
	\arrow["{\varpi_{a+b}}", from=1-3, to=2-3]
	\arrow["{p_{a,b}}", from=1-2, to=1-3]
	\arrow["{ q_{a,b}}"', from=1-2, to=1-1]
	\arrow["\oplus"', from=2-1, to=2-3]
\end{tikzcd}
\end{equation}

We define the morphism
\begin{equation}
\label{vfornow}
(\varpi_a\times\varpi_b)_*v_{\CC_{a,b}}\colon(\varpi_a\times\varpi_b)_* \BD \ulBoQ_{\FM_a\times\FM_b}\rightarrow (\varpi_a\times\varpi_b)_*(q_{a,b})_*\BD \ulBoQ_{\mathfrak{Exact}_{a,b}} \otimes \BoL^{\vrank(\CC^{\bullet}_{a,b})}
\end{equation}
as the virtual pullback along $q_{a,b}$ defined in \S \ref{subsection:virtual-pb-equivariant} for the chosen presentation by a $3$-term complex $\CC_{a,b}^{\bullet}$. Using Assumption \ref{p_assumption} (properness of $p$) we define the morphism
\[
(\varpi_{a+b})_*\beta\colon (\varpi_{a+b})_*(p_{a,b})_*\BD \ulBoQ_{\mathfrak{Exact}_{a,b}}\rightarrow (\varpi_{a+b})_*\BD \ulBoQ_{\FM_{a+b}}
\]
as the direct image of the pullback\footnote{Recall that we pull back to a scheme smoothly approximating the stack so that this morphism has a lift to MHM complexes.} to an acyclic cover of the Verdier dual of the adjunction $\ulBoQ_{\FM_{a+b}}\rightarrow (p_{a,b})_* \ulBoQ_{\mathfrak{Exact}_{a,b}}$.
 We define the relative CoHA multiplication on $\ulrelCoHA_{\varpi}$, restricted to $\CM_a\times \CM_b$, to be given by
\begin{align*}
m_{a,b}=&((\varpi_{a+b})_*\beta \otimes \BoL^{(-(a+b,a+b)_{\mathscr{C}}/2})\circ (\oplus_*(\varpi_a\times\varpi_b)_*v_{\CC_{a,b}}\otimes \BoL^{(-(a,a)_{\mathscr{C}}-(b,b)_{\mathscr{C}})/2}).
\end{align*}

\subsubsection{Associativity}

Given how general our context is regarding the category $\CA$, we take the associativity as an assumption.
\begin{assumption}{7}
\label{assumption:associativity}
We assume that the multiplication $m_{a,b}$ defined on $\underline{\SA}_{\varpi}$ is associative, that is, for any $a,b,c\in\pi_0(\CM)$, the morphisms $m_{a+b,c}\circ (m_{a,b}\boxdot\id)$ and $m_{a,b+c}\circ (\id\boxdot m_{b,c})$ coincide.
\end{assumption}
We verify this assumption for the stacks of interest in this paper, which are either the stack of representations of an algebra of  homological dimension $2$, or the stack of compactly supported coherent sheaves over a smooth quasi-projective surface, in Appendices~\ref{sec:2dim_categories} and \ref{algebra_constr_sec}.

\subsubsection{CoHA associated to a submonoid}
We assume, as in \S \ref{strict_mf_sec} that we are given an inclusion of a saturated submonoid $\imath\colon(\mathcal{N},\oplus)\rightarrow (\mathcal{M},\oplus)$ so that the diagram \eqref{equation:diagrampullbackmonoidal} is Cartesian. Then by Lemma \ref{lemma:strictmonoidalfunctor}, the complex of mixed Hodge modules $\imath^!\ulrelCoHA_{\varpi}$ inherits an algebra structure from the algebra structure on $\ulrelCoHA_{\varpi}$. Precisely, we define the morphism
\[
\imath^!\ulrelCoHA_{\varpi}\boxdot\imath^!\ulrelCoHA_{\varpi}\rightarrow \imath^!\ulrelCoHA_{\varpi}
\]
as the composition of the natural isomorphism
\[
\imath^!\ulrelCoHA_{\varpi}\boxdot\imath^!\ulrelCoHA_{\varpi}\cong \imath^!(\ulrelCoHA_{\varpi}\boxdot\ulrelCoHA_{\varpi})
\]
and $\iota^!m$.

\subsection{A hierarchy of cohomological Hall algebras associated to preprojective algebras}
\label{subsection:hierarchy}
It is often profitable to consider the Hall algebra on the Borel--Moore homology of the stack of representations in a Serre subcategory of the category of representations for a given algebra. This turns out to be especially true for preprojective algebras, for which we will prove our main theorem (Theorem \ref{theorem:freenesspreprojective}) by considering three different Serre subcategories.
\subsubsection{The full cohomological Hall algebra}
We obtain the \emph{full cohomological Hall algebra} by taking derived global sections.
\[
\HO^*\!\!\!\ulrelCoHA_{\Pi_Q}\coloneqq \HO^*\!(\JH_*\BD\ulBoQ_{\mathfrak{M}_{\mathcal{A}}}^{\vir})\cong \HO^{\BoMo}_*\!(\Mst_{\Pi_Q},\ulBoQ^{\vir}).
\]
In this case the Serre subcategory of the category of finite-dimensional $\Pi_Q$-representations is the entire category. This special case is the subject of Corollary \ref{corollary:BPSalgebraFree}.

\subsubsection{The cohomological Hall algebra of the strictly seminilpotent locus}
\label{subsubsection:cohaSSN}
Let $\mathcal{M}^{\SSN}_{\Pi_Q}$ be the closed subvariety of $\mathcal{M}_{\Pi_Q}$ parametrising semisimple representations of $\Pi_Q$ such that the only arrows acting possibly non-trivially are the loops $\alpha\in Q_1\subset \overline{Q}_1=Q_1\sqcup Q_1^*$.

Then, we have a pullback square
\[
 \begin{tikzcd}
	{\mathfrak{M}^{\SSN}_{\Pi_Q}} & {\mathfrak{M}_{\Pi_Q}} \\
	{\mathcal{M}^{\SSN}_{\Pi_Q}} & {\mathcal{M}_{\Pi_Q}}
	\arrow["\JH_{\Pi_Q}^{\SSN}"',from=1-1, to=2-1]
	\arrow[from=1-1, to=1-2]
	\arrow["\imath_{\SSN}",from=2-1, to=2-2]
	\arrow["\JH_{\Pi_Q}",from=1-2, to=2-2]
	\arrow["\lrcorner"{anchor=center, pos=0.125}, draw=none, from=1-1, to=2-2]
\end{tikzcd}
\]
where $\mathfrak{M}^{\SSN}_{\Pi_Q}$ parametrises strongly seminilpotent representations of $\Pi_Q$. It is called the \emph{strongly seminilpotent stack} and its $\BoC$-points parametrise the strongly seminilpotent representations of $\Pi_Q$. These are representations $M$ of $\Pi_Q$ such that there exists a flag $0=M_0\subset M_1\subset \hdots\subset M_r=M$ with $M_i/M_{i-1}$ supported at a single vertex for $1\leq i\leq r$, $x_{\alpha}M_i\subset M_i$ and $x_{\alpha^*}M_i\subset M_{i-1}$ for $\alpha\in Q_1$.

The algebra structure $\imath_{\SSN}^!m$ obtained on $\ulrelCoHA_{\Pi_Q}^{\SSN}:=(\JH_{\Pi_Q}^{\SSN})_*\BD\ulBoQ_{\mathfrak{M}_{\Pi_Q}^{\SSN}}^{\vir}=\imath_{\SSN}^!(\JH_{\Pi_Q})_*\BD\ulBoQ_{\mathfrak{M}_{\Pi_Q}}^{\vir}$ is the relative strongly seminilpotent Hall algebra. The strongly seminilpotent Hall algebra $\HO^*\!\!\ulrelCoHA_{\Pi_Q}^{\SSN}$, is obtained by applying the derived global sections functor to the algebra object $\ulrelCoHA_{\Pi_Q}^{\SSN}$.

This algebra plays a crucial role in our proofs; we use in an essential way that the cohomological degree zero piece of this algebra is the enveloping algebra of a generalised Kac--Moody algebra (via \cite{bozec2016quivers} and then \cite{hennecart2022geometric} for the translation into cohomology), and we reduce the general statement regarding the BPS algebra to this fact.

\subsubsection{The cohomological Hall algebra of the nilpotent locus}
\label{subsubsection:fullynilpotentCoHA}
Let $\CM_{\Pi_Q}^{\nil}$ be the closed (discrete) subvariety of $\CM_{\Pi_Q}$ parametrising semisimple representations of $\Pi_Q$ such that all arrows act via the zero morphism. I.e. for each dimension vector $\dd$ the inclusion $\iota_{\Nil,\dd}\colon \CM_{\Pi_Q,\dd}^{\nil}\hookrightarrow \CM_{\Pi_Q,\dd}$ is a single point, corresponding to the nilpotent module $S_{\dd}$. Then we define the \textit{fully nilpotent CoHA} of $\Pi_Q$
\[
\ulrelCoHA_{\Pi_Q}^{\nil}\coloneqq \HO^*(\CM_{\Pi_Q}^{\nil},\iota_{\Nil}^!\ulrelCoHA_{\Pi_Q}).
\]
This algebra also plays a key role in our proofs; although $\Nil\coloneqq\{S_i \in\Pi_Q\lmod\mid i\in Q_0\}$ is a specific $\Sigma$-collection, this CoHA models the CoHA of an \textit{arbitrary} $\Sigma$-collection: see Corollary \ref{corollary:CoHAcompatibiltiySigmacoll} for the precise statement. This is as close as we come in this paper to upgrading the analytic neighbourhood theorem (Theorem \ref{theorem:neighbourhood}) to a statement incorporating Hall algebra products, and is as close as we need to come in order to prove our main theorems.

\subsection{The cohomological algebra of a $\mathbf{\Sigma}$-collection}
\label{subsubsection:cohaSigmacollection}
Let $\mathscr{C}$, $\mathcal{A}$ be as at the start of \S \ref{subsubsection:ThecohaproductKV}, and assume moreover that Assumption \ref{gms_assumption} (existence of a good moduli space) holds, so we may take $\JH\colon\FM_{\mathcal{A}}\rightarrow \CM_{\mathcal{A}}$ for our morphism $\varpi$.  Let $\underline{x}=\{x_1,\hdots,x_r\}$ be pairwise distinct closed $\BoC$-points of $\mathcal{M}$ represented by a set $\underline{\mathcal{F}}=\{\mathcal{F}_1,\hdots,\mathcal{F}_r\}$ of simple objects of $\mathcal{A}$, where the full subcategory containing $\CF_1,\ldots,\CF_r$ carries a right 2CY structure, and thus these objects form a $\Sigma$-collection. Let $\mathcal{M}_{\underline{\CF}}\subset\mathcal{M}_{\mathcal{A}}$ be the submonoid generated by $\underline{\CF}$. We \textit{define} $\FM_{\underline{\CF}}$ via the Cartesian square
\[
\begin{tikzcd}
	{\mathfrak{M}_{\underline{\CF}}} & {\mathfrak{M}_{\mathcal{A}}} \\
	{\mathcal{M}_{\underline{\CF}}} & {\mathcal{M}_{\mathcal{A}}}
	\arrow["{\JH}", from=1-2, to=2-2]
	\arrow["{\imath_{\underline{\CF}}}"', from=2-1, to=2-2]
	\arrow["{\JH^{\underline{\CF}}}"', from=1-1, to=2-1]
	\arrow[from=1-1, to=1-2]
	\arrow["\lrcorner"{anchor=center, pos=0.125}, draw=none, from=1-1, to=2-2].
\end{tikzcd}
\]
The stack $\FM_{\underline{\CF}}$ parametrises objects of $\mathcal{A}$ whose simple subquotients are in $\underline{\mathcal{F}}$.

The cohomological Hall algebra of $\mathfrak{M}_{\underline{\CF}}$ is defined to be $\ulrelCoHA_{\underline{\CF}}:=(\JH^{\underline{\CF}})_*\BD\ulBoQ_{\mathfrak{M}_{\underline{\CF}}}\otimes \BoL^{-\chi/2}\cong \imath_{\underline{\CF}}^!(\JH)_*\BD\ulBoQ^{\vir}_{\mathfrak{M}_{\mathcal{A}}}$, with the multiplication $\imath_{\underline{\CF}}^!m$.

Let $Q=(Q_0,Q_1)$ be a quiver and let $\Pi_Q$ be its preprojective algebra. We let $\ul{\CN}$ be the $\Sigma$-collection of one-dimensional nilpotent $\Pi_Q$-modules.  We recover \emph{the fully nilpotent cohomological Hall algebra} of $\Pi_Q$ defined in \S\ref{subsubsection:fullynilpotentCoHA}: $\ulrelCoHA_{\Pi_Q}^{\nil}=\ulrelCoHA_{\underline{\CN}}$.

We now identify the cohomological Hall algebra of a $\Sigma$-collection.

\begin{lemma}[{\cite[Proof of Theorem 5.11]{davison2021purity}}] 
We denote by $\mathscr{C}$ the ambient dg category as in §\ref{subsubsection:categorical}.
 Let $\underline{\mathcal{F}}=\{\mathcal{F}_1,\hdots,\mathcal{F}_r\}$ be a $\Sigma$-collection in $\mathscr{C}$, and assume that the full subcategory of $\mathscr{C}$ containing $\CF_1,\ldots,\CF_r$ carries a right 2CY structure. Let $Q_{\underline{\mathcal{F}}}$ be a half of the Ext-quiver of $\mathcal{\underline{F}}$. Then the full dg-subcategory of $\mathscr{C}$ generated under extensions by $\underline{\mathcal{F}}$ is quasi-equivalent to the full dg-subcategory of $\CD^{\bdd}_{\dg}(\mathrm{mod}^{\mathscr{G}_2(Q_{\underline{\mathcal{F}}})})$ generated under extensions by nilpotent one-dimensional modules.
\end{lemma}

\begin{corollary}
\label{corollary:comparisonCoHAfiber}
 Let $\underline{\mathcal{F}}=\{\mathcal{F}_1,\hdots,\mathcal{F}_r\}$ be a $\Sigma$-collection in $\mathcal{A}$, and assume that the full subcategory of $\mathscr{C}$ containing $\CF_1,\ldots,\CF_r$ carries a right 2CY structure. Let ${\mathfrak{M}}_{\underline{\mathcal{F}}}$ be the closed substack of ${\mathfrak{M}}_{\mathcal{A}}$ parametrising objects whose simple subquotients are in $\underline{\mathcal{F}}$ (\S\ref{subsubsection:cohaSigmacollection}). Let $Q$ be a half of the Ext-quiver of $\underline{\mathcal{F}}$. Let ${\mathfrak{M}}_{\Pi_Q}^{\nil}$ be the closed substack of ${\mathfrak{M}}_{\Pi_Q}$ parametrising nilpotent representations of $\Pi_Q$ (\S\ref{subsubsection:cohaSigmacollection} with $\underline{\CF}=\CN$). We have an isomorphism $\Phi_{\underline{\mathcal{F}}}\colon\mathfrak{M}_{\Pi_Q}^{\nil}\rightarrow\mathfrak{M}_{\underline{\mathcal{F}}}$ such that if $\CC^{\bullet}_{\underline{\mathcal{F}}}$ is the restriction of the shifted RHom complex over $\mathfrak{M}_{\mathcal{A}}\times\mathfrak{M}_{\mathcal{A}}$ to $\mathfrak{M}_{\underline{\mathcal{F}}}\times\mathfrak{M}_{\underline{\mathcal{F}}}$ and $\mathcal{C}_{\Pi_Q}^{\nil,\bullet}$ is the restriction of the RHom complex over $\mathfrak{M}_{\Pi_Q}\times\mathfrak{M}_{\Pi_Q}$ to $\mathfrak{M}_{\Pi_Q}^{\nil}\times\mathfrak{M}_{\Pi_Q}^{\nil}$, then $\Phi_{\underline{\mathcal{F}}}^*\CC^{\bullet}_{\underline{\mathcal{F}}}$ is quasi-isomorphic to $\mathcal{C}_{\Pi_Q}^{\nil,\bullet}$.
\end{corollary}

\begin{corollary}
\label{corollary:CoHAcompatibiltiySigmacoll}
 The isomorphism of stacks $\Phi_{\underline{\mathcal{F}}}\colon\mathfrak{M}_{\Pi_Q}^{\nil}\rightarrow\mathfrak{M}_{\underline{\mathcal{F}}}$ induces an isomorphism of cohomological Hall algebras
 \[
 \Phi_{\underline{\mathcal{F}}}^*\colon \ulrelCoHA_{\underline{\mathcal{F}}}\rightarrow \ulrelCoHA_{\Pi_Q}^{\nil}.
 \]
\end{corollary}
\begin{proof}
 This is a consequence of Corollary \ref{corollary:comparisonCoHAfiber}, since the cohomological Hall algebra of the $\Sigma$-collection $\underline{\mathcal{F}}$ only depends on the stack $\mathfrak{M}_{\underline{\mathcal{F}}}$ together with the (shifted) RHom-complex $\mathcal{C}^{\bullet}_{\underline{\mathcal{F}}}$ over $\mathfrak{M}_{\underline{\mathcal{F}}}\times\mathfrak{M}_{\underline{\mathcal{F}}}$ up to quasi-isomorphism, by Propositions \ref{gqe_prop} and \ref{glob_res_prop}.
\end{proof}

\section{BPS algebras}
\label{section:lessperverse}
We introduce the BPS algebra of a 2-Calabi--Yau category. The BPS algebra is a smaller, more manageable, algebra than the cohomological Hall algebra. The general expectation from cohomological DT theory is that the cohomological Hall algebra should be (half) of a Yangian-type algebra associated to a Lie algebra, while what we define as the BPS algebra will be the universal enveloping algebra of (half of) this Lie algebra. We recall important facts about BPS algebras.

\subsection{The BPS algebra of a $\mathbf{2}$-Calabi--Yau category}
\label{subsection:BPSalg}
Let $\mathcal{A}$ and $\JH\colon\mathfrak{M}\rightarrow\mathcal{M}$ be as in \S \ref{BPS_assumptions_sec}. Recall that this means that as well as Assumptions \ref{p_assumption}, \ref{q_assumption1} and \ref{assumption:associativity}, we require that there is a good moduli space $\JH\colon\FM\rightarrow \CM$, the direct sum morphism $\oplus\colon\CM\times\CM\rightarrow \CM$ is finite, and a 2CY condition on categories generated by simple objects is satisfied (Assumptions \ref{gms_assumption}, \ref{ds_fin} and \ref{BPS_cat_assumption}).

The relative cohomological Hall algebra is defined in \S\ref{subsection:relativecohaproductKV} as an algebra structure on the complex of mixed Hodge modules $\ulrelCoHA_{\JH}\coloneqq\JH_*\BD\ulBoQ_{\mathfrak{M}}^{\vir}\in \D^+(\MHM(\Msp))$. Let $\KK(\Mst)$ be the groupification of $\pi_0(\Mst)$. As in the introduction (\S \ref{subsubsection:cohatotnegative}), we also make a choice of bilinear form $\psi$ on $\KK(\Mst)$ and define $\ulrelCoHA^{\psi}_{\JH}$ to be the same underlying object as $\ulrelCoHA_{\JH}$ but with the $\psi$-twisted multiplication
\begin{equation}
\label{psi_twist_def}
m^{\psi}_{a,b}=(-1)^{\psi(a,b)}m_{a,b}.
\end{equation}
We write $\ulrelCoHA_{\JH}^{(\psi)}$ in statements to indicate that the statement is true whether we include the $\psi$-twist or not. This $\psi$-twist is crucial for Theorems~\ref{theorem:degree0SSN} and \ref{theorem:pbwtotnegative2CY}.

\begin{lemma}
\label{lemma:nonnegativeperverse}
 The mixed Hodge module complex $\ulrelCoHA_{\JH}^{(\psi)}$ is concentrated in nonnegative degrees: $\CH^i(\ulrelCoHA^{(\psi)}_{\JH})=0$ for $i<0$.
\end{lemma}
\begin{proof}
This is part of \cite[Theorem B]{davison2021purity}. The statement is independent of whether we give $\ulrelCoHA_{\JH}$ the $\psi$-twisted multiplication or not because it concerns the underlying objects.
\end{proof}

\begin{corollary}
\label{corollary:bpsalgissubalg}
 The map $\CH^{0}(m^{(\psi)})\colon \CH^{0}(\ulrelCoHA_{\JH})\boxdot \CH^{0}(\ulrelCoHA_{\JH})\rightarrow \CH^{0}(\ulrelCoHA_{\JH})$ yields an algebra in the tensor category $(\MHM(\mathcal{M}),\boxdot)$.
\end{corollary}
\begin{proof}
 This follows from the facts that $\ulrelCoHA_{\JH}$ is concentrated in nonnegative degrees (Lemma \ref{lemma:nonnegativeperverse}), $\oplus_*$ is exact, and for any two mixed Hodge modules $\CF,\CG$ on $\mathcal{M}$, $\Hom_{\D^+(\MHM(\Msp))}(\CF,\CG[-n])=0$ if $n>0$ by general facts on t-structures.
\end{proof}

\begin{definition}
We define $\ulBPS_{\Alg}^{(\psi)}$ to be the mixed Hodge module $\CH^{0}(\ulrelCoHA^{(\psi)}_{\JH})$ with the multiplication $\CH^{0}(m^{(\psi)})$. We denote by $\mathcal{BPS}^{(\psi)}_{\Alg}=\pH{0}(\mathscr{A}^{(\psi)}_{\JH})$ the underlying perverse sheaf of $\ulBPS_{\Alg}$, along with its induced $\boxdot$-algebra structure $\rat(\CH^0(m^{(\psi)}))$. By abuse of terminology we call both \emph{the BPS algebra sheaf}. Its cohomology is denoted $\aBPS^{(\psi)}_{\Alg}:=\mathbf{H}(\mathcal{BPS}^{(\psi)}_{\Alg})$ and called \emph{the BPS algebra}.
\end{definition}

\begin{lemma}
\label{lemma:BPSsemisimple}
 The mixed Hodge module $\ulBPS^{(\psi)}_{\Alg}$ is semisimple.
\end{lemma}
\begin{proof}
Again, the statement is independent of whether we consider the $\psi$-twisted multiplication or not. This is a consequence of the decomposition theorem for 2-Calabi--Yau categories \cite[Theorem B]{davison2021purity}.
\end{proof}

\subsection{Primitive summands}
\label{primitives_sec}
\begin{lemma}[{\cite{davison2021purity}}]
\label{lemma:ICintoBPSAlg}
 Let $a\in \Sigma_{\CA}$ be a class such that $\mathcal{A}$ has simple objects of class $a$.
Then, we have a canonical monomorphism of mixed Hodge modules
\[
\phi_a\colon \underline{\mathcal{IC}}(\mathcal{M}_a)\hookrightarrow\ulBPS_{\Alg}.
\]
It induces a morphism of complexes of mixed Hodge modules via the $\HO_{\BoC^{\ast}}$-action on the target given by the determinant line bundle (Assumption~\ref{det_bun_assumption})
 \[
 \underline{\IC}(\mathcal{M}_a)\otimes \HO_{\BoC^{\ast}}\rightarrow \ulrelCoHA_{\JH}.
 \]
\end{lemma}
\begin{proof}
 Let $a$ be as in the lemma. The scheme $\mathcal{M}_a$ is irreducible with smooth locus $\mathcal{M}_a^s$, the locus of simple objects (Proposition \ref{proposition:geometryofgms}). Over this locus, $\JH$ is a $\BoC^*$-gerbe. Therefore, $\CH^0\!(\ulrelCoHA_{\JH})_{|\mathcal{M}_a^s}\cong \underline{\IC}(\mathcal{M}_a^s)$. By semisimplicity of $\ulBPS_{\Alg}$, we obtain the monomorphism $\phi_a$. Since $\ulrelCoHA_{\JH}$ has a $\HO_{\BoC^{\ast}}$-action coming from the determinant bundle on $\mathfrak{M}_a$, we obtain the second statement.
\end{proof}
By the universal property of free algebras in monoidal categories, the morphisms $\phi_{a}\colon \underline{\IC}(\Msp_{a}) \into \ulBPS_{\Alg}$ induce $\boxdot$-algebra morphisms
\begin{equation}
\label{equation:mainiso}
\Phi_{\CA}^{(\psi)}\colon \Free_{\boxdot-\mathrm{Alg}}\left( \bigoplus_{a \in \Sigma_{\CA}}\underline{\IC}(\Msp_{a}) \right) \longto \ulBPS_{\Alg}^{(\psi)}.
\end{equation}

\begin{lemma}
\label{lemma:ICPreprojPrimitive}
The summands $\underline{\IC}(\Msp_{a}) \subset \ulBPS^{(\psi)}_{\Alg}$ for $a \in \Sigma_{\CA}$ are primitive subobjects of $\ulBPS^{(\psi)}_{\Alg}$ as a $\boxdot$-algebra,
i.e. if $\mathcal{I}_{a}$ is the image by $\CH^{0}(m^{(\psi)})$ of $\bigoplus_{\substack{b+c=a\\b,c\neq 0}}\ulBPS^{(\psi)}_{\Alg,b}\boxtimes\ulBPS^{(\psi)}_{\Alg,c}$, the composition $\underline{\IC}(\Msp_{a})\rightarrow \ulBPS^{(\psi)}_{\Alg,a}/\mathcal{I}_{a}$ is non-zero. In particular, using semisimplicity of $\ulBPS^{(\psi)}_{\Alg,a}$, the image $\CI_a$ is contained in a direct sum complement of the inclusion $\underline{\IC}(\Msp_{a}) \into \ulBPS^{(\psi)}_{\Alg,a}$.
\end{lemma}
\begin{proof}
This follows by the proof of \cite[Theorem 7.35]{davison2021purity}, and is immediate by support considerations: the supports of all simple direct summands of $\mathcal{I}_a$ are included in the closed subvariety $\mathcal{M}_{\mathcal{A},a}\setminus\mathcal{M}_{\mathcal{A},a}^s$.
\end{proof}

\subsection{Borcherds--Bozec algebra of a quiver}
\label{subsection:borcherds_bozec_Lie_algebra}
Let $Q=(Q_0,Q_1)$ be a quiver. Following Bozec's definition in \cite{bozec2015quivers} we consider a Borcherds Lie algebra associated to this datum generalising the Kac--Moody Lie algebra of a loop-free quiver.

We decompose the set of vertices $Q_0=Q_0^{\real}\sqcup Q_0^{\iso}\sqcup Q_0^{\hyp}$ where the set of \emph{real} vertices $Q_0^{\real}$ is the set of vertices carrying no loops, the set of \emph{isotropic} vertices $Q_0^{\iso}$ is the set of vertices carrying exactly one loop and the set of \emph{hyperbolic} vertices $Q_0^{\hyp}$ is the set of vertices carrying at least two loops. The set of isotropic and hyperbolic vertices $Q_0^{\imag}=Q_{0}^{\iso}\sqcup Q_0^{\hyp}$ is the set of \emph{imaginary} vertices.

The set of simple positive roots of the Borcherds--Bozec Lie algebra $\mathfrak{g}_Q$ of $Q$, is
\[
 I_{\infty}=(Q_0^{\real}\times \{1\})\sqcup(Q_0^{\imag}\times \BoZ_{>0}).
\]
There is a natural projection
\begin{align*}
p\colon\BoZ^{(I_{\infty})}&\longrightarrow\BoZ^{Q_0}\\
(f\colon I_{\infty}\rightarrow\BoZ)&\longmapsto\left(pf\colon i'\mapsto\sum_{(i',n)\in I_{\infty}}nf(i',n)\right).
\end{align*}
The lattice $\BoZ^{Q_0}$ has a bilinear form given by the symmetrized Euler form $(-,-)=(-,-)_Q$ of $Q$. We endow $\BoZ^{(I_{\infty})}$ with the bilinear form $p^*(-,-)$ obtained by pulling back the Euler form, and by abuse of notation we also denote this form on $\BoZ^{(I_{\infty})}$ by $(-,-)$. In explicit terms,
\[
 (1_{(i',n)},1_{(j',m)})=mn(1_{i'},1_{j'})=mn(\langle 1_{i'},1_{j'}\rangle_Q+\langle 1_{j'},1_{i'}\rangle_Q).
\]
There is a Borcherds Lie algebra associated to the data $(\BoZ^{(I_{\infty})},p^*(-,-))$. It is the Lie algebra over $\BoQ$ with generators $h_{i'},e_i,f_i$, with $i'\in Q_0$, $i\in I_{\infty}$, satisfying the following set of relations.
\begin{equation}
\label{equation:relations}
 \begin{aligned}
{[h_{i'},h_{j'}]}&=0 &\text{ for $i',j'\in Q_0$}\\
 [h_{j'},e_{(i',n)}]&=n(1_{j'},1_{i'})e_{(i',n)}&\text{ for $j'\in Q_0$ and $(i',n)\in I_{\infty}$}\\
 [h_{j'},f_{(i',n)}]&=-n(1_{j'},1_{i'})f_{(i',n)}&\text{ for $j'\in Q_0$ and $(i',n)\in I_{\infty}$} \\
 \ad(e_j)^{1-(j,i)}(e_i)=\ad(f_j)^{1-(j,i)}(f_i)&=0&\text{ for $j\in Q_0^{\real}\times \{1\}$, $i\neq j$}\\
 [e_i,e_j]=[f_i,f_j]&=0 &\text{ if $(i,j)=0$}\\
 [e_i,f_j]&=\delta_{i,j}nh_{i'} &\text{ for $i=(i',n)$.}
 \end{aligned}
\end{equation}
The Lie algebra $\mathfrak{g}_{Q}$ has a triangular decomposition
\[
\mathfrak{g}_Q=\mathfrak{n}_Q^-\oplus\mathfrak{h}\oplus\mathfrak{n}_Q^+ 
\]
where $\mathfrak{n}_Q^-$ (resp. $\mathfrak{n}_Q^+$, resp. $\mathfrak{h}$) is the Lie subalgebra generated by $e_i, i\in I_{\infty}$ (resp. $f_i, i\in I_{\infty}$, resp. $h_{i'}, i'\in Q_0$) and we will only be interested in its positive part $\mathfrak{n}_Q^{+}$. It is generated by $e_i, i\in I_{\infty}$ with Serre relations
\[
 \begin{aligned}
 \ad(e_j)^{1-(j,i)}(e_i)&=0&\text{ for $j\in Q_0^{\real}\times \{1\}$, $i\neq j$}\\
 [e_i,e_j]&=0 &\text{ if $(i,j)=0$}.
 \end{aligned}
\]
If $Q$ is a totally negative quiver, then it has no real vertex and $(i,j)<0$ for every $i,j\in I_{\infty}$, so $\mathfrak{n}_Q^+$ is the \textit{free} Lie algebra generated by $e_i, i\in I_{\infty}$.

By considering the \emph{associative} algebra generated by $e_{i}, f_{i}$, $i\in I_{\infty}$ and $h_{i'}$, $i'\in Q_0$ subject to the relations \eqref{equation:relations}, one obtains an algebra $\UEA(\mathfrak{g}_Q)$.
It is the enveloping algebra of $\mathfrak{g}_Q$ and it admits a triangular decomposition
\[
 \UEA(\mathfrak{g}_Q)=\UEA(\mathfrak{n}_Q^-)\otimes \UEA(\mathfrak{h})\otimes \UEA(\mathfrak{n}_Q^+).
\]

\subsection{The degree zero BPS algebra of the strictly seminilpotent stack}

\subsubsection{The cohomological Hall algebra of the strictly seminilpotent stack}

The absolute cohomological Hall algebra of the category of strictly seminilpotent representations of $Q$ is
\[
 \HO^*\!\!\!\mathscr{A}_{\Pi_Q}^{\SSN}:=\bigoplus_{\dd\in \BoN^{Q_0}}\HO_*^{\BoMo}(\mathfrak{M}_{\Pi_Q,\dd}^{\SSN},\BoQ^{\vir}),
\]
and has been defined in \S \ref{subsubsection:cohaSSN}. We denote by $ \HO^*\!\!\!\mathscr{A}_{\Pi_Q}^{\psi,\SSN}$ the same algebra, with the multiplication twisted by the sign $(-1)^{\psi(-,-)}$, for some choice of bilinear form $\psi$ satisfying $\psi(a,b)+\psi(b,a)=(a,b)_{Q}$ modulo 2 (\S\ref{subsection:BPSalg}). For concreteness the reader may like to choose $\psi(a,b)=\langle a,b\rangle_Q$.

The seminilpotent stack $\mathfrak{M}_{\Pi_Q,\dd}^{\SSN}$ is a Lagrangian substack of $\mathfrak{M}_{\Pi_Q,\dd}$. Therefore, its irreducible components are of dimension $-\langle\dd,\dd\rangle$. By \cite[Theorem 1.15]{bozec2016quivers}, for $\dd=1_{i'}$, where $i'\in Q_0$ is a vertex with $g_{i'}$ loops, the irreducible components of $\Mst^{\SSN}_{\Pi_Q,\dd}$ are parametrised by
\begin{enumerate}
 \item $\{\star\}$ if $g_{i'}=0$,
 \item \label{isotropicvertex}The set of partitions $(n_1,\hdots,n_r)$ of $n$ if $g_{i'}=1$,
 \item \label{hyperbolicvertex}The set of tuples $(n_1,\hdots,n_r)$ with $n_j>0$ and $\sum_{j}n_j=n$ if $g_{i'}>1$.
\end{enumerate}

In the cases \eqref{isotropicvertex} and \eqref{hyperbolicvertex}, the partition/tuple can be described explicitly. Let $I_r\subset \Pi_Q$ be the two-sided ideal generated by paths in $\overline{Q}$ containing at least $r$ arrows of $Q_1^*$. Let $\mathfrak{M}_{\Pi_Q,\dd,c}^{\SSN}$ be an irreducible component of $\mathfrak{M}_{\Pi_Q,\dd}^{\SSN}$ and $\rho$ a $\Pi_Q$-module corresponding to a general point in $\mathfrak{M}_{\Pi_Q,\dd,c}^{\SSN}$. Let $r$ be the smallest integer such that $I_r\rho=0$. We have a sequence of inclusions
\[
 0=I_r\rho\subset I_{r-1}\rho \subset\hdots\subset I_0\rho=\rho.
\]
Then, $n_i=\dim I_{i-1}\rho/I_i\rho$ for $1\leq i\leq r$.

We let $\HO^0\!\!\!\relCoHA_{\Pi_Q}^{\psi,\SSN}$ be the \emph{degree zero strictly seminilpotent cohomological Hall algebra}.
It has a combinatorial basis given by fundamental classes of irreducible components $[\mathfrak{M}^{\SSN}_{\Pi_Q,\dd,c}]$ of $\mathfrak{M}_{\Pi_Q}^{\SSN}$. This algebra is identified with the enveloping algebra of the Borcherds--Bozec Lie algebra defined in \S\ref{subsection:borcherds_bozec_Lie_algebra}. This is proven in detail in \cite[Theorem 1.3]{hennecart2022geometric}:

\begin{theorem}
\label{theorem:degree0SSN}
 There is an isomorphism of algebras
 \[
 \UEA(\mathfrak{n}_Q^+)\rightarrow \HO^0\!\!\!\relCoHA_{\Pi_Q}^{\psi,\SSN}
 \]
 sending $e_{(i',n)}$ to $[\mathfrak{M}^{\SSN}_{n1_{i'},(n)}]$.
\end{theorem}
Let $\imath_{\SSN}\colon\mathcal{M}_{\Pi_Q}^{\SSN}\hookrightarrow\mathcal{M}_{\Pi_Q}$ be the natural inclusion (\S\ref{subsubsection:cohaSSN}).

\begin{lemma}
 The natural map $\ptau{\leq 0}\mathscr{A}^{(\psi)}_{\Pi_Q}\rightarrow \mathscr{A}^{(\psi)}_{\Pi_Q}$ induces a map $\imath_{\SSN}^!\ptau{\leq 0}\mathscr{A}^{(\psi)}_{\Pi_Q}\rightarrow \imath_{\SSN}^!\mathscr{A}^{(\psi)}_{\Pi_Q}=:\mathscr{A}_{\Pi_Q}^{(\psi),\SSN}$ which induces an isomorphism of algebras after taking global sections :
 \[
 \mathbf{H}^0(\imath_{\SSN}^!\ptau{\leq 0}\mathscr{A}^{(\psi)}_{\Pi_Q})\xrightarrow{\sim} \mathbf{H}^0(\imath_{\SSN}^!\mathscr{A}^{(\psi)}_{\Pi_Q}).
 \]
\end{lemma}
\begin{proof}
 This appears in \cite[Section 6.4]{davison2020bps}. Again, the proof is independent of whether we include the sign twist $\psi$ simultaneously on both sides or not.
\end{proof}

\begin{corollary}
\label{corollary:degree0SSN}
Let $Q$ be a totally negative quiver. The morphism 
\[
 \HO^{0}\left(\imath_{\SSN}^!\Phi^{(\psi)}_{\Pi_Q}\right)\colon \HO^{0}\!\left(\imath_{\SSN}^!\Free_{\boxdot-\Alg}\left(\bigoplus_{\dd\in\Sigma_{\Pi_Q}}\mathcal{IC}(\mathcal{M}_{\Pi_Q,\dd})\right)\right)\rightarrow \HO^{0}\!\left(\mathscr{A}_{\Pi_Q}^{(\psi),\SSN}\right)
\]
is an isomorphism.
\end{corollary}
\begin{proof}
We consider the $\psi$-twisted version of the statement: then the untwisted version will follow from Lemma \ref{free_signs}. The functor $\imath_{\SSN}^!$ is strict monoidal, so it commutes with the operator $\Free$. By Lemma~\ref{lemma:ICPreprojPrimitive}, the subobject $\mathcal{IC}(\mathcal{M}_{\Pi_Q})\coloneqq\bigoplus_{\dd\in\Sigma_Q}\mathcal{IC}(\mathcal{M}_{\Pi_Q,\dd})\rightarrow \mathcal{BPS}^{\psi}_{\Alg}=\ptau{\leq 0}\mathscr{A}_{\Pi_Q}$ admits a direct sum complement $\mathcal{BPS'}\subset \mathcal{BPS}_{\Alg}$ such that the multiplication map $m\colon \mathcal{BPS}^{\psi}_{\Alg,a}\boxtimes\mathcal{BPS}^{\psi}_{\Alg,b}\rightarrow \mathcal{BPS}^{\psi}_{\Alg,a+b}$ for $a,b\neq 0$ factors through the inclusion $\mathcal{BPS}'\subset\mathcal{BPS}^{\psi}_{\Alg}$. Consequently, the multiplication map $\HO^{0}(\imath_{\SSN}^!m^{\psi})\colon\HO^{0}(\mathscr{A}_{\Pi_Q}^{\psi,\SSN})\otimes\HO^{0}(\mathscr{A}_{\Pi_Q}^{\psi,\SSN})\rightarrow\HO^{0}(\mathscr{A}_{\Pi_Q}^{\psi,\SSN})$ factors through $\HO^{0}(\imath_{\SSN}^!\mathcal{BPS}')$. 

If $\dd=ne_{i'}$ for some $i'\in Q_0$, $n\geq 1$, then, $\HO^{0}(\imath_{\SSN}^!\mathcal{IC}(\mathcal{M}_{\Pi_Q,ne_{i'}}))=\BoQ e_{i',n}$ is one-dimensional \cite[\S 6.4.4]{davison2021purity}. By Theorem \ref{theorem:degree0SSN}, if $\dd\neq ne_{i'}$ for $i'\in Q_0$, $n\geq 1$, then $\HO^{0}(\imath_{\SSN}^!\mathcal{BPS}'_{\dd})=\HO^{0}(\mathscr{A}_{\Pi_Q,\dd}^{\psi,\SSN})$. For such $\dd$, $\HO^{0}(\imath_{\SSN}^!\mathcal{IC}(\mathcal{M}_{\Pi_Q,\dd}))$ is a direct summand of $\HO^{0}(\mathscr{A}_{\Pi_Q,\dd}^{\psi,\SSN})$ with complement $\HO^{0}(\imath_{\SSN}^!\mathcal{BPS}'_{\dd})$. It therefore vanishes. The LHS in the corollary is then equal to $\HO^{0}\!\left(\Free_{\Alg}\left(\bigoplus_{i\in I_{\infty}}\BoQ e_i\right)\right)$. By Theorem \ref{theorem:degree0SSN} again, $\HO^{0}(\imath_{\SSN}^!\Phi^{\psi}_{\Pi_Q})$ is an isomorphism.
\end{proof}

We finish this section by collecting some elementary lemmas regarding $\psi$-twisted algebras.
\begin{lemma}
\label{free_signs}
Let $\CB$ be a $L$-graded algebra, where $L$ is some Abelian group. Let $\psi$ be a bilinear form on $L$, and define $\CB^{\psi}$ via $m_{a,b}^{\psi}=(-1)^{\psi(a,b)}m_{a,b}$, where $m_{a,b}$ is the restriction of the multiplication in $\CB$ to the $a$th and $b$th graded pieces, and $m_{a,b}^{\psi}$ is the same restriction, for $\CB^{\psi}$. If $\CB$ is the free algebra generated by the graded subspace $V\subset \CB$, then $\CB^{\psi}$ is freely generated by the same subspace $V\subset \CB^{\psi}$.
\end{lemma}
\begin{proof}
Let $S$ be a homogeneous basis for $V$. Given a (possibly empty) word $w$ in the letters $S$, we write $m^{(\psi)}(w)$ for the evaluation of the products on $w$, setting this to be the unit if $w$ is empty. Then $m(w)=\pm m^{\psi}(w)$. Let $W$ be the set of words in $S$. Then $\{m(w)\;\lvert\; w\in W\}$ is a basis for $\CB$ if and only if $\{m^{\psi}(w)\;\lvert\; w\in W\}$ is.
\end{proof}
\begin{lemma}
Let $\CB$ be a $L$-graded algebra, where $L$ is some Abelian group. Let $\psi,\psi'$ be bilinear forms on $L$ such that
\begin{align*}
\psi(a,b)+\psi(b,a)=\psi'(a,b)+\psi'(b,a)&\quad \mathrm{mod }\;2.
\end{align*}
Let $V\subset \CB$ be a $L$-graded subspace of $\CB$. We consider $\CB$ as a Lie algebra in two different ways: firstly for the commutator Lie bracket $[-,-]_{\psi}$ defined by the $\psi$-twisted multiplication, secondly for $[-,-]_{\psi'}$ defined with respect to the $\psi'$-twisted multiplication. Let $\mathfrak{g}$ be the Lie subalgebra generated by $V$ under the first Lie bracket, let $\mathfrak{g}'$ be the Lie algebra generated under the second. Then as vector subspaces of $\CB$, we have $\mathfrak{g}=\mathfrak{g}'$.
\end{lemma}
\begin{proof}
If $\alpha$ and $\beta$ are homogeneous, the conditions imply that $[\alpha,\beta]_{\psi}=\pm [\alpha,\beta]_{\psi'}$.
\end{proof}

\section{Some examples}
The framework of \S \ref{subsection:relativecohaproductKV} concerns very general 2-dimensional categories satisfying the assumptions of \S \ref{subsubsection:assumptionCoHAproduct}. In subsequent sections, we first restrict ourselves to categories with a left 2-Calabi--Yau structure satisfying the additional assumptions of \S \ref{BPS_assumptions_sec} in order to define the BPS algebra of \S \ref{section:lessperverse}, and then restrict further to totally negative 2-Calabi--Yau categories, in order to prove our main results. 

Before making these restrictions, we discuss general examples of 2-dimensional categories that the above Hall algebra construction applies to, before moving on to example applications of our main theorems for totally negative 2CY categories.

\subsection{Degree zero sheaves on surfaces}
We use our constructions to generalise a PBW result concerning the CoHA of zero-dimensional support coherent sheaves on a smooth surface $S$ from \cite{kapranov2019cohomological} to mixed Hodge structures, and further to the level of mixed Hodge modules on the coarse moduli spaces $\Sym^n(S)$. 

Let $S$ be a smooth quasi-projective complex surface. We do not require that $S$ is projective, or that there is an isomorphism $K_S\cong\mathcal{O}_S$, or that $S$ is cohomologically pure. We denote by $\FM_n(S)$ the stack of coherent sheaves on $S$ with zero-dimensional support, and length $n$. Denoting by $\varpi_n \colon \FM_n(S) \to \CM_n(S)$ the good moduli space, there is an isomorphism $\CM_n(S)\cong \Sym^n(S)$. We drop the subscripts and superscripts $n$ when considering all possible lengths at once, that is in particular $\CM(S)=\bigsqcup_{n\geq 0}\CM_n(S)$. Let $\Delta_n\colon S\rightarrow \Sym^n(S)$ be the inclusion of the small diagonal. Then one shows as in \cite[Appendix A]{davison2023nonabelian} that there is an injection of mixed Hodge modules
\[
(\Delta_n)_*\underline{\IC}_S\hookrightarrow \tau^{\leq 0}\ulrelCoHA_n\coloneqq\ulrelBPSalg{n}
\]
where $\ulrelCoHA=\varpi_*\BD \ulBoQ_{\FM(S)}$ and $\ulrelCoHA_n$ is the restriction of $\ulrelCoHA$ to $\Sym^n(S)$. Furthermore, the BPS algebra is commutative: since perverse sheaves form a stack, this is a local calculation, and follows from the case $S=\BoA^2$. The resulting morphism of algebra objects in $\MHM(\Sym(S))$
\[
\Sym_{\boxdot}\left(\bigoplus_{n\geq 1}(\Delta_n)_*\underline{\IC}_S\right)\rightarrow \ul{\BPS}_{\Alg}
\]
is an isomorphism of algebra objects: again, this is a local calculation, and is shown in the local case $S=\BoA^2$ in \cite[Appendix A]{davison2023nonabelian}.

The algebra $\ulrelCoHA$ carries the usual $\HO_{\BoC^*}$-action given by the morphism $\det\colon \FM(S)\rightarrow \B\BoC^*$ \S\ref{subsection:detlbsurfaces}, and so we obtain a morphism of complexes of mixed Hodge modules
\[
\Phi\colon \Sym_{\boxdot}\left(\bigoplus_{n\geq 1}(\Delta_n)_*\underline{\IC}_S\otimes\HO_{\BoC^*}\right)\rightarrow \ulrelCoHA
\]
combining the $\HO_{\BoC^*}$-action with the evaluation morphism given by the relative Hall algebra product.
\begin{proposition}
The morphism $\Phi$ is an isomorphism of complexes of mixed Hodge modules. Taking derived global sections, there is an isomorphism of cohomologically graded mixed Hodge structures
\[
\HO^*(\Phi)\colon\Sym\left(\bigoplus_{n\geq 1}\HO^*(S,\BoQ^{\vir})\otimes\HO_{\BoC^*}\right)\cong \HO_*^{\BoMo}(\FM(S),\BoQ).
\]
Above, $\HO^*(S,\BoQ^{\vir})$ is obtained by applying the derived global sections functor to the constant pure weight zero mixed Hodge module $\underline{\IC}(S)=\ul{\BoQ}_S\otimes \BoL^{-1}$.
\end{proposition}
The statement regarding $\Phi$ is again a local statement, that can be checked by reducing to the case $S=\BoA^2$. Here, it is the original PBW isomorphism of \cite{davison2020cohomological} for the three-loop quiver with potential $X[Y,Z]$, using the description of the BPS sheaves found in \cite[\S 5]{davison2016integrality}. The statement regarding $\HO^*(\Phi)$ at the level of graded vector spaces recovers \cite[Theorem 7.1.6]{kapranov2019cohomological}.

\subsection{Semistable coherent sheaves on surfaces}

Let $S$ be a smooth quasi-projective complex surface, which for now we do not assume has a trivial canonical bundle. Let $H$ be an ample line bundle on $\overline{S}$, a projective compactification of $S$. Let $p(t)\in\BoQ[t]/\sim$ be a reduced Hilbert polynomial. We form as in Appendix \ref{geometry_constr_sec} the moduli stack $\Cohst^{\sst}_{p(t)}(S)$ of compactly supported coherent sheaves which are either the zero sheaf or semistable with normalised Hilbert polynomial $p(t)$, and the coarse moduli space $\Cohsp^{\sst}_{p(t)}(S)$. For simplicity, we assume that $\chi(\CF,\CF)$ is even for all coherent sheaves with reduced Hilbert polynomial $p(t)$: this condition can be relaxed, at the cost of dealing with half Tate twists (to define $\underline{\BoQ}^{\vir}$) and some sign difficulties. By Proposition \ref{geometry_sumup_prop}, the stack $\Cohst^{\sst}_{p(t)}(S)$ satisfies Assumptions \ref{p_assumption}, \ref{q_assumption1}, \ref{gms_assumption} and \ref{assumption:associativity}, and so we may form the relative CoHA
\[
\ulrelCoHA_{p(t)}(S)\coloneqq \JH_*\BD\ulBoQ_{\Cohst^{\sst}_{p(t)}(S)}^{\vir}\in\CD^+(\MHM(\Cohsp^{\sst}_{p(t)}(S))).
\]
This object then carries the $\boxdot$-algebra structure provided in \S \ref{subsubsection:ThecohaproductKV}. Note that this structure will involve some Tate twists if $\chi(-,-)$ is not symmetric, and so at the level of constructible complexes, the cohomological grading is not respected unless $\chi(-,-)$ is symmetric.

Now we assume that $S$ satisfies $\mathcal{O}_S\cong\omega_S$. Then there are no Tate twists appearing in the product, and (see Appendix \ref{geometry_constr_sec}) Assumptions \ref{ds_fin} and \ref{BPS_cat_assumption} are also satisfied, so that (possibly after picking a form $\psi$ as in \eqref{first_psi}) we may form the BPS algebra MHM
\[
\ulBPS_{p(t),\Alg}^{(\psi)}(S)=\tau^{\leq 0}\ulrelCoHA^{(\psi)}_{p(t)}(S).
\]
Finally, we make the assumption that the category of compactly supported semistable coherent sheaves on $S$ with reduced Hilbert polynomial is totally negative. Then as special cases of Theorems \ref{theorem:freenesstotneg2CY} and \ref{theorem:pbwtotnegative2CY}, we deduce the following
\begin{theorem}
\label{surf_BPS_thm}
Let $S$ be a smooth quasi-projective complex surface with $\mathcal{O}_S\cong \omega_S$. Let $\CA$ be the category of compactly supported semistable coherent sheaves on $S$ with reduced Hilbert polynomial $p(t)$. Assume moreover that this category is totally negative. Then the natural morphism
\[
\Phi_{\CA}^{(\psi)}\colon \Free_{\boxdot-\mathrm{Alg}}\left( \bigoplus_{a \in \Sigma_{\CA}}\underline{\IC}(\Cohsp_{p(t),a}^{\sst}(S)) \right) \longto \ulBPS_{p(t),\Alg}^{(\psi)}(S).
\]
is an isomorphism of $\boxdot$-algebras in $\MHM(\Cohsp_{p(t)}^{\sst}(S))$.
\end{theorem}

\begin{theorem}
Let $S$ be a smooth quasi-projective complex surface satisfying the same conditions as Theorem \ref{surf_BPS_thm}. Define
\[
\ulBPS_{p(t),\Lie}(S)\coloneqq \Free_{\boxdot-\Lie}\left(\bigoplus_{a\in\Sigma_{\mathcal{A}}}\ul{\mathcal{IC}}(\Cohsp_{p(t),a}^{\sst}(S))\right).
\]
Then the morphism
\[
\tilde{\Phi}^{\psi}\colon \Sym_{\boxdot}\left(\ulBPS_{p(t),\Lie}(S)\otimes \HO_{\BoC^*}\right)\rightarrow \ulrelCoHA^{\psi}_{p(t)}(S)
\]
is an isomorphism in $\CD^+(\MHM(\Cohsp^{\sst}_{p(t)}(S)))$ (though not of algebra objects). Taking derived global sections, and setting $L\subset \BoQ[t]$ to be the monoid of polynomials in the equivalence class $p(t)$, we deduce that there is an isomorphism of $L$-graded mixed Hodge structures
\[
\Sym\left(\Free_{\Lie}\left(\bigoplus_{a\in\Sigma_{\mathcal{A}}}\ICA(\Cohsp_{p(t),a}^{\sst}(S))\right)\otimes\HO_{\BoC^*}\right)\cong \HO^{\BoMo}_*(\Cohst^{\sst}_{p(t)}(S),\BoQ^{\vir}).
\]
\end{theorem}
\subsection{Semistable $\bs{B}$-modules}
Let $B$ be an algebra, which we presume is presented in the form $B=A/\langle R\rangle$ as in \eqref{standard_pres}. Recall that we assume that $A$ is the universal localisation of the path algebra $\BoC Q$ of a quiver. We fix a stability condition $\zeta\in\BoQ^{Q_0}$. Recall that the \textit{slope} of a $B$-module $M$ is defined to be the number
\[
\mu(M)\coloneqq \frac{1}{\dim_{\BoC}(M)}\dim_{Q_0}(M)\cdot \zeta
\]
where $\dim_{Q_0}(M)\coloneqq (\dim_{\BoC}(e_i\cdot M))_{i\in Q_0}$. A $B$-module is called \textit{semistable} if for all proper nonzero submodules $M'\subset M$ we have $\mu(M')\leq \mu(M)$. We denote by $\FM^{\zeta\sst}_{B,\dd}$ the moduli stack of $\dd$-dimensional semistable $B$-modules. There is a good moduli space $\FM^{\zeta\sst}_{B,\dd}\rightarrow \CM^{\zeta\sst}_{B,\dd}$ constructed by King via GIT \cite{king1994moduli} in the special case $B= \BoC Q$, and then defined in general via base change from this special case. So Assumption \ref{gms_assumption} is always satisfied.

Fix $\theta\in\BoQ$. We denote by $\FM^{\zeta\sst}_{B,\theta}$ the disjoint union of all $\FM^{\zeta\sst}_{B,\dd}$ such that $\dd\cdot\zeta/\lvert \dd\lvert=\theta$, and define $\CM^{\zeta\sst}_{B,\theta}$ to be the union of $\CM^{\zeta\sst}_{B,\dd}$ over the same set of dimension vectors. Assumptions \ref{p_assumption}-\ref{q_assumption1} and \ref{assumption:associativity} similarly hold for the stack $\FM^{\zeta\sst}_{B,\theta}$ via base change from the case $\zeta=(0,\ldots,0)$, i.e. the case in which we consider all $B$-modules to be semistable.

For finiteness of the direct sum map, consider the commutative diagram
\[
\begin{tikzcd}
\CM^{\zeta\sst}_{B,\theta}\times\CM^{\zeta\sst}_{B,\theta}&\CM^{\zeta\sst}_{B,\theta}
\\
\CM_{B,\theta}\times \CM_{B,\theta}& \CM_{B,\theta}
\arrow["l\times l", from=1-1, to=2-1]
\arrow["\oplus^{\zeta}", from=1-1, to=1-2]
\arrow["l", from=1-2, to=2-2]
\arrow["\oplus", from=2-1, to=2-2]
\end{tikzcd}
\]
where $l\colon \CM^{\zeta\sst}_{B,\theta}\rightarrow \CM_{B,\theta}$ is the GIT quotient map, which is projective by construction. Then $\oplus \circ (l\times l)$ is projective by Assumption \ref{ds_fin} for the category of $B$-modules, and so $\oplus^{\zeta}$ is proper, since properness satisfies the 2 out of 3 property. Finally, the category of semistable $B$-modules is a full subcategory of the category of $\tilde{B}$-modules (where $\tilde{B}$ is the derived enhancement of $B$ as in \S\ref{Dalg_sec}), so as long as $\tilde{B}$ carries a left 2CY structure, Assumption \ref{BPS_cat_assumption} is satisfied. We summarise in a theorem.
\begin{theorem}
Let $B$ be an algebra presented as $B=A/\langle R\rangle$ where $A$ is the localisation of a path algebra of a quiver by a finite set of linear combinations of cyclic paths with the same endpoints, and $R$ is a finite set of relations in $A$. Let $\zeta\in\BoQ^{Q_0}$ be a stability condition, let $\theta\in\BoQ$ be a slope, and let $\JH\colon\FM^{\zeta\sst}_{\theta}(B)\rightarrow \CM^{\zeta\sst}_{\theta}(B)$ be the morphism from the stack of $\zeta$-semistable $B$-modules of slope $\theta$ to the GIT moduli space. Then 
\begin{enumerate}
\item
$\ulrelCoHA_{B,\theta}^{\zeta}\coloneqq \JH_*\BD\ulBoQ^{\vir}_{\FM^{\zeta\sst}_{\theta}(B)}$ carries a (Tate-twisted) algebra structure in $\CD^+(\MHM(\CM^{\zeta\sst}_{\theta}(B)))$. 
\item
Assume moreover that the derived enhancement $\tilde{B}$ is a 2-Calabi--Yau algebra. Then the algebra structure on $\ulrelCoHA_{B,\theta}^{\zeta}$ is untwisted, and the subcomplex $\ulrelBPSalg{B,\theta}^{\zeta}\coloneqq\tau^{\leq 0}\ulrelCoHA_{B,\theta}^{\zeta}$ is a $\boxdot$-algebra in $\MHM(\CM^{\zeta\sst}_{\theta}(B))$. 
\item 
Assume finally that the category $\tilde{B}\lmod$ is a totally negative 2CY category. Then there is an isomorphism of $\boxdot$-algebras
\[
\ulrelBPSalg{B,\theta}^{\zeta}\cong \Free_{\boxdot-\Alg}\left(\bigoplus_{\dd\in\Sigma_{\theta}^{\zeta}}\ul\IC(\CM_{\dd}^{\zeta\sst}(B))\right)
\]
where $\Sigma^{\zeta}_{\theta}$ is the set of dimension vectors $\dd$ such that $\dd\cdot\zeta/\lvert \dd\lvert=\theta$ and there exists a $\zeta$-stable $\dd$-dimensional $B$-module. In addition, there is a PBW isomorphism of MHM complexes
\[
\Sym_{\boxdot}\left(\ulrelBPSLie{B,\theta}^{\zeta}\otimes\HO_{\BoC^*}\right)\rightarrow \ulrelCoHA_{B,\theta}^{\zeta}
\]
where
\[
\ulrelBPSLie{B,\theta}^{\zeta}\coloneqq \Free_{\boxdot\mathrm{-Lie}}\left(\bigoplus_{\dd\in\Sigma_{\theta}^{\zeta}}\ul\IC(\CM_{\dd}^{\zeta\sst}(B))\right).
\]
\end{enumerate}
\end{theorem}
\section{Freeness of the BPS algebra of totally negative $2$-Calabi--Yau categories}
\label{section:freeness}

In this section we prove the following theorem, which determines the BPS algebra of a totally negative $2$-Calabi--Yau category.

\begin{theorem}[{= MHM version of Theorem~\ref{theorem:freenesstotneg2CY}}]
\label{theorem:FreeAlg2CY}
For any totally negative $2$-Calabi--Yau category $\CA$ satisfying the Assumptions \ref{p_assumption}-\ref{BPS_cat_assumption} of §\ref{section:modulistackobjects2CY} and Assumption~\ref{assumption:associativity},
the morphism $\Phi_{\CA}$ \eqref{equation:mainiso} is an isomorphism of $\boxdot$-algebras in $\MHM(\Msp_{\CA})$.
\end{theorem} 

Since the functor $\rat\colon\MHM(\CM_{\mathcal{A}})\rightarrow\Perv(\CM_{\mathcal{A}})$ is exact, the mixed Hodge module version implies the perverse sheaf version of the theorem.  Those readers who prefer to think about perverse sheaves rather than mixed Hodge modules may read the proofs in the categories of perverse sheaves/constructible complexes. The key point is that a semisimple mixed Hodge module is sent under $\rat$ to a semisimple perverse sheaf \cite{deligne1987theoreme}.\footnote{However, we issue a minor warning here that simple mixed Hodge modules are sent to semisimple perverse sheaves that need not be simple.}

We prove Theorem~\ref{theorem:FreeAlg2CY} by reducing to the the case of preprojective algebras of totally negative quivers.

\begin{theorem}[{= MHM version of Theorem~\ref{theorem:freenesspreprojective}}]
\label{theorem:FreeAlgPreProj}
For every totally negative quiver $Q$, the morphism $\Phi_{\Pi_Q}$ (defined in \eqref{equation:mainiso}) is an isomorphism of $\boxdot$-algebras in $\MHM(\Msp_{\Pi_Q})$.
\end{theorem}
We prove that Theorem \ref{theorem:FreeAlgPreProj} implies Theorem \ref{theorem:FreeAlg2CY} in \S\ref{subsubsection:FreeAlgreductiontoPreProj}. The proof of Theorem \ref{theorem:FreeAlgPreProj} is given in \S\ref{subsection:proofforpreproj}.

\subsection{Reduction to preprojective algebras of totally negative quivers}
\label{subsection:reductiontopreproj}

\subsubsection{Free algebra of a $\Sigma$-collection}
\label{subsubsection:freealgfibre}
Let $\underline{\CF}=\{ \CF_1,\ldots,\CF_r \}$ be a collection of simple objects in $\CA$ of classes $a_i = [\CF_i] \in \Sigma_{\CA}$ and for $1\leq i\leq r$ let $x_i \in \Msp_{\CA,a_i}$ be the closed $\BoC$-point corresponding to $\CF_i$. The inclusions $x_i \into \Msp_{\CA}$ induce a monoid morphism $\imath_{\underline{\CF}} \colon \BoN^{\underline{\CF}} \into \Msp_{\CA}$ sending $1_{\CF_i}$ to $ x_i$.
Let $Q$ be half of the Ext-quiver of the collection $\underline{\CF}$ (\S\ref{subsubsection:extquiver}). 
The inclusions $\{S_i\} \into \Msp_{\Pi_{Q},e_i}$ induce a monoid morphism $\imath_{\Nil} \colon \BoN^{Q_0} \into \Msp_{\Pi_{Q}}$ sending $e_i$ to $S_i$.
By Lemma~\ref{lemma:strictmonoidalfunctor}, $\imath_{\Nil}^!$ and $\imath_{\underline{\CF}}^!$ are strict monoidal.

For every $\vec{m} \in \BoN^{\underline{\CF}}$ pick an analytic Ext-quiver neighbourhood
$\CU_{\vec{m}}$ of the point $x_\vec{m} \in \Msp_{\CA,a_{\vec{m}}}$
corresponding to the semisimple object $\bigoplus_{i} \CF_i^{\oplus m_i}$ of
class $a_{\vec{m}} \coloneqq \sum_{i}m_ia_i$ as in Theorem \ref{theorem:neighbourhood}. Set $\CU = \bigsqcup_{\vec{m} \in \BoN^{I}} \CU_{\vec{m}}$.
As part of an analytic Ext-quiver neighbourhood we have a commutative diagram of analytic spaces 
\begin{equation*}
\begin{tikzcd}
&\BoN^{\underline{\CF}} \ar[dl,"\imath_{\Nil}"',hook']\ar[d,"y",hook] \ar[dr,"\imath_{\underline{\CF}}",hook] & \\
 \Msp_{\Pi_Q} & \CU \ar[l,"\jmath'"',hook']\ar[r,"\jmath",hook] & \Msp_{\CA}
\end{tikzcd}
\end{equation*}
such that the horizontal morphisms $\jmath,\jmath'$ are analytic-open embeddings.

\begin{lemma}
\label{lemma:ICSigmaonfibre}
There is a natural isomorphism of $\BoN^{\underline{\CF}}$-graded mixed Hodge structures
\begin{equation*}
(\imath_{\Nil})^!\left(\bigoplus_{\vec{m} \in \Sigma_{\Pi_Q}} \underline{\IC}(\Msp_{\Pi_Q,\vec{m}})\right) \cong \imath_{\underline{\CF}}^!\left(\bigoplus_{a \in \Sigma_{\CA}} \underline{\IC}(\Msp_{\CA,a}) \right).
\end{equation*}
\end{lemma}
\begin{proof}
Since intersection complexes are stable under pullback along open embeddings, we have isomorphisms
\begin{equation*}
(\imath_{\Nil})^{!} \underline{\IC}(\Msp_{\Pi_Q,\vec{m}}) \cong y^{!}\underline{\IC}(\CU_{\vec{m}}) \cong \imath_{\underline{\CF}}^{!} \underline{\IC}(\Msp_{\CA,a_{\vec{m}}}).
\end{equation*}
Now it suffices to note that 
$\vec{m} \in \Sigma_{\Pi_Q}$ if and only if $a_\vec{m} \in \Sigma_{\CA}$ (Proposition~\ref{proposition:geometryofgms} \emph{(3)}).
\end{proof}

\begin{corollary}
There is a natural isomorphism of algebras 
\begin{equation*}
\gamma_{\Free}\colon (\imath_{\Nil})^{!}\Free_{\boxdot-{\mathrm{Alg}}}\left(\bigoplus_{\vec{m} \in \Sigma_{\Pi_Q}} \underline{\IC}(\Msp_{\Pi_{Q},\vec{m}})\right)
\cong
\imath_{\underline{\CF}}^!\Free_{\boxdot-{\mathrm{Alg}}}\left(\bigoplus_{a \in \Sigma_{\CA}} \underline{\IC}(\Msp_{\CA,a})\right).
\end{equation*}
\end{corollary}
\begin{proof}
Since $\imath_{\underline{\CF}}^!$ and $(\imath_{\Nil})^!$ are both strict monoidal (Lemma \ref{lemma:strictmonoidalfunctor}), they commute with the free algebra construction. The statement now follows from Lemma~\ref{lemma:ICSigmaonfibre}.
\end{proof}

\subsubsection{BPS algebra of a $\Sigma$-collection}
We keep the notation from §\ref{subsubsection:freealgfibre}.

The analytic Ext-quiver neighbourhoods are compatible with good moduli space morphisms in the sense that the diagram in Theorem~\ref{theorem:neighbourhood} commutes. 
Hence the canonical morphisms
\begin{equation}
\label{eq:localDTsheavesagree}
\CH^{0}\!\left((\jmath')^{\ast}\ulrelCoHA_{\Pi_Q}\right) \longto \CH^0\!\left(p_{\ast}\BD\ulBoQ_{\FU}^{\vir}\right) \longfrom \CH^{0}\!\left((\jmath)^{\ast}\ulrelCoHA_{\CA}\right)
\end{equation}
are isomorphisms in $\MHM(\CU)$, where $p\colon \FU=\bigsqcup_{\vec{m}\in \BoN^{\underline{\CF}}}\FU_{\vec{m}} \to \bigsqcup_{\vec{m} \in \BoN^{\underline{\CF}}}\CU_{\vec{m}}=\CU$ is the good moduli space morphism over $\CU$. 

Since pullback for mixed Hodge modules by an analytic-open embedding is t-exact, the isomorphisms \eqref{eq:localDTsheavesagree} induce an isomorphism of cohomologically graded mixed Hodge structures
\begin{equation}
\label{eq:fibreBPSagree}
\gamma_{\aBPS} \colon (\imath_{\Nil})^{!}\ulrelBPSalg{\Pi_Q} \longto \imath_{\underline{\CF}}^!\ulrelBPSalg{\CA}.
\end{equation}

\begin{corollary}
\label{corollary:gammaisoalg}
 The isomorphism $\gamma_{\aBPS}$ is an isomorphism of algebras such that the diagram 
\begin{equation}
\label{eq:isoonfibreagree}
\begin{tikzcd}
\imath_{\Nil}^{!}\Free_{\boxdot-{\mathrm{Alg}}}\left(\bigoplus_{\vec{d} \in \Sigma_{\Pi_Q}}\underline{\IC}(\Msp_{\Pi_{Q},\vec{d}})\right) 
\ar[r,"\gamma_{\Free}"] \ar[d,"(\imath_{\Nil})^{!}\Phi_{\Pi_{Q}}"]
&
\imath_{\underline{\CF}}^{!}\Free_{\boxdot-{\mathrm{Alg}}}
\left(\bigoplus_{a\in \Sigma_{\CA}} \underline{\IC}(\Msp_{\CA,a})\right)\ar[d,"\imath_{\underline{\CF}}^{!}\Phi_{\CA}"]
 \\
\imath_{\Nil}^{!}\ulrelBPSalg{\Pi_Q} \ar[r,"\gamma_{\aBPS}"]& \imath_{\underline{\CF}}^{!}\ulrelBPSalg{\CA}
\end{tikzcd}
\end{equation}
commutes.
\end{corollary}

\begin{proof}
By Corollary~\ref{corollary:CoHAcompatibiltiySigmacoll} we have the isomorphism of algebras 
$\gamma\colon (\imath_{\Nil})^{!}\ulrelCoHA_{\Pi_Q} \to \imath_{\underline{\CF}}^{!} \ulrelCoHA_{\CA}$. 
Since the multiplication on $\ulrelBPSalg{\CA}$ and $\ulrelBPSalg{\Pi_Q}$ is obtained by applying $\tau^{\leq 0}$ to the multiplication on $\ulrelCoHA_{\CA}$ and $\ulrelCoHA_{\Pi_Q}$ (\S \ref{section:lessperverse}), respectively, 
it follows that $\gamma_{\Free}$ restricts to a morphism of algebras $\gamma_{\aBPS}$.
Commutativity follows from the fact that $\gamma_{\aBPS}$ restricts to $\imath_{\underline{\CF}}^{!}\underline{\IC}(\Msp_{\CA,a_\vec{m}}) \cong (\imath_{\Nil})^{!}\underline{\IC}({\Msp_{\Pi_Q,\vec{m}}})$ for $\vec{m} \in \Sigma_{\Pi_Q}$. 
\end{proof}

\subsubsection{Proof that Theorem~\ref{theorem:FreeAlgPreProj} implies Theorem~\ref{theorem:FreeAlg2CY}}
\label{subsubsection:FreeAlgreductiontoPreProj} 

We start with an easy lemma.
\begin{lemma}
\label{lemma:nonvanishingcstblecomplex}
 Let $\mathscr{B}\in \CD^+_{\mathrm{c}}(X)$ be a constructible complex on a complex algebraic variety $X$. Then, if $\mathscr{B}\neq 0$, there exists a $\BoC$-point $i_x\colon \pt\rightarrow X$ such that $i_x^!\mathscr{B}\neq 0$.
\end{lemma}

\begin{proof}
By Verdier duality, it suffices to prove the same result for $i_x^*$ (and $\mathscr{B}\in \CD^-_{\mathrm{c}}(X)$) instead of $i_x^!$. For any $\BoC$-point $x$ of $X$, the functor $i_x^*$ is exact for the natural $t$-structures, so that for any $n\in\BoZ$, $i_x^*\mathcal{H}^n(\mathscr{B})\cong \mathcal{H}^n(i_x^*\mathscr{B})$. Therefore, if $i_x^*\mathscr{B}=0$ for every $x\in X$, the constructible complex $\mathscr{B}$ has vanishing cohomology sheaves. By conservativity of the system of cohomology functors (\cite[Proposition 1.3.7]{beilinson2018faisceaux}), $\mathscr{B}$ itself vanishes.
\end{proof}
\begin{corollary}
\label{corollary:nonvanishingcomplexMHM}
Let $\underline{\SB} \in \D^{+}(\MHM(X))$ be a complex of mixed Hodge modules on a complex algebraic variety $X$. Then, if $\underline{\SB} \neq 0$, there exists a $\BoC$-point $i_x\colon \pt \rightarrow X$ such that $i_x^{!}\underline{\SB} \neq 0$.
\end{corollary}
\begin{proof}
Apply Lemma~\ref{lemma:nonvanishingcstblecomplex} to $\rat(\SB)$.
\end{proof}

\begin{lemma}
\label{lemma:NonIsoImpliesNonIsoforExtQuiver}
Suppose that $\Phi_{\CA}$ is not an isomorphism. 
Then there is an $x \in \supp(\ker(\Phi_{\CA}) \oplus \coker(\Phi_{\CA}))$ with half Ext-quiver and dimension vector $(Q',\vec{m})$ such that the morphism
\begin{equation*}
\imath_0^{!}\Phi_{\Pi_{Q'},\vec{m}}\colon \imath_0^{!}\left(\Free_{\boxdot-{\mathrm{Alg}}}\left(\bigoplus_{\vec{d} \in \Sigma_{\Pi_{Q}}} \underline{\IC}(\Msp_{\Pi_{Q},\vec{d}})\right)\right)_{\vec{m}} \longto \iota^!_{0}\ulBPS_{\Pi_{Q},\vec{m}}
\end{equation*}
is {\emph{not}} an isomorphism, where $\Phi_{\Pi_{Q},\vec{m}}$ is the $\vec{m}$th graded piece of $\Phi_{\Pi_{Q}}$ and 
 $\imath_0 \colon \pt \into \Msp_{\Pi_{Q},\vec{m}}$ is the inclusion of the point corresponding to the trivial $\vec{m}$-dimensional representation $S_{\vec{m}}$. 
 If $\CT\subset\ker(\Phi_{\CA}) \oplus \coker(\Phi_{\CA})$ is a direct summand, we can take any $x$ inside an open subset of $\supp(\CT)$ on which $\CT$ is given by an admissible variation of Hodge structures.
\end{lemma}
\begin{proof}
Write $\mathcal{K}\coloneqq \ker(\Phi_{\CA})$ and $\mathcal{C}\coloneqq \coker (\Phi_{\CA})$.
Both $\CK$ and $\CC$ are semisimple, because they are subquotients of semisimple mixed Hodge modules (Lemma~\ref{lemma:BPSsemisimple} and Proposition~\ref{proposition:FreeSemiSimp}). 
Let $x \in \supp(\CK \oplus \CC)$ be such that $\imath_x^{!}\CK \oplus \imath_x^{!}\CC \neq 0$ (Corollary \ref{corollary:nonvanishingcomplexMHM}).
Let $(Q,\vec{m})$ be the half Ext-quiver and dimension vector of $x$. 
By the commutative diagram \eqref{eq:isoonfibreagree} and an analytic Ext-quiver neighbourhood argument as above, we have $\imath_x^{!}\CK \oplus \imath_x^{!}\CC = \imath_{0}^{!}\ker(\Phi_{\Pi_{Q},\vec{m}}) \oplus \imath_{0}^{!}\coker(\Phi_{\Pi_{Q},\vec{m}})$.
\end{proof}

\begin{proof}[End of the proof of Theorem \ref{theorem:FreeAlg2CY} assuming Theorem \ref{theorem:FreeAlgPreProj}]
By Proposition~\ref{proposition:totnegextquiv} every half Ext-quiver $Q$ of a collection of simples in a totally negative 2CY category must be a totally negative quiver.
Thus, if $\Phi_{\CA}$ is not an isomorphism, Lemma~\ref{lemma:NonIsoImpliesNonIsoforExtQuiver} implies that there exists a totally negative quiver $Q$, for which $\Phi_{\Pi_Q}$ is not an isomorphism, contradicting Theorem~\ref{theorem:FreeAlgPreProj}.
\end{proof}

\subsection{The proof for preprojective algebras}
\label{subsection:proofforpreproj}
We turn to the proof of Theorem~\ref{theorem:FreeAlgPreProj}.

We begin by restating Theorem~\ref{theorem:FreeAlgPreProj} so that we can induct on the (cross-sum of the) dimension vector $\vec{d}$ and reverse induct on the number of vertices of the quiver $Q$.
The morphism $\Phi_{\Pi_Q} = \bigoplus_{\vec{d} \in \BoN^{Q_0}}\Phi_{\Pi_Q,\vec{d}}$ is graded by dimension vector. Thus Theorem~\ref{theorem:FreeAlgPreProj} is equivalent to the following theorem.
\begin{theorem}
\label{theorem:FreeAlgPreProjGraded}
For every totally negative quiver $Q$ and dimension vector $\vec{d} \in \BoN^{Q_0}$ the morphism
\begin{equation*}
\Phi_{\Pi_Q,\vec{d}} \colon \bigg(\Free_{\boxdot-\mathrm{Alg}}\bigg(\bigoplus_{\vec{f} \in \Sigma_{Q}} \underline{\IC}(\Msp_{\Pi_Q,\vec{f}}) \bigg)\bigg)_{\vec{d}} \longto \CH^{0}\left(\JH_{\vec{d},\ast} \BD\ulBoQ_{\Mst_{\Pi_Q,\vec{d}}}^{\vir}\right)=\BPS_{\Pi_Q,\Alg,\dd}
\end{equation*}
is an isomorphism.
\end{theorem}

Consider the set $\QuivDim$ consisting of pairs $(Q,\vec{d})$ of quivers with dimension vectors supported on the entire quiver. 
To formalise the simultaneous induction on the cross-sum of the dimension vectors and the reverse induction on the number of vertices we introduce the function
\begin{equation*}
\begin{split}
\mu\colon \QuivDim &\longto \BoZ_{>0} \times \BoZ_{<0} \\
(Q,\vec{d}) &\longmapsto (\abs{\vec{d}},-\abs{Q_0}) = \left(\sum_{i\in Q_0}d_i,-\abs{Q_0}\right).
\end{split}
\end{equation*}
The induction will be with respect to the partial order on $\QuivDim$ pulled back from the lexicographical order on $\BoZ_{>0}\times \BoZ_{<0}$ along $\mu$.

\subsubsection{The action of adding a scalar to a loop}

Let $Q$ be a quiver. Let $L$ be the set of loops of $Q$ and let $\bar{L} = L \sqcup L^{\ast}$ be the set of loops in the doubled quiver $\bar{Q}$. For every loop $l \in L$ there are two $\BG_a$-actions on $\Mst_{\Pi_Q}$ and $\Msp_{\Pi_Q}$ given by adding a scalar multiple of the identity to the loop $l$ or to the loop $l^*$. 
Combining both of the actions for all of the loops, we have a $\BG_a^{\bar{L}}$-action on $\Mst_{\Pi_Q}$ which in formulas is given by
\begin{equation*}
(x_l,x^*_l)_{l \in L}(\rho_a)_{a \in \overline{Q}_1} = \left(\left\{\begin{matrix}
\rho_l + x_l \id_{e_i\cdot \rho},& \text{ if $a = l\colon i \to i$ for $l \in L$}\\
\rho_{l^*} + x^*_l \id_{e_i\cdot \rho}, & \text{ if $a = l^*\colon i \to i$ for $l \in L$}\\
\rho_a & \text{ otherwise}
\end{matrix}\right\} \right)_{a \in \overline{Q}_1}
\end{equation*}
for $(x_l,x^*_l)_{l \in L} \in \BG_a^{\bar{L}}$ and a $\Pi_Q$-module $(\rho_a)_{a \in \overline{Q}_1}$.

\begin{lemma}
\label{lemma:loopequivariant}
All summands of the pure mixed Hodge modules $\Free_{\boxdot-\mathrm{Alg}}(\bigoplus_{\vec{d} \in \Sigma_{\Pi_Q}} \underline{\IC}(\Msp_{\Pi_Q,\vec{d}}))$ and $\CH^0(\JH_{\vec{d},\ast} \BD\ulBoQ_{\Mst_{\Pi_Q,\vec{d}}}^{\vir})$ are $\BG_a^{\bar{L}}$-equivariant.
\end{lemma}
\begin{proof}
The intersection complexes $\underline{\IC}(\Msp_{\Pi_Q,\vec{d}})$ for $\vec{d} \in \Sigma_{\Pi_Q}$ are $\BG_a^{\bar{L}}$-equivariant as they can be defined by intermediate extension of the constant variation of Hodge structure $\ulBoQ_{\Msp_{\Pi_Q}^{s}}\otimes \BoL^{-\dim(\Msp_{\Pi_Q,\vec{d}})/2}$ on the $\BG_a^{\bar{L}}$-invariant dense open subset $\Msp_{\Pi_Q,\vec{d}}^{s} \subset \Msp_{\Pi_Q,\vec{d}}$. The monoidal product $\boxdot$ is evidently $\BG_a^{\bar{L}}$-equivariant. It follows that so is $\Free_{\boxdot-\Alg}(\bigoplus_{\vec{d} \in \Sigma_{\Pi_Q}}\underline{\IC}(\Msp_{\Pi_Q,\vec{d}}))$.

By inspecting the presentation of $\JH$ as in \S\ref{subsection:preprojectivealgebra} one sees that the good moduli space morphism $\JH\colon \Mst_{\Pi_Q} \to \Msp_{\Pi_Q}$ is equivariant with respect to this $\BG_a^{\bar{L}}$-action.
It follows that $\CH^0(\JH_{\vec{d},\ast}\BD\ulBoQ_{\Mst_{\Pi_Q,\vec{d}}}^{\vir})$ is $\BG_a^{\bar{L}}$-equivariant. Since $\BG_a^{\bar{L}}$ is connected, the direct summands of $\CH^0(\JH_{\vec{d},\ast}\BD\ulBoQ_{\Mst_{\Pi_Q,\vec{d}}}^{\vir})$ are also $\BG_a^{\bar{L}}$-equivariant. 
\end{proof}

\subsubsection{Comparison of Ext-quivers}

To every closed point $x$ in the good moduli space $\Msp_{\Pi_Q}$, we associate a quiver with dimension vector $(Q'_x,\vec{m}_x) \in \QuivDim$ given by half of the Ext-quiver and multiplicity vector of the corresponding semisimple $\Pi_Q$-module $\bigoplus_i \CF_i^{m_i}$.
\begin{lemma}
\label{lemma:ExtQuivnotworse}
For all $(Q,\vec{d}) \in \QuivDim$ and for all closed points $x \in \Msp_{\Pi_Q,\vec{d}}$ we have
\begin{enumerate}[(i)]
\item $\abs{\vec{m}_x} \leq \abs{\vec{d}}$ 
\item $\mu(Q'_x,\vec{m}_x) \leq \mu(Q,\vec{d})$ 
\item The following are equivalent 
\begin{enumerate}[(a)]
\item $\mu(Q_{x}',\vec{m}_x) = \mu(Q,\vec{d})$
\item $(Q_{x}',\vec{m}_x) = (Q,\vec{d})$
\item $x$ is in the $\BG_a^{\bar{L}}$-orbit of (the point corresponding to) $S_{\vec{d}}$ (the $\dd$-dimensional $0$-representation of $\Pi_Q$).
\end{enumerate}
\end{enumerate}
\end{lemma}
\begin{proof}
For convenience we write $(Q',\vec{m})$ for $(Q'_x,\vec{m}_x)$.
Let $\CF = \bigoplus_{j\in Q'_{0}}\CF_j^{\oplus m_j}$ be the semisimple $\Pi_Q$-module of dimension vector $\vec{d}$ corresponding to $x$ (such that all $m_i$ are nonzero). We have for all $i \in Q_0$
\begin{equation}
\label{eq:SemiSimpDimvsDimVec}
\sum_{j \in Q'_{0}}m_j \dimvec(\CF_j)_i = d_i.
\end{equation}
Summing \eqref{eq:SemiSimpDimvsDimVec} over all $i$ we have
\begin{equation}
\label{eq:totalSemiSimpDimvsDimVec}
\sum_{j\in Q'_0} m_j \abs{\dimvec(\CF_j)} = \abs{\vec{d}}.
\end{equation}

Part \emph{(i)} of the lemma follows immediately from $\abs{\dimvec(\CF_j)} \geq 1$.
We now prove part \emph{(ii)}.
The case $\abs{\vec{m}} > \abs{\vec{d}}$, is impossible because it contradicts \eqref{eq:totalSemiSimpDimvsDimVec}.
If $\abs{\vec{m}} < \abs{\vec{d}}$, then by definition of $\mu$ we have $\mu(Q',\vec{m}) < \mu(Q,\vec{d})$ (for the lexicographic order).

On the other hand suppose $\abs{\vec{m}} = \abs{\vec{d}}$, then we wish to show $\abs{Q'_0} \geq \abs{Q_0}$. By \eqref{eq:totalSemiSimpDimvsDimVec} we must have $\abs{\dimvec{(\CF_j)}} = 1$, hence each $\CF_j$ is supported at a single vertex, call it $v(j)$. By \eqref{eq:SemiSimpDimvsDimVec} for every $i \in Q_0$ there is a $w(i) \in Q'_0$ such that $v(w(i)) = i$ (choose $w(i)$ so that $m_{w(i)}\dimvec({\CF_{w(i)}})_i\neq 0$). Thus $w\colon Q_0 \to Q'_0$ is an injective map with left inverse $v\colon Q'_0 \to Q_0$, showing $\abs{Q'_0} \geq \abs{Q_0}$.

It remains to characterize the saturation of the inequality. The implications $(c) \implies (b) \implies (a)$ are clear. Suppose $\mu(Q',\vec{m}) = \mu(Q,\vec{d})$. By definition of $\mu$ we can identify the vertex sets of the quivers $Q'$ and $Q$.
As before, by \eqref{eq:SemiSimpDimvsDimVec} we deduce that each $\CF_i$ is a 1-dimensional representation supported at the vertex $i$ and $\vec{m} = \vec{d}$. Hence $\CF_i$ is the data of a scalar $x_l$ for every loop $l$ in $\bar{Q}$ at $i$. Altogether we see that $\CF$ is in the $\BG_a^{\bar{L}}$-orbit of $0_{\vec{d}}$. This proves part \emph{(iii)}.
\end{proof}

\subsubsection{The recursion}
The key ingredient for the recursion is the following lemma.

\begin{lemma}
\label{lemma:InductionStep}
Fix a totally negative quiver $Q$ and a dimension vector $\vec{d} \in \BoN^{Q_0}$. 
Suppose $\Phi_{\Pi_{Q'_x},\vec{m}_x}$ is an isomorphism for all half Ext quivers with dimension vectors $(Q'_x,\vec{m}_x) \in \QuivDim$ for closed points $x \in \Msp_{\Pi_Q,\vec{d}}$ such that $\mu(Q'_x,\vec{m}_x) < \mu(Q,\vec{d})$. Then $\Phi_{\Pi_{Q},\vec{d}}$ is an isomorphism.
\end{lemma}

\begin{proof}
Suppose $\Phi_{\Pi_Q,\vec{d}}$ is not an isomorphism. Its kernel $\mathcal{K}_{\Pi_Q,\vec{d}}$ and cokernel $\mathcal{C}_{\Pi_Q,\vec{d}}$ are semisimple mixed Hodge modules on $\Msp_{\Pi_Q,\vec{d}}$.
Hence there is a (nonzero) simple direct summand $\CT \subset \mathcal{K}_{\Pi_Q,\vec{d}} \oplus \mathcal{C}_{\Pi_Q,\vec{d}}$. 
We have the following chain of inclusions
\begin{align*}
\supp(\CT) &\subset \overline{ \{ x \in \Msp_{\Pi_{Q},\vec{d}} \mid \Phi_{\Pi_x,\vec{m}_x} \text{ is not an iso.} \} } 
&\qquad \text{(by Lemma \ref{lemma:NonIsoImpliesNonIsoforExtQuiver})}\\
&\subset \overline{\{ x \in \Msp_{\Pi_Q,\vec{d}} \mid \mu(Q_x,\vec{m}_x) = \mu(Q,\dd) \}} 
 &\qquad \text{(by hypothesis and Lemma~\ref{lemma:ExtQuivnotworse} \emph{(i)})}\\
&=\BG_a^{\overline{L}} \cdot 0_{\vec{d}} &\qquad \text{(by Lemma~\ref{lemma:ExtQuivnotworse} \emph{(iii)})}
\end{align*}
In words, the hypothesis guarantees that the support of $\CT$ is contained in the $\BG_a^{\bar{L}}$-orbit $\bar{Z} = \BG_a^{\bar{L}}\cdot 0_{\vec{d}} \cong \BoC^{\bar{L}}$ of $0_\vec{d}$.
Since $\CT$ is a $\BG_a^{\bar{L}}$-equivariant mixed Hodge module (Lemma~\ref{lemma:loopequivariant}), its support is $\BG_a^{\bar{L}}$-stable and hence equal to the entire orbit $\bar{Z}$.

The groups $\BG_a^{\bar{L}}$ and $\BG_a^{L}$ are contractible, and so are any of their orbits. Therefore taking total cohomology induces equivalences 
$\MHM_{\BG_a^{\bar{L}}}(\bar{Z}) \simeq \MHM(\pt) \simeq \MHM_{\BG_a^{L}}(Z)$
where $Z$ is the $\BG_a^{L}$-orbit of $0_{\vec{d}}$. Let $\imath_{\SSN}\colon \Msp_{\Pi_{Q},\vec{d}}^{\SSN} \into \Msp_{\Pi_Q,\vec{d}}$ be the inclusion.
We have a Cartesian square of inclusions
\begin{equation*}
\begin{tikzcd}
\BoC^{L} \cong Z \ar[d]\ar[r] & \Msp_{\Pi_Q,\vec{d}}^{\SSN} \ar[d,"\iota_\SSN"] \\
\BoC^{\bar{L}}\cong \bar{Z} \ar[r] & \Msp_{\Pi_Q,\vec{d}}.
\arrow["\lrcorner"{anchor=center, pos=0.125}, draw=none, from=1-1, to=2-2]
\end{tikzcd}
\end{equation*}
Since $\CT$ is simple, we have $\CT = \ul{\IC}_{\BG_a^{\overline{L}}\cdot 0_{\dd}}\otimes V=\underline{\BoQ}_{\BG_a^{\overline{L}}\cdot 0_{\dd}}\otimes\BoL^{-\abs{L}}\otimes V$ for some simple polarizable pure mixed Hodge structure $V$, and $\imath_{\SSN}^{!}\CT \cong \ulBoQ_{\BG_a^{L}}\otimes V$. 

Hence $\HO^0({\imath_{\SSN}^!\CT})\cong \rat(V)$ is a summand of the kernel or cokernel of $\HO^{0}(\imath_{\SSN}^{!}\Phi_{\Pi_Q})$ which contradicts Corollary~\ref{corollary:degree0SSN}.
\end{proof}

\begin{proof}[Proof of Theorem~\ref{theorem:FreeAlgPreProjGraded}]
Fix a quiver with dimension vector $(Q,\vec{d}) \in \QuivDim$. Without loss of generality we assume that $\vec{d}$ is supported on all of $Q$. Consider the following subset of $\BoZ_{>0} \times \BoZ_{<0}$.
\begin{align*}
R(Q,\vec{d}) &\coloneqq \{\mu(Q'_x,\vec{m}_x) \mid x \in \Msp_{\Pi_Q,\vec{d}} \text{ and }\mu(Q'_x,\vec{m}_x) \neq \mu(Q,\vec{d})\}& \\
&= \{\mu(Q'_x,\vec{m}_x) \mid x \in \Msp_{\Pi_Q,\vec{d}} \text{ and }\mu(Q'_x,\vec{m}_x) < \mu(Q,\vec{d}) \}. & (\text{Lemma~\ref{lemma:ExtQuivnotworse}})
\end{align*}
In order to apply Lemma~\ref{lemma:InductionStep} we need to show that for all $(Q^{\mathrm{new}},\vec{d}^\mathrm{new}) \in \mu^{-1}(R(Q,\vec{d}))$ the morphism
$\Phi_{\Pi_{Q^{\mathrm{new}}},\vec{d}^\mathrm{new}}$ is an isomorphism. We do so inductively.

For all $(Q^{\mathrm{new}},\vec{d}^\mathrm{new}) \in \mu^{-1}(R(Q,\vec{d}))$\footnote{The preimage $E(Q,\vec{d})$ of $R(Q,\vec{d})$ under $\mu$ is precisely the set of ``better'' half Ext-quivers with dimension vectors appearing for $(Q,\vec{d})$.} we have the strict inclusion of finite sets $R(Q^{\mathrm{new}},\vec{d}^{\mathrm{new}}) \subsetneqq R(Q,\vec{d})$.
By Lemma~\ref{lemma:ExtQuivnotworse} \emph{(i)}, the set $R(Q,\vec{d})$ is bounded hence finite. Thus, after iterating we eventually end at the case where $R(Q^{\mathrm{fin}},\vec{d}^{\mathrm{fin}}) = \emptyset$, so that the hypothesis in Lemma~\ref{lemma:InductionStep} for $(Q^{\mathrm{fin}},\vec{d}^{\mathrm{fin}})$ is vacuous.

Applying Lemma~\ref{lemma:InductionStep} recursively we deduce that $\Phi_{\Pi_Q,\vec{d}}$ is an isomorphism.
\end{proof}

\section{PBW theorem for the CoHA of a totally negative $2$-Calabi--Yau category}
In this section, we prove the PBW theorem, Theorem \ref{theorem:pbwtotnegative2CY}.
In the introduction we proposed for the definition of the BPS Lie algebra sheaf of a totally negative 2CY category $\CA$, the free Lie algebra object generated by the intersection complexes 
\begin{equation*}
\ulrelBPSLie{\CA} = \Free_{\boxdot-\Lie}\left( \bigoplus_{a \in \Sigma_{\CA}} \underline{\IC}(\Msp_{\CA,a})\right).
\end{equation*}
To justify the name BPS sheaf this must satisfy the cohomological integrality theorem, meaning that there should be \textit{some} isomorphism as in Theorem \ref{theorem:pbwtotnegative2CY} (not necessarily defined in terms of Hall algebras). In this section we prove Theorem~\ref{theorem:pbwtotnegative2CY} for totally negative 2CY categories, that is, we show that the morphism
\begin{equation*}
\Sym_{\boxdot}\left(\ulrelBPSLie{\CA} \otimes \HO_{\BoC^{\ast}}\right) \longrightarrow \underline{\relCoHA}_{\CA}^{\psi}
\end{equation*}
defined in terms of the relative Hall algebra product of \S \ref{subsubsection:ThecohaproductKV}, twisted as in \eqref{psi_twist_def}, is an isomorphism.
\subsection{BPS Lie algebras for preprojective algebras of totally negative quivers}

\label{subsubsection:3dBPSLiepreproj} 

An a priori different definition of the BPS Lie algebra sheaf for preprojective algebras of quivers, which we denote by $\ulrelBPSLie{\Pi_Q}^{\mathrm{3d},\psi}$, appears in \cite[Theorem B]{davison2018purity} (see also \cite[Theorem/Definition~4.1]{davison2020bps}). We will show that our definition of the BPS Lie algebra agrees with this ``3d'' definition of the BPS Lie algebra. 
\subsubsection{PBW theorem for critical CoHAs}
\label{general_3d_sssec}
In \cite{davison2020cohomological} the authors define the BPS mixed Hodge module $\ulBPS_{Q,W}$ for the 3-Calabi--Yau category of representations of the Jacobi algebra $\Jac(Q,W)$ of a symmetric quiver with potential $(Q,W)$. We define 
\begin{align*}
\ulBPS_{Q,W}\coloneqq& \left(\tau^{\leq 1}\tilde{\JH}_*\phi_{\Tr(W)}\ulBoQ^{\vir}_{\FM_{Q}}\right)[1]\\
\mathfrak{g}_{Q,W}\coloneqq &\HO^*(\CM_{Q,W},\tau^{\leq 1}\tilde{\JH}_*\phi_{\Tr(W)}\ulBoQ^{\vir}_{\FM_{Q}})
\end{align*}
where $\tilde{\JH}\colon \FM_{Q,W}\rightarrow \CM_{Q,W}$ is the usual affinization morphism to the coarse moduli space. 
We choose a bilinear form $\psi$ on $\BoZ^{Q_0}$ satisfying
\begin{equation}
\label{equation:psi_bilinear_form}
\psi(a,b)+\psi(b,a)=\langle a,a\rangle \langle b,b\rangle +\langle a,b\rangle \quad\pmod{2}
\end{equation}
for all $a,b\in\BoZ^{Q_0}$. 
In \cite{kontsevich2010cohomological} Kontsevich and Soibelman explain how to endow the vanishing cycle cohomology
\[
\HO^*\!\!\mathscr{A}_{Q,W}\coloneqq \bigoplus_{\dd\in\BoN^{Q_0}}\HO^*(\FM_{\dd,Q,W},\phi_{\Tr(W)}\ulBoQ^{\vir}_{\FM_{Q}})
\]
with a cohomological Hall algebra structure, defined by pushforward and pullback in vanishing cycle cohomology. We denote by $\HO^*\!\!\mathscr{A}^{\psi}_{Q,W}$ the algebra obtained by multiplying the degree $(\dd,\dd')$-piece of the multiplication by the sign $(-1)^{\psi(\dd,\dd')}$. Then we have the following theorem.
\begin{theorem}[\cite{davison2020cohomological}]
The natural morphism $\mathfrak{g}_{Q,W}\rightarrow \HO^*\!\!\mathscr{A}^{\psi}_{Q,W}$ is injective, and its image is moreover closed under the commutator Lie bracket in $\HO^*\!\!\mathscr{A}^{\psi}_{Q,W}$. The morphism
\[
\Sym\left(\mathfrak{g}_{Q,W}\otimes\HO_{\BoC^*}\right)\rightarrow \HO^*\!\!\mathscr{A}^{\psi}_{Q,W}
\]
induced by the ($\psi$-twisted) multiplication, and the $\HO_{\BoC^*}$-action on the target induced by a determinant line bundle \S\ref{subsection:detlbalgebras}, is an isomorphism.
\end{theorem}

We denote by $\mathfrak{g}_{Q,W}^{\psi}$ the resulting Lie algebra; it is called the \textit{BPS Lie algebra} associated to $(Q,W)$. We emphasize that the moduli stack $\Mst_{(Q,W)}$ of representations of $\Jac(Q,W)$ admits a global critical locus description, which makes vanishing cycle techniques especially effective.

\subsubsection{Dimensional reduction}
There is a specific choice of quiver with potential that is relevant to the study of $\Pi_Q$. We first form the tripled quiver $\tilde{Q}$ by adding a loop $\omega_i$ to every vertex $i$ of the doubled quiver $\overline{Q}$.
Then we consider the canonical cubic potential $\tilde{W}=\left(\sum_{i\in Q_0} \omega_i\right)\left(\sum_{a\in Q_1}[a,a^*]\right)$, and consider the vanishing cycle mixed Hodge module $\phi_{\Tr(\tilde{W})}\ulBoQ^{\vir}_{\FM_{\tilde{Q}}}$ on the stack of finite-dimensional $\BoC\tilde{Q}$-modules, which is supported on $\FM_{(\tilde{Q},\tilde{W})}\coloneqq \crit(\Tr(\tilde{W}))$.
We define $\ulBPS_{\tilde{Q},\tilde{W}}=\left(\tau^{\leq 1}\tilde{\JH}_*\phi_{\Tr(\tilde{W})}\ulBoQ^{\vir}_{\FM_{\tilde{Q}}}\right)[1]$ as above. 

Applying this theory to the tripled quiver with potential, a connection to the category of $\Pi_Q$-representations is made in \cite{davison2016integrality}. The category $\Jac(\tilde{Q},\tilde{W})$ is equivalent to the category of pairs $(M,f)$ of a $\Pi_Q$-module $M$ with an endomorphism $f\colon M \to M$.
Forgetting the endomorphism induces a morphism of stacks $\Mst_{(\tilde{Q},\tilde{W})} \to \Mst_{\Pi_Q}$. 
The method of dimensional reduction is used to define $\ul{\BPS}_{\Pi_Q}^{\mathrm{3d}}$ from $\ul\BPS_{\tilde{Q},\tilde{W}}$ \cite{davison2016integrality}: we define
\[
\ul{\BPS}_{\Pi_Q}^{\mathrm{3d}}=(\CM_{(\tilde{Q},\tilde{W})}\rightarrow \CM_{\Pi_Q})_*\ul{\BPS}_{\tilde{Q},\tilde{W}}[-1].
\]
This is a mixed Hodge module (this is an application of the ``support lemma'' of \cite{davison2020bps}), and $\ul{\BPS}_{\Pi_Q}^{\mathrm{3d}}$ is moreover closed under the commutator Lie bracket in $\ulrelCoHA^{\psi}_{\Pi_Q}$ and satisfies the cohomological integrality theorem for $\ulrelCoHA_{\JH}$ via a PBW-type isomorphism (cf. Theorem~\ref{theorem:pbwPiQ}). We denote by $\ulrelBPSLie{\Pi_Q}^{\mathrm{3d},\psi}$ the resulting $\boxdot$-Lie algebra.

In order to relate the $\psi$-twists occurring in the general theory of \S \ref{general_3d_sssec} with the rest of this paper, we use that for all pairs of $\dd,\dd'\in\BoN^{Q_0}$ we have
\begin{align}
\langle \dd,\dd'\rangle_{\tilde{Q}}=(\dd,\dd')_{Q}& \quad\textrm{mod } 2.
\end{align}
It follows that $\psi(\dd,\dd')\coloneqq \langle \dd,\dd'\rangle_Q$ is a valid choice for the bilinear form $\psi$ as in \eqref{equation:psi_bilinear_form}.

For the comparison of $\ulrelBPSLie{\Pi_Q}$ with $\ulrelBPSLie{\Pi_Q}^{\mathrm{3d},\psi}$ we use the following theorem that relates the ``3d'' BPS Lie algebra to the BPS algebra for preprojective algebras of quivers.

\begin{theorem}[{\cite[Theorem~6.1]{davison2020bps}}]
\label{theorem:envBPSLieisBPSalg}
Let $Q$ be a quiver. There is a canonical inclusion $\ulrelBPSLie{\Pi_Q}^{\mathrm{3d},\psi} \into \ulrelBPSalg{\Pi_Q}^{\psi}$ of $\boxdot$-Lie algebras which induces an isomorphism of $\boxdot$-algebras $\relEnv{\ulrelBPSLie{\Pi_Q}^{\mathrm{3d},\psi}} \isoto \ulrelBPSalg{\Pi_Q}^{\psi}$.
\end{theorem}

Suppose now that $Q$ is a totally negative quiver. By similar arguments as in the proof of Lemma~\ref{lemma:ICintoBPSAlg}, there exist morphisms 
$\underline{\IC}(\Msp_{\Pi_Q,\vec{d}}) \to \ulrelBPSalg{\Pi_Q}^{\psi}\subset \relCoHA_{\Pi_Q}^{\psi}$ which 
factor through the inclusion $\ulrelBPSLie{\Pi_Q}^{\mathrm{3d},\psi} \into \ulrelBPSalg{\Pi_Q}^{\psi}$ (see \cite[§7.1.1]{davison2020bps} for details). Thus, by the universal property of free Lie algebras, they induce a morphism of Lie algebras 
\begin{equation*}
\hat{\Phi}_{\Pi_Q}^{\psi}\colon \ulrelBPSLie{\Pi_Q} \longto \ulrelBPSLie{\Pi_Q}^{\mathrm{3d},\psi}.
\end{equation*}
By composing universal properties, the universal enveloping algebra of a free Lie algebra is always the free
algebra on the same generator(s). Altogether we have the following commutative diagram.
\begin{equation} \label{equation:FreeBPSLievsFreeBPSAlg}
\begin{tikzcd}
\Free_{\boxdot-\mathrm{Lie}}(\underline{\IC}(\Msp_{\Pi_Q,\Sigma})) 
\ar[d,hook] \ar[r,"\hat{\Phi}^{\psi}_{\Pi_Q}"]
& \ulrelBPSLie{\Pi_Q}^{\mathrm{3d},\psi} \ar[d,hook]
\\
\Free_{\boxdot-\mathrm{Alg}}(\underline{\IC}(\Msp_{\Pi_Q,\Sigma}))
\ar[r,"\Phi^{\psi}_{\Pi_Q}"]
& \ulrelBPSalg{\Pi_Q}^{\psi}.
\end{tikzcd}
\end{equation}

\begin{corollary} 
\label{corollary:FreeLiePreproj}
For all totally negative quivers $Q$ the morphism $\hat{\Phi}^{\psi}_{\Pi_Q}$ of $\boxdot$-Lie algebras is an isomorphism.
\end{corollary}

\begin{proof}
By the PBW theorem for Lie algebra objects in symmetric tensor categories we
have the equations in $\K(\MHM(\Msp_{\Pi_Q}))$
\begin{align*}
\left[ \Sym_{\boxdot}(\Free_{\boxdot-\Lie}(\underline{\IC}(\Msp_{\Pi_Q,\Sigma})))\right] &= \left[\Free_{\boxdot-\Alg}(\underline{\IC}(\Msp_{\Pi_Q,\Sigma}))\right] & (\relEnv{\Free_{\boxdot-\mathrm{Lie}}}=\Free_{\boxdot-{\mathrm{Alg}}}) \\
&= \left[\ulrelBPSalg{\Pi_Q}\right] & (\text{Theorem~\ref{theorem:FreeAlgPreProj}})\\
&= \left[\Sym_{\boxdot}(\ulrelBPSLie{\Pi_Q}^{\mathrm{3d}})\right] &(\text{Theorem~\ref{theorem:envBPSLieisBPSalg} = \cite[Theorem~6.1]{davison2020bps}}).
\end{align*}
It follows that in $\K(\MHM(\Msp_{\Pi_Q}))$ we have $\left[\Free_{\boxdot-\Lie}(\underline{\IC}(\Msp_{\Pi_Q,\Sigma}))\right] = \left[ \ul{\BPS}_{\Pi_Q}^{\mathrm{3d}}\right]$. 

Since $\Phi^{\psi}_{\Pi_Q}$ is a monomorphism (by Theorem~\ref{theorem:FreeAlgPreProj})
and the diagram \eqref{equation:FreeBPSLievsFreeBPSAlg} is commutative, it follows that $\hat{\Phi}_{\Pi_Q}^{\psi}$ is also a monomorphism.
Thus $\hat{\Phi}_{\Pi_Q}^{\psi}$ is a monomorphism between two semisimple mixed Hodge modules of the same class in $\K(\MHM(\Msp_{\Pi_Q}))$, hence
it must be an isomorphism.
\end{proof}

\subsection{PBW and Cohomological integrality for totally negative 2CY categories}

\label{subsection:PBW-2CY}

We recall the following theorem:
\begin{theorem}[{\cite{davison2016integrality}}]
\label{theorem:pbwPiQ}
Let $Q$ be a quiver. The morphism $\ulrelBPSLie{\Pi_Q}\otimes \HO_{\BoC^{\ast}} \to \ulrelCoHA_{\Pi_Q}^{\psi}$ induced by the $\HO_{\BoC^{\ast}}$-action on the target coming from a positive determinant line bundle on $\FM_{\Pi_Q}$ \S\ref{subsection:detlbalgebras}, along with the multiplication on the target, induces an isomorphism in $\D^{+}(\MHM(\Msp_{\Pi_Q}))$
\begin{equation*}
\tilde{\Phi}_{\Pi_Q}^{\psi}\colon \Sym_{\boxdot}(\ulrelBPSLie{\Pi_Q}^{\mathrm{3d},\psi} \otimes \HO_{\BoC^{\ast}}) \longisoto \ulrelCoHA_{\Pi_Q}^{\psi}.
\end{equation*}
\end{theorem}
We can restrict this isomorphism to the submonoid $\imath_{\Nil}\colon\BoN^{Q_0}\rightarrow\mathcal{M}_{\Pi_Q}$.

\begin{corollary}
\label{corollary:PBWfibre}
 We have a PBW isomorphism
 \[
 \imath_{\Nil}^!\tilde{\Phi}_{\Pi_Q}^{\psi}\colon\Sym_{\boxdot-\Alg}\left(\imath_{\Nil}^!(\ulrelBPSLie{\Pi_Q}^{\mathrm{3d},\psi})\otimes \HO_{\BoC^{\ast}}\right)\rightarrow \imath_{\Nil}^!(\underline{\mathscr{A}}_{\Pi_Q}^{\psi})=\ulrelCoHA_{\Pi_Q}^{\nil,\psi},
 \]
 in the symmetric tensor category $\D^+(\MHM(\BoN^{Q_0}))$.
\end{corollary}
\begin{proof}
 The functor $\imath^{!}_{\Nil}\colon \D^{+}(\MHM(\Msp_{\Pi_Q})) \to \D^+(\MHM(\BoN^{Q_0}))$ is a strict monoidal functor by Lemma \ref{lemma:strictmonoidalfunctor}. 
\end{proof}

\begin{proof}[Proof of Theorem \ref{theorem:pbwtotnegative2CY}]
We want to prove that the morphism of objects in $\D^{+}(\MHM(\Msp_{\CA}))$ 
\[
 \tilde{\Phi}^{\psi}\colon \Sym_{\boxdot}\left(\Free_{\boxdot-\Lie}(\underline{\IC}(\mathcal{M}_{\CA,\Sigma}))\otimes \HO_{\BoC^{\ast}}\right)\rightarrow \ulrelCoHA_{\JH}^{\psi}
\]
defined via the Hall algebra product on the target is an isomorphism.

Let $\cone(\tilde{\Phi}^{\psi})$ be the cone of $\tilde{\Phi}^{\psi}$, so that we have a distinguished triangle
\[
 \Sym_{\boxdot}\left(\Free_{\boxdot-\Lie}(\underline{\IC}(\mathcal{M}_{\CA,\Sigma}))\otimes \HO_{\BoC^{\ast}}\right)\xrightarrow{\tilde{\Phi}^{\psi}} \ulrelCoHA_{\JH} \rightarrow \cone(\tilde{\Phi}^{\psi})\rightarrow.
\]
If $\imath_S\colon S\rightarrow \mathcal{M}$ is a subspace, we have a distinguished triangle 
 \[
 \imath_S^!\Sym_{\boxdot}\left(\Free_{\boxdot-\Lie}(\underline{\IC}(\mathcal{M}_{\CA,\Sigma}))\otimes \HO_{\BoC^{\ast}}\right)\xrightarrow{\imath_S^!\tilde{\Phi}^{\psi}} \imath_S^!\ulrelCoHA_{\JH}\rightarrow \imath_S^!\cone(\tilde{\Phi}^{\psi})\rightarrow.
\]
The set $S$ will later be specialised to the closed discrete submonoid of $\mathcal{M}$ corresponding to a $\Sigma$-collection. 

If, for a contradiction, $\cone(\tilde{\Phi}^{\psi})$ does not vanish, there exists a closed $\BoC$-point $x\in\mathcal{M}_{\CA}$ such that, if $\imath_x$ is the inclusion of $x$ in $\mathcal{M}_{\CA}$, $\imath_x^!\cone(\tilde{\Phi}^{\psi})\neq 0$ (by Corollary~\ref{corollary:nonvanishingcomplexMHM}). Let $\underline{\mathcal{F}}=\{\mathcal{F}_1,\hdots,\mathcal{F}_r\}$ be the $\Sigma$-collection of simple objects of $\mathcal{A}$ associated to $x$. We let $Q$ be a half of the Ext-quiver of $\underline{\mathcal{F}}$. The morphism of monoids $\imath_{\underline{\mathcal{F}}}\colon \BoN^{\underline{\mathcal{F}}}\rightarrow \mathcal{M}_{\CA}$ of \S\ref{subsubsection:extquiver} induces a distinguished triangle
\[
 \Sym_{\boxdot}\left(\Free_{\boxdot-\Lie}(\imath_{\underline{\mathcal{F}}}^!\underline{\IC}(\mathcal{M}_{\CA,\Sigma}))\otimes \HO_{\BoC^{\ast}}\right)\xrightarrow{\imath_{\underline{\mathcal{F}}}^!\tilde{\Phi}^{\psi}} \imath_{\underline{\mathcal{F}}}^!(\ulrelCoHA_{\JH})\rightarrow \imath_{\underline{\mathcal{F}}}^!\cone(\tilde{\Phi}^{\psi})\rightarrow.
\]

By the compatibility of the CoHA multiplication on fibres (Corollary~\ref{corollary:CoHAcompatibiltiySigmacoll}), and by the PBW theorem for preprojective algebras of quivers (more specifically by Corollary~\ref{corollary:PBWfibre} and Corollary~\ref{corollary:FreeLiePreproj}), $\imath_{\underline{\mathcal{F}}}^!\tilde{\Phi}^{\psi}$ is an isomorphism. Therefore, $\imath_{\underline{\mathcal{F}}}^!\cone(\tilde{\Phi}^{\psi})$ vanishes and so does $\iota_x^!\cone(\tilde{\Phi}^{\psi})$.  We obtain a contradiction, proving that $\tilde{\Phi}^{\psi}$ is an isomorphism.
\end{proof}
\subsection{Comparison with the ``3d'' definition of BPS sheaves and BPS cohomology for Higgs bundles}
\label{KK_Higgs_sec}
Let $C$ be a smooth projective connected curve of genus $g \geq 2$. In \cite{kinjo2021cohomological} Kinjo--Koseki give a definition of the BPS sheaf $\ulBPS_\theta^{\Dol,\mathrm{3d}}$ and BPS cohomology $\aBPS_\theta^{\Dol,\mathrm{3d}}$ for the category $\Higgs_{\theta}^{\sst}(C)$.

We recall their definition of $\ulBPS_{\theta}^{\Dol,\mathrm{3d}}(C)$.
In \cite[Theorem~5.6]{kinjo2021global} the moduli stack $\Mst_{X,\theta} $ of 1-dimensional semistable sheaves on $X \coloneqq \Tot_{C}(K_C \oplus \CO_C)$ of slope $\theta$ is realized as a critical locus of a function $f\colon \Mst^{\sst}_{\Tot_{C}(K_C(p)),\theta} \to \BoC$ on the moduli stack of one-dimensional semistable sheaves on $\Tot_{C}(K_C(p))$ of the same slope $\theta$, where $p \in C$ is a closed point.
The BPS sheaf for 1-dimensional semistable sheaves on $X$ is defined just as in the case of symmetric quivers with potential. Let $\Mst_{X,\theta}^{\sst} \subset \Mst_{X,\theta}$ be the open substack of semistable sheaves of slope $\theta$.
Consider the vanishing cycle sheaf $\phi_f \ulBoQ^{\vir}_{\Mst^{\sst}_{X}}$ and the good moduli space $\tilde{\JH}\colon \Mst^{\sst}_{X} \to \Msp^{\sst}_{X}$. Define the BPS sheaf by $\ulBPS_{X} \coloneqq \left(\tau^{\leq 1} \tilde{\JH}_{*}\phi_f \ulBoQ^{\vir}_{\Mst^{\sst}_{X}}\right)[1]$. By pushing forward along the projection $\Msp^{\sst}_{X} \to \Msp_{\Tot_C(K_C),\theta}^{\sst} = \Msp^{\Dol}_{\theta}(C)$ we give the ``3d'' definition of the BPS sheaf 
$\ulBPS_{\theta}^{\Dol,\mathrm{3d}} \coloneqq (\Msp^{\sst}_{X} \to \Msp^{\Dol}_{\theta}(C))_{\ast} \ulBPS_{X}[-1]$, which is a semisimple mixed Hodge module by \cite[Proposition 5.12]{kinjo2021cohomological}.

Let $\underline{\IC}(\Msp_{\theta}^{\Dol}(C)) \coloneqq \bigoplus_{d/r=\theta}\underline{\IC}(\Msp_{d,r}^{\Dol}(C))$. Since $\ulBPS_{\theta}^{\Dol,\mathrm{3d}}$ and $\Free_{\boxdot-\Lie}(\underline{\IC}(\Msp_{\theta}^{\Dol}(C))$ are semisimple mixed Hodge modules, the following computation in $\K(\D^+(\MHM(\Msp_{\theta}^{\Dol}(C))))$ implies that they represent the same class in $\K(\MHM(\Msp_{\theta}^{\Dol}(C)))$. 
\begin{align*}
[\Sym_{\boxdot}(\ulBPS_{\theta}^{\Dol,\mathrm{3d}}(C) \otimes \HO_{\BoC^{\ast}})] &= [\ulrelCoHA_{\theta}^{\Dol}(C)] &\text{\cite[Theorem~5.16]{kinjo2021cohomological}}\\
&= \left[\Sym_{\boxdot}\left(\Free_{\boxdot-\Lie}\left(\underline{\IC}(\Msp_{\theta}^{\Dol}(C))\right) \otimes \HO_{\BoC^{\ast}}\right)\right] & (\text{Theorem~\ref{theorem:pbwtotnegative2CY}}).
\end{align*}
This implies that there is some non-canonical isomorphism 
\begin{equation}
\label{3dHiggsIT}
\ulBPS_{\theta}^{\Dol,\mathrm{3d}}(C) \cong \Free_{\boxdot-\Lie}(\underline{\IC}(\Msp_{\theta}^{\Dol}(C)))
\end{equation}

Unlike for the case of quivers, the 3d-definition of the BPS sheaf $\ulBPS_{\theta}^{\Dol,\text{3d}}(C) \subset \ulrelCoHA_{\theta}^{\Dol}(C)$ and BPS cohomology $\aBPS_{\theta}^{\Dol,\text{3d}}(C)\subset \HO^{*}(\ulrelCoHA_{\theta}^{\Dol}(C))$ have not been shown to be closed under the commutator bracket and therefore are not obviously Lie algebra objects\footnote{Since the first version of this paper appeared, this situation has been remedied: see \cite[\S 8.2.8]{BDINKP}}. Without showing this, we cannot mimic \S\ref{subsubsection:3dBPSLiepreproj} to obtain a canonical morphism $\Free_{\boxdot-\Lie}(\underline{\IC}(\Msp_{\theta}^{\Dol}(C))) \to \ulBPS_{\theta}^{\Dol,\mathrm{3d}}(C)$.
Here we see the advantage of defining the BPS sheaf and BPS cohomology as we do in this paper: they automatically have the structure of a Lie algebra object.

\section{Nonabelian Hodge theory for stacks}
\label{NAHT_sec}

Let $C$ be a smooth projective complex curve of genus $g$. For $\theta\in\BoQ\cup \{\infty\}$ we recall from \S\ref{Higgs_background_sec} the definition of the category $\Higgs^{\sst}_{\theta}(C)$: the full subcategory of the category of Higgs sheaves on $C$ that are either semistable of slope $\theta$, or the zero Higgs sheaf. This category is Abelian, finite length, and closed under extensions in $\Higgs(C)$.

By the classical nonabelian Hodge isomorphism \cite{hitchin1987self,donaldson1987twisted,simpson1992higgs}, there is a homeomorphism
\[
\Psi_{r,d}\colon \CM^{\Dol}_{r,d}(C)\xrightarrow{\cong}\CM^{\Betti}_{g,r,d}.
\]
This homeomorphism induces an isomorphism in Borel--Moore homology
\begin{equation}
\label{NAHTBoMo}
\Phi\colon \HO^{\BoMo}_*(\CM^{\Dol}_{r,d}(C),\BoQ)\rightarrow \HO^{\BoMo}_*(\CM^{\Betti}_{g,r,d},\BoQ).
\end{equation}
One of our main motivations was to produce an isomorphism in Borel--Moore homology for moduli \textit{stacks}, analogous to \eqref{NAHTBoMo}. Note that if the genus of $C$ is zero or one, such an isomorphism is known to exist, since the Borel--Moore homology of the Dolbeault and Betti stacks can be explicitly calculated, see \cite{davison2023nonabelian} for details. So in this paper we concentrate on the case $g\geq 2$, i.e. the case in which $\Higgs^{\sst}_{\theta}(C)$ is totally negative (Proposition \ref{g_t_n}).
\smallbreak
For a nonzero rational number $\theta$, which we may write as $\theta=a/b$ with $a,b\in\BoZ$, $a>0$ and $\gcd(a,b)=1$, we define
\[
\CM^{\Dol}_{\theta}(C)\coloneqq\coprod_{n\in\BoZ_{\geq 0}}\CM^{\Dol}_{na,nb}(C)\quad\quad
\CM^{\Betti}_{g,\theta}\coloneqq\coprod_{n\in \BoZ_{\geq 0}}\CM^{\Betti}_{g,na,nb}.
\]
We define $\FM^{\Dol}_{\theta}(C)$ and $\FM^{\Betti}_{g,\theta}$ similarly. We define 
\begin{align}
\aBPS^{\Betti}_{\Lie,g,\theta}\coloneqq &\Free_{\Lie}\left(\bigoplus_{n\geq 0}\ICA(\Msp^{\Betti}_{g,na,nb} \label{naht_bps})\right)\\
\aBPS^{\Dol}_{\Lie,\theta}(C)\coloneqq &\Free_{\Lie}\left(\bigoplus_{n\geq 0}\ICA(\Msp^{\Dol}_{na,nb}(C))\right). \nonumber
\end{align}
The Lie algebra $\aBPS^{\Dol}_{\Lie,\theta}(C)$ carries a bigrading (taking into account ranks and degrees), and for $r\geq 1$ and $d\in\BoZ$ we define $\aBPS^{\Dol}_{\Lie,r,d}(C)$ to be the piece of $\aBPS^{\Dol}_{\Lie,\theta}(C)$ of bidegree $(r,d)$, where $\theta=d/r$. We define $\aBPS^{\Betti}_{\Lie,g,r,d}$ similarly.
\begin{lemma}
\label{Lem:monNAHT}
Let $\theta\in\BoQ$. The morphism $\Psi\colon \CM^{\Dol}_{\theta}(C)\rightarrow \CM^{\Betti}_{g,\theta}$ is an isomorphism of monoid objects in the category of topological spaces.
\end{lemma}
\begin{proof}
Since we are working in the category of topological spaces it is sufficient to show that the two morphisms
\[
\CM^{\Dol}_{n'a,n'b}(C)\times \CM^{\Dol}_{n''a,n''b}(C)\rightarrow \CM^{\Betti}_{g,(n'+n'')a,(n'+n'')b}
\]	
given by $\Psi_{(n'+n'')a,(n'+n'')b}\circ \oplus$ and $\oplus\circ(\Psi_{n'a,n'b}\times\Psi_{n''a,n''b})$ are the same at the level of points. This follows from the construction of $\Psi$: given a polystable Higgs bundle $\CF=\CF_1\oplus\ldots\oplus \CF_l$ the corresponding semisimple twisted $\pi_1(C)$-module is explicitly constructed via the nonabelian Hodge isomorphism applied to the summands $\CF_1,\ldots,\CF_l$, see \cite{simpson1992higgs} for details.
\end{proof}

We denote by 
\[
p_{\Dol}\colon \FM^{\Dol}_{\theta}(C)\rightarrow \CM^{\Dol}_{\theta}(C),\quad\quad
p_{\Betti}\colon \FM^{\Betti}_{g,\theta}\rightarrow \CM^{\Betti}_{g,\theta}
\]
the morphisms from the moduli stacks to the moduli spaces of objects, and define $\ulrelCoHA^{\Dol}_{\theta}(C)\coloneqq p_{\Dol,*}\BD\ulBoQ_{\FM^{\Dol}_{\theta}(C)}^{\vir}$ and $\ulrelCoHA^{\Betti}_{g,\theta}\coloneqq p_{\Betti,*}\BD\ulBoQ_{\FM^{\Betti}_{g,\theta}}^{\vir}$. For the next theorems we assume that the genus of $C$ is at least $2$ so that the relevant categories of Higgs bundles and twisted representations of the fundamental group are totally negative. The following is a special case of (the MHM version of) Theorem \ref{theorem:freenesspreprojective}.
\begin{theorem}
Assume $g(C)\geq 2$. There are isomorphisms of algebra objects in $\MHM(\Msp^{\Dol}_{\theta}(C))$ and $\MHM(\Msp^{\Betti}_{g,\theta})$ respectively:
\begin{align*}
\Free_{\boxdot\mathrm{-Alg}}\left(\underline{\IC}(\CM^{\Dol}_{\theta}(C))\right)\rightarrow &\CH^0(\ulrelCoHA^{\Dol}_{\theta}(C))\\
\Free_{\boxdot\mathrm{-Alg}}\left(\underline{\IC}(\CM^{\Betti}_{\theta})\right)\rightarrow &\CH^0(\ulrelCoHA^{\Betti}_{g,\theta})
\end{align*}
extending the inclusions $\underline{\IC}(\CM^{\Dol}_{\theta}(C))\hookrightarrow \CH^0(\ulrelCoHA^{\Dol}_{\theta}(C))$ and $\underline{\IC}(\CM^{\Betti}_{g,\theta}))\hookrightarrow \CH^0(\ulrelCoHA^{\Betti}_{g,\theta})$ from \S \ref{primitives_sec}.
\end{theorem}
The following is a special case of Theorem \ref{theorem:pbwtotnegative2CY}.
\begin{theorem}
\label{Thm:naFree}
Assume $g(C)\geq 2$. The natural morphisms of mixed Hodge structures
\begin{align*}
\Sym_{\boxdot}\left(\Free_{\mathrm{Lie}}(\ICA(\CM^{\Dol}_{\theta}(C)))\otimes\HO_{\BoC^{\ast}}\right)&\rightarrow \HO^*(\ulrelCoHA^{\Dol}_{\theta}(C))\\
\Sym_{\boxdot}\left(\Free_{\mathrm{Lie}}(\ICA(\CM^{\Betti}_{g,\theta}))\otimes\HO_{\BoC^{\ast}}\right)&\rightarrow \HO^*(\ulrelCoHA^{\Betti}_{g,\theta})
\end{align*}
are isomorphisms.
\end{theorem}

The following theorem is proved for $g(C)\leq 1$ in \cite{davison2023nonabelian}. It is our version of the nonabelian Hodge isomorphism for stacks:
\begin{theorem}[=Theorem~\ref{NAHT_main_thm}]
\label{theorem:NAHT-body}
There is a natural isomorphism in $\CD_{\mathrm{c}}^+(\CM^{\Betti}_{g,r,d})$
\begin{equation}
\label{relNAHT}
\Psi_*p_{\Dol,*}\BD\BoQ^{\vir}_{\FM^{\Dol}_{r,d}(C)}\cong p_{\Betti,*}\BD\BoQ^{\vir}_{\FM^{\Betti}_{g,r,d}}.
\end{equation}
Taking derived global sections, we deduce that there is a natural isomorphism in Borel--Moore homology between the Dolbeault and the Betti stacks:
\begin{equation}
\label{absNAHT}	
\Phi\colon\HO^{\BoMo}_*(\FM^{\Dol}_{r,d}(C),\BoQ)\cong \HO^{\BoMo}_*(\FM^{\Betti}_{g,r,d},\BoQ).
\end{equation}
\end{theorem}
\begin{proof}
The intersection complex is a topological invariant \cite{goresky1983intersection}. Since $\Psi_{r,d}$ is a diffeomorphism, there is a canonical isomorphism
\[
\Psi_{r,d,*}\IC(\CM^{\Dol}_{r,d}(C))\cong \IC(\CM^{\Betti}_{g,r,d}).
\]	
Fix $\theta=r/d$. By Lemma \ref{Lem:monNAHT} there is an isomorphism
\begin{equation}
\label{ac_na}
\Sym_{\boxdot}\left(\Free_{\boxdot \mathrm{-Lie}}(\Psi_*\IC(\CM^{\Dol}_{\theta}(C)))\otimes\HO_{\BoC^*}\right)\cong \Psi_*\!\left(\Sym_{\boxdot}\left(\Free_{\boxdot \mathrm{-Lie}}(\IC(\CM^{\Dol}_{\theta}(C))\otimes\HO_{\BoC^*}\right)\right).
\end{equation}
Combining \eqref{ac_na} with the two isomorphisms of Theorem \ref{Thm:naFree} yields the isomorphism \eqref{relNAHT}. Then restricting to $\CM^{\Betti}_{g,r,d}$ and taking derived global sections yields the isomorphism \eqref{absNAHT}.
\end{proof}

We list some immediate consequences of the construction of the isomorphism $\Phi$:
\begin{enumerate}
\item	
The isomorphism $\Phi$ respects the perverse filtration on $\HO^{\BoMo}_*(\FM^{\Dol}_{r,d}(C),\BoQ)$ (respectively, $\HO^{\BoMo}_*(\FM^{\Betti}_{g,r,d},\BoQ)$) induced by the morphisms $p_{\Dol}$ (respectively $p_{\Betti}$). 
\item
There are natural inclusions $\ICA(\CM^{\Dol}_{r,d}(C))\subset \HO^{\BoMo}_*(\FM^{\Dol}_{r,d}(C),\BoQ^{\vir})$ and $\ICA(\CM^{\Betti}_{g,r,d})\subset \HO^{\BoMo}_*(\FM^{\Betti}_{g,r,d},\BoQ^{\vir})$, and $\Phi(\ICA(\CM^{\Dol}_{r,d}(C)))=\ICA(\CM^{\Betti}_{g,r,d})$.  Furthermore the isomorphism
\begin{equation}
\label{ACI_iso}
\Phi\colon \ICA(\CM^{\Dol}_{r,d}(C))\rightarrow \ICA(\CM^{\Betti}_{g,r,d})
\end{equation}
is the isomorphism induced by the homeomorphism $\Psi_{r,d}$.
\end{enumerate}
\smallbreak
We recall the definition of the \textit{Hitchin morphism}
\begin{equation}
\label{hmap}
\Hit_{r,d}\colon \Msp^{\Dol}_{r,d}(C)\rightarrow \Lambda_r\coloneqq \prod_{i=1}^r\HO^0(C,\omega_C^{\otimes i}),\quad\quad
(\CF,\eta)\mapsto \mathrm{char}(\eta)= (\Tr(\eta),\Tr(\wedge^2\eta)),\ldots).
\end{equation}
The target is a monoid: it records the supports of possible spectral curves (with multiplicity). The intersection cohomology $\ICA(\CM^{\Dol}_{r,d}(C))$ carries a perverse filtration, defined with respect to the Hitchin morphism. Precisely, we set
\[
\mathfrak{P}^i_{\Hit}\ICA(\CM^{\Dol}_{r,d}(C))\coloneqq \HO^*(\Lambda_r,{}^{\mathfrak{p}}\!\tau^{\leq i}\Hit_{r,d,*}\IC(\CM^{\Dol}_{r,d}(C))).
\]
The PI=WI conjecture of de Cataldo and Maulik \cite{de2018perverse} states that we have the identities
\[
\Phi(\mathfrak{P}^i_{\Hit}\ICA(\CM^{\Dol}_{r,d}(C)))=W^{2i}\ICA(\CM^{\Betti}_{g,r,d})
\]
for all $i$. 
By \cite{davison2021purity} the Borel--Moore homology $\HO^{\BoMo}_*(\FM^{\Betti}_{g,r,d},\BoQ^{\vir})$ carries a perverse $\FL^{\bullet}\HO^{\BoMo}_*(\FM^{\Betti}_{g,r,d},\BoQ^{\vir})$ filtration with respect to the morphism $p^{\Betti}$, while $\HO^{\BoMo}_*(\FM^{\Dol}_{r,d}(C),\BoQ^{\vir})$ carries a perverse filtration $\FL^{\bullet}\HO^{\BoMo}_*(\FM^{\Dol}_{r,d}(C),\BoQ^{\vir})$ with respect to $p^{\Dol}$, as well as a perverse filtration $\FP^{\bullet}_{\Hit^{\sta}}\HO^{\BoMo}_*(\FM^{\Dol}_{r,d}(C),\BoQ^{\vir})$ with respect to $\Hit^{\sta}\coloneqq \Hit\circ p^{\Dol}$.   By construction of $\Phi$, it preserves the $\FL^{\bullet}$ filtrations.  %As in \cite[Sec.1.3]{davison2023nonabelian} we define a mixed perverse filtration
%\begin{equation}
%\label{Comb_def}
%\FF^{i}\!\HO^{\BoMo}_*(\Mst^{\Dol}_{r,d}(C),\BoQ^{\vir})\coloneqq\sum_{j+k=i}\left(\mathfrak{P}_{p^{\Dol}}^{2k}\!\HO^{\BoMo}_*(\Mst^{\Dol}_{r,d}(C),\BoQ^{\vir})\right)\cap \left(\mathfrak{P}_{{\Hit}^{\sta}}^{j+2k}\HO^{\BoMo}_*(\Mst^{\Dol}_{r,d}(C),\BoQ^{\vir})\right).
%\end{equation}
To state the next theorem we restrict to slope zero Higgs sheaves, and (untwisted) representations of $\pi_1(\Sigma_g)$. In \cite{davison2023nonabelian} it is explained how to prove the following theorem in the presence of the nonabelian Hodge isomorphism \eqref{relNAHT}.
\begin{theorem}
The following three statements are equivalent
\begin{itemize}
\item
The PI=WI conjecture is true.
\item
There is an equality of filtrations of $\mathfrak{P}_{p^{\Betti}}^0\HO^{\BoMo}_*(\FM^{\Betti}_{g,r,0},\BoQ^{\vir})$
\[
\Phi\left(\mathfrak{P}^i_{\Hit^{\sta}}\HO^{\BoMo}_*(\FM^{\Dol}_{r,0}(C),\BoQ^{\vir})\cap \mathfrak{P}^0_{p^{\Dol}}\HO^{\BoMo}_*(\FM^{\Dol}_{r,0}(C),\BoQ^{\vir})\right)=W^{2i}\mathfrak{P}^0_{p^{\Betti}}\HO^{\BoMo}_*(\FM^{\Betti}_{g,r,0},\BoQ^{\vir})
\]
\item
There is an equality of filtrations of $\mathrm{Gr}^{\FL}_j\HO^{\BoMo}_*(\FM^{\Betti}_{g,r,0},\BoQ^{\vir})$
\[
\Phi\left( \FP^{i-j}_{\Hit^{\sta}}\mathrm{Gr}^{\FL}_j\!\HO^{\BoMo}_*(\Mst^{\Dol}_{r,0}(C),\BoQ^{\vir})\right)=W^{2i-j}\mathrm{Gr}^{\FL}_j\HO^{\BoMo}_*(\FM^{\Betti}_{g,r,0},\BoQ^{\vir}).
\]
\end{itemize}
\end{theorem}

\subsection{$\bs{\chi}$-independence results}
\subsubsection{Dolbeault side}
Note that the target of the Hitchin map $\Hit\colon \CM^{\Dol}_{r,d}(C)\rightarrow \Lambda_r$ depends only on $r$ and not $d$. 
We start by recalling the following theorem of Kinjo and Koseki:
\begin{theorem}\cite{kinjo2021cohomological}
\label{KKthm}
Let $C$ be a smooth complex projective curve of arbitrary genus. For all $d,d'\in\BoZ$ and all $r\in\BoZ_{>0}$ there is an isomorphism of mixed Hodge module complexes \[
(\Hit_{r,d})_*\ulBPS_{r,d}^{\Dol,\mathrm{3d}}(C)\cong (\Hit_{r,d'})_*\ulBPS_{r,d'}^{\Dol,\mathrm{3d}}(C).
\]
Taking derived global sections, we deduce that there is an isomorphism 
\[
\HO^*(\CM^{\Dol}_{r,d}(C),\ulBPS_{r,d}^{\Dol,\mathrm{3d}}(C))\cong \HO^*(\CM^{\Dol}_{r,d'}(C),\ulBPS_{r,d'}^{\Dol,\mathrm{3d}}(C))
\]
respecting the perverse filtrations induced by the Hitchin maps $\Hit_{r,d}, \Hit_{r,d'}$.
\end{theorem}
We give $\CD^{\mathrm{b}}(\MHM(\Lambda))$ the tensor structure $\boxdot$ induced by the monoid structure $+\colon\Lambda\times\Lambda\rightarrow \Lambda$, i.e. we set $\CF\boxdot\CG\coloneqq +_*(\CF\boxtimes\CG)$.  Comparing $d=0,1$ in Theorem \ref{KKthm} and applying our main theorem (more precisely, \eqref{3dHiggsIT} for $\theta=0$), we get:
\begin{corollary}
\label{chi_cor}
Let $C$ be a smooth complex projective curve of genus at least two. There is an isomorphism of complexes of pure mixed Hodge modules
\[
\Free_{\boxdot-\Lie}\left(\bigoplus_{r\geq 1}\Hit_*\ul{\IC}(\CM^{\Dol}_{r,0}(C))\right)\cong \bigoplus_{r\geq 1}\Hit_*\ulBoQ^{\vir}_{\CM^{\Dol}_{r,1}(C)}.
\]
Taking derived global sections, there is a $\BoZ_{\geq 1}$-graded isomorphism
\begin{equation}
\label{mystery_Lie}
\Free_{\Lie}\left(\bigoplus_{r\geq 1}\ICA(\CM^{\Dol}_{r,0}(C))\right)\cong \bigoplus_{r\geq 1}\HO^*(\CM^{\Dol}_{r,1}(C),\ulBoQ^{\vir}).
\end{equation}
respecting the perverse filtrations induced by the Hitchin maps $\Hit$ on both sides.
\end{corollary}

\begin{remark}
In particular, the cohomology of the moduli scheme of degree one semistable Higgs bundles carries a Lie algebra structure via the isomorphism \eqref{mystery_Lie}. Although there is surely a more down-to-earth way of writing it down, we are not sure what it is. 
\end{remark}
The Hodge-to-Singular correspondence of Mauri and Migliorini \cite{mauri2022hodge} (see also \cite{mauri2022combinatorial}) asserts that (after restricting to the reduced locus in $\Lambda$) the direct image $\Hit_*\IC(\CM^{\Dol}_{r,d}(C))$ (for arbitrary $d$) decomposes into isotypic components of tensor powers of $\Hit_*\IC(\CM^{\Dol}_{r_i,0}(C))$ (for decompositions $r=(r_1,\ldots,r_l)$). Corollary \ref{chi_cor} gives a new interpretation of this fact, as well as showing that it extends over the whole of $\Lambda$. We finish this subsection with a final corollary, extending \cite[Corollary 1.9]{mauri2022hodge} over the entire Hitchin base:
\begin{corollary}
Let $C$ be a complex projective curve of genus at least two. Let $r\in\BoZ_{\geq 1}$ and let $d,d'$ satisfy $\gcd(r,d)=\gcd(r,d')$. Then there is an isomorphism in $\CD^{\mathrm{b}}(\MHM(\Lambda_r))$:
\[
(\Hit_{r,d})_*\ul{\IC}(\CM^{\Dol}_{r,d}(C))\cong (\Hit_{r,d'})_*\ul{\IC}(\CM^{\Dol}_{r,d'}(C)).
\]
\end{corollary}
\begin{proof}
Write $(r,d)=m(\overline{r},\overline{d})$ and $(r,d')=m(\overline{r},\overline{d}')$ where $m=\gcd(r,d)$. By Theorem \ref{KKthm} there is an isomorphism of complexes of mixed Hodge modules
\[
\Sym_{\boxdot}\left(\bigoplus_{n\geq 1}\Hit_*\ulBPS_{n\overline{r},n\overline{d}}^{\Dol}(C)\right) \cong \Sym_{\boxdot}\left(\bigoplus_{n\geq 1}\Hit_*\ulBPS_{n\overline{r},n\overline{d}'}^{\Dol}(C)\right).
\]
Thus by \eqref{3dHiggsIT} there is an isomorphism of complexes of mixed Hodge modules
\[
\Free_{\boxdot-\Alg}\left(\bigoplus_{n\geq 1}\Hit_*\ul{\IC}(\CM^{\Dol}_{n\overline{r},n\overline{d}}(C))\right)\cong\Free_{\boxdot-\Alg}\left(\bigoplus_{n\geq 1}\Hit_*\ul{\IC}(\CM^{\Dol}_{n\overline{r},n\overline{d}'}(C))\right)
\]
and the result follows by comparing summands at $n=m$.
\end{proof}
\subsubsection{Betti side}
We record the following Betti $\chi$-independence result
\begin{theorem}
\label{Betti_chi_thm}
Let $g\geq 2$. Then for $\aBPS^{\Betti}_{\Lie,r,d}$ as defined in \eqref{naht_bps}, there is an isomorphism of cohomologically graded vector spaces $\aBPS^{\Betti}_{\Lie,g,r,d}\cong \aBPS^{\Betti}_{\Lie,g,r,d'}$ for all $d,d'\in\BoZ$.
\end{theorem}
\begin{proof}
From $\aBPS^{\Betti}_{\Lie,g,\theta}=\Free_{\Lie}(\ICA(\Msp_{g,\theta}^{\Betti}))$, $\aBPS^{\Dol}_{\Lie,\theta}(C)=\Free_{\Lie}(\ICA(\Msp_{\theta}^{\Dol}(C))$ and the nonabelian Hodge homeomorphism \eqref{ACI_iso} it follows that
\begin{equation}
\label{BPSata}
\aBPS^{\Dol}_{\Lie,r,d}(C)\cong \aBPS^{\Betti}_{\Lie,g,r,d}
\end{equation}

By \cite{kinjo2021cohomological} and \S \ref{KK_Higgs_sec} there are isomorphisms
\begin{equation}
\label{KKiso}
\aBPS^{\Dol}_{\Lie,r,d}(C)\cong \aBPS^{\Dol}_{\Lie,r,d'}(C)
\end{equation}
for all $d,d'$, and the result follows.
\end{proof}
Note that if $(r,d)=1$, by definition 
\[
\aBPS^{\Dol}_{\Lie,r,d}(C)\cong \HO^*(\Msp_{r,d}^{\Dol}(C),\BoQ^{\vir})\quad\text{and}\quad \aBPS^{\Betti}_{\Lie,r,d}(C)\cong \HO^*(\Msp_{g,r,d}^{\Betti},\BoQ^{\vir}).
\]
By the proof of the P=W conjecture \cite{maulik2022p, hausel2022p} if $\gcd(r,d)=1$ the isomorphism \eqref{BPSata} carries the perverse filtration on the Dolbeault side to (twice) the weight filtration on the Betti side. Furthermore, \cite{kinjo2021cohomological} shows that the isomorphism \eqref{KKiso} respects the perverse filtration for all $d,d'$. It follows that for coprime $d,d'$ the isomorphism of Theorem \ref{Betti_chi_thm} respects the weight filtration. We conjecture that this is in fact the case for \textit{all} $d,d'$, regardless of coprimality. This would imply, amongst other things, the PI=WI conjecture (see \cite{davison2023nonabelian} for details). Note that by \cite{hausel2008mixed} the E-polynomials (recording weights, but taking alternating sums over cohomological degrees) of $\aBPS^{\Betti}_{\Lie,g,r,d}$ and $\aBPS^{\Betti}_{\Lie,g,r,d'}$ coincide. See \cite{davison2016cohomological} for details.

\section{Cuspidal polynomials of totally negative quivers}
\label{cuspidals_sec}
Let $Q=(Q_0,Q_1)$ be a quiver. Ringel and Green defined in the beginning of the 1990s in \cite{ringel1990hall,ringel1992hall,green1995hall} the \emph{Hall algebra} of $Q$ over a finite field $\BoF_q$ as an algebra structure on the vector space having as basis the set of isomorphism classes of finite-dimensional representations of $Q$ over $\BoF_q$:
\[
 H_{Q,\BoF_q}:=\bigoplus_{[M]\in\Rep_Q(\BoF_q)/\sim}\BoC[M].
\]
Green \cite{green1995hall} defined a coproduct $\Delta$ on $ H_{Q,\BoF_q}$ and Xiao \cite{xiao1997drinfeld} expressed the antipode, so that $H_{Q,\BoF_q}$ has the structure of a twisted\footnote{The twist is explained in Xiao's paper. It is not relevant here.} Hopf algebra. A good introduction to this subject is \cite{schiffmann2012hall}. Sevenhant and Van den Bergh proved \cite{sevenhant2001relation} that $H_{Q,\BoF_q}$ has the structure of the quantized enveloping algebra of a Borcherds--Kac--Moody algebra with deformation parameter specialized at $\sqrt{q}$, where the generators are given by a basis of the space of \emph{cuspidal functions}
\[
 H_{Q,\BoF_q}^{\cusp}=\{f\in H_{Q,\BoF_q}\mid \Delta(f)=f\otimes 1+1\otimes f\},
\]
satisfying Serre relations\footnote{They assume the quiver is loop-free. This assumption can be removed.}, see \cite[Theorem 3.4]{hennecart2021isotropic} for a formulation.

For $\dd\in\BoN^{Q_0}$, let $M_{Q,\dd}(q)$ be the number of isomorphism classes of $\BoF_q$--representations of $Q$ of dimension $\dd$, $I_{Q,\dd}(q)$ be the number of isomorphism classes of \emph{indecomposable} $\BoF_q$--representations of $Q$ of dimension $\dd$, and $A_{Q,\dd}(q)$ be the number of isomorphism classes of \emph{absolutely indecomposable} $\BoF_q$--representations of $Q$ of dimension $\dd$. By \cite{kac1983root}, these counting functions are polynomials in $q$ and moreover, by \cite{hausel2013positivity} (or \cite{davison2018purity,dobrovolska2016moduli} for different approaches), the coefficients of $A_{Q,\dd}(q)$ are nonnegative.

The graded character of the Hall algebra is given by the formula
\[
 \ch(H_{Q,F_q})\coloneqq\sum_{\dd\in\BoN^{Q_0}}M_{Q,\dd}(q)z^{\dd}=\Exp_{z}\left(\sum_{\dd \neq 0}I_{Q,\dd}(q)z^{\dd}\right)=\Exp_{q,z}\left(\sum_{\dd \neq 0}A_{Q,\dd}(q)z^{\dd}\right),
\]
where $\Exp_z$ and $\Exp_{z,t}$ denote the plethystic exponentials, see \cite[Section 1.5]{bozec2019counting}. The second equality follows from the Krull--Schmidt property of the category of representations of the quiver and the third from Galois descent for quiver representations.

The space $H_{Q,\BoF_q}^{\cusp}$ is naturally graded by the dimension vector: $H_{Q,\BoF_q}^{\cusp}=\bigoplus_{\dd\in\BoN^{Q_0}}H_{Q,\BoF_q}^{\cusp}[\dd]$. There has been a growing interest in understanding this space and to find a parametrisation of cuspidal functions \cite{bozec2019counting, hennecart2021isotropic}, in connection and analogy with the Langlands programme for smooth projective curves. The first step was to compute its dimension. We have the following result.

\begin{theorem}[{\cite[Theorem 1.1]{bozec2019counting}}]
 The dimension $\dim_{\BoC} H_{Q,\BoF_q}^{\cusp}[\dd]$ is given by a polynomial with rational coefficients $C_{Q,\dd}(q)\in \BoQ[q]$.
\end{theorem}
Bozec and Schiffmann combinatorially defined a new family of polynomials $(C_{Q,\dd}^{\abso}(q))_{\dd\in\BoN^{Q_0}}$ from the family $(C_{Q,\dd}(q))_{\dd\in\BoN^{Q_0}}$, expected to enjoy more favourable properties:
\[
 \left\{
 \begin{aligned}
 & C_{Q,\dd}^{\abso}(q)=C_{Q,\dd}(q) \text{ if $\langle\dd,\dd\rangle<0$},\\
 &\Exp_{z}\left(\sum_{l\in\BoZ_{>0}}C_{Q,l\dd}(q)z^{l \dd}\right)=\Exp_{q,z}\left(\sum_{l\in\BoZ_{>0}}C_{Q,l\dd}^{\abso}(q)z^{l \dd}\right) \text{ if $\dd\in(\BoN^{Q_0})_{\prim}$ and $\langle\dd,\dd\rangle=0$},\\
 &C_{Q,\dd}^{\abso}(q)=0\quad \text{ otherwise.}
 \end{aligned}
 \right.
\]
The set $(\BoN^{Q_0})_{\prim}\subset (\BoN^{Q_0})$ is the subset of $\dd$ that cannot be written in the form $\dd=n\dd'$ with $\dd\in\BoN^{Q_0}$ and $n\geq 2$.  The definition is motivated by the fact that if there exists a $\BoN\times \BoN^{Q_0}$-graded Borcherds Lie algebra $\mathfrak{n}_{Q}^{\BoN}$ associated with the lattice $(\BoZ^{Q_0},(-,-))$, with graded character
\[
 \ch(\mathfrak{n}_Q^{\BoN})=\sum_{\dd\in\BoN^{Q_0}}A_{Q,\dd}(q)z^{\dd},
\]then the $\BoN\times\BoN^{Q_0}$-graded dimension of the space of simple roots is given by the generating series
\[
 \dim_{\BoC}\mathfrak{n}_Q^{\BoN}/[\mathfrak{n}_Q^{\BoN},\mathfrak{n}_Q^{\BoN}]=\sum_{\dd\in\BoN^{Q_0}}C_{Q,\dd}^{\abso}(q)z^{\dd},
\]
see \cite{bozec2019counting}.

For totally negative quivers, $\langle\dd,\dd\rangle<0$ for every dimension vector $\dd\in\BoN^{Q_0}$ so that the two families of polynomials $(C_{Q,\dd}(q))_{\dd\in\BoN^{Q_0}}$ and $(C_{Q,\dd}^{\abso}(q))_{\dd\in\BoN^{Q_0}}$ coincide.
\begin{theorem}[{\cite[Theorem 1.4, Theorem 1.6]{bozec2019counting}}]
 The polynomials $C_{Q,\dd}^{\abso}(q)$ have integer coefficients. If $Q$ is a totally negative quiver, they satisfy the equality
 \[
 1-\sum_{\dd>0}C_{Q,\dd}^{\abso}(q)z^{\dd}=\Exp_{q,z}\left(-\sum_{\dd>0}A_{Q,\dd}(q)z^{\dd}\right).
 \]

\end{theorem}
The equality for totally negative quivers witnesses the fact that by the Sevenhant and Van den Bergh theorem, the constructible Hall algebra of such a quiver is free, generated by the subspace of cuspidal functions.

Bozec and Schiffmann conjecture the following.

\begin{conj}[]
\label{conjecture:BozecSchiffmann}
 For any $\dd\in\BoN^{Q_0}$, $C_{Q,\dd}^{\abso}(q)\in\BoN[q]$.
\end{conj}
This conjecture is known for isotropic dimension vectors \cite{deng2003new,bozec2019counting, hennecart2021isotropic}, but is open in general.

We fix now a totally negative quiver $Q$. Such quivers have no isotropic dimension vectors. As mentioned above, for such quivers the cuspidal polynomials $C_{Q,\dd}(q)$ and absolutely cuspidal polynomials $C_{Q,\dd}^{\abso}(q)$ coincide. The case of general quivers will be the object of a subsequent paper. The theorem of Sevenhant and Van den Bergh \cite[Theorem 1.1]{sevenhant2001relation} implies that $H_{Q,\BoF_q}$ is the free associative algebra having $\dim_{\BoC} H_{Q,\BoF_q}[\dd]$ generators in dimension $\dd$.

Assuming that the coefficients of the polynomials $C_{Q,\dd}^{\abso}(q)$ are nonnegative, there exists a free $\BoN^{Q_0}\times \BoN$-graded Lie algebra $\mathfrak{n}^{\BoN}_Q$ with the dimension of the $\BoN$-graded space of generators in dimension $\dd$ given by $C^{\abso}_{Q,\dd}(q)$ such that 
\begin{equation}
\label{equation:charnQN}
 \ch(\UEA(\mathfrak{n}^{\BoN}_Q))=\Exp_{q,z}\left(\sum_{\dd\in\BoN^{Q_0}\setminus\{0\}}A_{Q,\dd}(q)z^{\dd}\right).
\end{equation}
By the graded PBW theorem, this is equivalent to the equality
\begin{equation}
 \ch(\mathfrak{n}_Q^{\BoN})=\sum_{\dd\in\BoN^{Q_0}\setminus\{0\}}A_{Q,\dd}(q)z^{\dd}.
\end{equation}
Conversely, if such a free Lie algebra $\mathfrak{n}_Q^{\BoN}$ exists, then the polynomials $C^{\abso}_{Q,\dd}(q)$ have nonnegative coefficients: they are given by the equality

\begin{equation}
\label{equation:Ngraded}
 \sum_{\dd\in\BoN^{Q_0}}C_{Q,\dd}^{\abso}(q)z^{\dd}=\ch(\mathfrak{n}_Q^{\BoN}/[\mathfrak{n}_Q^{\BoN},\mathfrak{n}_Q^{\BoN}]).
\end{equation}
By \cite[Section 1.2]{davison2020bps}, the 3d BPS Lie algebra of \S\ref{subsubsection:3dBPSLiepreproj}, $\mathfrak{g}_{\Pi_Q}^{\aBPS}\coloneqq\HO^*\BPS_{\Pi_Q,\Lie}^{\mathrm{3d}}$, is a $2\BoN_{\leq 0}$-graded Lie algebra with character

 \begin{equation}
 \label{equation:charBPS}
 \ch(\mathfrak{g}_{\Pi_Q}^{\aBPS})=\sum_{\dd\in\BoN^{Q_0}}A_{Q,\dd}(q^{-2})z^{\dd}
 \end{equation}
which by Corollary \ref{corollary:FreeLiePreproj} is isomorphic to $\aBPS_{\Pi_Q,\Lie}$.  Via this isomorphism $\mathfrak{g}_{\Pi_Q}^{\aBPS}$ is identified with the free Lie algebra with graded dimension of the spaces of generators given by
\begin{equation}
\label{equation:BPSchar}
 \ch(\mathfrak{g}_{\Pi_Q}^{\aBPS}/[\mathfrak{g}_{\Pi_Q}^{\aBPS},\mathfrak{g}_{\Pi_Q}^{\aBPS}])=\sum_{\vec{d}\in\BoN^{Q_0}}\IP(\mathcal{M}_{\Pi_Q,\dd},q)z^{\dd},
\end{equation}
where\footnote{Note that since we take the derived global sections of the perverse intersection complex here, this differs from the standard definition of the intersection Poincar\'e polynomial by a factor of $q^{\dim(\mathcal{M}_{\Pi_Q,\dd})/2}$.} $\IP(\mathcal{M}_{\Pi_Q,\dd},q)\coloneqq \sum_{i\in \BoZ} \dim(\mathrm{IH}^i(\mathcal{M}_{\Pi_Q,\dd}))q^{i}$ is the intersection Poincar\'e polynomial.
\begin{theorem}
 The Conjecture \ref{conjecture:BozecSchiffmann} is true for totally negative quivers: $C^{\abso}_{Q,\dd}(q)\in\BoN[q]$. Moreover, for any $\dd\in\BoN^{Q_0}$, $\dd\in\Sigma_Q$,
 \[
 C_{Q,\dd}(q^{-2})=\IP(\mathcal{M}_{\Pi_Q,\dd},q).
 \]
\end{theorem}
\begin{proof}
This comes from the comparison of Equations \eqref{equation:Ngraded} and \eqref{equation:BPSchar}, given that the character of the free Lie algebra $\mathfrak{g}_{\Pi_Q}^{\aBPS}$ is given by Equation \eqref{equation:charBPS}, which is \eqref{equation:charnQN} up to the change of variables $q\leftrightarrow q^{-2}$.
\end{proof}

\appendix

\section{Cohomological Hall algebras for preprojective algebras}
\label{section:relativeCoHA}
In this section, we explain the original construction of the product for cohomological Hall algebras of preprojective algebras of quivers following \cite{schiffmann2020cohomological,davison2020bps}. The goal is to prove that the product defined using the $3$-term RHom complex \S\ref{subsubsection:ThecohaproductKV} agrees with this original definition \S\ref{subsection:cohaproductpreproj}, Theorem \ref{theorem:cohappcoincide}. This allows us to use the results of \cite{davison2020bps} and \cite{davison2020cohomological} in the body of this paper.

\subsection{Notations for quiver representations}
\label{subsection:notationsquivreps}

Let $Q=(Q_0,Q_1)$ be a quiver and let $ \Pi_Q\coloneqq \BoC\overline{Q}/\langle\rho\rangle$ be the associated preprojective algebra defined as in \S\ref{subsection:preprojectivealgebra}. 

We establish the notations for the rest of the section, following on from \S \ref{subsection:preprojectivealgebra}:

\begin{enumerate}
 
\item \label{item:repspacepp}The representation space of $\dd$-dimensional representations of the preprojective algebra is $X_{\Pi_Q,\dd}=\mu^{-1}_{\dd}(0)$. We have a closed immersion $i_{\dd}\colon X_{\Pi_Q,\dd}\rightarrow X_{\overline{Q},\dd}$. The stack of $\dd$-dimensional representations of the preprojective algebra is denoted $\mathfrak{M}_{\Pi_Q,\dd}\simeq\mu^{-1}_{\dd}(0)/\GL_{\dd}$. The scheme parametrising semisimple $\dd$-dimensional representations of $\Pi_Q$ is denoted $\mathcal{M}_{\Pi_Q,\dd}\cong\mu_{\dd}^{-1}(0)\cms \GL_{\dd}$. The semisimplification map is $\JH_{\dd}\colon \mathfrak{M}_{\Pi_Q,\dd}\rightarrow\mathcal{M}_{\Pi_Q,\dd}$.
 
\item For $\dd^{(1)},\dd^{(2)}\in\BoN^{Q_0}$, we fix a $ \dd\coloneqq \dd^{(1)}+\dd^{(2)}$-dimensional $Q_0$-graded complex vector space $\BoC^{\dd^{(1)}+\dd^{(2)}}$ and a $\dd^{(1)}$-dimensional $Q_0$-graded subspace $\BoC^{\dd^{(1)}}\subset \BoC^{\dd}$.

\item 
\label{Pequivstr}
We let $P_{\dd^{(1)},\dd^{(2)}}\subset \GL_{\dd}$ be the stabiliser of the subspace $\BoC^{\dd^{(1)}}$. It is a parabolic subgroup. Its unipotent radical is $U_{\dd^{(1)},\dd^{(2)}}$ and its Levi quotient is $\GL_{\dd^{(1)}}\times\GL_{\dd^{(2)}}$. Its Lie algebra is denoted by $\mathfrak{p}_{\dd^{(1)},\dd^{(2)}}$. The Lie algebra of the Levi subgroup is $\mathfrak{l}_{\dd^{(1)},\dd^{(2)}}=\mathfrak{gl}_{\dd^{(1)}}\times\mathfrak{gl}_{\dd^{(2)}}$. We have a $P_{\dd^{(1)},\dd^{(2)}}$-equivariant quotient map
\[
l\colon \mathfrak{p}_{\dd^{(1)},\dd^{(2)}}\rightarrow \mathfrak{l}_{\dd^{(1)},\dd^{(2)}}
\]
where $P_{\dd^{(1)},\dd^{(2)}}$ acts on the target via the quotient morphism $P_{\dd^{(1)},\dd^{(2)}}\rightarrow \GL_{\dd^{(1)}}\times\GL_{\dd^{(2)}}$ and the conjugation action. We give $\mathfrak{n}_{\dd^{(1)},\dd^{(2)}}\coloneqq \ker(l)$ the induced $P_{\dd^{(1)},\dd^{(2)}}$-equivariant structure. We let $i_{\mathfrak{n},\dd^{(1)},\dd^{(2)}}\colon\mathfrak{n}_{\dd^{(1)},\dd^{(2)}}\rightarrow\mathfrak{p}_{\dd^{(1)},\dd^{(2)}}$ be the natural inclusion.

\item We set 
\[
\overline{\mathfrak{M}_{\overline{Q},\dd^{(1)}}\times \mathfrak{M}_{\overline{Q},\dd^{(2)}}}\coloneqq (X_{\overline{Q},\dd^{(1)}}\times X_{\overline{Q},\dd^{(2)}})/P_{\dd^{(1)},\dd^{(2)}}
\]
where $P_{\dd^{(1)},\dd^{(2)}}$ acts on $X_{\overline{Q},\dd^{(1)}}\times X_{\overline{Q},\dd^{(2)}}$ via the quotient morphism $P_{\dd^{(1)},\dd^{(2)}}\rightarrow \GL_{\dd^{(1)}}\times\GL_{\dd^{(2)}}$.

\item \label{item:extensionoverlineQ} We let $F_{\overline{Q},\dd^{(1)},\dd^{(2)}}$ be the closed subvariety of $X_{\overline{Q},\dd}$ consisting of elements preserving the subspace $\BoC^{\dd^{(1)}}\subset\BoC^{\dd}$. It is acted on by $P_{\dd^{(1)},\dd^{(2)}}$ and we set $\mathfrak{M}_{\overline{Q},\dd^{(1)},\dd^{(2)}}=F_{\overline{Q},\dd^{(1)},\dd^{(2)}}/P_{\dd^{(1)},\dd^{(2)}}$. This is the stack of short exact sequences of $\overline{Q}$-representations for which the subrepresentation has dimension vector $\dd^{(1)}$ and the quotient $\dd^{(2)}$. We denote by
$
\overline{p}_{\dd^{(1)},\dd^{(2)}}\colon F_{\overline{Q},\dd^{(1)},\dd^{(2)}}\rightarrow X_{\overline{Q},\dd}
$
the obvious closed immersion. We also denote by 
\[\overline{p}_{\dd^{(1)},\dd^{(2)}}\colon\mathfrak{M}_{\overline{Q},\dd^{(1)},\dd^{(2)}}\rightarrow\mathfrak{M}_{\overline{Q},\dd^{(1)}+\dd^{(2)}}\]
the induced map between the stacks. We denote by 
\[
\overline{q}_{\dd^{(1)},\dd^{(2)}}\colon F_{\overline{Q},\dd^{(1)},\dd^{(2)}}\rightarrow X_{\overline{Q},\dd^{(1)}}\times X_{\overline{Q},\dd^{(2)}}
\]
the projection. It induces a vector bundle $\mathfrak{M}_{\overline{Q},\dd^{(1)},\dd^{(2)}}\rightarrow \overline{\mathfrak{M}_{\overline{Q},\dd^{(1)}}\times\mathfrak{M}_{\overline{Q},\dd^{(2)}}}$ and a vector bundle stack $\mathfrak{M}_{\overline{Q},\dd^{(1)},\dd^{(2)}}\rightarrow \mathfrak{M}_{\overline{Q},\dd^{(1)}}\times\mathfrak{M}_{\overline{Q},\dd^{(2)}}$.

\item \label{item:extensionsPi_Q} We let $F_{\Pi_Q,\dd^{(1)},\dd^{(2)}}$ be the closed subvariety of $X_{\Pi_{Q},\dd}$ of elements preserving the subspace $\BoC^{\dd^{(1)}}\subset\BoC^{\dd}$. It is acted on by $P_{\dd^{(1)},\dd^{(2)}}$ and we let $\mathfrak{M}_{\Pi_Q,\dd^{(1)},\dd^{(2)}}=F_{\Pi_Q,\dd^{(1)},\dd^{(2)}}/P_{\dd^{(1)},\dd^{(2)}}$. This is the stack of short exact sequences of $\Pi_Q$-representations where the subrepresentation has dimension vector $\dd^{(1)}$ and the quotient has dimension vector $\dd^{(2)}$. We let 
\[
q_{\dd^{(1)},\dd^{(2)}}\colon F_{\Pi_Q,\dd^{(1)},\dd^{(2)}}\rightarrow X_{\Pi_Q,\dd^{(1)}}\times X_{\Pi_Q,\dd^{(2)}}
\]
be the natural projection and denote by $p_{\dd^{(1)},\dd^{(2)}}\colon F_{\Pi_Q,\dd^{(1)},\dd^{(2)}}\rightarrow X_{\Pi_Q,\dd}$ the inclusion. We still denote by $p_{\dd^{(1)},\dd^{(2)}}\colon\mathfrak{M}_{\Pi_Q,\dd^{(1)},\dd^{(2)}}\rightarrow\mathfrak{M}_{\Pi_Q,\dd^{(1)}+\dd^{(2)}}$ the map between the stacks. It is proper and representable.

\item \label{item:extensionoverlineQPiQ} We set $\overline{F}_{\overline{Q},\dd^{(1)},\dd^{(2)}}=\overline{q}_{\dd^{(1)},\dd^{(2)}}^{-1}(\mu_{\dd^{(1)}}^{-1}(0)\times \mu_{\dd^{(2)}}^{-1}(0))$. We have closed immersions 
\[
F_{\Pi_Q,\dd^{(1)},\dd^{(2)}}\xrightarrow{i'_{\dd^{(1)},\dd^{(2)}}}\overline{F}_{\overline{Q},\dd^{(1)},\dd^{(2)}}\xrightarrow{i_{\dd^{(1)},\dd^{(2)}}} F_{\overline{Q},\dd^{(1)},\dd^{(2)}}
\]
and the projection $q_{\dd^{(1)},\dd^{(2)}}'\colon \overline{F}_{\overline{Q},\dd^{(1)},\dd^{(2)}}\rightarrow X_{\Pi_Q,\dd^{(1)}}\times X_{\Pi_Q,\dd^{(2)}}$ induced by $\overline{q}_{\dd^{(1)},\dd^{(2)}}$. The stack $\overline{\mathfrak{M}}_{\Pi_Q,\dd^{(1)},\dd^{(2)}}=\overline{F}_{\overline{Q},\dd^{(1)},\dd^{(2)}}/P_{\dd^{(1)},\dd^{(2)}}$ is the stack of short exact sequences of $\overline{Q}$-representations such that the subobject and the quotient are representations of $\Pi_Q$, respectively of dimension $\dd^{(1)}$ and $\dd^{(2)}$. The morphism $q'_{\dd^{(1)},\dd^{(2)}}$ induces a morphism $\overline{\mathfrak{M}}_{\Pi_Q,\dd^{(1)},\dd^{(2)}}\rightarrow \overline{\mathfrak{M}_{\Pi_Q,\dd^{(1)}}\times\mathfrak{M}_{\Pi_Q,\dd^{(2)}}}$, still denoted $q'_{\dd^{(1)},\dd^{(2)}}$. We let $\tilde{i}_{\dd^{(1)},\dd^{(2)}}\coloneqq i_{\dd^{(1)},\dd^{(2)}}\circ i_{\dd^{(1)},\dd^{(2)}}'$ and the same letter denotes the map after quotient by $P_{\dd^{(1)},\dd^{(2)}}$.

\item We let
\[\overline{\mathfrak{M}_{\Pi_Q,\dd^{(1)}}\times\mathfrak{M}_{\Pi_Q,\dd^{(2)}}}=(X_{\Pi_Q,\dd^{(1)}}\times X_{\Pi_Q,\dd^{(2)}})/P_{\dd^{(1)},\dd^{(2)}}
\]
where $P_{\dd^{(1)},\dd^{(2)}}$ acts on $X_{\Pi_{Q},\dd^{(1)}}\times X_{\Pi_{Q},\dd^{(2)}}$ via the quotient $P_{\dd^{(1)},\dd^{(2)}}\rightarrow \GL_{\dd^{(1)}}\times\GL_{\dd^{(2)}}$.

\item We still denote by $q_{\dd^{(1)},\dd^{(2)}}$ the map $\mathfrak{M}_{\Pi_Q,\dd^{(1)},\dd^{(2)}}\rightarrow\overline{\mathfrak{M}_{\Pi_Q,\dd^{(1)}}\times\mathfrak{M}_{\Pi_Q,\dd^{(2)}}}$ obtained from the map $q_{\dd^{(1)},\dd^{(2)}}$ of \eqref{item:extensionsPi_Q} after quotienting by $P_{\dd^{(1)},\dd^{(2)}}$.

\item \label{item:stackZ} We let $\mathfrak{Z}_{\dd^{(1)},\dd^{(2)}}=Z_{\dd^{(1)},\dd^{(2)}}/P_{\dd^{(1)},\dd^{(2)}}$ where $Z_{\dd^{(1)},\dd^{(2)}}\subset X_{\overline{Q},\dd^{(1)}}\times X_{\overline{Q},\dd^{(2)}}\times\mathfrak{p}_{\dd^{(1)},\dd^{(2)}}$ is the space of triples $(\rho^{(1)},\rho^{(2)},g)$ defined by the condition $\mu_{\dd^{(1)}}(\rho^{(1)})\times \mu_{\dd^{(1)}}(\rho^{(2)})=l(g)$ ($l$ is the projection onto the Levi \eqref{Pequivstr}). We have a closed immersion 
$
\overline{i_{\dd^{(1)}}\times i_{\dd^{(2)}}}:=(i_{\dd^{(1)}}\times i_{\dd^{(2)}}\times \{0\})\colon X_{\Pi_Q,\dd^{(1)}}\times X_{\Pi_Q,\dd^{(2)}}\rightarrow Z_{\dd^{(1)},\dd^{(2)}}.
$
We still denote by 
\[
\overline{i_{\dd^{(1)}}\times i_{\dd^{(2)}}}\colon \overline{\mathfrak{M}_{\Pi_Q,\dd^{(1)}}\times\mathfrak{M}_{\Pi_Q,\dd^{(2)}}}\rightarrow \mathfrak{Z}_{\dd^{(1)},\dd^{(2)}}
\]
the morphism obtained after quotienting by $P_{\dd^{(1)},\dd^{(2)}}$. We denote by
\[
\overline{q}'_{\dd^{(1)},\dd^{(2)}}\colon F_{\overline{Q},\dd^{(1)},\dd^{(2)}}\rightarrow Z_{\dd^{(1)},\dd^{(2)}}
\]
the morphism taking $\rho\mapsto (q_{\dd^{(1)},\dd^{(2)}}(\rho),\mu_{\dd}(\rho))$, and we denote by the same symbol the induced morphism of $P_{\dd^{(1)},\dd^{(2)}}$-quotients.

\item \label{item:vectorbundlesection} We let $\overline{\mathfrak{Z}}_{\dd^{(1)},\dd^{(2)}}=\overline{Z}_{\dd^{(1)},\dd^{(2)}}/P_{\dd^{(1)},\dd^{(2)}}$ where
$
\overline{Z}_{\dd^{(1)},\dd^{(2)}}\subset F_{\overline{Q},\dd^{(1)},\dd^{(2)}}\times\mathfrak{p}_{\dd^{(1)},\dd^{(2)}}
$
is the subspace of pairs $(\rho,g)$ such that $(\mu_{\dd^{(1)}}\times\mu_{\dd^{(2)}})q_{\dd^{(1)},\dd^{(2)}}(\rho)=l(g)$, with its natural $P_{\dd^{(1)},\dd^{(2)}}$-equivariant structure. The moment map induces a section $s_{\mu}\colon F_{\overline{Q},\dd^{(1)},\dd^{(2)}}\rightarrow \overline{Z}_{\dd^{(1)},\dd^{(2)}}$, defined by $s_{\mu}(x)=(x,\mu_{\dd}(x))$. By taking the quotient by $P_{\dd^{(1)},\dd^{(2)}}$, we obtain a morphism
\[
r\colon \overline{\mathfrak{Z}}_{\dd^{(1)},\dd^{(2)}}\rightarrow\FM_{\overline{Q},\dd^{(1)},\dd^{(2)}}
\]
with a section still denoted $s_{\mu}$. 

We let $\mathfrak{F}_{\Pi_Q,\dd^{(1)},\dd^{(2)}}=(\overline{F}_{\overline{Q},\dd^{(1)},\dd^{(2)}}\times\mathfrak{n}_{\dd^{(1)},\dd^{(2)}})/P_{\dd^{(1)},\dd^{(2)}}$. The morphism $r$ restricts to the projection of the total space of the vector bundle 
\[
\mathfrak{F}_{\Pi_Q,\dd^{(1)},\dd^{(2)}}\rightarrow \overline{\mathfrak{M}}_{\Pi_Q,\dd^{(1)},\dd^{(2)}}
\] 
with a section denoted by $s'_{\mu}$. Note that $\FM_{\Pi_Q,\dd^{(1)},\dd^{(2)}}$ is precisely the vanishing locus of this section.

\item \label{item:qtilde}We let $\tilde{q}_{\dd^{(1)},\dd^{(2)}}\colon \overline{\mathfrak{M}_{\Pi_Q,\dd^{(1)}}\times\mathfrak{M}_{\Pi_Q,\dd^{(2)}}}\rightarrow\mathfrak{M}_{\Pi_Q,\dd^{(1)}}\times\mathfrak{M}_{\Pi_Q,\dd^{(2)}}$ be the map induced by the group homomorphism $P_{\dd^{(1)},\dd^{(2)}}\rightarrow \GL_{\dd^{(1)}}\times\GL_{\dd^{(2)}}$. It is smooth of dimension $-\dd^{(1)}\cdot\dd^{(2)}=-\sum_{i\in Q_0}(\dd^{(1)})_i(\dd^{(2)})_i$.

\end{enumerate}
\subsection{The relative cohomological Hall algebra product for a preprojective algebra}
\label{subsection:cohaproductpreproj}
In this section, we recall how the relative cohomological Hall algebra product is defined for the stack of representations of the preprojective algebra of a quiver, following \cite{schiffmann2013cherednik,yang2018cohomological,davison2020bps}.

Let $\ulrelCoHA_{\Pi_Q}\coloneqq (\JH_{\Pi_Q})_*\BD\ulBoQ_{\mathfrak{M}_{\Pi_Q}}^{\vir}$. The relative cohomological Hall algebra is an algebra structure on the object $\ulrelCoHA_{\Pi_Q}$ of $\CD^+(\MHM(\mathcal{M}_{\Pi_Q}))$, i.e. a map
\[
 m\colon \ulrelCoHA_{\Pi_Q}\boxdot\ulrelCoHA_{\Pi_Q}\rightarrow \ulrelCoHA_{\Pi_Q}
\]
as in Definition \ref{ma_alg_def} and \S \ref{unbounded_cplx_sec}.

\subsubsection{Construction at the sheaf level}
\label{subsubsection:constructionproductSV}
Let $\dd^{(1)}, \dd^{(2)}\in\BoN^{Q_0}$ and $\dd=\dd^{(1)}+\dd^{(2)}$. Consider the following commutative diagram.

\begin{equation}
\label{equation:cohadiagram}
\begin{tikzcd}[column sep=large]
	{\mathfrak{M}_{\Pi_Q,\dd^{(1)}}\times\mathfrak{M}_{\Pi_Q,\dd^{(2)}}} & {\overline{\mathfrak{M}_{\Pi_Q,\dd^{(1)}}\times\mathfrak{M}_{\Pi_Q,\dd^{(2)}}}} & {\mathfrak{M}_{\Pi_Q,\dd^{(1)},\dd^{(2)}}} & {\mathfrak{M}_{\Pi_Q,\dd}} \\
	& {\mathfrak{Z}_{\dd^{(1)},\dd^{(2)}}} & {\mathfrak{M}_{\overline{Q},\dd^{(1)},\dd^{(2)}}} & {\mathfrak{M}_{\overline{Q},\dd}}
	\arrow["\tilde{q}_{\dd^{(1)},\dd^{(2)}}"',"\eqref{item:qtilde}", from=1-2, to=1-1]
	\arrow["{\overline{i_{\dd^{(1)}}\times i_{\dd^{(2)}}}}"',"\eqref{item:stackZ}", from=1-2, to=2-2]
	\arrow["{\tilde{i}_{\dd^{(1)},\dd^{(2)}}}","\eqref{item:extensionoverlineQPiQ}"', from=1-3, to=2-3]
	\arrow["{i_{\dd}}","\eqref{item:repspacepp}"', from=1-4, to=2-4]
	\arrow["{\overline{p}_{\dd^{(1)},\dd^{(2)}}}"',"\eqref{item:extensionoverlineQ}", from=2-3, to=2-4]
	\arrow["{\overline{q}'_{\dd^{(1)},\dd^{(2)}}}","\eqref{item:stackZ}"', from=2-3, to=2-2]
	\arrow["q_{\dd^{(1)},\dd^{(2)}}"',"\eqref{item:extensionsPi_Q}", from=1-3, to=1-2]
	\arrow["p_{\dd^{(1)},\dd^{(2)}}","\eqref{item:extensionsPi_Q}"', from=1-3, to=1-4]
	\arrow["\lrcorner"{anchor=center, pos=0.125, rotate=-90}, draw=none, from=1-3, to=2-2]
	\arrow["\lrcorner"{anchor=center, pos=0.125}, draw=none, from=1-3, to=2-4]
\end{tikzcd}
 \end{equation}
where both the squares are Cartesian. The vertical arrows are closed immersions. In the sequel, we drop the indices $\dd^{(1)},\dd^{(2)}$ for horizontal maps since the dimension vectors $\dd^{(1)},\dd^{(2)}$ are fixed.

The map $p=p_{\dd^{(1)},\dd^{(2)}}$ is proper and representable so that we have a morphism of complexes
\begin{equation}
\label{equation:pushforward}
\alpha_{\dd^{(1)},\dd^{(2)}}\colon p_*\BD\BoQ_{\mathfrak{M}_{\Pi_Q,\dd^{(1)},\dd^{(2)}}}\rightarrow \BD\BoQ_{\mathfrak{M}_{\Pi_Q,\dd}} 
\end{equation}
obtained by dualizing the canonical adjunction map $\BoQ_{\mathfrak{M}_{\Pi_Q,\dd}}\rightarrow p_*\BoQ_{\mathfrak{M}_{\Pi_Q,\dd^{(1)},\dd^{(2)}}}$.

Since $\mathfrak{Z}_{\dd^{(1)},\dd^{(2)}}$ and $\mathfrak{M}_{\overline{Q},\dd^{(1)},\dd^{(2)}}$ are smooth, the map $\overline{q}'=\bar{q}'_{\dd^{(1)},\dd^{(2)}}$ is strongly l.c.i., and therefore there exists a refined pullback by $q=q_{\dd^{(1)},\dd^{(2)}}$ which we recall at the level of constructible complexes (see also \S \ref{lci_pb_sec}). The map $\overline{q}'=\overline{q}'_{\dd^{(1)},\dd^{(2)}}$ induces the adjunction map
\begin{equation}
\label{equation:adjunctionqbar}
\BoQ_{\mathfrak{Z}_{\dd^{(1)},\dd^{(2)}}}\rightarrow \overline{q}'_*\BoQ_{\mathfrak{M}_{\overline{Q},\dd^{(1)},\dd^{(2)}}} \end{equation}
which dualizes to
\begin{equation}
 \label{equation:dualadjunctionqbar}
 \overline{q}'_!\BD\BoQ_{\mathfrak{M}_{\overline{Q},\dd^{(1)},\dd^{(2)}}}\rightarrow \BD\BoQ_{\mathfrak{Z}_{\dd^{(1)},\dd^{(2)}}}.
\end{equation}
Since for a smooth stack $\mathfrak{X}$ , we have $\BD\BoQ_{\mathfrak{X}}=\BoQ_{\mathfrak{X}}[2\dim\FX]$, we can rewrite \eqref{equation:dualadjunctionqbar} as
\begin{equation}
 \label{equation:dualadjunctionqbarrewritten}
\overline{q}'_!\BoQ_{\mathfrak{M}_{\overline{Q},\dd^{(1)},\dd^{(2)}}}[2\dim\mathfrak{M}_{\overline{Q},\dd^{(1)},\dd^{(2)}}]\rightarrow \BoQ_{\mathfrak{Z}_{\dd^{(1)},\dd^{(2)}}}[2\dim\mathfrak{Z}_{\dd^{(1)},\dd^{(2)}}].
\end{equation}
Applying $(\overline{i_{\dd^{(1)}}\times i_{\dd^{(2)}}})^*$ to \eqref{equation:dualadjunctionqbarrewritten} together with the base change isomorphism $(\overline{i_{\dd^{(1)}}\times i_{\dd^{(2)}}})^*\overline{q}'_!\cong q_!\tilde{i}_{\dd^{(1)},\dd^{(2)}}^*$, we get
\begin{equation}
 \label{equation:basechange}
 q_!\BoQ_{\mathfrak{M}_{\Pi_Q,\dd^{(1)},\dd^{(2)}}}[2\dim\mathfrak{M}_{\overline{Q},\dd^{(1)},\dd^{(2)}}]\rightarrow\BoQ_{\overline{\FM_{\Pi_Q,\dd^{(1)}\times\FM_{\Pi_Q,\dd^{(2)}}}}}[2\dim\mathfrak{Z}_{\dd^{(1)},\dd^{(2)}}].
\end{equation}
Dualizing \eqref{equation:basechange}, we get the morphism
\begin{equation}
 \label{equation:virtualpullback}
(\BD\BoQ_{\overline{\FM_{\Pi_Q,\dd^{(1)}}\times\FM_{\Pi_Q,\dd^{(2)}}}})[-2\dim\mathfrak{Z}_{\dd^{(1)},\dd^{(2)}}]\rightarrow q_*(\BD\BoQ_{\mathfrak{M}_{\Pi_{Q},\dd^{(1)},\dd^{(2)}}})[-2\dim \mathfrak{M}_{\overline{Q},\dd^{(1)},\dd^{(2)}}]
\end{equation}
which is the virtual pullback by $q$ at the level of complexes of sheaves.

The map $\tilde{q}$ is smooth so that denoting by $\dim \tilde{q}$ its relative dimension, we may identify $(\tilde{q})^!= (\tilde{q})^* [2\dim \tilde{q}]$. Using this identity, the inverse of the isomorphism $(\tilde{q})^*\BoQ_{\mathfrak{M}_{\Pi_Q,\dd^{(1)}}\times\mathfrak{M}_{\Pi_Q,\dd^{(2)}}}\rightarrow \BoQ_{\overline{\mathfrak{M}}_{\Pi_Q,\dd^{(1)},\dd^{(2)}}}$ can be rewritten
\begin{equation}
\BoQ_{\overline{\mathfrak{M}}_{\Pi_Q,\dd^{(1)},\dd^{(2)}}}\rightarrow(\tilde{q})^!\BoQ_{\mathfrak{M}_{\Pi_Q,\dd^{(1)}}\times\mathfrak{M}_{\Pi_Q,\dd^{(2)}}}[-2\dim \tilde{q}].
\end{equation}
By the adjunction $((\tilde{q})_!,(\tilde{q})^!)$, we get a map
\begin{equation}
 (\tilde{q})_!\BoQ_{\overline{\mathfrak{M}}_{\Pi_Q,\dd^{(1)},\dd^{(2)}}}\rightarrow\BoQ_{\mathfrak{M}_{\Pi_Q,\dd^{(1)}}\times\mathfrak{M}_{\Pi_Q,\dd^{(2)}}}[-2\dim \tilde{q}].
\end{equation}
whose Verdier dual is the pullback by $\tilde{q}$:
\begin{equation}
 \label{equation:pullbackqprime}
 (\BD\BoQ_{\mathfrak{M}_{\Pi_Q,\dd^{(1)}}\times\mathfrak{M}_{\Pi_Q,\dd^{(2)}}})[2\dim \tilde{q}]\rightarrow (\tilde{q})_*\BD\BoQ_{\overline{\mathfrak{M}}_{\Pi_Q,\dd^{(1)},\dd^{(2)}}}.
\end{equation}

By composing \eqref{equation:pullbackqprime} shifted by $2(\dim\mathfrak{M}_{\overline{Q}\,\dd^{(1)},\dd^{(2)}}-\dim\mathfrak{Z}_{\dd^{(1)},\dd^{(2)}})$ with $(\tilde{q})_*$ applied to \eqref{equation:virtualpullback} shifted by $2\dim\mathfrak{M}_{\overline{Q},\dd^{(1)},\dd^{(2)}}$, we get the morphism
\begin{equation}
 \label{equation:pullbackqqprime}
v_{\dd',\dd''}\colon \BD\BoQ_{\mathfrak{M}_{\Pi_Q,\dd^{(1)}}\times\mathfrak{M}_{\Pi_Q,\dd^{(2)}}}[2(\dim \tilde{q}+\dim\mathfrak{M}_{\overline{Q},\dd^{(1)},\dd^{(2)}}-\dim\mathfrak{Z}_{\dd^{(1)},\dd^{(2)}})]\rightarrow (\tilde{q})_*q_*\BD\BoQ_{\mathfrak{M}_{\Pi_{Q},\dd^{(1)},\dd^{(2)}}} 
\end{equation}
which is the virtual pullback by $\tilde{q}\circ q$.
\begin{lemma}
\label{lemma:shifts}
 We have the equalities
 \[
 \begin{aligned}
 \dim \tilde{q}+\dim\mathfrak{M}_{\overline{Q},\dd^{(1)},\dd^{(2)}}-\dim\mathfrak{Z}_{\dd^{(1)},\dd^{(2)}}&=-\langle\dd^{(1)},\dd^{(2)}\rangle_{Q}-\langle\dd^{(2)},\dd^{(1)}\rangle_{Q}\\
 &=-\frac{1}{2}(\dd,\dd)_{Q}+\frac{1}{2}(\dd^{(1)},\dd^{(1)})_{Q}+\frac{1}{2}(\dd^{(2)},\dd^{(2)})_{Q}.
 \end{aligned}
 \]
 \end{lemma}
 \begin{proof}
 This is a straightforward calculation.
 \end{proof}

Recall that we define $\BoQ_{\FM_{\Pi_Q,\dd}}^{\vir}\coloneqq \BoQ_{\FM_{\Pi_Q,\dd}}[-(\dd,\dd)_{Q}]$. We now introduce the coarse moduli space of the preprojective algebra with the first row of the diagram \eqref{equation:cohadiagram} as in the following commutative diagram:
\begin{equation}
\label{equation:diagramcoarsemodspace}
 \begin{tikzcd}
	{{\mathfrak{M}_{\Pi_Q,\dd^{(1)}}\times\mathfrak{M}_{\Pi_Q,\dd^{(2)}}}} & {{\overline{\mathfrak{M}_{\Pi_Q,\dd^{(1)}}\times\FM_{\Pi_Q,\dd^{(2)}}}}} & {{\mathfrak{M}_{\Pi_Q,\dd^{(1)},\dd^{(2)}}}} & {{\mathfrak{M}_{\Pi_Q\dd}}} \\
	{\mathcal{M}_{\Pi_Q,\dd^{(1)}}\times\mathcal{M}_{\Pi_Q,\dd^{(2)}}} &&& {\mathcal{M}_{\Pi_Q,\dd}.}
	\arrow["{\tilde{q}}"', from=1-2, to=1-1]
	\arrow["q"', from=1-3, to=1-2]
	\arrow["p", from=1-3, to=1-4]
	\arrow["\oplus"', from=2-1, to=2-4]
	\arrow["{\JH_{\dd}}", from=1-4, to=2-4]
	\arrow["{\JH_{\dd^{(1)}}\times\JH_{\dd^{(2)}}}"', from=1-1, to=2-1]
\end{tikzcd}
\end{equation}
By composing $\oplus_{*}(\JH_{\dd^{(1)}}\times\JH_{\dd^{(2)}})_*$ applied to \eqref{equation:pullbackqqprime} shifted by $(\dd,\dd)_{Q}$ with $\JH_{*}$ applied to $\eqref{equation:pushforward}$ shifted by $(\dd,\dd)_{Q}$ and using that \eqref{equation:diagramcoarsemodspace} commutes, we get the map
\[
 m_{\dd^{(1)},\dd^{(2)}}\colon \oplus_*(\JH_{\dd^{(1)}}\times\JH_{\dd^{(2)}})_*\BD\BoQ^{\vir}_{\mathfrak{M}_{\Pi_Q,\dd^{(1)}}\times \mathfrak{M}_{\Pi_Q,\dd^{(2)}}}\rightarrow(\JH_{\dd})_*\BD\BoQ^{\vir}_{\mathfrak{M}_{\Pi_Q,\dd}}
\]
which is the $(\dd^{(1)},\dd^{(2)})$-graded piece of the relative CoHA product. 
\subsubsection{Upgrade to mixed Hodge modules}

Since each $\FM_{\Pi_Q,\dd}$ is a global quotient stack, we can upgrade the morphisms $m_{\dd^{(1)},\dd^{(2)}}$ in the standard way recalled in \S \ref{unbounded_cplx_sec}. The pullback of \eqref{equation:pullbackqqprime} along a smooth morphism $j\colon U_N\rightarrow \FM_{\Pi_Q,\dd^{(1)}}\times\FM_{\Pi_Q,\dd^{(2)}}$ admits a canonical upgrade to the category of mixed Hodge modules on $U_N$, which we then apply $\oplus_{*}(\JH_{\dd^{(1)}}\times\JH_{\dd^{(2)}})_*j_*$ to. Picking $U_N$ the schemes appearing in an acyclic cover of $\FM_{\Pi_Q,\dd^{(1)}}\times\FM_{\Pi_Q,\dd^{(2)}}$ enables us to upgrade the morphisms $\ptau{\leq n}v_{\dd',\dd''}$ to morphisms of mixed Hodge module complexes, yielding a morphism $\ul{v}_{\dd',\dd''}$ in $\CD^+(\MHM(\CM_{\Pi_Q}))$. We upgrade $\JH_*\alpha_{\dd^{(1)},\dd^{(2)}}$ the same way, and composing the resulting morphisms, we define the Hall algebra multiplication on $\ulrelCoHA_{\Pi_Q}$. We generalise this Hall algebra, defined in the category of complexes of mixed Hodge modules, to arbitrary categories satisfying Assumptions \ref{p_assumption}-\ref{ds_fin} and \ref{assumption:associativity} in the next section.

\subsection{Comparison of the products for preprojective CoHAs}
\label{subsection:comparison_preproj}
If $\mathfrak{M}_{\Pi_Q}$ is the moduli stack of representations of the preprojective algebra of a quiver, the constructions of \S\ref{subsection:cohaproductpreproj} and \S\ref{subsubsection:ThecohaproductKV} provide us with two a priori different products on $\ulrelCoHA_{\varpi}=\varpi_*\BD \ulBoQ^{\vir}_{\mathfrak{M}}$. In this section, we show that these two products coincide.

\subsubsection{Total space of the RHom complex for the doubled quiver}

We give an explicit presentation of the RHom complex (in the category of $\BoC\overline{Q}$-modules) on $\mathfrak{M}_{\overline{Q},\dd^{(1)}}\times\mathfrak{M}_{\overline{Q},\dd^{(2)}}$ for $\dd^{(1)},\dd^{(2)}\in\BoN^{Q_0}$.

Let $V_{\dd^{(j)}}$, $j=1,2$ be trivial $\BoN^{Q_0}$-graded vector bundles on $X_{\overline{Q},\dd^{(j)}}$ of rank $\dd^{(j)}$. We have a $2$-term complex of vector bundles on $X_{\overline{Q},\dd^{(1)}}\times X_{\overline{Q},\dd^{(2)}}$, situated in cohomological degrees $-1$ and $0$:
\[
 \mathcal{C}^{\bullet}_{\overline{Q}}=\left(\bigoplus_{i\in Q_0}\Hom_{\BoC}((V_{\dd^{(2)}})_i,(V_{\dd^{(1)}})_i)\xrightarrow{d}\bigoplus_{i\xrightarrow{\alpha}j\in \overline{Q}_1}\Hom_{\BoC}((V_{\dd^{(2)}})_i,(V_{\dd^{(1)}})_j)\right).
\]
Let $x=(x_{\alpha})_{\alpha\in \overline{Q}_1}\in X_{\overline{Q},\dd^{(1)}}$ and $y=(y_{\alpha})_{\alpha\in \overline{Q}_1}\in X_{\overline{Q},\dd^{(2)}}$. For $z\in\bigoplus_{i\in Q_0}\Hom_{\BoC}((V_{\dd^{(1)}})_i,(V_{\dd^{(2)}})_i)$, we have
\[
 d_{(x,y)}z=(z_jy_{\alpha}-x_{\alpha}z_i)_{i\xrightarrow{\alpha} j\in\overline{Q}_1}.
\]
We have a natural $\GL_{\dd^{(1)}}\times\GL_{\dd^{(2)}}$-equivariant structure on $\mathcal{C}_{\overline{Q}}$ so that we can consider $\mathcal{C}_{\overline{Q}}$ as a $2$-term complex on $\mathfrak{M}_{\overline{Q},\dd^{(1)}}\times \mathfrak{M}_{\overline{Q},\dd^{(2)}}$. The following is easy and well-known (cf. Appendix \ref{algebra_constr_sec}).

\begin{proposition}
 The $2$-term complex $\mathcal{C}^{\bullet}_{\overline{Q}}$ on $\mathfrak{M}_{\overline{Q},\dd^{(1)}}\times \mathfrak{M}_{\overline{Q},\dd^{(2)}}$ is quasi-isomorphic to the RHom complex shifted by one.
\end{proposition}

\begin{corollary}
 The total space of $\mathcal{C}^{\bullet}_{\overline{Q}}$ is isomorphic as a vector bundle stack over $\mathfrak{M}_{\overline{Q},\dd^{(1)}}\times \mathfrak{M}_{\overline{Q},\dd^{(2)}}$ to the map
 \[
 \pi_{\dd^{(1)},\dd^{(2)}}\colon\mathfrak{M}_{\overline{Q},\dd^{(1)},\dd^{(2)}}\rightarrow \mathfrak{M}_{\overline{Q},\dd^{(1)}}\times \mathfrak{M}_{\overline{Q},\dd^{(2)}}
 \]
 given in \eqref{item:extensionoverlineQ}.
\end{corollary}

\subsubsection{Total space of the RHom complex for preprojective algebras}
Recall the explicit presentation of the RHom complex for preprojective algebras $\CC_{\mathscr{G}_2(Q)}$ given in \eqref{equation:RHompreprojective}. From the definitions of the complexes $\mathcal{C}^{\bullet}_{\mathscr{G}_2(Q)}$ and $\mathcal{C}^{\bullet}_{\overline{Q}}$, the following three lemmas are immediate.

\begin{lemma}
 The restriction of $\mathcal{C}^{\bullet}_{\overline{Q}}$ to $\mathfrak{M}_{\Pi_{Q},\dd^{(1)}}\times \mathfrak{M}_{\Pi_{Q},\dd^{(2)}}$ is equal to $\mathcal{C}_{\mathscr{G}_2(Q)}^{\leq 0}$.
\end{lemma}

\begin{lemma}
 The total space $\Tot(\mathcal{C}_{\mathscr{G}_2(Q)}^{\leq 0})$ is isomorphic as a vector bundle stack over $\mathfrak{M}_{\Pi_Q,\dd^{(1)}}\times\mathfrak{M}_{\Pi_Q,\dd^{(2)}}$ to the map $q'_{\dd^{(1)},\dd^{(2)}}\colon\overline{\mathfrak{M}}_{\Pi_Q,\dd^{(1)},\dd^{(2)}}\rightarrow\mathfrak{M}_{\Pi_Q,\dd^{(1)}}\times\mathfrak{M}_{\Pi_Q,\dd^{(2)}}$ defined in \eqref{item:extensionoverlineQPiQ} \S\ref{subsection:notationsquivreps}.
\end{lemma}

\begin{lemma}
 The vector bundle $V=(q'_{\dd^{(1)},\dd^{(2)}})^*
 \mathcal{C}_{\mathscr{G}_2(Q)}^1$ on $\overline{\mathfrak{M}}_{\Pi_Q,\dd^{(1)},\dd^{(2)}}$ is isomorphic to the vector bundle $\mathfrak{F}_{\Pi_Q,\dd^{(1)},\dd^{(2)}}\rightarrow\overline{\mathfrak{M}}_{\Pi_Q,\dd^{(1)},\dd^{(2)}}$ defined in \eqref{item:vectorbundlesection} of \S\ref{subsection:notationsquivreps}. The section of this vector bundle induced by $\mu$ is the section $s'_{\mu}$ described in \eqref{item:vectorbundlesection}.
\end{lemma}

\subsubsection{The comparison diagram}

We have the following commutative diagram with Cartesian squares.

\[
\begin{tikzcd}[column sep=large]
	{\mathfrak{M}_{\Pi_Q,\dd^{(1)}}\times\mathfrak{M}_{\Pi_Q,\dd^{(2)}}} & {\overline{\mathfrak{M}_{\Pi_Q,\dd^{(1)}}\times\mathfrak{M}_{\Pi_Q,\dd^{(2)}}}} & {\overline{\mathfrak{M}}_{\Pi_Q,\dd^{(1)},\dd^{(2)}}} & {\mathfrak{M}_{\Pi_Q,\dd^{(1)},\dd^{(2)}}} \\
	&& {\mathfrak{F}_{\Pi_Q,\dd^{(1)},\dd^{(2)}}} & {\overline{\mathfrak{M}}_{\Pi_Q,\dd^{(1)},\dd^{(2)}}} \\
	& {\mathfrak{Z}_{\dd^{(1)},\dd^{(2)}}} & {\overline{\mathfrak{Z}}_{\overline{Q},\dd^{(1)},\dd^{(2)}}} & {\mathfrak{M}_{\overline{Q},\dd^{(1)},\dd^{(2)}}}
	\arrow["\overline{i_{\dd^{(1)}}\times i_{\dd^{(2)}}}"',from=1-2, to=3-2]
	\arrow[from=1-3, to=2-3,"0_V"]
	\arrow[from=2-3, to=3-3]
	\arrow[from=3-3, to=3-2]
	\arrow[from=1-3, to=1-2]
	\arrow[from=1-4, to=1-3]
	\arrow[from=1-4, to=2-4]
	\arrow[from=2-4, to=3-4]
	\arrow[from=3-4, to=3-3,"s_{\mu}"]
	\arrow[from=2-4, to=2-3, "s'_{\mu}"]
	\arrow["\lrcorner"{anchor=center, pos=0.125, rotate=-90}, draw=none, from=1-3, to=3-2]
	\arrow["\lrcorner"{anchor=center, pos=0.125, rotate=-90}, draw=none, from=1-4, to=2-3]
	\arrow["\lrcorner"{anchor=center, pos=0.125, rotate=-90}, draw=none, from=2-4, to=3-3]
	\arrow["\tilde{q}_{\dd^{(1)},\dd^{(2)}}"',from=1-2, to=1-1]
\end{tikzcd}
\]
This diagram (without the left-most map) comes from the quotient by $P_{\dd^{(1)},\dd^{(2)}}$ of the following diagram (the numbers refer to the item of \S\ref{subsection:notationsquivreps} where the corresponding map is defined):
\[
 \begin{tikzcd}[column sep=large]
	{X_{\Pi_Q,\dd^{(1)}}\times X_{\Pi_Q,\dd^{(2)}}} & {\overline{F}_{\overline{Q},\dd^{(1)},\dd^{(2)}}} & {F_{\Pi_Q,\dd^{(1)},\dd^{(2)}}} \\
	& {\overline{F}_{\overline{Q},\dd^{(1)},\dd^{(2)}}\times\mathfrak{n}_{\dd^{(1)},\dd^{(2)}}} & {\overline{F}_{\overline{Q},\dd^{(1)},\dd^{(2)}}} \\
	{Z_{\dd^{(1)},\dd^{(2)}}} & {\overline{Z}_{\dd^{(1)},\dd^{(2)}}} & {F_{\overline{Q},\dd^{(1)},\dd^{(2)}}}
	\arrow["{q'_{\dd^{(1)},\dd^{(2)}}}"',"\eqref{item:extensionoverlineQPiQ}", from=1-2, to=1-1]
	\arrow["{i_{\dd^{(1)}}\times i_{\dd^{(2)}}\times 0}"',"\eqref{item:stackZ}", from=1-1, to=3-1]
	\arrow["\id\times 0"',from=1-2, to=2-2]
	\arrow["{i_{\dd^{(1)},\dd^{(2)}}\times i_{\mathfrak{n},\dd^{(1)},\dd^{(2)}}}"',"\eqref{Pequivstr}", from=2-2, to=3-2]
	\arrow["{q_{\dd^{(1)},\dd^{(2)}}}", from=3-2, to=3-1]
	\arrow["{s_{\mu}}","\eqref{item:vectorbundlesection}"', from=3-3, to=3-2]
	\arrow["{i'_{\dd^{(1)},\dd^{(2)}}}"',"\eqref{item:extensionoverlineQPiQ}", from=1-3, to=1-2]
	\arrow["{s'_{\mu}}"',"\eqref{item:vectorbundlesection}", from=2-3, to=2-2]
	\arrow["{i'_{\dd^{(1)},\dd^{(2)}}}","\eqref{item:extensionoverlineQPiQ}"', from=1-3, to=2-3]
	\arrow["{i_{\dd^{(1)},\dd^{(2)}}}","\eqref{item:extensionoverlineQPiQ}"', from=2-3, to=3-3]
	\arrow["\lrcorner"{anchor=center, pos=0.125, rotate=-90}, draw=none, from=1-2, to=3-1]
	\arrow["\lrcorner"{anchor=center, pos=0.125, rotate=-90}, draw=none, from=1-3, to=2-2]
	\arrow["\lrcorner"{anchor=center, pos=0.125, rotate=-90}, draw=none, from=2-3, to=3-2]
\end{tikzcd}
\]
The outer square
\[
 \begin{tikzcd}[column sep=large]
	{X_{\Pi_Q,\dd^{(1)}}\times X_{\Pi_Q,\dd^{(2)}}} & {F_{\Pi_Q,\dd^{(1)},\dd^{(2)}}} \\
	{Z_{\dd^{(1)},\dd^{(2)}}} & {F_{\overline{Q},\dd^{(1)},\dd^{(2)}}}
	\arrow["q_{\dd^{(1)}\times\dd^{(2)}}"',"\eqref{item:extensionsPi_Q}",from=1-2, to=1-1]
	\arrow["\overline{i_{\dd^{(1)}}\times i_{\dd^{(2)}}}"',"\eqref{item:stackZ}",from=1-1, to=2-1]
	\arrow["\tilde{i}_{\dd^{(1)},\dd^{(2)}}","\eqref{item:extensionoverlineQPiQ}"',from=1-2, to=2-2]
	\arrow["\overline{q}'_{\dd^{(1)},\dd^{(2)}}","\eqref{item:stackZ}"',from=2-2, to=2-1]
	\arrow["\lrcorner"{anchor=center, pos=0.125, rotate=-90}, draw=none, from=1-2, to=2-1]
\end{tikzcd}
\]
is the left-most Cartesian square of \ref{equation:cohadiagram} (before quotienting by $P_{\dd^{(1)},\dd^{(2)}}$) used to defined the refined pullback in the relative version of the cohomological Hall algebra product of Schiffmann--Vasserot \S\ref{subsubsection:constructionproductSV}. The top-right square together with the top-left arrow is the diagram used to define the relative version of the cohomological Hall algebra in \S \ref{subsubsection:ThecohaproductKV} (the version with $3$-term complexes).

The map $q'_{\dd^{(1)},\dd^{(2)}}$ is smooth, so in particular strongly l.c.i. and so the left-most square has no excess intersection bundle ($q'_{\dd^{(1)},\dd^{(2)}}$ and $q_{\dd^{(1)},\dd^{(2)}}$ are both of codimension $-\sum_{\i\xrightarrow{\alpha}j\in\overline{Q}_1}(\dd^{(1)})_i(\dd^{(2)})_j$). The morphisms $s_{\mu}$, $s'_{\mu}$ are also strongly l.c.i. (as sections of vector bundles) of the same codimension (that equals $\dim\mathfrak{n}_{\dd^{(1)},\dd^{(2)}}$). By Lemma \ref{zei_lemma} we obtain the identification of the two Hall algebra products constructed in \S\ref{subsection:cohaproductpreproj} and \S\ref{subsection:relativecohaproductKV}, yielding the following theorem:

\begin{theorem}
\label{theorem:cohappcoincide}
 For any $\dd^{(1)},\dd^{(2)}\in\BoN^{Q_0}$, the morphisms $m_{\dd^{(1)},\dd^{(2)}}$ defined in \S\ref{subsection:cohaproductpreproj} and $m_{\dd^{(1)},\dd^{(2)}}$ defined in \S\ref{subsubsection:ThecohaproductKV} (for the stack of representations of a preprojective algebra) coincide.
\end{theorem}

\section{2-dimensional categories from geometry: sheaves on surfaces}
\label{sec:2dim_categories}
The categories conforming to the geometric setup of \S \ref{subsubsection:gsetup} and \S \ref{subsubsection:assumptionCoHAproduct} that we consider in this paper come from two constructions, and are of geometric or algebraic origin. In the next two appendices we explain these two constructions, and check that the various assumptions of \S \ref{subsubsection:gsetup} and \S \ref{subsubsection:assumptionCoHAproduct} are met.
\label{geometry_constr_sec}

Let $X$ be a complex projective variety. To start with we do not make any assumption on the smoothness or dimension of $X$.
\subsection{The moduli stack of semistable coherent sheaves}
\label{subsection:propernessp}
We fix an embedding $X\rightarrow \mathbf{P}^N$ of $X$ in some projective space, so that we have a distinguished choice of a very ample line bundle $\mathcal{O}_X(1)$ on $X$ obtained as the restriction to $X$ of $\mathcal{O}_{\mathbf{P}^N}(1)$. For $\CF$ a coherent sheaf on $X$ we write $\CF(n)\coloneqq \CF\otimes\CO_X(1)^{\otimes n}$.

Let $\mathcal{F}$ be a coherent sheaf on $X$. We say that $\mathcal{F}$ is strongly generated by $\mathcal{O}_X(n)$ if the natural map
\[
 \Hom(\mathcal{O}_X(n),\mathcal{F})\otimes_{\BoC}\mathcal{O}_X(n)\rightarrow\mathcal{F}
 \]
is an epimorphism of coherent sheaves and
\[
 \Ext^i(\mathcal{O}_X(n),\mathcal{F})=\HO^i(X,\mathcal{F}(-n))=0
 \]
 for $i>0$.

 Let $P(t)\in\BoQ[t]$ be a polynomial and $\mathfrak{Coh}_{n,P(t)}(X)$ be the substack of $\mathfrak{Coh}_{P(t)}(X)$ parametrising coherent sheaves on $X$ with Hilbert polynomial $P(t)$, strongly generated by $\mathcal{O}_X(-n)$. As the conditions of being an epimorphism and of vanishing of cohomology are open conditions, this is an open substack of $\mathfrak{Coh}_{P(t)}(X)$. It can be realised as the global quotient of an open subscheme $Q_X(n,P(t))$ of the Quot-scheme $\Quot_X(n,P(t)):=\Quot_X(\mathcal{O}_X(-n)\otimes\BoC^{P(n)},P(t))$ of quotients $\mathcal{O}_X(-n)^{P(n)}\xrightarrow{\varphi}\mathcal{F}$ where $\mathcal{F}$ has Hilbert polynomial $P(t)$. The open subscheme 
 \[
 Q_X(n,P(t))\subset \Quot_X(n,P(t))
 \]
 consists of epimorphisms $\mathcal{O}_X(-n)\otimes\BoC^{P(n)}\xrightarrow{\varphi}\mathcal{F}$ such that $\mathcal{F}$ is strongly generated by $\mathcal{O}_X(-n)$ and $\varphi$ induces an isomorphism $\BoC^{P(n)}\rightarrow \HO^0(X,\mathcal{F}(n))$. 

 There is a natural action of $\GL_{P(n)}$ on $Q_X(n,P(t))$, via the action on the factor $\BoC^{P(n)}$ of $\mathcal{O}_X(-n)\otimes\BoC^{P(n)}$. We have
 \[
 \mathfrak{Coh}_{n,P(t)}(X)\simeq Q_X(n,P(t))/\GL_{P(n)}.
 \]
It is known (see \cite[Sec.2]{huybrechts2010geometry}) that the $\GL_{P(n)}$ action on $\Quot_X(n,P(t))$ admits a linearisation, from which it follows that the $\GL_{P(n)}$-quotient $\mathfrak{Coh}_{n,P(t)}(X)$ satisfies the resolution property (Assumption \ref{q_assumption1}).

\subsection{Projectivity}
 Let $P(t)$ and $R(t)$ be polynomials. We let $\Quot_X(n,P(t),R(t))$ be the Quot-scheme parametrising quotients $\mathcal{O}_X(-n)^{P(n)}\xrightarrow{\varphi}\mathcal{F}\xrightarrow{\psi}\mathcal{G}$ where $\mathcal{F}$ (resp. $\mathcal{G}$) has Hilbert polynomial $P(t)$ (resp. $R(t)$). We let $Q_X(n,P(t),R(t))$ be the open subscheme of such quotients for which $\mathcal{F}$ is strongly generated by $\mathcal{O}_X(-n)$ and $\varphi$ induces an isomorphism $\BoC^{P(n)}\rightarrow \HO^0(X,\mathcal{F}(n))$. The action of $\GL_{P(n)}$ on $\mathcal{O}_X(-n)\otimes\BoC^{P(n)}$ induces an action of $\GL_{P(n)}$ on $Q_X(n,P(t),R(t))$. Then,
 \[
 Q_{X}(n,P(t),R(t))/\GL_{P(n)}\simeq \mathfrak{Exact}_{n,P(t),R(t)}(X),
 \]
 where $\mathfrak{Exact}_{n,P(t),R(t)}(X)$ is the stack of short exact sequences $0\rightarrow\mathcal{F}\rightarrow\mathcal{H}\rightarrow\mathcal{G}\rightarrow0$ where $\mathcal{H}$ is strongly generated by $\mathcal{O}_X(-n)$, $\mathcal{H}$ has Hilbert polynomial $P(t)$ and $\mathcal{G}$ has Hilbert polynomial $R(t)$. We have a map $\Quot_X(n,P(t),R(t))\rightarrow \Quot_X(n,P(t))$ forgetting the map $\psi$, and a similar map obtained by composition $\Quot_X(n,P(t),R(t))\rightarrow \Quot_X(n,R(t))$, $(\varphi,\psi)\mapsto \psi\circ\varphi$.

 The map $\Quot_X(n,P(t),R(t))\rightarrow \Quot_X(n,P(t))$ is projective as both $\Quot_X(n,P(t),R(t))$ and $\Quot_X(n,P(t))$ are projective schemes over $\BoC$.
 
By definition, we have a Cartesian square
\[
\begin{tikzcd}
	{Q_X(n,P(t),R(t))} & {Q_{X}(n,P(t))} \\
	{\Quot_X(n,P(t),R(t))} & {\Quot_X(n,P(t))}.
	\arrow[from=1-1, to=2-1]
	\arrow[from=2-1, to=2-2]
	\arrow[from=1-1, to=1-2]
	\arrow[from=1-2, to=2-2]
	\arrow["\lrcorner"{anchor=center, pos=0.125}, draw=none, from=1-1, to=2-2]
\end{tikzcd} 
\]

 By base change, the morphism $Q_X(n,P(t),R(t))\rightarrow Q_X(n,P(t))$ is projective.

\begin{proposition}
\label{gl_pr}
The morphism of stacks
\begin{align*}
p_{P(t),R(t)}\colon\mathfrak{Exact}_{n,P(t),R(t)}(X)&\longrightarrow\mathfrak{Coh}_{n,P(t)}(X)\\
(0\rightarrow\mathcal{F}\rightarrow\mathcal{H}\rightarrow\mathcal{G}\rightarrow0)&\longmapsto\mathcal{H}
\end{align*}
is a representable projective morphism of stacks. In particular, $p_{P(t),R(t)}$ satisfies Assumption \ref{p_assumption} from \S\ref{subsubsection:gsetup}.
\end{proposition}
\begin{proof}
This morphism is obtained from the quotient by $\GL_{P(n)}$ of the natural projective morphism $Q_X(n,P(t),R(t))\rightarrow Q_X(n,P(t))$ between schemes, so it is representable and proper.
\end{proof}

We define the set of \textit{reduced polynomials} as equivalence classes, by imposing the equivalence relation on $\BoQ[t]$: $P(t)\sim Q(t)$ if $P(t)=\lambda Q(t)$ for some $\lambda\in\BoQ\setminus\{0\}$. We denote equivalence classes by lowercase letters $p(t)\in\BoQ[t]/\sim$. For $p(t)\in\BoQ[t]/\sim $ we denote by $\mathfrak{Coh}^{\sst}_{p(t)}(X)$ the stack of coherent sheaves which are either the zero sheaf, or semistable with reduced Hilbert polynomial $p(t)$. We denote by $\mathfrak{Exact}^{\sst}_{p(t)}(X)$ the stack of short exact sequences of semistable coherent sheaves with reduced Hilbert polynomial $p(t)$.

\begin{corollary}
Let $p(t)\in\BoQ[t]/\sim$ be a reduced polynomial. The morphism of stacks
\[
 \mathfrak{Exact}^{\sst}_{p(t)}(X)\rightarrow\mathfrak{Coh}^{\sst}_{p(t)}(X)
\]
mapping a short exact sequence of Gieseker semistable coherent sheaves on $X$ to the middle term is representable and projective, i.e. $\mathfrak{Coh}^{\sst}_{p(t)}(X)$ satisfies Assumption \ref{p_assumption} from \S\ref{subsubsection:gsetup}.
\end{corollary}
\begin{proof}
Fix a polynomial $P(t)$ in the equivalence class $p(t)$. Since the leading term of any $Q(t)$ for which $\mathfrak{Coh}_{Q(t)}(X)$ is nonempty is a positive integer divided by $\deg(p(t))!$, there are finitely may $R(t)\sim P(t)$ for which $p_{P(t),R(t)}^{-1}(\mathfrak{Coh}_{n,P(t)}(X))$ is nonempty.
If $j\colon \mathcal{G}\hookrightarrow\mathcal{F}$ is a monomorphism of sheaves having the same reduced Hilbert polynomial where $\mathcal{F}$ is semistable, then $\mathcal{G}$ and $\coker(j)$ are semistable. By boundedness of the moduli stack \cite{huybrechts2010geometry}, for sufficiently large $n$ 
 \[
 \mathfrak{Coh}^{\sst}_{P(t)}(X)\subset \mathfrak{Coh}_{n,P(t)}(X).
 \]
Projectivity of $\mathfrak{Exact}^{\sst}_{p(t)}(S)\rightarrow\mathfrak{Coh}_{p(t)}^{\sst}(X)$ then follows from Proposition \ref{gl_pr} by base change.
\end{proof}
\subsection{Resolving RHom}
Fix a reduced Hilbert polynomial $p(t)$. So far we have placed few constraints on the projective variety $X$. In order for the stack $\FM=\mathfrak{Coh}^{\sst}_{p(t)}(X)$ to satisfy all of the assumptions detailed in \S\ref{subsubsection:gsetup}, we next assume that $S=X$ has dimension at most 2 and is smooth.

Fix a polynomial $P(t)$ in the equivalence class $p(t)$. We denote by $\FM_{P(t)}\subset \FM$ the substack of coherent sheaves with Hilbert polynomial $P(t)$. Let $\CU\in\Coh(\FM_{P(t)}\times S)$ be the universal coherent sheaf. By boundedness of the moduli space $\FM_{P(t)}$ we can pick $n_1\gg 0$ such that $\CU_{y}$ is strongly generated by $\CO_S(-n_1)$ for every sheaf $\CU_y$ corresponding to a point $y$ of $\FM_{P(t)}$. Let 
\[
\CL^1=\CH om_{\FM_{P(t)}\times S}(\pi_2^*\CO_S(-{n_1}),\CU)\otimes \pi_2^*\CO_S(-{n_1}),
\]
and define $\CK$ to be the kernel of the adjunction morphism $\CL^1\rightarrow \CU$. By boundedness again, there is a $n_2\gg n_1$ such that all sheaves $\CK_y$ parametrised by points $y$ of $\FM_{P(t)}$ are strongly generated by $\CO_S(-n_2)$ and we define $\CL^2=\CH om_{\FM_{P(t)}\times S}(\pi_2^*\CO_S(-{n_2}),\CK)\otimes \pi_2^*\CO_S(-{n_2})$. Let $\CL^3$ be the kernel of the adjunction morphism $\CL^2\rightarrow \CL^1$. The following is well-known
\begin{lemma}
The coherent sheaf $\CL^3$ is a flat family of vector bundles on $S$.
\end{lemma}
\begin{proof}
Let $y$ be a $K$-point of $\FM_{P(t)}$ for some field $K\supset \BoC$ and let $x$ be a point of $S_K$. Then $\Ext_{S_K}^i(\CL_y^3,\mathcal{O}_x)\cong \Ext_{S_K}^{i+2}(\CU_y,\mathcal{O}_x)=0$ if $i\geq 1$, and the result follows.
\end{proof}
We have thus defined a global resolution $\CL^{\bullet}\rightarrow \CU$ of the universal sheaf by vector bundles. Since the resolution depends on numbers $n_2\gg n_1\gg 0$ we will denote this resolution by $\CL^{\bullet}_{n_1,n_2}$ when the dependence on these numbers is significant. Fix $R(t)\sim P(t)$. We abuse notation and also denote by $\CL^{\bullet}_{n_1,n_2}$ the resolution of the universal sheaf on $\FM_{R(t)}\times S$. For the next lemma we use boundedness again:

\begin{lemma}
\label{lemma:vector-bundle}
Pick numbers $n'_2\gg n'_1\gg n_2\gg n_1\gg 0$. Let $y,y'$ be $K$-points of $\FM_{R(t)}$ and $\FM_{P(t)}$ respectively for a field $K \supset \BoC$. Then $\Ext_{S_K}^n((\CL^i_{n'_1,n'_2})_{y'},(\CL^j_{n_1,n_2})_{y})=0$ for $n\geq 1$ and all $i,j\in\{1,2,3\}$.
\end{lemma}
The lemma implies that the dimension of $((\pi_{1,2})_*\CH om(\pi_{1,3}^*\CL^i_{n'_1,n'_2},\pi_{2,3}^*\CL^j_{n_1,n_2}))_{y''}$ is constant across all $K$-points $y''$ of $\FM_{P(T)}\times \FM_{R(t)}$, from which we deduce the following.
\begin{corollary}
Pick numbers $n'_2\gg n'_1\gg n_2\gg n_1\gg 0$, and $i,j\in\{1,2,3\}$. For numbers $a\neq b$ let $\pi_{a,b}$ denote the projection of $\FM_{P(t)}\times\FM_{R(t)}\times S$ onto its $a$th and $b$th factors. The coherent sheaf
\[
\CV^{i,j}_{(n'_1,n'_2),(n_1,n_2)}\coloneqq (\pi_{1,2})_*\CH om(\pi_{1,3}^*\CL^i_{n'_1,n'_2},\pi_{2,3}^*\CL^j_{n_1,n_2})
\]
is a vector bundle on $\FM_{P(t)}\times \FM_{R(t)}$.
\end{corollary}

Note that $\CV^{\bullet,\bullet}_{(n'_1,n'_2),(n_1,n_2)}$ is a double complex, with differentials induced by the differentials on $\CL^{\bullet}_{n_1,n_2}$ and $\CL^{\bullet}_{n'_1,n'_2}$. We denote by $\CR^{\bullet}_{(n_1,n_2),(n'_1,n'_2)}$ the resulting total complex.
\begin{corollary}
	\label{corollary:RHom_complexes_surfaces}
The complex $\CR^{\bullet}_{(n_1,n_2),(n'_1,n'_2)}$ provides a 5-term complex of vector bundles resolving the RHom complex on $\FM_{P(t)}\times \FM_{R(t)}$.
\end{corollary}
\subsection{Stacks of extensions}
\label{new_exts_sec}
Set $\FX=t_0(\Tot_{\FM_{P(t)}\times\FM_{R(t)}}(\CR^{\bullet}_{(n_1,n_2),(n'_1,n'_2)}[1]))\times S$. We denote by $\pi'_{a,b}$ the projection of $\FX$ onto the $a$th and $b$th factors of $\FM_{P(t)}\times\FM_{R(t)}\times S$. By construction, $\FX$ carries a universal cone double complex
\[
\begin{tikzcd}
&\pi'^*_{2,3}\CL^3_{n_1,n_2}
\\
\pi'^*_{1,3}\CL^3_{n'_1,n'_2}&\pi'^*_{2,3}\CL^2_{n_1,n_2}
\\
\pi'^*_{1,3}\CL^2_{n'_1,n'_2}&\pi'^*_{2,3}\CL^1_{n_1,n_2}
\\
\pi'^*_{1,3}\CL^1_{n'_1,n'_2}
\arrow[from=1-2,to=2-2]
\arrow[from=2-1,to=2-2]
\arrow[from=2-1,to=3-1]
\arrow[from=2-2,to=3-2]
\arrow[from=3-1,to=3-2]
\arrow[from=3-1,to=4-1]
\end{tikzcd}
\]
The first column is the pullback of the universal resolution $\CL^{\bullet}_{n'_1,n'_2}$ along the projection to $\FM_{P(t)}\times S$, the second column is the pullback of the universal resolution $\CL^{\bullet}_{n_1,n_2}$ along the projection to $\FM_{R(t)}\times S$, and the horizontal morphisms are determined by the $\CR^{\bullet}_{(n_1,n_2),(n'_1,n'_2)}$-fibre. We denote the total complex by $\CE^{\bullet}$. Then since $\CE^i=0$
for $i>0$ there is a morphism of complexes
\[
\CE^{\bullet}\rightarrow \CH^0(\CE^{\bullet})
\]
which is moreover a quasi-isomorphism, defining a morphism $q'\colon \FX\rightarrow \FM_{a+b}$ and a tautological morphism
\begin{equation}
\label{taut_out_morph}
T\colon \CE^{\bullet}\rightarrow q'^*\CU.
\end{equation}
We obtain the standard
\begin{proposition}
\label{cone_prop}
There is an equivalence of stacks
\[
t_0(\Tot_{\FM_{P(t)}\times\FM_{R(t)}}((\tau^{\geq 0}\CR^{\bullet}_{(n_1,n_2),(n'_1,n'_2)})[1]))\simeq \mathfrak{Exact}^{\sst}_{R(t),P(t)}(S).
\]
\end{proposition}
\begin{remark}
\label{3t_complex}
The complex $\CR^{\bullet}_{(n_1,n_2),(n'_1,n'_2)}$ is situated in cohomological degrees $[-2,2]$. There is a more economical presentation of the RHom complex, that will suffice for defining the relative CoHA product. We define 
\[
\CR^{i}_{(n_1,n_2),-}\coloneqq (\pi_{1,2})_*\CH om(\pi_{1,3}^*\CL^i_{n_1,n_2},\pi_{2,3}^*\CU),
\]
For $n_2 \gg n_1\gg0$ this defines a complex of vector bundles, in cohomological degrees $[0,2]$, and for $n_2 \gg n_1\gg n'_2\gg n'_1\gg 0$ the quasi-isomorphism $\CL^j_{n'_1,n'_2}\rightarrow \CU$ induces a quasi-isomorphism $\CR^{\bullet}_{(n_1,n_2),(n'_1,n'_2)}\rightarrow \CR^{\bullet}_{(n_1,n_2),-}$. Therefore, if $d^{-1}\colon \CR^{-1}_{(n_1,n_2),(n'_1,n'_2)}\rightarrow\CR^{0}_{(n_1,n_2),(n'_1,n'_2)}$ is the differential, there is a morphism of complexes $\tau^{\geq 0}\CR^{\bullet}_{(n_1,n_2),(n'_1,n'_2)}=[\coker(d^{-1})\rightarrow \CR^{1}_{(n_1,n_2),(n'_1,n'_2)}\rightarrow \CR^{2}_{(n_1,n_2),(n'_1,n'_2)}]\rightarrow \CR^{\bullet}_{(n_1,n_2),-}$ that is a quasi isomorphism.
\end{remark}
In particular we obtain the following (well-known) proposition:
\begin{proposition}
\label{proposition:3t-complex}
Fix two non-zero polynomials $P(t),R(t)$ with the same reduction (i.e. differing by scalar multiplication). Fix numbers $n_2\gg n_1\gg 0$. Then the RHom complex on $ \mathfrak{Coh}^{\sst}_{P(t)}(S)\times \mathfrak{Coh}^{\sst}_{R(t)}(S)$ is resolved by the 3-term complex of vector bundles $\CR^{\bullet}_{(n_1,n_2),-}$. In particular, Assumption \ref{q_assumption1} is satisfied by the morphism $\bs{q}\colon \Tot_{\FM_{P(t)}\times\FM_{R(t)}}(\CR^{\bullet}_{(n_1,n_2),-})\rightarrow \mathfrak{Coh}^{\sst}_{P(t)}(S)\times \mathfrak{Coh}^{\sst}_{R(t)}(S)$.
\end{proposition}

We summarize our results as follows.
\begin{proposition}
Let $n'_2\gg n'_1\gg n_2\gg n_1\gg 0$, $\CR^{\bullet}_{(n_1',n_2'),(n_1,n_2)}$, $\CR^{\bullet}_{(n_1,n_2),-}$ and $\CR^{\bullet}_{(n'_1,n'_2),-}$ be as in \ref{corollary:RHom_complexes_surfaces} and Remark~\ref{3t_complex}. Then, there are equivalences
\[
 \Tot_{\FM_{P(t),R(t)}} (\tau^{\geq -1}\CR^{\bullet}_{(n_1',n_2'),(n_1,n_2)}[1]) \rightarrow \Tot_{\FM_{P(t),R(t)}}(\CR^{\bullet}_{(n'_1,n'_2),-}[1])
\]
and
\[
 \Tot_{\FM_{P(t),R(t)}} (\CR^{\bullet}_{(n'_1,n'_2),-}[1]) \simeq \Tot_{\FM_{P(t),R(t)}}(\CR^{\bullet}_{(n_1,n_2),-}[1])
\]
such that the virtual pullbacks defined in terms of the $3$-term complexes $\tau^{\geq -1}\CR^{\bullet}_{(n_1,n_2),(n'_1,n'_2)}[1]$, $\CR^{\bullet}_{(n_1,n_2),-}[1]$, and $\CR^{\bullet}_{(n'_1,n'_2),-}[1]$ all coincide.
\end{proposition}
\begin{proof}
The equivalences of stacks come from Propositions~\ref{cone_prop} and \ref{proposition:3t-complex}. The virtual pullbacks for the $3$-term complexes $\tau^{\geq -1}\CR^{\bullet}_{(n_1,n_2),(n'_1,n'_2)}[1]$ and $\CR^{\bullet}_{(n'_1,n'_2),-}[1]$ coincide since they are related by a morphism of complex that is a quasi-isomorphism (Remark~\ref{3t_complex}). The virtual pullbacks defined by the $3$-term complexes $\CR^{\bullet}_{(n'_1,n'_2),-}[1]$ and $\CR^{\bullet}_{(n_1,n_2),-}[1]$ coincide by Proposition~\ref{glob_res_prop}.
\end{proof}

\subsection{Three-step filtrations}
Fix non-zero polynomials $P_m(t)$, with $m=1,2,3$, such that all three polynomials have the same reduction. We denote $\FM_m=\mathfrak{Coh}_{P_m(t)}^{\sst}(S)$ and $\FM_{m+m'}=\mathfrak{Coh}_{P_m(t)+P_{m'}(t)}^{\sst}(S)$. We consider some $3$-term complexes $\CR_{1,2}$ and $\CR_{1+2,3}$ (constructed as in \S\ref{new_exts_sec}) resolving the RHom complexes over $\FM_1\times \FM_2$ and $\FM_{1+2}\times\FM_3$ respectively. There is a natural morphism $p_{1+2}\colon \Tot_{\FM_1\times\FM_2}(\CR_{(n_1^{(2)},n_2^{(2)}),-}[1])\rightarrow\FM_{1+2}$.

We define
\[
 \bsF^l\coloneqq \Tot_{\Tot_{\FM_1\times\FM_2}(\CR_{1,2}[1])\times\FM_3}((p_{1+2}\times\id_{\FM_3})^*\CR_{1+2,3}[1])\,.
\]

Considering instead $3$-term complexes of vector bundles $\CR_{2,3}$ and $\CR_{1,2+3}$ resolving the RHom complexes over $\FM_2\times \FM_3$ and $\FM_1\times\FM_{2+3}$ respectively, there is a natural morphism $p_{2+3}\colon\Tot_{\FM_2\times\FM_3}(\CR_{2,3}[1])$ and we define
\[
 \bsF^r\coloneqq \Tot_{\FM_1\times\Tot_{\FM_2\times\FM_3}(\CR_{2,3}[1])}((\id_{\FM_1}\times p_{2+3})^*\CR_{1,2+3}[1])\,.
\]

The following is easily proven by comparing the functor of points of the stacks under consideration using Proposition~\ref{cone_prop}, Remark~\ref{3t_complex}.
\begin{proposition}
There are equivalences of stacks $\bsF^l\simeq\bsF_{P_1(t),P_2(t),P_3(t)}\simeq \bsF^r$, where $\bsF_{P_1(t),P_2(t),P_3(t)}$ is the stack of 3-term filtrations $0=\CF_0\subset \CF_1\subset \CF_2\subset \CF_3$ of semistable coherent sheaves on $S$, with the Hilbert polynomials of $\CF_i/\CF_{i-1}$ equal to $P_i(t)$ for $i=1,2,3$ respectively.
\end{proposition}

\begin{proposition}
The stack $\mathfrak{Coh}^{\sst}_{p(t)}(S)$ satisfies Assumption \ref{assumption:associativity}, that is, the CoHA product on $\underline{\SA}_{\varpi}$ is associative.
\end{proposition}

\emph{Strategy of the proof: } We fix non-zero Hilbert polynomials $P_{m}(t)$ with $m=1,2,3$ such that they have the same reduction (they coincide modulo multiplication by non-zero scalars).  We denote by $\mathfrak{Exact}_{i,j}^{\sst}(S)$ the stack of short exact sequences over $\FM_i\times\FM_j$. There are natural morphisms $\pr_i$, $\pr_j$, $p_{i+j}$ from $\mathfrak{Exact}_{i,j}^{\sst}(S)$ to $\FM_i, \FM_j, \FM_{i+j}$ respectively keeping the last, first or middle term of a short exact sequence. There are various projections $\pr_{i,j}\colon \FM_1\times\FM_2\times\FM_3\rightarrow\FM_i\times\FM_j$ forgetting one summand.

We let $\CC_{i,j}$ be the RHom complex over $\FM_i\times\FM_j$, $\CC_{1+2,3}$ the RHom complex over $\FM_{1+2}\times\FM_3$, etc. We give explicit resolutions $\CA^{\bullet}$ of the RHom complex $\CC_{1,2}$ over $\FM_1\times\FM_2$, $\CB^{\bullet}$ of $(\pr_{1+2}\times\id_{\FM_3})^*\CC_{1+2,3}$ over $\mathfrak{Exact}^{\sst}_{1,2}\times\FM_3$, $\CD^{\bullet}$ of the RHom complex over $\FM_2\times\FM_3$ and $\CE^{\bullet}$ of $(\id_{\FM_1}\times\pr_{2+3})^*\CC_{1,2+3}$ over $\FM_1\times\mathfrak{Exact}_{2,3}^{\sst}(S)$ such that the compositions of the projections for total spaces of the naively truncated (i.e. the complexes obtained by removing the negative parts) complexes
\[
 \alpha\colon\Tot_{\Tot_{\FM_1\times\FM_2\times\FM_3}(\pr_{1,2}^*\CA^{\geq 0})}(\pr_{\geq 0}^*\CB^{\geq 0})\rightarrow\Tot_{\FM_1\times\FM_2\times\FM_3}(\pr_{1,2}^*\CA^{\geq 0})\rightarrow\FM_1\times\FM_2\times\FM_3
\]
and
\[
 \beta \colon \Tot_{\Tot_{\FM_1\times\FM_2\times\FM_3}(\pr_{2,3}^*\CD^{\geq 0})}(\pr_{\geq 0}^*\CE^{\geq 0})\rightarrow\Tot_{\FM_1\times\FM_2\times\FM_3}(\pr_{2,3}^*\CD^{\geq 0})\rightarrow\FM_1\times\FM_2\times\FM_3
\]
admit a common presentation as the total space of a complex concentrated in degrees $0$ and $1$ with derivation.  Here, in the first and second compositions, we denote by $\pr_{\geq 0}$ the respective (distinct) morphisms
\[
 \Tot_{\FM_1\times\FM_2\times\FM_3}(\pr_{1,2}^*\CA^{\geq 0})\rightarrow\Tot_{\FM_1\times\FM_2\times\FM_3}(\pr_{1,2}^*\CA)\cong\mathfrak{Exact}_{1,2}\times\FM_3
\]
and
\[
 \Tot_{\FM_1\times\FM_2\times\FM_3}(\pr_{2,3}^*\CD^{\geq 0})\rightarrow\Tot_{\FM_1\times\FM_2\times\FM_3}(\pr_{2,3}^*\CD)\cong\FM_1\times\mathfrak{Exact}_{2,3}.
\]

 Then, the sources of the morphisms $\alpha$ and $\beta$ are the same and the pullback morphisms by $\alpha$ and $\beta$ naturally coincide. We then use Proposition~\ref{proposition:naivetruncation-associativitytool}, which allows us to forget about negative terms of complexes when comparing pullback morphisms, to conclude the associativity of the CoHA multiplication.

The upshot is that, given resolutions  $\CP^{\bullet}$ and $\CL^{\bullet}$ of the tautological sheaves $\CU_{i}$ ($i=1,2$) on $\FM_1\times S$ and $\FM_2\times S$, although it is not a priori clear how to produce a resolution of the tautological sheaf on $\FM_{1+2}\times S$, it is possible to build out of $\CP^{\bullet}$ and $\CL^{\bullet}$ a resolution of the pullback of the tautological sheaf $(p_{1+2}\times\id_S)^*\CU_{1+2}$ over $\mathfrak{Exact}_{1,2}^{\sst}(S)\times S$, such that the underlying vector bundles of this resolution are $\pr_{3}^*\CP^j\oplus\pr_1^*\CL^j$, where $\pr_i\colon\mathfrak{Exact}_{1,2}^{\sst}(S)\times S\rightarrow \FM_i\times S$ are the natural projections keeping the extreme terms of exact sequences.
\begin{proof}
Consider the following diagram. To show associativity is suffices to show that the two possible compositions of the virtual pullbacks and proper pushforwards along the boundary of the big square from the bottom left to the top right are equal. By base-change, it suffices to show that the two possible virtual pullbacks to $\mathfrak{Filt}$ by the compositions $(q_{1,2}\times\id_3)\circ r_{1+2,3}$ and $(\id_1\times q_{2,3})\circ r_{1,2+3}$ are the same.
%% https://q.uiver.app/#q=WzAsOSxbMCwyLCJcXG1hdGhmcmFre019XzEgXFx0aW1lcyBcXG1hdGhmcmFre019XzIgXFx0aW1lcyBcXG1hdGhmcmFre019XzMiXSxbMCwxLCJcXG1hdGhmcmFre0V4YWN0fV97MSwyfSBcXHRpbWVzIFxcbWF0aGZyYWt7TX1fMyJdLFswLDAsIlxcbWF0aGZyYWt7TX1fezErMn0gXFx0aW1lcyBcXG1hdGhmcmFre019XzMiXSxbMSwyLCJcXG1hdGhmcmFre019XzEgXFx0aW1lcyBcXG1hdGhmcmFre0V4YWN0fV97MiwzfSJdLFsyLDIsIlxcbWF0aGZyYWt7TX1fMSBcXHRpbWVzIFxcbWF0aGZyYWt7TX1fezIrM30iXSxbMiwxLCJcXG1hdGhmcmFre0V4YWN0fV97MSwyKzN9Il0sWzIsMCwiXFxtYXRoZnJha3tNfV97MSArIDIgKyAzfSJdLFsxLDAsIlxcbWF0aGZyYWt7RXhhY3R9X3sxKzIsM30iXSxbMSwxLCJcXG1hdGhmcmFre0ZpbHR9Il0sWzMsMCwiXFxtYXRocm17aWR9X3sxfSBcXHRpbWVzIHFfezIsM30gIl0sWzEsMCwicV97MSwyfSBcXHRpbWVzIFxcbWF0aHJte2lkfV8zIiwyXSxbMSwyXSxbNywyXSxbNyw2XSxbNSw2XSxbNSw0XSxbMyw0XSxbOCwxLCJyX3sxKzIsM30iXSxbOCw3XSxbOCwzLCJyX3sxLDIrM30iXSxbOCw1XV0=
\[\begin{tikzcd}
	{\mathfrak{M}_{1+2} \times \mathfrak{M}_3} & {\mathfrak{Exact}_{1+2,3}} & {\mathfrak{M}_{1 + 2 + 3}} \\
	{\mathfrak{Exact}_{1,2} \times \mathfrak{M}_3} & {\mathfrak{Filt}} & {\mathfrak{Exact}_{1,2+3}} \\
	{\mathfrak{M}_1 \times \mathfrak{M}_2 \times \mathfrak{M}_3} & {\mathfrak{M}_1 \times \mathfrak{Exact}_{2,3}} & {\mathfrak{M}_1 \times \mathfrak{M}_{2+3}.}
	\arrow[from=1-2, to=1-1]
	\arrow[from=1-2, to=1-3]
	\arrow[from=2-1, to=1-1]
	\arrow["{q_{1,2} \times \mathrm{id}_3}"', from=2-1, to=3-1]
	\arrow[from=2-2, to=1-2]
	\arrow["{r_{1+2,3}}", from=2-2, to=2-1]
	\arrow[from=2-2, to=2-3]
	\arrow["{r_{1,2+3}}", from=2-2, to=3-2]
	\arrow[from=2-3, to=1-3]
	\arrow[from=2-3, to=3-3]
	\arrow["{\mathrm{id}_{1} \times q_{2,3} }", from=3-2, to=3-1]
	\arrow[from=3-2, to=3-3]
\end{tikzcd}\]
We construct a presentation of the morphism  $\bsF \to \FM_1 \times \FM_2 \times \FM_3$ by combining presentations of $q_{1,2} \times \id_3$ and $r_{1+2,3}$. Similarly, we construct a presentation $\bsF \to \FM_1 \times \FM_2 \times \FM_3$ that can be obtained by combining the presentations of $\id_1 \times q_{2,3}$ and $r_{1,2+3}$.
The equality of the two virtual pullbacks then follows by observing that the two presentations (when pulled back to a certain atlas of $\bsF$) are the same.

We now proceed with the constructions of the explicit resolutions $\CA^{\bullet}$ and $\CB^{\bullet}$.

\textbf{The first presentation of $\bsF \to \FM_1 \times \FM_2 \times \FM_3 $.}
We let
\begin{equation}
\label{CPdef}
\CP^{\bullet}\quad\colon\quad \CP^2\xrightarrow{u_2}\CP^1\xrightarrow{u_1}\CP^0
\end{equation}
be a resolution by vector bundles of the universal sheaf $\CU_1$ over $\FM_1\times S$ and
\begin{equation}
\label{CLdef}
\CL^{\bullet}\quad\colon\quad \CL^2\xrightarrow{v_2}\CL^1\xrightarrow{v_1}\CL^0
\end{equation}
a resolution by vector bundles of the universal sheaf $\CU_2$ over $\FM_2\times S$. We assume that the vector bundles $\CP^i$ and $\CL^i$ are chosen such that $\CP^1\ll\CP^0\ll\CL^1\ll\CL^0\ll 0$. One can take $\CP^1=\CL_{n'_1,n'_2}^2, \CP^0=\CL_{n_1',n_2'}^1, \CL^1=\CL_{n_1,n_2}^2, \CL^0=\CL_{n_1,n_2}^1$ as in Lemma~\ref{lemma:vector-bundle} and $\CP^2=\ker(\CP^1\rightarrow\CP^0)=\CL_{n'_1,n'_2}^3$, $\CL^2=\ker(\CL^1\rightarrow\CL^0)=\CL_{n_1,n_2}^3$.

\textit{Presentation of $q_{1,2}$.}

We let $\mathfrak{Exact}_{1,2}^{\sst}(S)$ be the stack of short exact sequences over $\FM_1\times\FM_2$. As we have seen in Proposition~\ref{cone_prop}, it can be presented as the total complex $\CC_{1,2}^{\bullet}$ of the double complex $(\pr_{1,2})_*\CH om(\pr_{1,S}^*\CP^{\bullet},\pr_{2,S}^*\CL^{\bullet})$ over $\FM_1\times\FM_2$ (Here $\pr_{j,S}$ are the projections $\FM_1\times\FM_2\times S\rightarrow\FM_j\times S$). It admits the middle term projection $p_{1+2}\colon \mathfrak{Exact}_{1,2}^{\sst}(S)\rightarrow\FM_{1+2}$.

The degree $0$ part of the complex $(\pr_{1,2})_*\CH om(\pr_{1,S}^*\CP^{\bullet},\pr_{2,S}^*\CL^{\bullet})[1]$  is
\[
(\pr_{1,2})_*\CH om(\pr_{1,S}^*\CP^1,\pr_{2,S}^*\CL^0)\oplus (\pr_{1,2})_*\CH om(\pr_{1,S}^*\CP^2,\pr_{2,S}^*\CL^1)\,.
\]
By the choice of $\CP^\bullet$ and $\CL^{\bullet}$, it is a vector bundle over $\FM_1\times\FM_2$.

The degree $1$ part of the complex $(\pr_{1,2})_*\CH om(\pr_{1,S}^*\CP^{\bullet},\pr_{2,S}^*\CL^{\bullet})[1]$ is
\[
 (\pr_{1,2})_*\CH om(\pr_{1,S}^*\CP^2,\pr_{2,S}^*\CL^0)\,.
\]
It is a vector bundle over $\FM_1\times\FM_2$.

There is a homomorphism over $\FM_1\times\FM_2\times S$
\[
 \begin{matrix}
  \CH om(\pr_{1,S}^*\CP^1,\pr_{2,S}^*\CL^0)\oplus \CH om(\pr_{1,S}^*\CP^2,\pr_{2,S}^*\CL^1)&\rightarrow& \CH om(\pr_{1,S}^*\CP^2,\pr_{2,S}^*\CL^0)\\
  (f,g)&\mapsto&(\pr_{2,S}^*v_1)\circ g+f\circ (\pr_{1,S}^*u_2)\, ,
 \end{matrix}
\]
where the morphisms $v_1,u_2$ are from \eqref{CPdef} and \eqref{CLdef}.  After applying $(\pr_{1,2})_*$, this gives a morphism of vector bundles over $\FM_1\times\FM_2$
\[
 D^0\colon(\pr_{1,2})_*\CH om(\pr_{1,S}^*\CP^1,\pr_{2,S}^*\CL^0)\oplus (\pr_{1,2})_*\CH om(\pr_{1,S}^*\CP^2,\pr_{2,S}^*\CL^1)\rightarrow (\pr_{1,2})_*\CH om(\pr_{1,S}^*\CP^2,\pr_{2,S}^*\CL^0).
\]
By the definition of total spaces of complexes, $\Tot_{\FM_1\times\FM_2}(\CC^{\bullet}_{1,2})$ is the quotient by the smooth unipotent group scheme $\Tot_{\FM_1\times\FM_2}(\CC^{\leq -1}_{1,2}[-1])$ over $\FM_1\times\FM_2$ of $\ker(D^0)$.

The following diagram summarises the maps that appear in the construction of $\CQ$ below.

%% https://q.uiver.app/#q=WzAsMTAsWzAsMiwiXFxtYXRoZnJha3tNfV8xIFxcdGltZXMgXFxtYXRoZnJha3tNfV8yIFxcdGltZXMgXFxtYXRoZnJha3tNfV8zIFxcdGltZXMgUyJdLFsxLDMsIlxcbWF0aGZyYWt7TX1fMSBcXHRpbWVzIFxcbWF0aGZyYWt7TX1fMiBcXHRpbWVzIFMiXSxbMSwyLCJcXG1hdGhmcmFre0V4YWN0fV97MSwyfSBcXHRpbWVzIFMiXSxbMywyLCJcXG1hdGhmcmFre019X3sxKzJ9IFxcdGltZXMgUyJdLFswLDQsIlxcbWF0aGZyYWt7TX1fMSBcXHRpbWVzIFxcbWF0aGZyYWt7TX1fMiAiXSxbMSw0LCJcXG1hdGhmcmFre019XzEgXFx0aW1lcyBTIl0sWzIsNCwiXFxtYXRoZnJha3tNfV8yIFxcdGltZXMgUyJdLFsxLDEsIlxcbWF0aGZyYWt7RXhhY3R9X3sxLDJ9IFxcdGltZXMgXFxtYXRoZnJha3tNfV8zIFxcdGltZXMgUyJdLFsxLDAsIlxcbWF0aGZyYWt7RXhhY3R9X3sxLDJ9IFxcdGltZXMgXFxtYXRoZnJha3tNfV8zIl0sWzMsMCwiXFxtYXRoZnJha3tNfV97MSsyfSBcXHRpbWVzIFxcbWF0aGZyYWt7TX1fMyJdLFsxLDUsIlxcbWF0aHJte3ByfV97MSxTfSJdLFsxLDQsIlxcbWF0aHJte3ByfV97MSwyfSIsMl0sWzEsNiwiXFxtYXRocm17cHJ9X3syLFN9Il0sWzIsMywicF97MSsyfSBcXHRpbWVzIFxcbWF0aHJte2lkfV9TIl0sWzIsMSwicV97MSwyfSBcXHRpbWVzIFxcbWF0aHJte2lkfV9TIiwyXSxbMCwxXSxbNywyXSxbNywwXSxbNyw4LCJcXG1hdGhybXtwcn1fezEsMn0iLDJdLFs4LDksInBfezErMX0gXFx0aW1lcyBcXG1hdGhybXtpZH1fe1xcbWF0aGZyYWt7TX1fM30iXV0=
\[\begin{tikzcd}
	& {\mathfrak{Exact}_{1,2} \times \mathfrak{M}_3} && {\mathfrak{M}_{1+2} \times \mathfrak{M}_3} \\
	& {\mathfrak{Exact}_{1,2} \times \mathfrak{M}_3 \times S} \\
	{\mathfrak{M}_1 \times \mathfrak{M}_2 \times \mathfrak{M}_3 \times S} & {\mathfrak{Exact}_{1,2} \times S} && {\mathfrak{M}_{1+2} \times S} \\
	& {\mathfrak{M}_1 \times \mathfrak{M}_2 \times S} \\
	{\mathfrak{M}_1 \times \mathfrak{M}_2 } & {\mathfrak{M}_1 \times S} & {\mathfrak{M}_2 \times S}
	\arrow["{p_{1+2} \times \mathrm{id}_{\mathfrak{M}_3}}", from=1-2, to=1-4]
	\arrow["{\mathrm{pr}_{1,2}}"', from=2-2, to=1-2]
	\arrow[from=2-2, to=3-1]
	\arrow[from=2-2, to=3-2]
	\arrow[from=3-1, to=4-2]
	\arrow["{p_{1+2} \times \mathrm{id}_S}", from=3-2, to=3-4]
	\arrow["{q_{1,2} \times \mathrm{id}_S}", from=3-2, to=4-2]
	\arrow["{\mathrm{pr}_{1,2}}"', from=4-2, to=5-1]
	\arrow["{\mathrm{pr}_{1,S}}", from=4-2, to=5-2]
	\arrow["{\mathrm{pr}_{2,S}}", from=4-2, to=5-3]
\end{tikzcd}\]

\textit{Presentation of $r_{1+2,3}$.}

We obtain the resolution
 by vector bundles of the pullback of the universal sheaf
$(p_{1+2}\times\id_S)^*\CU_{1+2}$ to $\mathfrak{Exact}_{1,2}^{\sst}(S)\times S$ by the complex $\CQ^{\bullet}$ whose fiber over $(f,g)\in\ker(D^0)$ is given by
\[
 \CQ^{\bullet}|_{(f,g)}\quad\colon\quad\CP^2\oplus\CL^2\xrightarrow{\begin{pmatrix}
                               u_2&0\\
                               g&v_2
                              \end{pmatrix}
}\CP^1\oplus\CL^1\xrightarrow{\begin{pmatrix}
                               u_1&0\\
                               f&v_1
                              \end{pmatrix}}\CP^0\oplus\CL^0.
\]
Therefore, we obtain a resolution of $(p_{1+2}\times\id_{\FM_3})^*\CC_{1+2,3}$ over $\mathfrak{Exact}_{1,2}^{\sst}(S)\times\FM_3$ by the complex
\[
 (\pr_{1,2})_*\CH om(\pr_{1,S}^*\CQ^{\bullet},\pr_{2,S}^*\CU_3)[1]\,.
\]

Here $\pr_{i,j}$ refers to the projection from $\mathfrak{Exact}_{1,2}^{\sst}(S)\times\FM_3\times S$ onto the $i,j$ factors.

The degree $0$ part of this complex is
\[
 (\pr_{1,2})_*\CH om(\pr_{1,S}^*\CQ^1,\pr_{2,S}^*\CU_3)=(\pr_{1,2})_*\CH om(\pr_{1,S}^*\pr_3^*\CP^1\oplus\pr_{1,S}^*\pr_1^*\CL^1,\pr_{2,S}^*\CU_3)\,.
\]
This is a vector bundle over $\mathfrak{Exact}_{1,2}^{\sst}(S)\times\FM_3$.  The degree $1$ part is
\[
 (\pr_{1,2})_*\CH om(\pr_{1,S}^*\CQ^2,\pr_{2,S}^*\CU_3)=(\pr_{1,2})_*\CH om(\pr_{1,S}^*\pr_3^*\CP^2\oplus\pr_{1,S}^*\pr_1^*\CL^2,\pr_{2,S}^*\CU_3)\,.
\]

There is a homomorphism over $\mathfrak{Exact}_{1,2}^{\sst}(S)\times\FM_3\times S$
\[
\begin{matrix}
 \CH om(\pr_{1,S}^*\pr_3^*\CP^1\oplus\pr_{1,S}^*\pr_1^*\CL^1,\pr_{2,S}^*\CU_3)&\rightarrow&\CH om(\pr_{1,S}^*\pr_3^*\CP^2\oplus\pr_{1,S}^*\pr_1^*\CL^2,\pr_{2,S}^*\CU_3)\\
f&\mapsto&f\circ\begin{pmatrix}
                 u_2&0\\
                 g&v_2
                \end{pmatrix}\,.
\end{matrix}
\]
which, after applying $(\pr_{1,2})_*$, gives a morphism of vector bundles
\[
 D'^0\colon (\pr_{1,2})_*\CH om(\pr_{1,S}^*\pr_3^*\CP^1\oplus\pr_{1,S}^*\pr_1^*\CL^1,\pr_{2,S}^*\CU_3)\rightarrow(\pr_{1,2})_*\CH om(\pr_{1,S}^*\pr_3^*\CP^2\oplus\pr_{1,S}^*\pr_1^*\CL^2,\pr_{2,S}^*\CU_3)
\]
such that $\Tot_{\mathfrak{Exact}_{1,2}^{\sst}(S)\times\FM_3}((p_{1+2}\times\id_{\FM_3})^*\CC_{1+2,3})$ is the quotient of $\ker(D'^0)$ by a smooth unipotent group scheme.

We let $\CA^{\bullet}\coloneqq(\pr_{1,2})_*\CH om(\pr_{1,S}^*\CP^{\bullet},\pr_{2,S}^*\CL^{\bullet})[1]$, a complex of vector bundles over $\FM_1\times\FM_2$. We let $\CB^{\bullet}\coloneqq (\pr_{1,2})_*\CH om(\pr_{1,S}^*\CQ^{\bullet},\pr_{2,S}^*\CU_3)[1]$, a complex of vector bundles over $\mathfrak{Exact}_{1,2}^{\sst}(S)\times\FM_3$. We let $\pr_{\geq 0}\colon\Tot_{\FM_1\times\FM_2}(\CA^{\geq 0})\rightarrow\Tot_{\FM_1\times\FM_2}(\CA^{\bullet})\simeq\mathfrak{Exact}_{1,2}^{\sst}(S)$ be the natural projection.

\textit{Combining the presentations to obtain a presentation of $(q_{1,2}\times \id_3) \circ r_{1+2,3}$}.

Then, by the preceding discussion, the composition
\[
 \Tot_{\Tot_{\FM_1\times\FM_2\times\FM_3}(\pr_{1,2}^*\CA^{\geq 0})}((\pr_{\geq 0}\times\id_{\FM_3})^*\CB^{\geq 0})\rightarrow\Tot_{\FM_1\times\FM_2\times\FM_3}(\pr_{1,2}^*\CA^{\geq 0})\rightarrow\FM_1\times\FM_2\times\FM_3
\]
admits the global presentation by the two-term complex with derivation over $\FM_1\times\FM_2\times\FM_3$
\begin{equation}
\label{equation:global2termpresentation}
 \begin{matrix}
  \begin{matrix}
(\pr_{1,2,3})_*\CH om(\pr_{1,S}^*\CP^2,\pr_{2,S}^*\CL^1) \\
\oplus \\
(\pr_{1,2,3})_*\CH om(\pr_{1,S}^*\CP^1,\pr_{3,S}^*\CU_3)  \\
\oplus \\
(\pr_{1,2,3})_*\CH om(\pr_{2,S}^*\CL^1,\pr_{3,S}^*\CU_3) \\
\oplus\\
(\pr_{1,2,3})_*\CH om(\pr_{1,S}^*\CP^1,\pr_{2,S}^*\CL^0)
\end{matrix}&\rightarrow&
\begin{matrix}
(\pr_{1,2,3})_*\CH om(\pr_{1,S}^*\CP^2,\pr_{2,S}^*\CL^0) \\
\oplus\\
(\pr_{1,2,3})_*\CH om(\pr_{1,S}^*\CP^2,\pr_{3,S}^*\CU_3) \\
\oplus\\
(\pr_{1,2,3})_*\CH om(\pr_{2,S}^*\CL^2,\pr_{3,S}^*\CU_3)
 \end{matrix}\\
 (g,h,k,f)&\mapsto&(fu_2+v_1g,hu_2+kg,kv_2)
\end{matrix}
\end{equation}
The derivation comes from the fact that the morphism is not linear but has instead a quadratic part (more precisely the term $kg$).

\textbf{The second presentation of $\bsF \to \FM_1 \times \FM_2 \times \FM_3 $}

Going the other way, that is by first considering the resolution of the RHom complex over $\FM_2\times\FM_3$ obtained using the resolution $\CL^{\bullet}$ of $\CU_2$ and then a resolution of $(\id_{\FM_1}\times p_{2+3})^*\CC_{1,2+3}$ gives the complexes $\CD^{\bullet}$ and $\CE^{\bullet}$ as in the sketch of the proof, and yields exactly the same $2$-term complex with derivation.

More precisely the stack of short exact sequences $\mathfrak{Exact}_{2,3}^{\mathrm{sst}}(S)$ can be presented as the complex $\CC_{2,3}^{\bullet}=(\pr_{1,2})_*\CH om(\pr_{1,S}^*\CL^{\bullet},\pr_{2,S}^*\CU_3)$ over $\FM_2\times\FM_3$. It admits the middle term projection $p_{2+3}\colon\mathfrak{Exact}_{2,3}^{\mathrm{sst}}(S)\rightarrow\FM_{2+3}$. The degree $0$ and one parts of the complex $\CH om(\pr_{1,S}^*\CL^{\bullet},\pr_{2,S}^*\CU_3)$ are given by
\[
\begin{matrix}
 \CH om(\pr_{1,S}^*\CL^1,\pr_{2,S}^*\CU_3)&\rightarrow&\CH om(\pr_{1,S}^*\CL^2,\pr_{2,S}^*\CU_3)\\
 k&\mapsto&kv_2
\end{matrix}
\]
After applying $(\pr_{1,2})_*$, it gives a morphism of vector bundles
\[
 \tilde{D}_0\colon (\pr_{1,2})_* \CH om(\pr_{1,S}^*\CL^1,\pr_{2,S}^*\CU_3)\rightarrow(\pr_{1,2})_*\CH om(\pr_{1,S}^*\CL^2,\pr_{2,S}^*\CU_3)\,.
\]
By definition, $\Tot_{\FM_2\times\FM_3}(\CC_{2,3}^{\bullet})$ is the quotient by the smooth unipotent group scheme $\Tot_{\FM_2\times\FM_3}(\CC_{2,3}^{\leq 1}[-1])$ over $\FM_2\times\FM_3$ of $\ker(\tilde{D}_0)$. We obtain a resolution of the pullback of the universal sheaf $(p_{2+3}\times\id_S)^*\CU_{2+3}$ over $\mathfrak{Exact}_{2+3}^{\mathrm{sst}}(S)\times S$ by the complex whose fiber over $k\in\ker(\tilde{D}_0)$ is
\[
 \CQ'^{\bullet}\quad\colon\quad \CL^2\xrightarrow{v_2}\CL^1\xrightarrow{\begin{pmatrix}v_1\\k\end{pmatrix}}\CL^0\oplus\CU_3
\]
Therefore, we obtain a resolution of $(\id_{\FM_1}\times p_{2+3})^*\CC_{1,2+3}$ over $\FM_1\times\mathfrak{Exact}_{2,3}^{\mathrm{sst}}(S)$ by the complex
\[
 (\pr_{1,2})_*\CH om(\pr_{1,S}^*\CP^{\bullet},\pr_{2,S}^*\CQ'^{\bullet})[1]
\]
Here, $\pr_{i,j}$ refers to the projection from $\mathfrak{Exact}_{1,2}^{\mathrm{sst}}(S)\times\FM_3\times S$ onto the $i,j$ factors.

The degree zero part of this complex is
\[
\begin{split}
 (\pr_{1,2})_*\CH om(\pr_{1,S}^*\CP^{1},\pr_{2,S}^*\CQ'^0)\oplus(\pr_{1,2})_*\CH om(\pr_{1,S}^*\CP^2,\pr_{2,S}^*\CQ'^1)\\=(\pr_{1,2})_*\CH om(\pr_{1,S}^*\CP^{1},\pr_{2,S}^*\CL^0\oplus\pr_{2,S}^*\CU_3)\oplus(\pr_{1,2})_*\CH om(\pr_{1,S}^*\CP^2,\pr_{2,S}^*\CL^1)
\end{split}
\]

The degree one part of this complex is given by
\[
 (\pr_{1,2})_*\CH om(\pr_{1,S}^*\CP^2,\pr_{2,S}^*\CQ'^0)=(\pr_{1,2})_*\CH om(\pr_{1,S}^*\CP^2,\pr_{2,S}^*\CL^0\oplus\pr_{2,S}^*\CU_3)
\]

There is a homomorphism over $\FM_1\times\mathfrak{Exact}_{2,S}^{\mathrm{sst}}(S)\times S$:
\[
\begin{matrix}
 \tilde{D}'^0&\colon &\begin{matrix}\CH om (\pr_{1,S}^*\CP^1,\pr_{2,S}^*\pr_1^*\CL^0\oplus\pr_{2,S}^*\pr_3^*\CU_3)\\ \oplus \\ \CH om (\pr_{1,S}^*\CP^2,\pr_{2,S}^*\CL^1)\end{matrix}&\rightarrow &\CH om(\pr_{1,S}^*\CP^2,\pr_{2,S}^*\pr_1^*\CL^0\oplus\pr_{2,S}^*\pr_3^*\CU_3)\\
&&(g,h)&\mapsto&g\circ u_2+\begin{pmatrix}v_1\\k\end{pmatrix}\circ h\,.
\end{matrix}
\]
such that $\Tot_{\FM_1\times\mathfrak{Exact}_{2,3}^{\mathrm{sst}}(S)}((\id_{\FM_1}\times p_{2+3})^*\CC_{1,2+3})$ is the quotient of $\ker(\tilde{D}'^0)$ by a smooth unipotent group scheme.

We let $\CD^{\bullet}\coloneqq (\pr_{1,2})_*\CH om(\pr_{1,S}^*\CL^{\bullet},\CU_3)[1]$, a complex of vector bundles over $\FM_2\times\FM_3$ and $\CE^{\bullet}\coloneqq (\pr_{1,2})_*\CH om(\pr_{1,S}^*\CP^{\bullet},\pr_{2,S}^*\CQ'^{\bullet})[1]$. We let $\pr_{\geq 0}\colon\Tot_{\FM_2\times\FM_3}(\CD^{\geq 0})\rightarrow\Tot_{\FM_2\times\FM_3}(\CD^{\bullet})\simeq\mathfrak{Exact}_{2,3}^{\mathrm{sst}}(S)$. Then, by the preceding discussion, the composition
\[
 \Tot_{\Tot_{\FM_1\times\FM_2\times\FM_3}(\pr_{2,3}^*\CD^{\geq 0})}((\id_{\FM_1}\times\pr_{\geq 0})^*\CE^{\geq 0})\rightarrow \Tot_{\FM_1\times\FM_2\times\FM_3}(\pr_{2,3}^*\CD^{\geq 0})\rightarrow \FM_1\times\FM_2\times\FM_3
\]
has exactly the presentation \eqref{equation:global2termpresentation}.

Therefore, we may conclude that the two possible pullbacks to the stack of three-step filtrations coincide thanks to Proposition~\ref{proposition:naivetruncation-associativitytool}. This proves that the sheafified cohomological Hall algebra multiplication on $\underline{\SA}_{\varpi}$ is associative.
\end{proof}

\subsection{Good moduli spaces}

The existence of a good moduli space $\JH\colon \Cohst^{\sst}_{P(t)}(X) \to \Cohsp^{\sst}_{P(t)}(X)$ is well-known and can be deduced from Alper's original paper on good moduli spaces \cite[§13]{alper2013good} and the standard GIT construction of moduli spaces of sheaves on projective varieties as in \cite{simpson1994moduliI}. The stack $\Cohst^{\sst}_{p(t)}(X)$ also admits a good moduli space, because it is a disjoint union of stacks of the form $\Cohst^{\sst}_{P(t)}(X)$. Thus we have
\begin{proposition}
The stack $\Cohst^{\sst}_{p(t)}(X)$ satisfies Assumption~\ref{gms_assumption}.
\end{proposition}

\subsubsection{The quasi-projective case}
If $X$ is a quasi-projective variety, we let $X\subset \overline{X}$ be a projective compactification. We fix an ample divisor of $\overline{X}$ and for $P(t)\in\BoQ[t] $, define $\Cohst_{P(t)}^{\sst}(X)\subset \Cohst_{P(t)}^{\sst}(\overline{X})$ to be the open substack of sheaves with support avoiding $\overline{X}\setminus X$. Since Assumptions \ref{p_assumption}-\ref{q_assumption1} and \ref{assumption:associativity} are stable under base change along open inclusions, we can generalise the above discussion with the following
\begin{proposition}
\label{geometry_sumup_prop}
Let $X$ be a smooth quasi-projective variety. Then the moduli stack $\Cohst_{p(t)}^{\sst}(X)$ satisfies Assumptions \ref{p_assumption} and \ref{gms_assumption}. If $X=S$ is a surface, then $\Cohst_{p(t)}^{\sst}(S)$ additionally satisfies Assumptions~\ref{q_assumption1} and \ref{assumption:associativity}.
\end{proposition}
\begin{proof}
It remains to check Assumption~\ref{gms_assumption}. For the good moduli spaces $\JH\colon \Cohst^{\sst}_{P(t)}(\overline{X}) \to \Cohsp^{\sst}_{P(t)}(\overline{X})$ we have $\JH^{-1}(\JH(\Cohst^{\sst}_{P(t)}(X))) =\Cohst^{\sst}_{P(t)}(X)$ (i.e. $\Cohst^{\sst}_{P(t)}(X) \subset \Cohst^{\sst}_{P(t)}(\overline{X})$ is \emph{saturated} in the sense of \cite{alper2013good}). Hence the good moduli space $\Cohst_{P(t)}^{\sst}(X)\rightarrow \Cohsp_{P(t)}^{\sst}(X)$ is obtained from $\Cohst_{P(t)}^{\sst}(\overline{X})\rightarrow \Cohsp_{P(t)}^{\sst}(\overline{X})$ by restriction.
\end{proof}

\subsection{Finiteness of $\oplus$}
Recall that one of the assumptions in the construction of the BPS Lie algebra is the finiteness of the direct sum map $\oplus$ on the good moduli space. For a reduced Hilbert polynomial $p(t)\in\BoQ[t]/\sim$ we let $\Cohsp^{\sst}_{p(t)}(X)$ be the disjoint union of all $\Cohsp^{\sst}_{P(t)}(X)$ with $P(t)$ either zero, or in the equivalence class $p(t)$.
\begin{proposition}
\label{proposition:coh_pfinite}
Let $X$ be a quasiprojective variety, with a fixed ample divisor on a compactifictaion $\overline{X}\supset X$. Then,
\[
\oplus\colon \Cohsp^{\sst}_{p(t)}(X)\times\Cohsp^{\sst}_{p(t)}(X)\rightarrow \Cohsp^{\sst}_{p(t)}(X)
\]
is finite, and thus proper. In particular, the morphism $\oplus$ satisfies Assumption \ref{ds_fin} from \S \ref{subsubsection:gsetup}.
\end{proposition}
\begin{proof}
The fibre of a point $y$ parametrising a coherent sheaf $\CF$ under the morphism $\oplus$ is identified with the set of pairs $(\CF',\CF'')$ for which there is an isomorphism $\CF\cong\CF'\oplus\CF''$. Since the category of semistable coherent sheaves on $X$ with fixed normalised Hilbert polynomial is of finite length, this fibre is finite. So the proposition boils down to proving properness.
\smallbreak
If $X=\overline{X}$ properness is immediate, since each connected component of the the variety $\Cohsp^{\sst}_{p(t)}(X)$ is itself projective. For the general case, we use the valuative criterion for properness. Denote by $\overline{\oplus}$ the extension of $\oplus$ to the direct sum morphism for $\Cohsp^{\sst}_{p(t)}(\overline{X})$; we have just seen that this morphism is finite. Denote by $R$ a discrete valuation ring with residue field $k$, and by $K$ its field of fractions. Then given the diagram
\[
\xymatrix{
&\mathrm{Spec}(K)\ar[d]\ar[r]&\ar[d]\Cohsp^{\sst}_{p(t)}(X)\times\Cohsp^{\sst}_{p(t)}(X)\ar[d]^{\oplus}\ar@{^{(}->}[r]^j&\Cohsp^{\sst}_{p(t)}(\overline{X})\times\Cohsp^{\sst}_{p(t)}(\overline{X})\ar[d]^{\overline{\oplus}}\\
\mathrm{Spec}(k)\ar[r]^{r}&\mathrm{Spec}(R)\ar[r]^l&\Cohsp^{\sst}_{p(t)}(X)\ar@{^{(}->}[r]^{j'}&\Cohsp^{\sst}_{p(t)}(\overline{X})
}
\]
there is a unique morphism $\mathrm{Spec}(R)\rightarrow \Cohsp^{\sst}_{p(t)}(\overline{X})\times\CM^{\sst}_{p(t)}(\overline{X})$ making the whole diagram commute. This lift factors through $j$ if the induced morphism $\alpha\colon \mathrm{Spec}(k)\rightarrow \Cohsp^{\sst}_{p(t)}(\overline{X})\times\Cohsp^{\sst}_{p(t)}(\overline{X})$ does. But this morphism corresponds to a direct sum decomposition of the coherent sheaf $\CF$ corresponding to the morphism $j'lr$. Since this morphism factors through $j'$, the sheaf $\CF$ is supported on $X$, and so in any direct sum decomposition $\CF\cong \CF'\oplus\CF''$ both $\CF'$ and $\CF''$ are supported on $X$. It follows that $\alpha$ factors through $j$.
\end{proof}

For $S$ a smooth quasi-projective surface, the only remaining assumption that will go into the construction of the BPS Lie algebra is Assumption \ref{BPS_cat_assumption}. In order to satisfy this assumption we will either restrict to coherent sheaves on smooth quasi-projective $S$ with $\mathcal{O}_S\cong K_S$, or consider coherent sheaves on smooth quasi-projective $S$ with zero-dimensional support.

\subsection{Determinant line bundles}
\label{subsection:detlbsurfaces}
We continue to use the notation $\Mst = \Cohst^{\sst}_{p(t)}(X)$.
Let $n \in \BZ$ be such that $p(n) \neq 0$. 
Let $\CU \in \Coh(\Mst \times S)$
be the universal sheaf.

The derived pushforward $(\R\pi_1)_*(\CU \otimes \pi_2^*\CO_X(-n))$ is a perfect complex on $\Mst$ and thus has a well-defined determinant $\CL \coloneqq \det((\R\pi_1)_*(\CU \otimes \pi_2^{*}\CO_X(-n)))$. 
Let $x \in \FM$ be a closed point corresponding to a polystable sheaf $\CF$ with reduced Hilbert polynomial $p(t)$ and let $i_x^* \colon \B \BoC^{\times} \to \FM$ be the inclusion of the scalar automorphisms at $x$.
Then $i_x^* \CL $ corresponds to the one-dimensional $\BoC^{\times}$-representation of weight $\chi(i_x^*(\R\pi_1)_*(\CU \otimes \pi_2^{*}\CO_X(n)))=\chi(\CF(n))\neq 0$.
Thus $\Mst$ satisfies Assumption~\ref{det_bun_assumption}.

\section{2-dimensional categories from algebra: representations of algebras of homological dimension $2$}
\label{algebra_constr_sec}
For categories of representations over algebras, checking that the Assumptions \ref{p_assumption}-\ref{ds_fin} of \S \ref{section:modulistackobjects2CY} hold is much more straightforward than the analogous problems for coherent sheaves. Nonetheless, we spell out some of the details below.

Let $B$ be an algebra, and let us assume that we have presented $B$ in the form 
\begin{align}
\label{standard_pres}
B=A/\langle R\rangle.
\end{align}
Here, $A$ is the universal localisation (as in \cite{schofield1985representations}) of the free path algebra of a quiver $kQ$, obtained by formally inverting a finite set of elements $b_1,\ldots,b_l$, where each $b_i$ is a linear combination of cyclic paths with the same start-points. We assume that $R=\{r_1,\ldots,r_e\}$ is a finite set of relations in $A$, and without loss of generality we assume that each $r_i$ is a linear combination of paths with the same starting point, and the same ending point. We define the dimension vector $(M_i\coloneqq e_i\cdot M)_{i\in Q_0}\in \BoN^{Q_0}$ of a $B$-module $M$ in the usual way.
\subsection{Good moduli spaces}
For $\dd\in Q_0$ the stack $\FM_{B,\dd}$ of $\dd$-dimensional $B$-modules is a global quotient of an affine scheme, and thus satisfies the resolution property. The coarse moduli space $\CM_{B,\dd}$ is the affinization of $\FM_{B,\dd}$, and the morphism
\[
\JH\colon \FM_{B,\dd}\rightarrow \CM_{B,\dd}
\]
is a good moduli space for $\FM_{B,\dd}$, i.e. $\FM_B$ satisfies Assumption \ref{gms_assumption}. The following is proved in exactly the same way as Proposition \ref{proposition:directsumproperpreprojective}:
\begin{proposition}
The direct sum map $\oplus\colon\CM_{B}\times \CM_B\rightarrow \CM_B$ is finite. In particular, the category of finite-dimensional $B$-modules satisfies Assumption \ref{ds_fin}.
\end{proposition}
We have shown that the geometric Assumptions \ref{gms_assumption} and \ref{ds_fin} in \S \ref{BPS_assumptions_sec} that go into the definition of the BPS algebra are always met for $B\lmod$. In the rest of this section we check Assumptions \ref{p_assumption}-\ref{q_assumption1} and \ref{assumption:associativity}.
\subsection{Properness}
Base changing from the stack of $Q$-representations also deals with Assumption \ref{p_assumption}:
\begin{proposition}
Let $\mathfrak{Exact}_B$ denote the stack of short exact sequences of finite-dimensional $B$-modules. Then the forgetful morphism $p\colon \mathfrak{Exact}_B\rightarrow \FM_B$ taking a short exact sequence to its central term is a projective representable morphism of stacks. In particular, the category of finite-dimensional $B$-modules satisfies Assumption \ref{p_assumption}.
\end{proposition}
\begin{proof}
We have a Cartesian square of stacks
\[
\begin{tikzcd}
\mathfrak{Exact}_B& \FM_B\\
\mathfrak{Exact}_{kQ}&\FM_{ kQ}
\arrow[hookrightarrow,from=1-1,to=2-1]
\arrow[hookrightarrow,from=1-2,to=2-2]
\arrow["p",from=1-1,to=1-2]
\arrow["p'",from=2-1,to=2-2]
\arrow["\lrcorner"{anchor=center, pos=0.125}', draw=none, from=1-1, to=2-2]
\end{tikzcd}
\]
in which $p'$ is easily shown to be representable and proper. The result follows by base change.
\end{proof}
\subsection{Derived enhancement of $B$}
\label{Dalg_sec}
Our explicit presentation of $B$ determines a derived enhancement of it. If $R=\{r_1,\ldots,r_e\}$ is a finite set of relations on the localised path algebra $A$, and each $r_i$ is a linear combination of paths with the same starting point and ending point, we first define the graded quiver $\tilde{Q}$ by setting $\tilde{Q}_0=Q_0$ and
\begin{align*}
(\tilde{Q}_1)^0&=Q_1,&
(\tilde{Q}_1)^{-1}&=x_1,\ldots ,x_e,&
(\tilde{Q}_1)^i&=\emptyset \quad\textrm{if }i\neq 0,-1.
\end{align*}
We set $t(x_i)=t(r_i)$ and $s(x_i)=s(r_i)$. Then we localise $k\tilde{Q}$ with respect to the same localising set as $A$, and denote by $\tilde{A}$ the resulting graded algebra. Finally we set $\tilde{B}=(\tilde{A},d)$, where $d$ is determined by the Leibniz rule and by setting $d(x_i)=r_i$. Then $\tilde{B}$ is a dga, concentrated in nonpositive degrees, with $\HO^0(\tilde{B})\cong B$, which we call the \textit{derived enhancement} of $B$. We remark that $\tilde{B}$ depends on a presentation of $B$, and not just $B$ itself. We will be particularly interested in examples in which the natural morphism of dgas $\tilde{B}\rightarrow B$ is a quasi-isomorphism.
\begin{example}
For $Q$ a quiver, $\overline{Q}$ its double, $A=k \overline{Q}$ and $R=\{e_i\sum_{a\in Q_1}[a,a^*]e_i\;\lvert\; i\in Q_0\}$, $B=\Pi_Q$ is the preprojective algebra, while $\tilde{B} = \SG_{2}(Q)$ is the derived preprojective algebra.
\end{example}
\subsection{Bimodule resolution}
Consider the complex of $\tilde{A}$-bimodules
\[
0\rightarrow \ker(m')\xrightarrow{i} \bigoplus_{i\in Q_0}\tilde{A}e_i\otimes e_i \tilde{A}\xrightarrow{m'} \tilde{A}\rightarrow 0
\]
where $m'$ is the multiplication.  This is in fact a double complex, since each term carries an internal differential induced by $d$. 
\begin{lemma}
\label{A_bimod_lemma}
There is an isomorphism of $\tilde{A}$-bimodules
\[
\ker(m')\cong\bigoplus_{a\in \tilde{Q}_1}\tilde{A}e_{t(a)}\otimes e_{s(a)} \tilde{A}
\]
and under this isomorphism $i$ is sent to the morphism $\iota$ sending $e_{t(a)}\otimes e_{s(a)}$ to $a\otimes 1-1\otimes a$. 
\end{lemma}
We introduce some preparation for the proof of this lemma.
\begin{definition}
An algebra $C$ is called \textit{formally smooth} if and only if it satisfies the two equivalent properties
\begin{enumerate}
\item
For any algebra $D$ with two-sided nilpotent ideal $I$, any morphism $C\rightarrow D/I$ lifts to a morphism $C\rightarrow D$.
\item
The bimodule $\ker(C\otimes C\xrightarrow{m} C)$ is a projective $C$-bimodule.
\end{enumerate}
\end{definition}
\begin{lemma}
\label{prep_lemma}
Let $C$ be a formally smooth algebra, and let $C'$ be the universal localisation at a set of elements $c_1,\ldots,c_i$. Then $C'$ is formally smooth.
\end{lemma}
\begin{proof}
Let $C'$ be as in the statement of the lemma.  By the universal property of universal localisation, for any algebra $D$ there is an embedding
\[
\Hom_{\Alg}(C',D)\subset \Hom_{\Alg}(C,D)
\]
as the subset of homomorphisms sending all the elements $c_i$ to units in $D$. Let $f\in \Hom_{\Alg}(C',D/I)$. It is sufficient to show that $f$ lifts to a homomorphism in $\Hom_{\Alg}(C',D/I^2)$. By formal smoothness of $C$ there is a $g\in\Hom_{\Alg}(C,D/I^2)$ inducing $f$ under restriction. By assumption, for each $c_j$ there is a $d_j\in D/I^2$ such that $g(c_j)d_j=1+n$, where $n\in I$. But then $g(c_j)d_j(1-n)=1+n'$ where $n'\in I^2$, and so $g\in\Hom_{\Alg}(C',D/I^2)$.
\end{proof}
\begin{proof}[Proof of Lemma \ref{A_bimod_lemma}]
For the purposes of the proof, we may forget the cohomological grading on $\tilde{A}$, and we denote the resulting ungraded algebra $\tilde{A}_{\ungr}$. It is known that
\[
0\rightarrow \bigoplus_{a\in \tilde{Q}_1}k\tilde{Q}_{\ungr}e_{t(a)}\otimes e_{s(a)} k\tilde{Q}_{\ungr}\xrightarrow{\iota}\bigoplus_{i\in Q_0}k \tilde{Q}_{\ungr}e_i\otimes e_ik\tilde{Q}_{\ungr}\xrightarrow{m_{\mathrm{red}}} k\tilde{Q}_{\ungr}\rightarrow 0
\]
is exact, where $m_{\mathrm{red}}$ is the restriction of the multiplication morphism $m\colon k \tilde{Q}_{\ungr}\otimes k \tilde{Q}_{\ungr} \rightarrow k \tilde{Q}_{\ungr}$. Since 
\[
\ker(m)\cong \ker(m_{\mathrm{red}})\oplus\bigoplus_{\substack{i,j\in Q_0\\ i\neq j}} k\tilde{Q}_{\ungr}e_i\otimes e_jk\tilde{Q}_{\ungr}
\]
we deduce that $k\tilde{Q}_{\ungr}$ is formally smooth, and so $\tilde{A}$ is by Lemma \ref{prep_lemma}.

Consider the complex
\[
0\rightarrow\bigoplus_{a\in \tilde{Q}_1}\tilde{A}_{\ungr}e_{t(a)}\otimes e_{s(a)} \tilde{A}_{\ungr}\xrightarrow{\iota'} \bigoplus_{i\in Q_0}\tilde{A}_{\ungr} e_i\otimes e_i\tilde{A}_{\ungr}\xrightarrow{m'}\tilde{A}_{\ungr}\rightarrow 0.
\]
Since the image of $\iota$ lies in the kernel of $m$, and $\tilde{A}_{\ungr}\otimes_{k\tilde{Q}_{\ungr}}-$ and $-\otimes_{k\tilde{Q}_{\ungr}}\tilde{A}_{\ungr}$ are right exact, we obtain the split surjection of projective bimodules
\[
p\colon \bigoplus_{a\in \tilde{Q}_1}\tilde{A}_{\ungr}e_{t(a)}\otimes e_{s(a)} \tilde{A}_{\ungr}\rightarrow \ker(m'_{\mathrm{red}}).
\]
Each of the summands appearing in the domain are indecomposable: since each summand $P$ is a cyclic $\tilde{A}\otimes\tilde{A}^{\mathrm{op}}$-module, an endomorphism of $P$ is provided by an element $p=p'\otimes p''\in P$, which is a projection if and only if $p'p'=p'$ (up to scalars) and $p''p''=p''$ (up to scalars): this forces $p'=e_{t(a)}$ and $p''=e_{s(a)}$ (up to scalars). We deduce that there is a subset $\Omega\subset \tilde{Q}_1$ such that 
\[
\bigoplus_{a\in \Omega}\tilde{A}_{\ungr}e_{t(a)}\otimes e_{s(a)} \tilde{A}_{\ungr}\cong \ker(m'_{\mathrm{red}}).
\]
By assumption we localised at linear combinations of cyclic paths with the same start points, hence for every $i\in Q_0$ there is an $\tilde{A}$-module $N_i$ of dimension $1_i$. Let $i,j\in Q_0$. Then we may calculate $\Ext^n_{k\tilde{Q}_{\ungr}}(N_i,N_j)$ as the $n$th cohomology of the Hom complex
\[
\Hom\left(\left(\bigoplus_{a\in \tilde{Q}_1}k\tilde{Q}_{\ungr}e_{t(a)}\otimes e_{s(a)} k\tilde{Q}_{\ungr}\xrightarrow{\iota'}\bigoplus_{i\in Q_0}k \tilde{Q}_{\ungr}e_i\otimes e_ik\tilde{Q}_{\ungr}\right)\otimes_{k\tilde{Q}_{\ungr}} N_i,N_j\right)
\]
and we find that $\ext^1_{k\tilde{Q}_{\ungr}}(N_i,N_j)$ is the number of arrows $i\rightarrow j$. The same calculation gives that $\ext^1_{\tilde{A}}(N_i,N_j)$ is the number of such arrows belonging to $\Omega$. By \cite[Thm.4.7]{schofield1985representations} we have the equality $\ext^1_{k\tilde{Q}_{\ungr}}(N_i,N_j)=\ext^1_{\tilde{A}}(N_i,N_j)$ and so $\Omega=\tilde{Q}_1$ and the result follows.
\end{proof}
We consider $\ker(m')\xrightarrow{\iota'} \bigoplus_{i\in Q_0}\tilde{B}e_i\otimes e_i\tilde{B}\xrightarrow{m'}\tilde{B}$ as a double complex, with the differential $d$ on $\tilde{B}$ inducing the second differential. This double complex has exact rows by Lemma \ref{A_bimod_lemma}, the total complex of the double complex of $\tilde{B}$-modules $\ker(m')\xrightarrow{\iota'} \tilde{B}\otimes\tilde{B}$ provides a resolution of the diagonal bimodule $\tilde{B}$ by perfect $\tilde{B}$-bimodules. The following corollaries of this construction are immediate.
\begin{corollary}
There is a fully faithful embedding of categories $B\lmod\rightarrow \Perf(\tilde{B})$ from the category of finite-dimensional $B$-modules to the perfect derived category of $\tilde{B}$-modules.
\end{corollary}
\begin{proof}
If $M$ is a finite-dimensional $B$-module, we may consider it as a finite-dimensional $\tilde{B}$-module via the canonical morphism $\tilde{B}\rightarrow \HO^0(\tilde{B})\cong B$. Then 
\[
\left(\ker(m')\xrightarrow{\iota'} \bigoplus_{i\in Q_0}\tilde{B}e_i\otimes e_i\tilde{B}\right)\otimes_{\tilde{B}} M\rightarrow \tilde{B}\otimes_{\tilde{B}} M 
\]
provides a resolution of $M$ by perfect $\tilde{B}$-modules.
\end{proof}
Let $\pi\colon \ker(m)\rightarrow \ker(m')$ be the natural projection. For the next corollary, we first define
\begin{align*}
D\colon &\tilde{A}\rightarrow \ker(m')\\
&\alpha\mapsto \pi(\alpha\otimes 1-1\otimes \alpha).
\end{align*}
Given an element $\alpha\otimes \alpha'\in \tilde{A}e_{t(a)}\otimes e_{s(a)} \tilde{A}$, $\tilde{A}$-modules $M,N$, and a homomorphism $f_a\colon \Hom_k(N_{s(a)},M_{t(a)})$ we define the evaluation $(M,f,N)(\alpha\otimes \alpha')=M(\alpha)\circ f_a\circ N(\alpha')$.
\begin{corollary}
\label{der_exact_cor}
\label{dalg_dhom}
Let $M$ and $N$ be two finite-dimensional $B$-modules. Then the cohomology of the complex
\begin{align}
\label{dhom_cmplx}
0\leftarrow \bigoplus_{r\in R} \Hom_k(M_{s(r)},N_{t(r)})\xleftarrow{\alpha} \bigoplus_{a\in Q_1}\Hom_k(M_{s(a)},N_{t(a)}) \xleftarrow{\beta} \bigoplus_{i\in Q_0}\Hom_k(M_i,M_j)\leftarrow 0
\end{align}
calculates $\Ext^n_{\tilde{B}}(M,N)$, where we define
\begin{align*}
\alpha((f_a)_{a\in Q_1})&=((M,f,N)(Dr))_{r\in R}\\
\beta((f_i)_{i\in Q_0})&=(N(a)f_{s(a)}-f_{t(a)}M(a))_{a\in Q_1}.
\end{align*}
In particular, Assumption \ref{q_assumption1} holds for $\FM_{B}$, considered as a substack of $t_0(\bsM_{\tilde{B}})$. 
\end{corollary}
\begin{remark}
If $B=\HO^0(\tilde{B})$ it follows that the complex \eqref{dhom_cmplx} calculates $\Ext^n_B(M,N)$.
\end{remark}
The following proposition is a result of the well-known explicit presentation of $\Filt$ over $\FM_{\dd^{(1)},\dd^{(2)},\dd^{(3)}}$. Since its proof is strictly easier than the analogous statement for coherent sheaves, provided in \S \ref{geometry_constr_sec}, we omit it.
\begin{proposition}
Assumption \ref{assumption:associativity}, for the choice of presentation of the stack of short exact sequences given by Corollary~\ref{der_exact_cor}, holds for the stack $\FM_B$.
\end{proposition}
In conclusion, all of the geometric assumptions (Assumptions \ref{p_assumption},\ref{q_assumption1} and \ref{assumption:associativity}) that are required in \S \ref{subsubsection:ThecohaproductKV} to define the CoHA structure on (an appropriate shift of) $\JH_*\BD k_{\FM_{B,\dd}}$ hold for arbitrary $B$ presented in the general form \eqref{standard_pres}. Moreover, the geometric conditions (Assumptions \ref{gms_assumption} and \ref{ds_fin}) in \S \ref{BPS_assumptions_sec} are also met. So as long as $\tilde{B}$ is a 2-Calabi--Yau algebra, Assumption \ref{BPS_cat_assumption} holds and we may define the BPS algebra for $B\lmod$ as in \S \ref{section:lessperverse} above.

\subsection{Determinant line bundle}
\label{subsection:detlbalgebras}
For each vertex $i \in Q_0$ there is a line bundle $\CL_i$ on $\FM_Q$
which is the determinant bundle of the tautological vector bundle on $\FM_Q$ whose fibre over a representation is the underlying vector space sitting at the vertex $i$.
For each dimension vector $\vec{d} \in \BN^{Q_0}$ of $Q$ let $\CL_{\vec{d}}$ be product $\otimes_{i \in \supp(\vec{d})} \CL_i$ of the determinant line bundles for the vertices supported by the dimension vector $\vec{d}$.

Let $x \in \Mst_{Q,\vec{d}}$ be a closed point and $i_x\colon \B \BoC^{\times} \to \Mst_{Q}$ be the inclusion of the scaling automorphisms.
Then $i_x^*\CL_{\vec{d}}$ is the one-dimensional $\BoC^{\times}$-representation of weight $\abs{\vec{d}} \neq 0$.

More generally for an algebra $B$ as in \eqref{standard_pres}, the restriction of $\CL_{\vec{d}}$ along the inclusion $\Mst_{B,\vec{d}} \to \Mst_{Q,\vec{d}}$ is a positive determinant line bundle. Thus $\Mst_B$ satisfies Assumption~\ref{det_bun_assumption}.

\bibliographystyle{alpha}
{\small{\bibliography{freealg_references.bib}}

\newcommand{\etalchar}[1]{$^{#1}$}
\begin{thebibliography}{BDInNn{\etalchar{+}}25}

\bibitem[Alp13]{alper2013good}
Jarod Alper.
\newblock Good moduli spaces for {A}rtin stacks.
\newblock In {\em Annales de l'Institut Fourier}, volume~63, pages 2349--2402,
  2013.

\bibitem[BBDG18]{beilinson2018faisceaux}
Alexander Beilinson, Joseph Bernstein, Pierre Deligne, and Ofer Gabber.
\newblock {\em Faisceaux pervers}.
\newblock Soci{\'e}t{\'e} math{\'e}matique de France, 2018.

\bibitem[BBS13]{behrend2013motivic}
Kai Behrend, Jim Bryan, and Bal{\'a}zs Szendr{\H{o}}i.
\newblock {Motivic degree zero Donaldson--Thomas invariants}.
\newblock {\em Inventiones Mathematicae}, 192(1):111--160, 2013.

\bibitem[BD19]{brav2019relative}
Christopher Brav and Tobias Dyckerhoff.
\newblock {Relative {C}alabi--{Y}au structures}.
\newblock {\em Compositio Mathematica}, 155(2):372--412, 2019.

\bibitem[BDInNn{\etalchar{+}}25]{BDINKP}
Chenjing Bu, Ben Davison, Andr\'es Ib\'a\~nez N\'u\~nez, Tasuki Kinjo, and
  Tudor P\u{a}durariu.
\newblock Cohomology of symmetric stacks.
\newblock {\em arXiv preprint arXiv:2502.04253}, 2025.

\bibitem[Beh09]{behrend2009donaldson}
Kai Behrend.
\newblock Donaldson-{T}homas type invariants via microlocal geometry.
\newblock {\em Annals of Mathematics}, 170(3):1307--1338, 2009.

\bibitem[Bor88]{borcherds1988generalized}
Richard Borcherds.
\newblock Generalized {K}ac-{M}oody algebras.
\newblock {\em Journal of Algebra}, 115(2):501--512, 1988.

\bibitem[Boz15]{bozec2015quivers}
Tristan Bozec.
\newblock Quivers with loops and perverse sheaves.
\newblock {\em Mathematische Annalen}, 362(3):773--797, 2015.

\bibitem[Boz16]{bozec2016quivers}
Tristan Bozec.
\newblock Quivers with loops and generalized crystals.
\newblock {\em Compositio Mathematica}, 152(10):1999--2040, 2016.

\bibitem[BRN89]{beauville1989spectral}
Arnaud Beauville, Sundararaman Ramanan, and Mudumbai~S. Narasimhan.
\newblock Spectral curves and the generalised theta divisor.
\newblock {\em Journal f{\"u}r die Reine und Angewandte Mathematik (Crelle's
  Journal)}, 1989(398):169--179, 1989.

\bibitem[BS19]{bozec2019counting}
Tristan Bozec and Olivier Schiffmann.
\newblock Counting absolutely cuspidals for quivers.
\newblock {\em Mathematische Zeitschrift}, 292(1):133--149, 2019.

\bibitem[CB01]{crawley2001geometry}
William Crawley-Boevey.
\newblock Geometry of the moment map for representations of quivers.
\newblock {\em Compositio Mathematica}, 126(3):257--293, 2001.

\bibitem[CB13]{crawley2013monodromy}
William Crawley-Boevey.
\newblock Monodromy for systems of vector bundles and multiplicative
  preprojective algebras.
\newblock {\em Bulletin of the London Mathematical Society}, 45(2):309--317,
  2013.

\bibitem[CBH98]{crawley1998noncommutative}
William Crawley-Boevey and Martin~P. Holland.
\newblock Noncommutative deformations of kleinian singularities.
\newblock {\em Duke Mathematical Journal}, 92(3):605, 1998.

\bibitem[CBK22]{crawley2022deformed}
William Crawley-Boevey and Yuta Kimura.
\newblock On deformed preprojective algebras.
\newblock {\em Journal of Pure and Applied Algebra}, 226(12):107--130, 2022.

\bibitem[CBS06]{crawley2006multiplicative}
William Crawley-Boevey and Peter Shaw.
\newblock Multiplicative preprojective algebras, middle convolution and the
  {D}eligne--{S}impson problem.
\newblock {\em Advances in Mathematics}, 201(1):180--208, 2006.

\bibitem[Cor88]{corlette1988flat}
Kevin Corlette.
\newblock {Flat $G$-bundles with canonical metrics}.
\newblock {\em Journal of Differential Geometry}, 28(3):361--382, 1988.

\bibitem[Dav12]{davison2012superpotential}
Ben Davison.
\newblock Superpotential algebras and manifolds.
\newblock {\em Advances in Mathematics}, 231(2):879--912, 2012.

\bibitem[Dav16]{davison2016cohomological}
Ben Davison.
\newblock Cohomological {H}all algebras and character varieties.
\newblock {\em International Journal of Mathematics}, 27(07):1640003, 2016.

\bibitem[Dav17]{davison2017critical}
Ben Davison.
\newblock The critical {CoHA} of a quiver with potential.
\newblock {\em Quarterly Journal of Mathematics}, 68(2):635--703, 2017.

\bibitem[Dav18]{davison2018purity}
Ben Davison.
\newblock Purity of critical cohomology and {K}ac's conjecture.
\newblock {\em Mathematical Research Letters}, 25(2):469--488, 2018.

\bibitem[Dav23a]{davison2016integrality}
Ben Davison.
\newblock The integrality conjecture and the cohomology of preprojective
  stacks.
\newblock {\em Journal f{\"u}r die reine und angewandte Mathematik (Crelle's
  Journal)}, 2023(804):105--154, 2023.

\bibitem[Dav23b]{davison2023nonabelian}
Ben Davison.
\newblock Nonabelian {H}odge theory for stacks and a stacky {P}={W} conjecture.
\newblock {\em Advances in Mathematics}, 415:108889, 2023.

\bibitem[Dav24]{davison2021purity}
Ben Davison.
\newblock Purity and 2-{C}alabi--{Y}au categories.
\newblock {\em Inventiones Mathematicae}, 238(1):69--173, 2024.

\bibitem[Dav25]{davison2020bps}
Ben Davison.
\newblock B{PS} {L}ie algebras and the less perverse filtration on the
  preprojective {C}o{HA}.
\newblock {\em Advances in Mathematics}, 463:Paper No. 110114, 75, 2025.

\bibitem[dCHM12]{de2012topology}
Mark Andrea~A. de~Cataldo, Tam{\'a}s Hausel, and Luca Migliorini.
\newblock Topology of {H}itchin systems and {H}odge theory of character
  varieties: the case $\mathrm{A}_1$.
\newblock {\em Annals of Mathematics}, pages 1329--1407, 2012.

\bibitem[dCM21]{de2018perverse}
Mark Andrea~A. de~Cataldo and Davesh Maulik.
\newblock The perverse filtration for the {H}itchin fibration is locally
  constant.
\newblock {\em Pure and Applied Mathematics Quarterly}, 16(5):1444--1464, 2021.

\bibitem[Del87]{deligne1987theoreme}
Pierre Deligne.
\newblock Un th{\'e}oreme de finitude pour la monodromie.
\newblock In {\em Discrete groups in geometry and analysis}, pages 1--19.
  Springer, 1987.

\bibitem[DGT16]{dobrovolska2016moduli}
Galyna Dobrovolska, Victor Ginzburg, and Roman Travkin.
\newblock Moduli spaces, indecomposable objects and potentials over a finite
  field.
\newblock {\em arXiv preprint arXiv:1612.01733}, 2016.

\bibitem[DHM23]{davison2023bps}
Ben Davison, Lucien Hennecart, and Sebastian~Schlegel Mejia.
\newblock {BPS} algebras and generalised {K}ac-{M}oody algebras from
  2-{C}alabi--{Y}au categories.
\newblock {\em arXiv preprint arXiv:2303.12592}, 2023.

\bibitem[DM20]{davison2020cohomological}
Ben Davison and Sven Meinhardt.
\newblock Cohomological {D}onaldson--{T}homas theory of a quiver with potential
  and quantum enveloping algebras.
\newblock {\em Inventiones Mathematicae}, 221(3):777--871, 2020.

\bibitem[Don87]{donaldson1987twisted}
Simon~K. Donaldson.
\newblock Twisted harmonic maps and the self-duality equations.
\newblock {\em Proceedings of the London Mathematical Society}, 3(1):127--131,
  1987.

\bibitem[DX03]{deng2003new}
Bangming Deng and Jie Xiao.
\newblock A new approach to {K}ac's theorem on representations of valued
  quivers.
\newblock {\em Mathematische Zeitschrift}, 245(1):183--199, 2003.

\bibitem[Gin06]{ginzburg2006calabi}
Victor Ginzburg.
\newblock Calabi-{Y}au algebras.
\newblock {\em arXiv preprint math/0612139}, 2006.

\bibitem[GM83]{goresky1983intersection}
Mark Goresky and Robert MacPherson.
\newblock Intersection homology {II}.
\newblock {\em Inventiones Mathematicae}, 71:77--129, 1983.

\bibitem[GP79]{gel1979model}
Izrail~Moiseevich Gel'fand and Vladimir~A Ponomarev.
\newblock Model algebras and representations of graphs.
\newblock {\em Functional Analysis and Its Applications}, 13:157--166, 1979.

\bibitem[Gre95]{green1995hall}
James~A. Green.
\newblock {H}all algebras, hereditary algebras and quantum groups.
\newblock {\em Inventiones Mathematicae}, 120(1):361--377, 1995.

\bibitem[Gro96]{grojnowski1996instantons}
Ian Grojnowski.
\newblock {Instantons and affine algebras I: the Hilbert scheme and vertex
  operators}.
\newblock {\em Mathematical Research Letters}, 3:275--291, 1996.

\bibitem[Gro17]{gross2017tensor}
Philipp Gross.
\newblock Tensor generators on schemes and stacks.
\newblock {\em Algebraic Geometry}, 4(4):501--522, 2017.

\bibitem[Hen21]{hennecart2021isotropic}
Lucien Hennecart.
\newblock Isotropic cuspidal functions in the {H}all algebra of a quiver.
\newblock {\em International Mathematics Research Notices},
  2021(15):11514--11564, 2021.

\bibitem[Hen23a]{Hennecart2023LCIMorphisms}
Lucien Hennecart.
\newblock Local complete intersection morphisms, May 2023.
\newblock Note available at \url{https://hennlu.github.io/lci-morphisms.pdf}.

\bibitem[Hen23b]{hennecart2023nonabelian}
Lucien Hennecart.
\newblock Nonabelian hodge isomorphisms for stacks and cohomological {H}all
  algebras.
\newblock {\em arXiv preprint arXiv:2307.09920}, 2023.

\bibitem[Hen24]{hennecart2022geometric}
Lucien Hennecart.
\newblock Geometric realisations of the unipotent enveloping algebra of a
  quiver.
\newblock {\em Advances in Mathematics}, 441:109536, 2024.

\bibitem[Hit87]{hitchin1987self}
Nigel~J. Hitchin.
\newblock {The self-duality equations on a Riemann surface}.
\newblock {\em Proceedings of the London Mathematical Society}, 3(1):59--126,
  1987.

\bibitem[HL10]{huybrechts2010geometry}
Daniel Huybrechts and Manfred Lehn.
\newblock {\em The geometry of moduli spaces of sheaves}.
\newblock Cambridge University Press, 2010.

\bibitem[HLRV13]{hausel2013positivity}
Tam{\'a}s Hausel, Emmanuel Letellier, and Fernando Rodriguez-Villegas.
\newblock Positivity for {K}ac polynomials and {DT}-invariants of quivers.
\newblock {\em Annals of Mathematics}, pages 1147--1168, 2013.

\bibitem[HM96]{harvey1996algebras}
Jeffrey~A. Harvey and Gregory Moore.
\newblock {Algebras, BPS states, and strings}.
\newblock {\em Nuclear Physics B}, 463(2-3):315--368, 1996.

\bibitem[HM98]{harvey1998algebras}
Jeffrey~A. Harvey and Gregory Moore.
\newblock {On the algebras of BPS states}.
\newblock {\em Communications in Mathematical Physics}, 197(3):489--519, 1998.

\bibitem[HMMS22]{hausel2022p}
Tam{\'a}s Hausel, Anton Mellit, Alexandre Minets, and Olivier Schiffmann.
\newblock {$ P= W $ via $ H_2$}.
\newblock {\em arXiv preprint arXiv:2209.05429}, 2022.

\bibitem[HRV08]{hausel2008mixed}
Tam{\'a}s Hausel and Fernando Rodriguez-Villegas.
\newblock {Mixed {H}odge polynomials of character varieties}.
\newblock {\em Inventiones Mathematicae}, 174(3):555--624, 2008.

\bibitem[Joy07]{joyce2007holomorphic}
Dominic Joyce.
\newblock {Holomorphic generating functions for invariants counting coherent
  sheaves on Calabi--Yau 3-folds}.
\newblock {\em Geometry \& Topology}, 11(2):667--725, 2007.

\bibitem[JS12]{joyce2011theory}
Dominic Joyce and Yinan Song.
\newblock {\em A theory of generalized Donaldson--Thomas invariants}, volume
  217.
\newblock American Mathematical Society, 2012.

\bibitem[Kac83]{kac1983root}
Victor~G. Kac.
\newblock Root systems, representations of quivers and invariant theory.
\newblock In {\em Invariant Theory}, pages 74--108. Springer, 1983.

\bibitem[Kac90]{kac1990infinite}
Victor~G. Kac.
\newblock {\em Infinite-dimensional Lie algebras}.
\newblock Cambridge university press, 1990.

\bibitem[Kas93]{kashiwara1993global}
Masaki Kashiwara.
\newblock Global crystal bases of quantum groups.
\newblock {\em Duke Mathematical Journal}, 69(2):455, 1993.

\bibitem[Kel08]{keller2008calabi}
Bernhard Keller.
\newblock Calabi-{Y}au triangulated categories.
\newblock {\em Trends in representation theory of algebras and related topics},
  pages 467--489, 2008.

\bibitem[Kel11]{keller2011deformed}
Bernhard Keller.
\newblock {Deformed Calabi--Yau completions. (With an appendix by Michel Van
  den Bergh)}.
\newblock {\em Journal f{\"u}r die Reine und Angewandte Mathematik (Crelle's
  journal)}, 2011(654), 2011.

\bibitem[Kin94]{king1994moduli}
Alastair~D. King.
\newblock Moduli of representations of finite dimensional algebras.
\newblock {\em The Quarterly Journal of Mathematics}, 45(4):515--530, 1994.

\bibitem[Kin22]{kinjo2022dimensional}
Tasuki Kinjo.
\newblock {Dimensional reduction in cohomological Donaldson--Thomas theory}.
\newblock {\em Compositio Mathematica}, 158(1):123--167, 2022.

\bibitem[KK24]{kinjo2021cohomological}
Tasuki Kinjo and Naoki Koseki.
\newblock Cohomological $\chi$-independence for {H}iggs bundles and
  {G}opakumar--{V}afa invariants.
\newblock {\em Journal of the European Mathematical Society}, 2024.

\bibitem[KM24]{kinjo2021global}
Tasuki Kinjo and Naruki Masuda.
\newblock Global critical chart for local {C}alabi-{Y}au threefolds.
\newblock {\em International Mathematics Research Notices}, (5):4062--4093,
  2024.

\bibitem[KPS24]{kinjo2024cohomological}
Tasuki Kinjo, Hyeonjun Park, and Pavel Safronov.
\newblock Cohomological {H}all algebras for 3-{C}alabi-{Y}au categories.
\newblock {\em arXiv preprint arXiv:2406.12838}, 2024.

\bibitem[KS08]{kontsevich2008stability}
Maxim Kontsevich and Yan Soibelman.
\newblock {Stability structures, motivic Donaldson--Thomas invariants and
  cluster transformations}.
\newblock {\em arXiv preprint arXiv:0811.2435}, 2008.

\bibitem[KS11]{kontsevich2010cohomological}
Maxim Kontsevich and Yan Soibelman.
\newblock {Cohomological Hall algebra, exponential Hodge structures and motivic
  Donaldson--Thomas invariants}.
\newblock {\em Communications in Number Theory and Physics}, 5(2):231--252,
  2011.

\bibitem[KS23]{kaplan2023multiplicative}
Daniel Kaplan and Travis Schedler.
\newblock Multiplicative preprojective algebras are 2-{C}alabi--{Y}au.
\newblock {\em Algebra \& Number Theory}, 17(4):831--883, 2023.

\bibitem[KV23]{kapranov2019cohomological}
Mikhail Kapranov and Eric Vasserot.
\newblock The cohomological {H}all algebra of a surface and factorization
  cohomology.
\newblock {\em Journal of the European Mathematical Society},
  25(11):4221--4289, 2023.

\bibitem[Lus91]{lusztig1991quivers}
George Lusztig.
\newblock Quivers, perverse sheaves, and quantized enveloping algebras.
\newblock {\em Journal of the American Mathematical Society}, 4(2):365--421,
  1991.

\bibitem[Lus93]{lusztig1993tight}
George Lusztig.
\newblock Tight monomials in quantized enveloping algebras.
\newblock In {\em Quantum deformations of algebras and their representations
  ({R}amat-{G}an, 1991/1992; {R}ehovot, 1991/1992)}, volume~7 of {\em Israel
  Math. Conf. Proc.}, pages 117--132. Bar-Ilan Univ., Ramat Gan, 1993.

\bibitem[Min20]{minets2020cohomological}
Alexandre Minets.
\newblock Cohomological {H}all algebras for {H}iggs torsion sheaves, moduli of
  triples and sheaves on surfaces.
\newblock {\em Selecta Mathematica}, 26(2):1--67, 2020.

\bibitem[Mis22]{mistry2022cohomological}
Vivek Mistry.
\newblock On the cohomological {H}all algebra of a character variety.
\newblock {\em arXiv preprint arXiv:2209.00680}, 2022.

\bibitem[MM24]{mauri2022hodge}
Mirko Mauri and Luca Migliorini.
\newblock Hodge-to-singular correspondence for reduced curves.
\newblock {\em Journal of the European Mathematical Society}, 2024.

\bibitem[MMP22]{mauri2022combinatorial}
Mirko Mauri, Luca Migliorini, and Roberto Pagaria.
\newblock {Combinatorial decomposition theorem for Hitchin systems via
  zonotopes}.
\newblock {\em arXiv preprint arXiv:2209.00621}, 2022.

\bibitem[MR19]{meinhardt2019donaldson}
Sven Meinhardt and Markus Reineke.
\newblock {D}onaldson--{T}homas invariants versus intersection cohomology of
  quiver moduli.
\newblock {\em Journal f{\"u}r die reine und angewandte Mathematik (Crelle's
  Journal)}, 2019(754):143--178, 2019.

\bibitem[MS24]{maulik2022p}
Davesh Maulik and Junliang Shen.
\newblock The {P}={W} conjecture for {GL}$_n$.
\newblock {\em Annals of Mathematics}, 200(2):529--556, 2024.

\bibitem[MSS11]{maxim2011symmetric}
Laurentiu Maxim, Morihiko Saito, and J{\"o}rg Sch{\"u}rmann.
\newblock {Symmetric products of mixed {H}odge modules}.
\newblock {\em Journal de Math{\'e}matiques Pures et Appliqu{\'e}es},
  96(5):462--483, 2011.

\bibitem[Nak94]{nakajima1994instantons}
Hiraku Nakajima.
\newblock {Instantons on ALE spaces, quiver varieties, and Kac--Moody
  algebras}.
\newblock {\em Duke Mathematical Journal}, 76(2):365--416, 1994.

\bibitem[Nak98]{nakajima1998quiver}
Hiraku Nakajima.
\newblock {Quiver varieties and Kac--Moody algebras}.
\newblock {\em Duke Mathematical Journal}, 91(3):515--560, 1998.

\bibitem[Ols15]{olsson2015borel}
Martin Olsson.
\newblock {B}orel--{M}oore homology, {R}iemann--{R}och transformations, and
  local terms.
\newblock {\em Advances in Mathematics}, 273:56--123, 2015.

\bibitem[PS23]{porta2022two}
Mauro Porta and Francesco Sala.
\newblock Two-dimensional categorified hall algebras.
\newblock {\em Journal of the European Mathematical Society}, 25(3):1113--1205,
  2023.

\bibitem[Rin90]{ringel1990hall}
Claus~Michael Ringel.
\newblock {H}all algebras and quantum groups.
\newblock {\em Inventiones Mathematicae}, 101(1):583--591, 1990.

\bibitem[Rin92]{ringel1992hall}
Claus~Michael Ringel.
\newblock {H}all algebras revisited.
\newblock {\em Available on author's website}, 1992.
\newblock Available at:
  \url{https://www.math.uni-bielefeld.de/~ringel/opus/hall-rev.pdf}, Accessed:
  October 21, 2024.

\bibitem[Sai89]{saito1989introduction}
Morihiko Saito.
\newblock {Introduction to mixed Hodge modules}.
\newblock {\em Ast{\'e}risque}, 179(180):145--162, 1989.

\bibitem[Sai90]{saito1990mixed}
Morihiko Saito.
\newblock {Mixed Hodge modules}.
\newblock {\em Publications of the Research Institute for Mathematical
  Sciences}, 26(2):221--333, 1990.

\bibitem[Sch85]{schofield1985representations}
Aidan~Harry Schofield.
\newblock {\em Representations of rings over skew fields}, volume~92.
\newblock Cambridge University Press, 1985.

\bibitem[Sch98]{schaub1998courbes}
Daniel Schaub.
\newblock {Courbes spectrales et compactifications de Jacobiennes}.
\newblock {\em Mathematische Zeitschrift}, 2(227):295--312, 1998.

\bibitem[Sch12]{schiffmann2012hall}
Olivier Schiffmann.
\newblock Lectures on {H}all algebras.
\newblock In {\em Geometric methods in representation theory. {II}}, volume
  24-II of {\em S\'{e}min. Congr.}, pages 1--141. Soc. Math. France, Paris,
  2012.

\bibitem[Sim92]{simpson1992higgs}
Carlos Simpson.
\newblock {H}iggs bundles and local systems.
\newblock {\em Publications Math{\'e}matiques de l'IH{\'E}S}, 75:5--95, 1992.

\bibitem[Sim94a]{simpson1994moduliI}
Carlos Simpson.
\newblock {Moduli of representations of the fundamental group of a smooth
  projective variety I}.
\newblock {\em Publications Math{\'e}matiques de l'IH{\'E}S}, 79:47--129, 1994.

\bibitem[Sim94b]{simpson1994moduli}
Carlos Simpson.
\newblock Moduli of representations of the fundamental group of a smooth
  projective variety. {II}.
\newblock {\em Publications Math{\'e}matiques de l'IH{\'E}S}, (80):5--79
  (1995), 1994.

\bibitem[Sim96]{simpson1996hodge}
Carlos Simpson.
\newblock The {H}odge filtration on nonabelian cohomology.
\newblock {\em arXiv preprint alg-geom/9604005}, 1996.

\bibitem[SS20]{sala2020cohomological}
Francesco Sala and Olivier Schiffmann.
\newblock Cohomological {H}all algebra of {H}iggs sheaves on a curve.
\newblock {\em Algebr. Geom.}, 7(3):346--376, 2020.

\bibitem[SV13]{schiffmann2013cherednik}
Olivier Schiffmann and Eric Vasserot.
\newblock {Cherednik algebras, W-algebras and the equivariant cohomology of the
  moduli space of instantons on $\mathbf{A}^2$}.
\newblock {\em Publications Math{\'e}matiques de l'IH{\'E}S}, 118(1):213--342,
  2013.

\bibitem[SV17]{schiffmann2017cohomological}
Olivier Schiffmann and Eric Vasserot.
\newblock On cohomological hall algebras of quivers: Yangians.
\newblock {\em arXiv preprint arXiv:1705.07491}, 2017.

\bibitem[SV20]{schiffmann2020cohomological}
Olivier Schiffmann and Eric Vasserot.
\newblock On cohomological {H}all algebras of quivers: generators.
\newblock {\em Journal f{\"u}r die reine und angewandte Mathematik (Crelle's
  Journal)}, 2020(760):59--132, 2020.

\bibitem[SVDB01]{sevenhant2001relation}
Bert Sevenhant and Michel Van Den~Bergh.
\newblock A relation between a conjecture of {K}ac and the structure of the
  {H}all algebra.
\newblock {\em Journal of Pure and Applied Algebra}, 160(2-3):319--332, 2001.

\bibitem[Sze08]{szendroi2008thomas}
Bal{\'a}zs Szendr{\H{o}}i.
\newblock Non-commutative {D}onaldson--{T}homas invariants and the conifold.
\newblock {\em Geometry \& Topology}, 12(2):1171--1202, 2008.

\bibitem[Tho00]{thomas2000holomorphic}
Richard Thomas.
\newblock {A holomorphic Casson invariant for Calabi--Yau 3-folds, and bundles
  on $ K3 $ fibrations}.
\newblock {\em Journal of Differential Geometry}, 54(2):367--438, 2000.

\bibitem[Tot04]{totaro2004resolution}
Burt Totaro.
\newblock The resolution property for schemes and stacks.
\newblock {\em Journal f{\"u}r die reine und angewandte Mathematik (Crelle's
  Journal)}, (577):1--22, 2004.

\bibitem[Tub24]{tubach2024mixed}
Swann Tubach.
\newblock Mixed {H}odge modules on stacks.
\newblock {\em arXiv preprint arXiv:2407.02256}, 2024.

\bibitem[TV07]{toen2007moduli}
Bertrand To{\"e}n and Michel Vaqui{\'e}.
\newblock Moduli of objects in dg-categories.
\newblock In {\em Annales Scientifiques de l'Ecole Normale Sup{\'e}rieure},
  volume~40, pages 387--444, 2007.

\bibitem[Xia97]{xiao1997drinfeld}
Jie Xiao.
\newblock Drinfeld double and {R}ingel--{G}reen theory of {H}all algebras.
\newblock {\em Journal of Algebra}, 190(1):100--144, 1997.

\bibitem[YZ18]{yang2018cohomological}
Yaping Yang and Gufang Zhao.
\newblock {The cohomological Hall algebra of a preprojective algebra}.
\newblock {\em Proceedings of the London Mathematical Society},
  116(5):1029--1074, 2018.

\end{thebibliography}
\end{document}